\documentclass[12pt,twoside]{preprint}
\usepackage{graphicx}
\usepackage{amssymb}
\usepackage{amsmath}
\usepackage{hyperref}
\usepackage{breakurl}
\usepackage{mhenvs}
\usepackage{mhequ}
\usepackage{mhsymb}
\usepackage{booktabs}
\usepackage{tikz}
\usepackage{mathrsfs}
\usepackage{times}
\usepackage{microtype}
\usepackage{comment}
\usepackage{wasysym}
\usepackage{centernot}
\usepackage{leftidx}
\usepackage{accents}
\usepackage{arydshln}
\usepackage{footnote}
\makesavenoteenv{tabular}

\usepackage{atbegshi}
\newcommand\showtimer{%
  \message{^^Jtimer: \the\numexpr\the\pdfelapsedtime*1000/65536\relax}%
  \pdfresettimer}
\AtBeginDocument{\showtimer}
\AtBeginShipout {\showtimer}

\colorlet{symbols}{blue!90!black}
\colorlet{testcolor}{green!60!black}

\def\symbol#1{\textcolor{symbols}{#1}}
\def\1{\mathbf{\symbol{1}}}

\def\enorm#1{| #1|_{\eps}}

\newcommand{\cev}[1]{\reflectbox{\ensuremath{\vec{\reflectbox{\ensuremath{#1}}}}}}

\def\bB{\boldsymbol{B}}
\def\bPsi{\boldsymbol{\Psi}}

\usetikzlibrary{decorations}
\usetikzlibrary{decorations.markings}
\usetikzlibrary{decorations.pathmorphing}

\def\drawx{\draw[-,solid] (-3pt,-3pt) -- (3pt,3pt);\draw[-,solid] (-3pt,3pt) -- (3pt,-3pt);}

\tikzset{
	root/.style={circle,fill=testcolor,inner sep=0pt, minimum size=2mm},
	dot/.style={circle,fill=black,draw=black,inner sep=0pt,minimum size=0.5mm},
	empt/.style={inner sep=0pt,minimum size=0.5mm},
	int/.style={circle,fill=black,draw=black,inner sep=0pt,minimum size=0.7mm},
	circ/.style={circle,draw=black,inner sep=0pt, minimum size=1mm},
	var/.style={circle,fill=black!10,draw=black,inner sep=0pt, minimum size=3.5mm},
	vab/.style={rectangle,fill=black!10,draw=black,inner sep=0pt, minimum size=3mm},
	dotred/.style={circle,fill=black!50,inner sep=0pt, minimum size=2mm},
	generic/.style={semithick,shorten >=1pt,shorten <=1pt},
	oddfunc/.style={generic, dotted},
	dist/.style={ultra thick,draw=testcolor,shorten >=1pt,shorten <=1pt},
	testfcn/.style={ultra thick,testcolor,shorten >=1pt,shorten <=1pt,<-},
	testfunction/.style={ultra thick,testcolor,shorten >=1pt,shorten <=1pt},
	testfcnx/.style={ultra thick,testcolor,shorten >=1pt,shorten <=1pt,<-,
		postaction={decorate,decoration={markings,mark=at position 0.6 with {\drawx}}}},
	kprime/.style={semithick,shorten >=1pt,shorten <=1pt,densely dashed,->},
	kprimex/.style={semithick,shorten >=1pt,shorten <=1pt,densely dashed,->,
		postaction={decorate,decoration={markings,mark=at position 0.4 with {\drawx}}}},
	kernel/.style={semithick,shorten >=1pt,shorten <=1pt,->,draw=black},
	kernelblue/.style={very thick,shorten >=1pt,shorten <=1pt,->,draw=black,color=blue},
	multx/.style={shorten >=1pt,shorten <=1pt,
		postaction={decorate,decoration={markings,mark=at position 0.5 with {\drawx}}}},
	kernelx/.style={semithick,shorten >=1pt,shorten <=1pt,->,
		postaction={decorate,decoration={markings,mark=at position 0.4 with {\drawx}}}},
	kernel1/.style={->,semithick,shorten >=1pt,shorten <=1pt, postaction={decorate,decoration={markings,mark=at position 0.45 with {\draw[-] (0,-0.1) -- (0,0.1);}}}},
	kernel2/.style={->,semithick,shorten >=1pt,shorten <=1pt,postaction={decorate,decoration={markings,mark=at position 0.45 with {\draw[-] (0.05,-0.1) -- (0.05,0.1);\draw[-] (-0.05,-0.1) -- (-0.05,0.1);}}}},
	kernelBig/.style={semithick,shorten >=1pt,shorten <=1pt,decorate, decoration={zigzag,amplitude=1.5pt,segment length = 3pt,pre length=2pt,post length=2pt}},
	rho/.style={dotted,semithick,shorten >=1pt,shorten <=1pt},
	renorm/.style={shape=circle,fill=white,inner sep=1pt},
	labl/.style={shape=rectangle,fill=white,inner sep=1pt},
cumu2n/.style={inner sep=3pt},
cumu2/.style={draw=red!80,fill=red!40},
cumu3/.style={regular polygon, regular polygon sides=3,draw=red!80,rounded corners=3pt,fill=red!40,minimum size=5mm},
cumu4/.style={regular polygon, regular polygon sides=4,draw=red!80,rounded corners=3pt,fill=red!40,minimum size=7mm},
cumu5/.style={regular polygon, regular polygon sides=5,draw=red!80,rounded corners=3pt,fill=red!40,minimum size=7mm},
	xi/.style={rectangle,fill=symbols!10,draw=symbols,inner sep=0.4pt,minimum size=2.3mm},
	xix/.style={crosscircle,fill=symbols!10,draw=symbols,inner sep=0pt,minimum size=1.2mm},
	xib/.style={circle,fill=symbols!10,draw=symbols,inner sep=0pt,minimum size=1.6mm},
	xibx/.style={crosscircle,fill=symbols!10,draw=symbols,inner sep=0pt,minimum size=1.6mm},
	not/.style={circle,fill=symbols,draw=symbols,inner sep=0pt,minimum size=0.5mm},
	zeta/.style={circle,fill=symbols!10,draw=symbols,inner sep=0pt,minimum size=2.5mm},
	cir/.style={circle, draw=black, inner sep=0pt,minimum size=3mm},
	>=stealth,
	}
\makeatletter
\def\DeclareSymbol#1#2#3{\expandafter\gdef\csname MH@symb@#1\endcsname{\tikz[baseline=#2,scale=0.23,draw=symbols]{#3}}\expandafter\gdef\csname MH@symb@#1s\endcsname{\scalebox{0.7}{\tikz[baseline=#2,scale=0.23,draw=symbols]{#3}}}}
\def\<#1>{\csname MH@symb@#1\endcsname}
\makeatother
\DeclareSymbol{Psi1dPsi2}{-2}{
	\draw (-1,0.8) node[zeta] {\tiny $1$} -- (0,-1) node[not] {}
		 -- node[right,near start] {\tiny $\!\!j$} (1,0.8) node[zeta] {\tiny $2$}; 
	}
\DeclareSymbol{Psi2dPsi1}{-2}{
	\draw (-1,0.8) node[zeta] {\tiny $2$} -- (0,-1) node[not] {}
		 -- node[right,near start] {\tiny $j$}  (1,0.8) node[zeta] {\tiny $1$}; 
	}
\DeclareSymbol{PsikdPsil}{-2}{
	\draw (-1,0.8) node[zeta] {\tiny $k$} -- (0,-1) node[not] {} 
		-- node[right,near start] {\tiny $\!\!j$}  (1,0.8) node[zeta] {\tiny $\ell$}; 
	}
\DeclareSymbol{PsildPsik}{-2}{
	\draw (-1,0.8) node[zeta] {\tiny $\ell$} -- (0,-1) node[not] {} 
		-- node[right,near start] {\tiny $\!\!j$}  (1,0.8) node[zeta] {\tiny $k$}; 
	}
\DeclareSymbol{BdPsi}{-2}{
	\draw (-1,0.8) node[xi] {\tiny $j$} -- (0,-1) node[not] {} 
		-- node[right,near start] {\tiny $\!\!j$}  (1,0.8) node[zeta] {\tiny $k$}; 
	}
\DeclareSymbol{Bd-pm-Psi}{-2}{
	\draw (-1,0.8) node[xi] {\tiny $j$} -- (0,-1) node[not] {} 
		-- node[right,near start] {\tiny $\!\!\pm j$}  (1,0.8) node[zeta] {\tiny $k$}; 
	}
\DeclareSymbol{BdPsi--1}{-2}{
	\draw (-1,0.8) node[xi] {\tiny $k$} -- (0,-1) node[not] {} 
		-- node[right,near start] {\tiny $\!\!k$}  (1,0.8) node[zeta] {\tiny $\ell$}; 
	}
\DeclareSymbol{BdPsi--2}{-2}{
	\draw (-1,0.8) node[xi] {\tiny $k$} -- (0,-1) node[not] {} 
		-- node[right,near start] {\tiny $\!\! \- k$}  (1,0.8) node[zeta] {\tiny $\ell$}; 
	}
\DeclareSymbol{dBPsi}{-2}{
	\draw (-1,0.8) node[xi] {\tiny $k$}  
	-- node[left,near end] {\tiny $\- k \!\!$} 
	 (0,-1) node[not] {} 
		--  (1,0.8) node[zeta] {\tiny $\ell$}; 
	}

\DeclareSymbol{Ij-zeta-k}{-2}{
	\draw (0,-1) node[not] {} --  node[right] {\tiny $\!\! j$} (0,1) node[zeta] {\tiny $k$}; 
	}
\DeclareSymbol{I-pm-j-zeta-k}{-2}{
	\draw (0,-1) node[not] {} --  node[right] {\tiny $\!\! \pm j$} (0,1) node[zeta] {\tiny $k$}; 
	}
\DeclareSymbol{X-Ij-zeta-k}{-2}{
	\node at (-1.2,0.3) {$\!\!\!X$};
	\draw (-0.5,-1) node[not] {} --  node[right] {\tiny $\!\! j$} (-0.5,1) node[zeta] {\tiny $k$}; 
	}
\DeclareSymbol{X-Ik-zeta-l}{-2}{
	\node at (-1.2,0.3) {$\!\!\!X$};
	\draw (-0.5,-1) node[not] {} --  node[right] {\tiny $\!\! k$} (-0.5,1) node[zeta] {\tiny $\ell$}; 
	}
\DeclareSymbol{X-I-k-zeta-l}{-2}{
	\node at (-1.2,0.3) {$\!\!\!X$};
	\draw (-0.5,-1) node[not] {} --  node[right] {\tiny $\!\! \-k$} (-0.5,1) node[zeta] {\tiny $\ell$}; 
	}
\DeclareSymbol{Ii-zeta-j}{-2}{
	\draw (0,-1) node[not] {} --  node[right] {\tiny $\!\! i$} (0,1) node[zeta] {\tiny $j$}; 
	}
\DeclareSymbol{Ik-zeta-l}{-2}{
	\draw (0,-1) node[not] {} --  node[right] {\tiny $\!\! k$} (0,1) node[zeta] {\tiny $\ell$}; 
	}
\DeclareSymbol{I-k-zeta-l}{-2}{
	\draw (0,-1) node[not] {} --  node[right] {\tiny $\!\! \-k$} (0,1) node[zeta] {\tiny $\ell$}; 
	}
\DeclareSymbol{I-pm-k-zeta-l}{-2}{
	\draw (0,-1) node[not] {} --  node[right] {\tiny $\!\! \pm k$} (0,1) node[zeta] {\tiny $\ell$}; 
	}
\DeclareSymbol{Ij-zeta-l}{-2}{
	\draw (0,-1) node[not] {} --  node[right] {\tiny $\!\! j$} (0,1) node[zeta] {\tiny $\ell$}; 
	}
\DeclareSymbol{I-k-xi-k}{-2}{
	\draw (0,-1) node[not] {} --  node[right] {\tiny $\!\! \-k$} (0,1) node[xi] {\tiny $k$}; 
	}
\DeclareSymbol{X-I-k-xi-k}{-2}{
	\node at (-1.2,0.3) {$\!\!\!X$};
	\draw (-0.5,-1) node[not] {} --  node[right] {\tiny $\!\! \-k$} (-0.5,1) node[xi] {\tiny $k$}; 
	}
	
\DeclareSymbol{BjPsikPsik}{-1}{
	\draw (-1.3,0.8) node[zeta] {\tiny $k$} -- (0,-1) node[not] {} -- (1.3,0.8) node[zeta] {\tiny $k$}; 
	\draw (0,-1) node[not] {} -- (0,1) node[xi] {\tiny $j$};
	}
\DeclareSymbol{BBPsi}{-1}{
	\draw (-1.3,0.8) node[xi] {} -- (0,-1) node[not] {} -- (1.3,0.8) node[zeta] {}; 
	\draw (0,-1) node[not] {} -- (0,1) node[xi] {};
	}
\DeclareSymbol{BjBjPsik}{-1}{
	\draw (-1.3,0.8) node[xi] {\tiny $j$} -- (0,-1) node[not] {} -- (1.3,0.8) node[zeta] {\tiny $k$}; 
	\draw (0,-1) node[not] {} -- (0,1) node[xi] {\tiny $j$};
	}
\DeclareSymbol{BkBkPsij}{-1}{
	\draw (-1.3,0.8) node[xi] {\tiny $k$} -- (0,-1) node[not] {} -- (1.3,0.8) node[zeta] {\tiny $j$}; 
	\draw (0,-1) node[not] {} -- (0,1) node[xi] {\tiny $k$};
	}
\DeclareSymbol{dPsi-IBdPsi}{-2}{
	\draw (-1,1.2) node[xi] {\tiny $j$} -- (0,-0.3) node[not] {} 
			-- node[right,near start] {\tiny $\!\!j$}  (1,1.2) node[zeta] {\tiny $\ell$}; 
	\draw (0,-0.3) node[not] {} -- (1,-1.7) node[not] {} 
			-- node[right,near start]  {\tiny $\!\!k$}  (2,-0.3) node[zeta] {\tiny $\ell$};
	}
\DeclareSymbol{dPsi-IBd-pm-Psi}{-2}{
	\draw (-1,1.2) node[xi] {\tiny $j$} -- (0,-0.3) node[not] {} 
			-- node[right,near start] {\tiny $\!\!\!\!\pm j$}  (1,1.2) node[zeta] {\tiny $\ell$}; 
	\draw (0,-0.3) node[not] {} -- (1,-1.7) node[not] {} 
			-- node[right,near start]  {\tiny $\!\!k$}  (2,-0.3) node[zeta] {\tiny $\ell$};
	}
\DeclareSymbol{dj-Psi-IBdkPsi}{-2}{
	\draw (-1,1.2) node[xi] {\tiny $k$} -- (0,-0.3) node[not] {} 
			-- node[right,near start] {\tiny $\!\!k$}  (1,1.2) node[zeta] {\tiny $\ell$}; 
	\draw (0,-0.3) node[not] {} -- (1,-1.7) node[not] {} 
			-- node[right,near start]  {\tiny $\!\!j$}  (2,-0.3) node[zeta] {\tiny $\ell$};
	}
\DeclareSymbol{dj-Psi-IBd-kPsi}{-2}{
	\draw (-1,1.2) node[xi] {\tiny $k$} -- (0,-0.3) node[not] {} 
			-- node[right,near start] {\tiny $\!\! \- k$}  (1,1.2) node[zeta] {\tiny $\ell$}; 
	\draw (0,-0.3) node[not] {} -- (1,-1.7) node[not] {} 
			-- node[right,near start]  {\tiny $\!\!j$}  (2,-0.3) node[zeta] {\tiny $\ell$};
	}
\DeclareSymbol{dPsi-IBdPsi-12}{-2}{
	\draw (-1,1.2) node[xi] {\tiny $j$} -- (0,-0.3) node[not] {} -- node[sloped] {\tiny $j$}  (1,1.2) node[zeta] {\tiny $1$}; 
	\draw (0,-0.3) node[not] {} -- (1,-1.7) node[not] {} -- node[sloped]  {\tiny $k$}  (2,-0.3) node[zeta] {\tiny $2$};
	}
\DeclareSymbol{dPsi-IBdPsi-21}{-2}{
	\draw (-1,1.2) node[xi] {\tiny $j$} -- (0,-0.3) node[not] {} -- node[sloped] {\tiny $j$}  (1,1.2) node[zeta] {\tiny $2$}; 
	\draw (0,-0.3) node[not] {} -- (1,-1.7) node[not] {} --   node[sloped]  {\tiny $k$}   (2,-0.3) node[zeta] {\tiny $1$};
	}
\DeclareSymbol{Psi-dIBdPsi}{-2}{
	\draw (-1,1.2) node[xi] {\tiny $j$} -- (0,-0.3) node[not] {} 
			-- node[right,near start] {\tiny $\!\!j$}  (1,1.2) node[zeta] {\tiny $\ell$}; 
	\draw (0,-0.3) node[not] {} --  node[left,near end] {\tiny $k\!\!$}   (1,-1.7) node[not] {} 
			-- (2,-0.3) node[zeta] {\tiny $\ell$};
	}
\DeclareSymbol{Psi-dIBd-pm-Psi}{-2}{
	\draw (-1,1.2) node[xi] {\tiny $j$} -- (0,-0.3) node[not] {} 
			-- node[right,near start] {\tiny $\!\!\!\!\pm j$}  (1,1.2) node[zeta] {\tiny $\ell$}; 
	\draw (0,-0.3) node[not] {} --  node[left,near end] {\tiny $k\!\!$}   (1,-1.7) node[not] {} 
			-- (2,-0.3) node[zeta] {\tiny $\ell$};
	}
\DeclareSymbol{Psi-dIj-BkdPsi}{-2}{
	\draw (-1,1.2) node[xi] {\tiny $k$} -- (0,-0.3) node[not] {} 
			-- node[right,near start] {\tiny $\!\!k$}  (1,1.2) node[zeta] {\tiny $\ell$}; 
	\draw (0,-0.3) node[not] {} --  node[left,near end] {\tiny $j\!\!$}   (1,-1.7) node[not] {} 
			-- (2,-0.3) node[zeta] {\tiny $\ell$};
	}
\DeclareSymbol{Psi-dIj-Bkd-Psi}{-2}{
	\draw (-1,1.2) node[xi] {\tiny $k$} -- (0,-0.3) node[not] {} 
			-- node[right,near start] {\tiny $\!\! \- k$}  (1,1.2) node[zeta] {\tiny $\ell$}; 
	\draw (0,-0.3) node[not] {} --  node[left,near end] {\tiny $j\!\!$}   (1,-1.7) node[not] {} 
			-- (2,-0.3) node[zeta] {\tiny $\ell$};
	}
\DeclareSymbol{Psi-dIBdPsi-12}{-2}{
	\draw (-1,1.2) node[xi] {\tiny $j$} -- (0,-0.3) node[not] {} -- node[sloped] {\tiny $j$}  (1,1.2) node[zeta] {\tiny $1$}; 
	\draw (0,-0.3) node[not] {} --  node[left] {\tiny $k\!\!\!$}   (1,-1.7) node[not] {} -- (2,-0.3) node[zeta] {\tiny $2$};
	}
\DeclareSymbol{Psi-dIBdPsi-21}{-2}{
	\draw (-1,1.2) node[xi] {\tiny $j$} -- (0,-0.3) node[not] {} -- node[sloped] {\tiny $j$}  (1,1.2) node[zeta] {\tiny $2$}; 
	\draw (0,-0.3) node[not] {} --  node[left] {\tiny $k\!\!\!$}   (1,-1.7) node[not] {} -- (2,-0.3) node[zeta] {\tiny $1$};
	}
\DeclareSymbol{dPsi-IPsidPsi}{-2}{
	\draw (-1,1.2) node[zeta] {\tiny $i$} -- (0,-0.3) node[not] {} 
			-- node[right,near start] {\tiny $\!\!j$}   (1,1.2) node[zeta] {\tiny $\ell$}; 
	\draw (0,-0.3) node[not] {} -- (1,-1.7) node[not] {} 
			-- node[right,near start] {\tiny $\!\!j$}   (2,-0.3) node[zeta] {\tiny $k$};
	}
 \DeclareSymbol{d-pm-Psi-IPsidPsi}{-2}{
	\draw (-1,1.2) node[zeta] {\tiny $i$} -- (0,-0.3) node[not] {} 
			-- node[right,near start] {\tiny $ \!\! j$}   (1,1.2) node[zeta] {\tiny $\ell$}; 
	\draw (0,-0.3) node[not] {} -- (1,-1.7) node[not] {} 
			-- node[right,near start] {\tiny $\!\!\! \pm j$}   (2,-0.3) node[zeta] {\tiny $k$};
	}
\DeclareSymbol{dkPsil-IPsijdkPsil}{-2}{
	\draw (-1,1.2) node[zeta] {\tiny $j$} -- (0,-0.3) node[not] {} 
			-- node[right,near start] {\tiny $\!\!k$}   (1,1.2) node[zeta] {\tiny $\ell$}; 
	\draw (0,-0.3) node[not] {} -- (1,-1.7) node[not] {} 
			-- node[right,near start] {\tiny $\!\!k$}   (2,-0.3) node[zeta] {\tiny $\ell$};
	}
\DeclareSymbol{d-kPsil-IPsijdkPsil}{-2}{
	\draw (-1,1.2) node[zeta] {\tiny $j$} -- (0,-0.3) node[not] {} 
			-- node[right,near start] {\tiny $\!\!k$}   (1,1.2) node[zeta] {\tiny $\ell$}; 
	\draw (0,-0.3) node[not] {} -- (1,-1.7) node[not] {} 
			-- node[right,near start] {\tiny $\!\!\-k$}   (2,-0.3) node[zeta] {\tiny $\ell$};
	}
\DeclareSymbol{dkPsil-IPsildkPsij}{-2}{
	\draw (-1,1.2) node[zeta] {\tiny $\ell$} -- (0,-0.3) node[not] {} 
			-- node[right,near start] {\tiny $\!\!k$}   (1,1.2) node[zeta] {\tiny $j$}; 
	\draw (0,-0.3) node[not] {} -- (1,-1.7) node[not] {} 
			-- node[right,near start] {\tiny $\!\!k$}   (2,-0.3) node[zeta] {\tiny $\ell$};
	}
\DeclareSymbol{d-kPsil-IPsildkPsij}{-2}{
	\draw (-1,1.2) node[zeta] {\tiny $\ell$} -- (0,-0.3) node[not] {} 
			-- node[right,near start] {\tiny $\!\!k$}   (1,1.2) node[zeta] {\tiny $j$}; 
	\draw (0,-0.3) node[not] {} -- (1,-1.7) node[not] {} 
			-- node[right,near start] {\tiny $\!\!\-k$}   (2,-0.3) node[zeta] {\tiny $\ell$};
	}
\DeclareSymbol{dPsi-IPsidPsi-12}{-2}{
	\draw (-1,1.2) node[zeta] {\tiny $1$} -- (0,-0.3) node[not] {} -- node[sloped] {\tiny $j$}   (1,1.2) node[zeta] {\tiny $2$}; 
	\draw (0,-0.3) node[not] {} -- (1,-1.7) node[not] {} -- node[sloped] {\tiny $j$}   (2,-0.3) node[zeta] {\tiny $k$};
	}
\DeclareSymbol{dPsi-IPsidPsi-21}{-2}{
	\draw (-1,1.2) node[zeta] {\tiny $2$} -- (0,-0.3) node[not] {} -- node[sloped] {\tiny $j$}   (1,1.2) node[zeta] {\tiny $1$}; 
	\draw (0,-0.3) node[not] {} -- (1,-1.7) node[not] {} -- node[sloped] {\tiny $j$}   (2,-0.3) node[zeta] {\tiny $k$};
	}

\DeclareSymbol{Ped-kI(PjdkPe)}{-2}{
	\draw (-1,1.2) node[zeta] {\tiny $j$} -- (0,-0.3) node[not] {} 
			-- node[right,near start] {\tiny $\!\!k$}   (1,1.2) node[zeta] {\tiny $\ell$}; 
	\draw (0,-0.3) node[not] {} -- 
	node[left,near end] {\tiny $\-k$} 
	(1,-1.7) node[not] {} 
			--   (2,-0.3) node[zeta] {\tiny $\ell$};
	}
\DeclareSymbol{Ped-kI(PedkPj)}{-2}{
	\draw (-1,1.2) node[zeta] {\tiny $\ell$} -- (0,-0.3) node[not] {} 
			-- node[right,near start] {\tiny $\!\!k$}   (1,1.2) node[zeta] {\tiny $j$}; 
	\draw (0,-0.3) node[not] {} --
	node[left,near end] {\tiny $\-k$}
	 (1,-1.7) node[not] {} 
			--    (2,-0.3) node[zeta] {\tiny $\ell$};
	}
		
\DeclareSymbol{B-dIBdPsi}{-1.8}{
	\draw (-1,1.2) node[xi] {\tiny $\ell$} -- (0,-0.3) node[not] {} 
			-- node[right,near start] {\tiny $\!\!\ell$} (1,1.2) node[zeta] {\tiny $j$}; 
	\draw (0,-0.3) node[not] {}  -- node[left,near end] {\tiny $k \!\!$} (1,-1.7) node[not] {} 
			-- (2,-0.3) node[xi] {\tiny $k$};
	}
\DeclareSymbol{B-d-IBdPsi}{-1.8}{
	\draw (-1,1.2) node[xi] {\tiny $\ell$} -- (0,-0.3) node[not] {} 
			-- node[right,near start] {\tiny $\!\!\ell$} (1,1.2) node[zeta] {\tiny $j$}; 
	\draw (0,-0.3) node[not] {}  -- node[left,near end] {\tiny $\-k \!\!$} (1,-1.7) node[not] {} 
			-- (2,-0.3) node[xi] {\tiny $k$};
	}
\DeclareSymbol{B-dIBd-Psi}{-1.8}{
	\draw (-1,1.2) node[xi] {\tiny $\ell$} -- (0,-0.3) node[not] {} 
			-- node[right,near start] {\tiny $\!\!\-\ell$} (1,1.2) node[zeta] {\tiny $j$}; 
	\draw (0,-0.3) node[not] {}  -- node[left,near end] {\tiny $k \!\!$} (1,-1.7) node[not] {} 
			-- (2,-0.3) node[xi] {\tiny $k$};
	}
\DeclareSymbol{B-d-IBd-Psi}{-1.8}{
	\draw (-1,1.2) node[xi] {\tiny $\ell$} -- (0,-0.3) node[not] {} 
			-- node[right,near start] {\tiny $\!\!\-\ell$} (1,1.2) node[zeta] {\tiny $j$}; 
	\draw (0,-0.3) node[not] {}  -- node[left,near end] {\tiny $\-k \!\!$} (1,-1.7) node[not] {} 
			-- (2,-0.3) node[xi] {\tiny $k$};
	}
\DeclareSymbol{B-d-sign-IBd-sign-Psi}{-1.8}{
	\draw (-1,1.2) node[xi] {\tiny $\ell$} -- (0,-0.3) node[not] {} 
			-- node[right,near start] {\tiny $\!\!\!\si \ell$} (1,1.2) node[zeta] {\tiny $j$}; 
	\draw (0,-0.3) node[not] {}  -- node[left,near end] {\tiny $\bar\si k \!\!$} (1,-1.7) node[not] {} 
			-- (2,-0.3) node[xi] {\tiny $k$};
	}
\DeclareSymbol{B-d-pm-IBd-pm-Psi}{-1.8}{
	\draw (-1,1.2) node[xi] {\tiny $\ell$} -- (0,-0.3) node[not] {} 
			-- node[right,near start] {\tiny $\!\!\!\!\pm\ell$} (1,1.2) node[zeta] {\tiny $j$}; 
	\draw (0,-0.3) node[not] {}  -- node[left,near end] {\tiny $\pm k \!\!$} (1,-1.7) node[not] {} 
			-- (2,-0.3) node[xi] {\tiny $k$};
	}
\DeclareSymbol{djPeI(d-kBkPe)}{-2}{
	\draw (-1,1.2) node[xi] {\tiny $k$}   
		-- node[left,near end] {\tiny $\- k\!\! $}
		(0,-0.3) node[not] {} 
			--  (1,1.2) node[zeta] {\tiny $\ell$}; 
	\draw (0,-0.3) node[not] {} --   (1,-1.7) node[not] {} 
			-- node[right,near start] {\tiny $j\!\!$}  
			(2,-0.3) node[zeta] {\tiny $\ell$};
	}
\DeclareSymbol{PedjI(d-kBkPe)}{-2}{
	\draw (-1,1.2) node[xi] {\tiny $k$}  
		-- node[left,near end] {\tiny $\- k\!\! $} 
		(0,-0.3) node[not] {} 
			--  (1,1.2) node[zeta] {\tiny $\ell$}; 
	\draw (0,-0.3) node[not] {} --  node[left,near end] {\tiny $j\!\!$}   (1,-1.7) node[not] {} 
			-- (2,-0.3) node[zeta] {\tiny $\ell$};
	}
\DeclareSymbol{Ped-kdjI(BkPe)}{-2}{
	\draw (-1,1.2) node[xi] {\tiny $k$}  
		-- 
		(0,-0.3) node[not] {} 
			--  (1,1.2) node[zeta] {\tiny $\ell$}; 
	\draw (0,-0.3) node[not] {} --  node[left,near end] {\tiny $-k,j\!\!$}   (1,-1.7) node[not] {} 
			-- (2,-0.3) node[zeta] {\tiny $\ell$};
	}
\DeclareSymbol{djPed-kI(BkPe)}{-2}{
	\draw (-1,1.2) node[xi] {\tiny $k$}  
		-- 
		(0,-0.3) node[not] {} 
			--  (1,1.2) node[zeta] {\tiny $\ell$}; 
	\draw (0,-0.3) node[not] {} --  node[left,near end] {\tiny $-k\!\!$}   (1,-1.7) node[not] {} 
			--  node[right,near start] {\tiny $j$} 
			(2,-0.3) node[zeta] {\tiny $\ell$};
	}
\DeclareSymbol{d-kPI(BdeP)}{-2}{
\draw (-1,1.2) node[xi] {\tiny $\ell$} -- (0,-0.3) node[not] {} 
			-- node[right,near start] {\tiny $\!\!\ell$} (1,1.2) node[zeta] {\tiny $j$}; 
	\draw (0,-0.3) node[not] {}  --  (1,-1.7) node[not] {} 
			--  node[right,near start] {\tiny $\- k$} (2,-0.3) node[xi] {\tiny $k$};
	}
\DeclareSymbol{PdkI(d-eBP)}{-2}{
\draw (-1,1.2) node[xi] {\tiny $\ell$} -- 
node[left,near end] {\tiny $\- \ell$} 
(0,-0.3) node[not] {} 
			-- (1,1.2) node[zeta] {\tiny $j$}; 
	\draw (0,-0.3) node[not] {}  -- node[left,near end] {\tiny $ k$} 
	 (1,-1.7) node[not] {} 
			--  (2,-0.3) node[xi] {\tiny $k$};
	}
\DeclareSymbol{d-kPI(d-eBP)}{-2}{
\draw (-1,1.2) node[xi] {\tiny $\ell$} -- node[left,near end] {\tiny $\-\ell$}
(0,-0.3) node[not] {} 
			--  (1,1.2) node[zeta] {\tiny $j$}; 
	\draw (0,-0.3) node[not] {}  --  (1,-1.7) node[not] {} 
			--  node[right,near start] {\tiny $\- k$} (2,-0.3) node[xi] {\tiny $k$};
	}
\DeclareSymbol{Pd-edkI(BP)}{-2}{
\draw (-1,1.2) node[xi] {\tiny $\ell$} -- 
(0,-0.3) node[not] {} 
			-- (1,1.2) node[zeta] {\tiny $j$}; 
	\draw (0,-0.3) node[not] {}  -- node[left,near end] {\tiny $-\ell, k \!\!\!$} 
	 (1,-1.7) node[not] {} 
			--  (2,-0.3) node[xi] {\tiny $k$};
	}
\DeclareSymbol{d-kPd-eI(BP)}{-2}{
\draw (-1,1.2) node[xi] {\tiny $\ell$} --
(0,-0.3) node[not] {} 
			--  (1,1.2) node[zeta] {\tiny $j$}; 
	\draw (0,-0.3) node[not] {}  --   node[left,near end] {\tiny $\-\ell$}
	(1,-1.7) node[not] {} 
			--  node[right,near start] {\tiny $\- k$} (2,-0.3) node[xi] {\tiny $k$};
	}

\DeclareSymbol{BPsi}{-2}{
	\draw (-1,0.8) node[xi] {\tiny $k$}  
	-- (0,-1) node[not] {} 
		--  (1,0.8) node[zeta] {\tiny $\ell$}; 
	}
	
\DeclareSymbol{PsiPsi}{-2}{
	\draw (-1,0.8) node[zeta] {} -- (0,-0.8) node[not] {} -- (1,0.8) node[zeta] {}; 
	}
\DeclareSymbol{PsikPsik}{-2}{
	\draw (-1,0.8) node[zeta] {\tiny $k$} -- (0,-1) node[not] {} -- (1,0.8) node[zeta] {\tiny $k$}; 
	}
\DeclareSymbol{BjPsik}{-2}{
	\draw (-1,0.8) node[xi] {\tiny $j$} -- (0,-1) node[not] {} -- (1,0.8) node[zeta] {\tiny $k$}; 
	}
\DeclareSymbol{BkPsij}{-2}{
	\draw (-1,0.8) node[xi] {\tiny $k$} -- (0,-1) node[not] {} -- (1,0.8) node[zeta] {\tiny $j$}; 
	}
\DeclareSymbol{BjBj}{-2}{
	\draw (-1,0.8) node[xi] {\tiny $j$} -- (0,-1) node[not] {} -- (1,0.8) node[xi] {\tiny $j$}; 
	}
\DeclareSymbol{BkBk}{-2}{
	\draw (-1,0.8) node[xi] {\tiny $k$} -- (0,-1) node[not] {} -- (1,0.8) node[xi] {\tiny $k$}; 
	}
	
\DeclareSymbol{Psi-dIdPsi}{-2}{
	\draw (0,-0.3) node[not] {} 
			-- node[left] {\tiny $k\!\!$}  (1,1.2) node[zeta] {\tiny $\ell$}; 
	\draw (0,-0.3) node[not] {} --  node[left,near end] {\tiny $j\!\!$}   (1,-1.7) node[not] {} 
			-- (2,-0.3) node[zeta] {\tiny $\ell$}; 
	}
\DeclareSymbol{Psi-dIdPsi}{-2}{
	\draw (0,-0.3) node[not] {} 
			-- node[left] {\tiny $k\!\!$}  (1,1.2) node[zeta] {\tiny $\ell$}; 
	\draw (0,-0.3) node[not] {} --  node[left,near end] {\tiny $j\!\!$}   (1,-1.7) node[not] {} 
			-- (2,-0.3) node[zeta] {\tiny $\ell$}; 
	}
\DeclareSymbol{Psi-dId-Psi}{-2}{
	\draw (0,-0.3) node[not] {} 
			-- node[left] {\tiny $\- k\!\!$}  (1,1.2) node[zeta] {\tiny $\ell$}; 
	\draw (0,-0.3) node[not] {} --  node[left,near end] {\tiny $j\!\!$}   (1,-1.7) node[not] {} 
			-- (2,-0.3) node[zeta] {\tiny $\ell$}; 
	}
\DeclareSymbol{Psi-dId-pm-Psi}{-2}{
	\draw (0,-0.3) node[not] {} 
			-- node[left] {\tiny $\pm k\!\!$}  (1,1.2) node[zeta] {\tiny $\ell$}; 
	\draw (0,-0.3) node[not] {} --  node[left,near end] {\tiny $j\!\!$}   (1,-1.7) node[not] {} 
			-- (2,-0.3) node[zeta] {\tiny $\ell$}; 
	}	

\DeclareSymbol{Ped-kdeIPe}{-2}{
	\draw (0,-0.3) node[not] {} 
			--  (1,1.2) node[zeta] {\tiny $\ell$}; 
	\draw (0,-0.3) node[not] {} --  node[left,near end] {\tiny $-k,j\!\!$}   (1,-1.7) node[not] {} 
			-- (2,-0.3) node[zeta] {\tiny $\ell$}; 
	}
	
\DeclareSymbol{Ped-kI(dePe)}{-2}{
	\draw  (0,-0.3) node[not] {} --   (1,1.2) node[zeta] {\tiny $\ell$}; 
	\draw (0,-0.3) node[not] {} -- node[left] {\tiny $ \- k\!\!$}  
	(1,-1.7) node[not] {} --  
		node[right,near start] {\tiny $\!\!j$} (2,-0.3) node[zeta] {\tiny $\ell$};
	}	
	
\DeclareSymbol{BkdkI(d-eBe)}{-2}{
	\draw  (0,-0.3) node[not] {} --  node[left] {\tiny $ \- \ell\!$}
	 (1,1.2) node[xi] {\tiny $\ell$}; 
	\draw (0,-0.3) node[not] {} --   
	node[left,near end] {\tiny $k\!$}
	(1,-1.7) node[not] {} --  
		 (2,-0.3) node[xi] {\tiny $k$};
	}
	
\DeclareSymbol{BldlI(d-eBe)}{-2}{
	\draw  (0,-0.3) node[not] {} --  node[left] {\tiny $ \- \ell\!$}
	 (1,1.2) node[xi] {\tiny $\ell$}; 
	\draw (0,-0.3) node[not] {} --   
	node[left,near end] {\tiny $\ell\!$}
	(1,-1.7) node[not] {} --  
		 (2,-0.3) node[xi] {\tiny $\ell$};
	}	

\DeclareSymbol{d-kBkI(d-eBe)}{-2}{
	\draw  (0,-0.3) node[not] {} --  node[left] {\tiny $ \- \ell\!$}
	 (1,1.2) node[xi] {\tiny $\ell$}; 
	\draw (0,-0.3) node[not] {} --   
	(1,-1.7) node[not] {} --  node[right,near start] {\tiny $\- k$}
		 (2,-0.3) node[xi] {\tiny $k$};
	}
\DeclareSymbol{d-lBlI(d-eBe)}{-2}{
	\draw  (0,-0.3) node[not] {} --  node[left] {\tiny $ \- \ell\!$}
	 (1,1.2) node[xi] {\tiny $\ell$}; 
	\draw (0,-0.3) node[not] {} --   
	(1,-1.7) node[not] {} --  node[right,near start] {\tiny $\- \ell$}
		 (2,-0.3) node[xi] {\tiny $\ell$};
	}

\DeclareSymbol{Bkd-edkI(Be)}{-2}{
	\draw  (0,-0.3) node[not] {} -- 
	 (1,1.2) node[xi] {\tiny $\ell$}; 
	\draw (0,-0.3) node[not] {} --    node[left,near end] {\tiny $ \- \ell,k\!\!$}
	(1,-1.7) node[not] {} --  
		 (2,-0.3) node[xi] {\tiny $k$};
	}
	
\DeclareSymbol{Bld-edlI(Be)}{-2}{
	\draw  (0,-0.3) node[not] {} -- 
	 (1,1.2) node[xi] {\tiny $\ell$}; 
	\draw (0,-0.3) node[not] {} --    node[left,near end] {\tiny $ \- \ell,\ell\!\!$}
	(1,-1.7) node[not] {} --  
		 (2,-0.3) node[xi] {\tiny $\ell$};
	}
		
\DeclareSymbol{d-kBkd-eI(Be)}{-2}{
	\draw  (0,-0.3) node[not] {} -- 
	 (1,1.2) node[xi] {\tiny $\ell$}; 
	\draw (0,-0.3) node[not] {} --    node[left] {\tiny $ \- \ell\!$}
	(1,-1.7) node[not] {} --  node[right,near start] {\tiny $\- k$}
		 (2,-0.3) node[xi] {\tiny $k$};
	}
	
\DeclareSymbol{d-lBld-eI(Be)}{-2}{
	\draw  (0,-0.3) node[not] {} -- 
	 (1,1.2) node[xi] {\tiny $\ell$}; 
	\draw (0,-0.3) node[not] {} --    node[left] {\tiny $ \- \ell\!$}
	(1,-1.7) node[not] {} --  node[right,near start] {\tiny $\- \ell$}
		 (2,-0.3) node[xi] {\tiny $\ell$};
	}

\DeclareSymbol{Psi-djIdjPsi}{-2}{
	\draw (0,-0.3) node[not] {} 
			-- node[left] {\tiny $j\!\!$}  (1,1.2) node[zeta] {\tiny $\ell$}; 
	\draw (0,-0.3) node[not] {} --  node[left,near end] {\tiny $j\!\!$}   (1,-1.7) node[not] {} 
			-- (2,-0.3) node[zeta] {\tiny $\ell$}; 
	}

\DeclareSymbol{B-dIdPsi}{-2.5}{
	\draw (0,-0.3) node[not] {} -- node[left] {\tiny $\ell\!\!$} (1,1.2) node[zeta] {\tiny $j$}; 
	\draw (0,-0.3) node[not] {}  -- node[left,near end] {\tiny $k \!\!$} (1,-1.7) node[not] {} -- (2,-0.3) node[xi] {\tiny $k$};
	}
\DeclareSymbol{B-d-sign-Id-sign-Psi}{-2.5}{
	\draw (0,-0.3) node[not] {} -- node[left] {\tiny $\si\ell\!\!$} (1,1.2) node[zeta] {\tiny $j$}; 
	\draw (0,-0.3) node[not] {}  -- node[left,near end] {\tiny $\bar\si k \!\!$} (1,-1.7) node[not] {} -- (2,-0.3) node[xi] {\tiny $k$};
	}
\DeclareSymbol{B-d-pm-Id-pm-Psi}{-2.5}{
	\draw (0,-0.3) node[not] {} -- node[left] {\tiny $\pm\ell\!\!$} (1,1.2) node[zeta] {\tiny $j$}; 
	\draw (0,-0.3) node[not] {}  -- node[left,near end] {\tiny $\pm k \!\!$} (1,-1.7) node[not] {} -- (2,-0.3) node[xi] {\tiny $k$};
	}

\DeclareSymbol{dk-Psi-IdkPsi}{-2.5}{
	\draw (0,-0.3) node[not] {} -- node[left] {\tiny $k\!\!$} (1,1.2) node[zeta] {\tiny $\ell$}; 
	\draw (0,-0.3) node[not] {}  -- node[left,near end] {\tiny $-k \!\!$} (1,-1.7) node[not] {} -- (2,-0.3) node[zeta] {\tiny $\ell$};
	}
\DeclareSymbol{dk-Psi-IdkPsi2}{-2.5}{
	\draw (0,-0.3) node[not] {} -- node[left] {\tiny $k\!\!$} (1,1.2) node[zeta] {\tiny $j$}; 
	\draw (0,-0.3) node[not] {}  -- node[left,near end] {\tiny $-k \!\!$} (1,-1.7) node[not] {} -- (2,-0.3) node[zeta] {\tiny $\ell$};
	}
	
\DeclareSymbol{dj-Psik-Idj-Psil}{-2}{
	\draw  (0,-0.3) node[not] {} -- node[left] {\tiny $j\!\!$}   (1,1.2) node[zeta] {\tiny $\ell$}; 
	\draw (0,-0.3) node[not] {} -- (1,-1.7) node[not] {} 
			-- node[right,near start] {\tiny $\!\!j$} (2,-0.3) node[zeta] {\tiny $k$};
	}
\DeclareSymbol{dj-Psik-Id-pm-j-Psil}{-2}{
	\draw  (0,-0.3) node[not] {} -- node[left] {\tiny $j\!\!$}   (1,1.2) node[zeta] {\tiny $\ell$}; 
	\draw (0,-0.3) node[not] {} -- (1,-1.7) node[not] {} 
			-- node[right,near start] {\tiny $\!\!\pm j$} (2,-0.3) node[zeta] {\tiny $k$};
	}
\DeclareSymbol{dk-Psil-Idj-Psil}{-2}{
	\draw  (0,-0.3) node[not] {} -- node[left] {\tiny $j\!\!$}   (1,1.2) node[zeta] {\tiny $\ell$}; 
	\draw (0,-0.3) node[not] {} -- (1,-1.7) node[not] {} --  
		node[right,near start] {\tiny $\!\!k$} (2,-0.3) node[zeta] {\tiny $\ell$};
	}
\DeclareSymbol{dj-Psil-Idk-Psil}{-2}{
	\draw  (0,-0.3) node[not] {} -- node[left] {\tiny $k\!\!$}   (1,1.2) node[zeta] {\tiny $\ell$}; 
	\draw (0,-0.3) node[not] {} -- (1,-1.7) node[not] {} --  
		node[right,near start] {\tiny $\!\!j$} (2,-0.3) node[zeta] {\tiny $\ell$};
	}
\DeclareSymbol{dj-Psil-Id-k-Psil}{-2}{
	\draw  (0,-0.3) node[not] {} -- node[left] {\tiny $ \- k\!\!$}   (1,1.2) node[zeta] {\tiny $\ell$}; 
	\draw (0,-0.3) node[not] {} -- (1,-1.7) node[not] {} --  
		node[right,near start] {\tiny $\!\!j$} (2,-0.3) node[zeta] {\tiny $\ell$};
	}
\DeclareSymbol{dj-Psil-Id-pm-k-Psil}{-2}{
	\draw  (0,-0.3) node[not] {} -- node[left] {\tiny $ \pm k\!\!$}   (1,1.2) node[zeta] {\tiny $\ell$}; 
	\draw (0,-0.3) node[not] {} -- (1,-1.7) node[not] {} --  
		node[right,near start] {\tiny $\!\!j$} (2,-0.3) node[zeta] {\tiny $\ell$};
	}
\DeclareSymbol{djPsil-IdjPsil}{-2}{
	\draw  (0,-0.3) node[not] {} -- node[left] {\tiny $j\!\!$}   (1,1.2) node[zeta] {\tiny $\ell$}; 
	\draw (0,-0.3) node[not] {} -- (1,-1.7) node[not] {} --  
		node[right,near start] {\tiny $\!\!j$} (2,-0.3) node[zeta] {\tiny $\ell$};
	}
\DeclareSymbol{dkPsil-IdkPsil}{-2}{
	\draw  (0,-0.3) node[not] {} -- node[left] {\tiny $k\!\!$}   (1,1.2) node[zeta] {\tiny $\ell$}; 
	\draw (0,-0.3) node[not] {} -- (1,-1.7) node[not] {} --  
		node[right,near start] {\tiny $\!\!k$} (2,-0.3) node[zeta] {\tiny $\ell$};
	}
\DeclareSymbol{d-kPsil-IdkPsil}{-2}{
	\draw  (0,-0.3) node[not] {} -- node[left] {\tiny $k\!\!$}   (1,1.2) node[zeta] {\tiny $\ell$}; 
	\draw (0,-0.3) node[not] {} -- (1,-1.7) node[not] {} --  
		node[right,near start] {\tiny $\!\!\-k$} (2,-0.3) node[zeta] {\tiny $\ell$};
	}
\DeclareSymbol{d-pm-kPsil-IdkPsil}{-2}{
	\draw  (0,-0.3) node[not] {} -- node[left] {\tiny $k\!\!$}   (1,1.2) node[zeta] {\tiny $\ell$}; 
	\draw (0,-0.3) node[not] {} -- (1,-1.7) node[not] {} --  
		node[right,near start] {\tiny $\!\! \- k$} (2,-0.3) node[zeta] {\tiny $\ell$};
	}
\DeclareSymbol{dj-Psii-Idk-Psil}{-3}{
	\draw  (0,-0.3) node[not] {} -- node[left] {\tiny $k\!\!$}   (1,1.2) node[zeta] {\tiny $\ell$}; 
	\draw (0,-0.3) node[not] {} -- (1,-1.7) node[not] {} --  
		node[right,near start] {\tiny $\!\!j$} (2,-0.3) node[zeta] {\tiny $i$};
	}
\DeclareSymbol{dkPsil-IdkPsij}{-2}{
	\draw  (0,-0.3) node[not] {} -- node[left] {\tiny $k\!\!$}   (1,1.2) node[zeta] {\tiny $j$}; 
	\draw (0,-0.3) node[not] {} -- (1,-1.7) node[not] {} --  
		node[right,near start] {\tiny $\!\!k$} (2,-0.3) node[zeta] {\tiny $\ell$};
	}	
\DeclareSymbol{d-kPsil-IdkPsij}{-2}{
	\draw  (0,-0.3) node[not] {} -- node[left] {\tiny $k\!\!$}   (1,1.2) node[zeta] {\tiny $j$}; 
	\draw (0,-0.3) node[not] {} -- (1,-1.7) node[not] {} --  
		node[right,near start] {\tiny $\!\! \-k$} (2,-0.3) node[zeta] {\tiny $\ell$};
	}

\DeclareSymbol{dI-BdPsi}{-2}{
	\draw (-1,1.2) node[xi] {\tiny $k$} -- (0,-0.3) node[not] {} 
		-- node[right,near start] {\tiny $\!\!k$}  (1,1.2) node[zeta] {\tiny $\ell$};
	\draw  (0,-0.3) node[not] {} --  node[right] {\tiny $\!\!j$} (0,-1.7) node[not] {} ;
	}
\DeclareSymbol{dIV}{-2}{
	\node[zeta]  at (1,1.2)  (topright) {};
	\draw (-1,1.2) node[zeta] {} -- (0,-0.3) node[not] {};
	\draw[very thick] (0,-0.3) node[not] {} to  (topright);
	\draw[very thick]  (0,-0.3) node[not] {} to  (0,-1.7) node[not] {} ;
	}
\DeclareSymbol{d-pm-I-Bd-pm-Psi}{-2}{
	\draw (-1,1.2) node[xi] {\tiny $k$} -- (0,-0.3) node[not] {} 
		-- node[right,near start] {\tiny $\!\! \pm k$}  (1,1.2) node[zeta] {\tiny $\ell$};
	\draw  (0,-0.3) node[not] {} --  node[right] {\tiny $\!\!\pm j$} (0,-1.7) node[not] {} ;
	}
\DeclareSymbol{dI-Bd-Psi}{-2}{
	\draw (-1,1.2) node[xi] {\tiny $k$} -- (0,-0.3) node[not] {} 
		-- node[right,near start] {\tiny $\!\! \- k$}  (1,1.2) node[zeta] {\tiny $\ell$};
	\draw  (0,-0.3) node[not] {} --  node[right] {\tiny $\!\!j$} (0,-1.7) node[not] {} ;
	}
\DeclareSymbol{diI-BdPsi}{-2}{
	\draw (-1,1.2) node[xi] {\tiny $k$} -- (0,-0.3) node[not] {} 
		-- node[right,near start] {\tiny $\!\!k$}  (1,1.2) node[zeta] {\tiny $\ell$};
	\draw  (0,-0.3) node[not] {} --  node[right] {\tiny $\!\!i$} (0,-1.7) node[not] {} ;
	}
\DeclareSymbol{diI-Bd-Psi}{-2}{
	\draw (-1,1.2) node[xi] {\tiny $k$} -- (0,-0.3) node[not] {} 
		-- node[right,near start] {\tiny $\!\!\- k$}  (1,1.2) node[zeta] {\tiny $\ell$};
	\draw  (0,-0.3) node[not] {} --  node[right] {\tiny $\!\!i$} (0,-1.7) node[not] {} ;
	}
\DeclareSymbol{dkI-BdPsij}{-2}{
	\draw (-1,1.2) node[xi] {\tiny $\ell$} -- (0,-0.3) node[not] {} 
		-- node[right,near start] {\tiny $\!\!\ell$}  (1,1.2) node[zeta] {\tiny $j$};
	\draw  (0,-0.3) node[not] {} --  node[right] {\tiny $\!\!k$} (0,-1.7) node[not] {} ;
	}
\DeclareSymbol{d-sign-kI-Bd-sign-Psij}{-2}{
	\draw (-1,1.2) node[xi] {\tiny $\ell$} -- (0,-0.3) node[not] {} 
		-- node[right,near start] {\tiny $\!\!\si \ell$}  (1,1.2) node[zeta] {\tiny $j$};
	\draw  (0,-0.3) node[not] {} --  node[right] {\tiny $\!\!\bar\si k$} (0,-1.7) node[not] {} ;
	}

\DeclareSymbol{Izeta}{-2}{
	\draw (0,-1) node[not] {} -- (0,1) node[zeta] {}; 
	}
\DeclareSymbol{I'zeta}{-2}{
	\node[zeta] at (0,1) (a) {};
	\draw[very thick] (0,-1) node[not] {} to (a); 
	}
\DeclareSymbol{Izetaj}{-2}{
	\draw (0,-1) node[not] {} -- (0,1) node[zeta] {\tiny $j$}; 
	}
\DeclareSymbol{Izetal}{-2}{
	\draw (0,-1) node[not] {} -- (0,1) node[zeta] {\tiny $\ell$}; 
	}
\DeclareSymbol{Izetak}{-2}{
	\draw (0,-1) node[not] {} -- (0,1) node[zeta] {\tiny $k$}; 
	}
\DeclareSymbol{Izeta1}{-2}{
	\draw (0,-1) node[not] {} -- (0,1) node[zeta] {\tiny $1$}; 
	}
\DeclareSymbol{Izeta2}{-2}{
	\draw (0,-1) node[not] {} -- (0,1) node[zeta] {\tiny $2$}; 
	}
\DeclareSymbol{Izetaj-e}{-2}{
	\draw (0,-1) node[not] {} -- (0,1) node[zeta] {\tiny $j$}; 
	\node at (0.5,-1) {$\be$};
	} 
\DeclareSymbol{Ixi}{-2}{
	\draw (0,-1) node[not] {} -- (0,1) node[xi] {}; 
	}
\DeclareSymbol{Ixij}{-2}{
	\draw (0,-1) node[not] {} -- (0,1) node[xi] {\tiny $j$}; 
	}
\DeclareSymbol{Ixik}{-2}{
	\draw (0,-1) node[not] {} -- (0,1) node[xi] {\tiny $k$}; 
	}
\DeclareSymbol{dIdIzeta}{-2}{
	\node[zeta] at (1,1.2) (a) {};
	\draw[very thick]  (0,-0.3) node[not] {} 
		--   (a);
	\draw[very thick]  (0,-0.3) node[not] {} --   (1,-1.7) node[not] {} ;
	}
\DeclareSymbol{dI-dPsi}{-2}{
	\draw  (0,-0.3) node[not] {} 
		-- node[right,near start] {\tiny $\!\!k$}  (1,1.2) node[zeta] {\tiny $\ell$};
	\draw  (0,-0.3) node[not] {} --  node[left,near end] {\tiny $j\!\!$} (1,-1.7) node[not] {} ;
	}
\DeclareSymbol{dI-d-Psi}{-2}{
	\draw  (0,-0.3) node[not] {} 
		-- node[right,near start] {\tiny $\!\! \- k$}  (1,1.2) node[zeta] {\tiny $\ell$};
	\draw  (0,-0.3) node[not] {} --  node[left,near end] {\tiny $j\!\!$} (1,-1.7) node[not] {} ;
	}
\DeclareSymbol{d-pm-I-d-pm-Psi}{-2}{
	\draw  (0,-0.3) node[not] {} 
		-- node[right,near start] {\tiny $\!\! \pm k$}  (1,1.2) node[zeta] {\tiny $\ell$};
	\draw  (0,-0.3) node[not] {} --  node[left,near end] {\tiny $\pm j\!\!$} (1,-1.7) node[not] {} ;
	}
\DeclareSymbol{diI-dkPsi}{-2}{
	\draw  (0,-0.3) node[not] {} 
		-- node[right,near start] {\tiny $\!\!k$}  (1,1.2) node[zeta] {\tiny $\ell$};
	\draw  (0,-0.3) node[not] {} --  node[left,near end] {\tiny $i\!\!$} (1,-1.7) node[not] {} ;
	}
\DeclareSymbol{diI-d-kPsi}{-2}{
	\draw  (0,-0.3) node[not] {} 
		-- node[right,near start] {\tiny $\!\! \- k$}  (1,1.2) node[zeta] {\tiny $\ell$};
	\draw  (0,-0.3) node[not] {} --  node[left,near end] {\tiny $i\!\!$} (1,-1.7) node[not] {} ;
	}
\DeclareSymbol{dkI-dlPsi}{-3}{
	\draw  (0,-0.3) node[not] {} 
		-- node[right,near start] {\tiny $\!\!\ell$}  (1,1.2) node[zeta] {\tiny $j$};
	\draw  (0,-0.3) node[not] {} --  node[left,near end] {\tiny $k\!\!$} (1,-1.7) node[not] {} ;
	}
\DeclareSymbol{d-sign-kI-d-sign-lPsi}{-3}{
	\draw  (0,-0.3) node[not] {} 
		-- node[right,near start] {\tiny $\!\!\si\ell$}  (1,1.2) node[zeta] {\tiny $j$};
	\draw  (0,-0.3) node[not] {} --  node[left,near end] {\tiny $\bar\si k\!\!$} (1,-1.7) node[not] {} ;
	}

\DeclareSymbol{I-PsidPsi}{-2}{
	\draw (-1,1.2) node[zeta] {} -- (0,-0.3) node[not] {};
	\node[zeta] at (1,1.2) (a) {};
	\draw[very thick] (0,-0.3)  to (a); 
	\draw (0,-0.3) node[not] {} -- (0,-1.7) node[not] {}; 
	}
	
\DeclareSymbol{I-Psi1dPsi2}{-2}{
	\draw (-1,1.2) node[zeta] {\tiny $1$} -- (0,-0.3) node[not] {}
		 -- node[right,near start] {\tiny $\!\!j$} (1,1.2) node[zeta] {\tiny $2$}; 
	\draw (0,-0.3) node[not] {} -- (0,-1.7) node[not] {}; 
	}
\DeclareSymbol{I-Psi2dPsi1}{-2}{
	\draw (-1,1.2) node[zeta] {\tiny $2$} -- (0,-0.3) node[not] {}
		 -- node[right,near start] {\tiny $\!\!j$} (1,1.2) node[zeta] {\tiny $1$}; 
	\draw (0,-0.3) node[not] {} -- (0,-1.7) node[not] {}; 
	}
\DeclareSymbol{I-PsikdPsil}{-2}{
	\draw (-1,1.2) node[zeta] {\tiny $k$} -- (0,-0.3) node[not] {}
		 -- node[right,near start] {\tiny $\!\!j$} (1,1.2) node[zeta] {\tiny $\ell$}; 
	\draw (0,-0.3) node[not] {} -- (0,-1.7) node[not] {}; 
	}
\DeclareSymbol{I-PsijdkPsil}{-2}{
	\draw (-1,1.2) node[zeta] {\tiny $j$} -- (0,-0.3) node[not] {}
		 -- node[right,near start] {\tiny $\!\!k$} (1,1.2) node[zeta] {\tiny $\ell$}; 
	\draw (0,-0.3) node[not] {} -- (0,-1.7) node[not] {}; 
	}
\DeclareSymbol{I-BdPsi--1}{-2}{
	\draw (-1,1.2) node[xi] {\tiny $k$} -- (0,-0.3) node[not] {} 
		-- node[right,near start] {\tiny $\!\!k$}  (1,1.2) node[zeta] {\tiny $\ell$}; 
	\draw (0,-0.3) node[not] {} -- (0,-1.7) node[not] {}; 
	}
\DeclareSymbol{I-BdPsi--2}{-2}{
	\draw (-1,1.2) node[xi] {\tiny $k$} -- (0,-0.3) node[not] {} 
		-- node[right,near start] {\tiny $\!\! \- k$}  (1,1.2) node[zeta] {\tiny $\ell$}; 
	\draw (0,-0.3) node[not] {} -- (0,-1.7) node[not] {}; 
	}
	
\DeclareSymbol{I-BdPsi}{-2}{
	\draw (-1,1.2) node[xi] {\tiny $k$} -- (0,-0.3) node[not] {} 
		-- node[right,near start] {\tiny $\!\!k$}  (1,1.2) node[zeta] {\tiny $\ell$};
	\draw  (0,-0.3) node[not] {} --   (0,-1.7) node[not] {} ;
	}
\DeclareSymbol{I-Bd-Psi}{-2}{
	\draw (-1,1.2) node[xi] {\tiny $k$} -- (0,-0.3) node[not] {} 
		-- node[right,near start] {\tiny $\!\!\- k$}  (1,1.2) node[zeta] {\tiny $\ell$};
	\draw  (0,-0.3) node[not] {} --   (0,-1.7) node[not] {} ;
	}
\DeclareSymbol{I-Bd-pm-Psi}{-2}{
	\draw (-1,1.2) node[xi] {\tiny $k$} -- (0,-0.3) node[not] {} 
		-- node[right,near start] {\tiny $\!\!\pm k$}  (1,1.2) node[zeta] {\tiny $\ell$};
	\draw  (0,-0.3) node[not] {} --   (0,-1.7) node[not] {} ;
	}
\DeclareSymbol{I(d-kBP)}{-2}{
	\draw (-1,1.2) node[xi] {\tiny $k$} -- 
	node[left,near end] {\tiny $\- k\!\!$} 
	(0,-0.3) node[not] {} 
		--  (1,1.2) node[zeta] {\tiny $\ell$};
	\draw  (0,-0.3) node[not] {} --   (0,-1.7) node[not] {} ;
	}
	
\DeclareSymbol{d-kI(BP)}{-2}{
	\draw (-1,1.2) node[xi] {\tiny $k$} -- 
	(0,-0.3) node[not] {} 
		--  (1,1.2) node[zeta] {\tiny $\ell$};
	\draw  (0,-0.3) node[not] {} --  
		node[right] {\tiny $\!\!\- k$} 
		 (0,-1.7) node[not] {} ;
	}
	
\DeclareSymbol{I-BjdjPsik}{-2}{
	\draw (-1,1.2) node[xi] {\tiny $j$} -- (0,-0.3) node[not] {} 
		-- node[right,near start] {\tiny $\!\!j$}  (1,1.2) node[zeta] {\tiny $k$};
	\draw  (0,-0.3) node[not] {} --   (0,-1.7) node[not] {} ;
	}
\DeclareSymbol{I-Bjd-pm-jPsik}{-2}{
	\draw (-1,1.2) node[xi] {\tiny $j$} -- (0,-0.3) node[not] {} 
		-- node[right,near start] {\tiny $\!\! \pm j$}  (1,1.2) node[zeta] {\tiny $k$};
	\draw  (0,-0.3) node[not] {} --   (0,-1.7) node[not] {} ;
	}

\DeclareSymbol{I-dPsi1}{-2}{
	\draw  (0,-0.3) node[not] {} 
		-- node[right,near start] {\tiny $\!\!j$}  (1,1.2) node[zeta] {\tiny $1$};
	\draw  (0,-0.3) node[not] {} -- (1,-1.7) node[not] {} ;
	}
\DeclareSymbol{I-dPsi2}{-2}{
	\draw  (0,-0.3) node[not] {} 
		-- node[right,near start] {\tiny $\!\!j$}  (1,1.2) node[zeta] {\tiny $2$};
	\draw  (0,-0.3) node[not] {} -- (1,-1.7) node[not] {} ;
	}
\DeclareSymbol{I-dPsil}{-2}{
	\draw  (0,-0.3) node[not] {} 
		-- node[right,near start] {\tiny $\!\!k$}  (1,1.2) node[zeta] {\tiny $\ell$};
	\draw  (0,-0.3) node[not] {} -- (1,-1.7) node[not] {} ;
	}
\DeclareSymbol{I-d-Psil}{-2}{
	\draw  (0,-0.3) node[not] {} 
		-- node[right,near start] {\tiny $\!\!\- k$}  (1,1.2) node[zeta] {\tiny $\ell$};
	\draw  (0,-0.3) node[not] {} -- (1,-1.7) node[not] {} ;
	}
	
\DeclareSymbol{I(I-k(Bk))}{-2}{
	\draw  (0,-0.3) node[not] {} 
		-- node[right,near start] {\tiny $\!\!\- k$}  (1,1.2) node[xi] {\tiny $k$};
	\draw  (0,-0.3) node[not] {} -- (1,-1.7) node[not] {} ;
	}
	
\DeclareSymbol{d-kI(I(Pe))}{-2}{
	\draw  (0,-0.3) node[not] {} 
		--  (1,1.2) node[zeta] {\tiny $\ell$};
	\draw  (0,-0.3) node[not] {} -- node[right,near start] {\tiny $\!\!\- k$}  (1,-1.7) node[not] {} ;
	}	
\DeclareSymbol{d-kI(I(Bk))}{-2}{
	\draw  (0,-0.3) node[not] {} 
		--  (1,1.2) node[xi] {\tiny $k$};
	\draw  (0,-0.3) node[not] {} -- node[right,near start] {\tiny $\!\!\- k$}  (1,-1.7) node[not] {} ;
	}

\DeclareSymbol{I-BPsidPsi}{-2}{
	\node at (-1.2,-0.2) {$\!\!\!\CE$};
	\draw (-1.2,1.2) node[xi] {\tiny $j$} -- (0,-0.3) node[not] {} 
		-- node[right,near start] {\tiny $\!\!j$}  (1.2,1.2) node[zeta] {\tiny $k$};
	\draw  (0,-0.3) node[not] {} --   (0,1.2) node[zeta] {\tiny $k$} ;
	\draw  (0,-0.3) node[not] {} --   (0,-1.7) node[not] {} ;
	}

\definecolor{darkred}{rgb}{0.9,0.1,0.1}

\def\hhao#1{}

\def\s{\mathfrak{s}}

\def\be{\mathbf e}
\def\br{\mathbf r}
\def\poly #1{{#1}_{\mathrm{poly}}}
\def\onorm#1{| #1|_0}

\newcommand{\e}{\varepsilon}

\def\MM{\mathscr{M}}

\def\DD{\mathscr{D}}

\def\${|\!|\!|}
\def\Wick#1{\,\colon\!\! #1 \colon}
\def\E{\mathbf{E}}
\def\T{\mathbf{T}}

\def\curl{\mbox{curl}}
\def\d{\mbox{d}}
\def\-{\mbox{-}}

\def\ST{\mathscr{T}}
\newcommand{\VERT}{\vert\!\vert\!\vert}

\newcommand{\si}{\sigma}
\newcommand{\Cdot}{{\boldsymbol{\cdot}}}

\begin{document}

\title{Stochastic quantization of an Abelian gauge theory}
\author{Hao Shen}
\institute{University of Wisconsin-Madison, US, \email{pkushenhao@gmail.com}}

\maketitle

\begin{abstract}
We study the Langevin dynamics of a $U(1)$ lattice gauge theory on the torus, and prove that they converge for short time in a suitable gauge to a system of stochastic PDEs driven by space-time white noises. This also yields convergence of some gauge invariant observables on a short time interval.
We fix gauge via a DeTurck trick,  and  prove a version of Ward identity which results in cancellation of renormalization constants that would otherwise break gauge symmetry.
The proof relies on a discrete version of the theory of regularity structures.
\end{abstract}


\setcounter{tocdepth}{1}
\tableofcontents

\section{Introduction}


Consider a $U(1)$ gauge theory defined by the following Hamiltonian, which is sometimes called the Higgs model or scalar quantum electromagnetic dynamics (scalar QED):
\begin{equ} [e:potential]
\mathcal H(A,\Phi) \eqdef
\frac12 \int_{\T^2} \big( F_A(x)^2 + \sum_{j=1,2} |D^A_j \Phi(x)|^2 \big) \,\d^2 x \;.
\end{equ}
Here $A$ and $\Phi$ are both $\R^2$ valued functions on the unit torus $\T^2$,
and by the standard convention we will  
call $A=(A_1,A_2)$ a vector or gauge field and $\Phi = \Phi_1+i\Phi_2$  a complex valued field ($i=\sqrt{-1}$). Moreover,
\[
F_A \eqdef \curl A =
\partial_1 A_2 -\partial_2 A_1
\]
is a scalar function on $\T^2$
and $D_j^A$ is the {\it gauge covariant} derivative
\[
D_j^A \Phi \eqdef \partial_j \Phi - i \lambda A_j \Phi  \qquad (j=1,2)
\]
where  $\lambda\in \R$ and finally $|D^A_j \Phi(x)|$ 
is the complex amplitude of $D^A_j \Phi(x)$. 
Here gauge covariance means that for any real valued function $f$, if we introduce a gauge transformation as
 \begin{equ} \label{e:gauge-trans}
 \tilde A=A+\nabla f \;,
 \qquad \tilde \Phi = e^{i\lambda f} \Phi \;,
 \end{equ}
then one can check that 
 \begin{equ} \label{e:covariance}
D_j^{\tilde A} \tilde\Phi = e^{i\lambda f} D_j^A \Phi \;.
 \end{equ}
Since the vector field $A$ can be also naturally viewed as a differential one-form, the quantity $F_A=dA$ where $d$ is the exterior differential and the gauge transformation on $A$ is simply adding an exact one-form $df$.
The gauge covariant derivative has an adjoint operator which is 
$(D_j^A)^* = -\partial_j + i \lambda A_j=-D_j^A $.
\footnote{If we view $D^A$ as an operator which gives a 1-form, then
for a 1-form $B$, $(D^A)^* B = -\mbox{div}B + i A\cdot B$.}

The model \eqref{e:potential} exhibits gauge symmetry, namely, under the transformation \eqref{e:gauge-trans}
one can check that 
\[
\mathcal H(\tilde A, \tilde\Phi) =\mathcal H(A,\Phi) 
\]
since both $ F_A(x)$ and $ |D^A_j \Phi(x)|^2$ remain unchanged,
assuming that $f\in\CC^2(\T^2)$ such that 
$\curl \nabla f =0$.
This symmetry is called $U(1)$ gauge  invariance 
due to the nature of the transformation \eqref{e:gauge-trans}:
the transformation $e^{i\lambda f}$ is a function taking values in the Lie group $U(1)$ and
the term $- i \lambda A_j$ which defines the gauge covariant derivative 
is a function taking values in the Lie algebra $\mathfrak u(1)$.

Because of this invariance,
the formal ``Gibbs measure" (i.e. quantum gauge theory)
\begin{equ}[e:formal-measure]
e^{-\mathcal H(A,\Phi)} \CD A \CD\Phi
\end{equ}
where $\CD A \CD\Phi$ is the formal ``Lebesgue measure" 
can not be naively defined as a probability measure,
 even if $\T^2$ is replaced by a finite lattice, since the measure would not be normalizable. 
To make sense of this probability measure one has to fix a gauge, which amounts to
selecting a suitable representative element from each trajectory of the action by the gauge transformation \eqref{e:gauge-trans}.



The aim of this article is to study the Langevin dynamic associated with the potential $\mathcal H(A,\Phi)$, which is {\it formally} given by
\minilab{e:APhi}
\begin{equs} 
\partial_{t}A
	&=-\curl^*\curl A
	- \frac{i \lambda}{2} \Big( \bar{\Phi}D^A\Phi-\Phi\overline{D^A\Phi} \Big)+\xi \;,
		\label{e:APhi-A}
	 \\
\partial_t \Phi &= \sum_{j=1}^2 D^A_j D^A_j\Phi + \zeta \;,
		\label{e:APhi-Phi}
\end{equs}
where $\xi=(\xi_1,\xi_2)$, $\zeta=\zeta_1 + i\zeta_2$
and $\xi_1,\xi_2,\zeta_1,\zeta_2$ are four independent space-time white noises,
and $\sum_{j=1}^2 D^A_j D^A_j= -\sum_{j=1}^2 (D^A_j)^* D^A_j$ is called the
gauge covariant Laplacian (which is nonlinear).
For $\zeta$ as a complex valued space-time white noise one has
$\E(\zeta(t,x)\bar \zeta(s,y))=2\delta(t-s)\delta(x-y)$.
Equation~\eqref{e:APhi} is often called the ``stochastic quantization'' of 
the quantum field theory defined by $\mathcal H(A,\Phi)$:
formally, \eqref{e:formal-measure} is the invariant measure of \eqref{e:APhi}.
Stochastic quantization was first proposed by physicists Parisi and Wu in 
the paper {\it ``perturbation theory without gauge fixing"}
\cite{ParisiWu}; 
 as a matter of fact their main motivation was to study gauge theory, especially
 the gauge-fixing issue in the dynamic setting; see Section~\ref{sec:nonAbelian} for more formal discussion.

The nonlinear equation \eqref{e:APhi} also exhibits $U(1)$ gauge invariance in the distributional sense.
In fact, at the purely formal level, we can define  the gauge transformed functions $\tilde A,\tilde \Phi$ as \eqref{e:gauge-trans}
with  $f:\T^2\to \R$ (independent of time).
One can formally check 
 \footnote{Of course the initial condition will change under this transformation.} 
using \eqref{e:covariance} that $\tilde A $ and $\tilde \Phi $ still satisfy
 \eqref{e:APhi-A}, while \eqref{e:APhi-Phi} becomes
\[
\partial_t \tilde\Phi  =e^{i\lambda f} \partial_t \Phi 
	= e^{i\lambda f} \sum_j D^{A }_j D^{A }_j\Phi  + e^{i\lambda f} \zeta 
	=  \sum_j D^{\tilde A }_j D^{\tilde A }_j \tilde\Phi  
		+ e^{i\lambda f} \zeta  \;.
\]
Note that the gauge transformed noise $e^{i\lambda f}\zeta $ is still distributed as a complex valued Gaussian space-time white noise, in particular
$\E(e^{i\lambda f} \zeta(t,x)\overline{e^{i\lambda f}\zeta(s,y)})=2\delta(t-s)\delta(x-y)$.

Mathematically there are several reasons
 that the study of this system is interesting and nontrivial. 
First, note that the equations in \eqref{e:APhi} are completely  formal.
Indeed, in two spatial dimensions, the solution of the linearized equation 
$\partial_t \Phi = \Delta\Phi + \zeta$ is distributional, so that 
the nonlinearities of the system lack a classical meaning.
This is the general problem shared by a large class of singular stochastic PDEs 
that have been intensively studied recently, such as 
the KPZ equation \cite{KPZ,gubinelli2015kpz}, 
generalized parabolic Anderson model \cite{MR3406823,Regularity},
stochastic Navier-Stokes equation \cite{MR1941997,Zhu2014NS},
as well as some other equations which arise from the stochastic quantization procedure
such as the
$\Phi^4$ equation \cite{MR2016604,Regularity}
 and sine-Gordon equation \cite{MR3452276,SineGordon8pi}.
The theory of regularity structures \cite{Regularity}, the paracontrolled distribution method \cite{MR3406823,BailleulBernicot}
and the renormalization group method \cite{Kupiainen} are developed to provide solution theories to equations with such ill-defined nonlinearities. 
Typically, these equations require renormalizations. For the formal equation \eqref{e:APhi},
we will see that the only renormalization we will have is a 
``mass renormalization'' 
for $\Phi$ in the second equation. This amounts to renormalizing the Hamiltonian by subtracting a term  ``$\infty|\Phi|^2$"
which is again invariant under the transformation \eqref{e:gauge-trans}.

Another difficulty comes from  the {\it linear} part; note that the equation \eqref{e:APhi-A} for $A$
is not parabolic. In fact, if we drop the field $\Phi$ (by setting $\lambda=0$)  Eq.~\eqref{e:APhi-A} then becomes a linear system
\begin{equs}[e:linearA]
\partial_t A_1 &= -\partial_2 F_A + \xi_1 =\partial_2^2 A_1 - \partial_1\partial_2 A_2 +\xi_1\\
\partial_t A_2 &= \partial_1 F_A + \xi_2=\partial_1^2 A_2 - \partial_1\partial_2 A_1 +\xi_2 \;.
\end{equs}
This is the dynamic for the potential
\begin{equ} [e:potentialA]
\frac12 \int_{\T^2} F_A(x)^2 \,d^2x
= \frac12 \int_{\T^2} (\partial_1 A_2 -\partial_2 A_1)(x)^2 \,d^2x \;.
\end{equ}
Obviously the system \eqref{e:linearA} is not parabolic.
The loss of parabolicity is related with the $U(1)$ gauge symmetry: 
the quadratic form \eqref{e:potentialA} annihilates all the gradients 
$A=\nabla f$.

Before the discussion on the full system \eqref{e:APhi},
we illustrate the main idea of gauge tuning
which allows us to recover the parabolicity
 using the simple linear case \eqref{e:linearA}.
Consider the system \eqref{e:linearA}
with an initial data $\mathring A = (\mathring A_1,\mathring A_2)$.
We apply a version of ``DeTurck trick"
(originally arising from the context of Ricci flow \cite{DeTurck})
 to ``gauge out the  non-parabolic part'' as follows.
Let $B=(B_1,B_2)$ solves the parabolic equation
\begin{equs} [e:OUprocess]
\partial_t B_j= \Delta B_j +\xi_j  \qquad (j=1,2)
\end{equs}
with the same initial condition $\mathring B=\mathring A$,
whose solution is a well-known Ornstein-Uhlenbeck process.
 Then, with
 $\mbox{div}B=\partial_1 B_1 + \partial_2 B_2$, and let
\begin{equs} [e:ABAB]
A_1(t) &\eqdef  B_1(t) -\int_0^t \partial_1 \mbox{div}B(s)\,ds
= B_1(t) 
	-\int_0^t \Big( \partial_1\partial_2 B_2(s) + \partial_1^2 B_1 (s) \Big)\,ds \\
A_2(t) &\eqdef  B_2(t) -\int_0^t \partial_2 \mbox{div} B(s)\,ds
=  B_2(t) 
	-\int_0^t \Big( \partial_1\partial_2 B_1(s) + \partial_2^2 B_2 (s) \Big)\,ds
\end{equs}
one can check that
\begin{equs}
\partial_t A_1 &=  \Delta B_1 +\xi_1 -\Big( \partial_1\partial_2 B_2 + \partial_1^2 B_1 \Big)
= -\partial_2 F_B + \xi_1 \;, \\
\partial_t A_2 &=  \Delta B_2 +\xi_2 -\Big( \partial_1\partial_2 B_1 + \partial_2^2 B_2 \Big)
= \partial_1 F_B + \xi_2 \;.
\end{equs}
The important point is that with $A$ defined via $B$ as in \eqref{e:ABAB}, one has
\begin{equ} [e:FAFB]
F_A =\partial_1  A_2 -\partial_2 A_1
=\partial_1 B_2 -\partial_2 B_1
= F_B
\end{equ}
so $A$ actually solves the original equation \eqref{e:linearA},
with the initial condition $A(t=0) = \mathring A$ obviously satisfied.

The above manipulation  can be rephrased as follows. For the equations \eqref{e:linearA} for $A$, there exists
a time-dependent family of gauge transformations $\mathcal G_t$, such that $\mathcal G_0$ is identity transformation and the transformed process
\[
B(t) =\mathcal G_t A(t) \eqdef A(t)+ \nabla f(t) \;,
\qquad f(t)\eqdef \int_0^t \mbox{div} B(s)\,ds 
\]
satisfies a parabolic equation. Note that the transformations $\mathcal G_t$ depend on the ``target" (i.e. transformed) process $B$, for which one has to solve first from \eqref{e:OUprocess}.

\begin{remark} \label{rem:gauge-fix-term}
In the differential form notation \eqref{e:linearA} can be written as 
$\partial_t A = - d^* d A + \xi$,
and \eqref{e:OUprocess} can be written as 
$\partial_t B = - d^* d B - d d^* B + \xi$.
So, in a sense we have inserted a ``gauge fixing" term $- d d^* B = -\nabla\mbox{div}B$.
Note that this term {\it breaks} the $U(1)$ gauge symmetry:
unlike \eqref{e:APhi} or \eqref{e:linearA} which enjoy gauge invariance as discussed above,
\eqref{e:OUprocess} does not exhibit gauge invariance,
because if $B$ satisfies \eqref{e:OUprocess}, the function
 $B + \nabla g$ for a  time-independent function $g$ will not satisfy 
\eqref{e:OUprocess} anymore ($\Delta\nabla g\neq 0$ in general).
\end{remark}

\begin{remark} \label{rem:Helmholtz}
There is of course another natural way to fix gauge for our Abelian model.
Define an equivalence relation $\sim$ as: $A\sim \bar A$ if and only if $A=\bar A+\nabla f+(c_1,c_2)$. One can define a Gaussian measure formally given by $\exp\big(-\frac12 \int F_A^2 \,dx\big) \CD A/Z$  on the space of equivalence classes, where $\CD A$ is the formal Lebesgue measure and $Z$ is the normalization factor.
In fact, by Helmholtz decomposition a smooth vector field on $\T^2$ can be uniquely decomposed into a divergence free part and a gradient, i.e.
\begin{equ}[e:Helmholtz]
A = (-\partial_2 g,\partial_1 g) + \nabla f + (c_1,c_2)
\end{equ}
where the function $g$ solves $\Delta g=F_A$ and $f $ solves $\Delta f=\mbox{div}A$,
and $c_1,c_2$ are constants.
One can easily check that the first term on the right hand side is divergence free, and $\nabla f $ is curl free. 
It seems natural to solve our equation by projecting $A$ onto the first subspace of this decomposition,
as in the case of stochastic Navier-Stokes equation \cite{MR1941997,Zhu2014NS}.

We instead choose the present approach in this paper
mainly because we aim to develop  methods here which are potentially able to study SPDEs with non-Abelian gauge symmetry.
For non-Abelian gauge theories the gauge-fixing can not be achieved by the above {\it linear} decomposition. 
On the other hand, DeTurck trick has been successful in the study of {\it deterministic} Yang-Mills flow with non-Abelian gauge groups (see for instance  \cite[Section~6.3]{Donaldson} or the more recent monograph \cite{FeehanBook} and many references therein).
For stochastic quantization of non-Abelian gauge theories,
the use of the  gauge fixing method that is close to the DeTurck trick
goes back to physicists \cite{Zwanziger1981,sadun1987continuum} (see \cite[Eq.~(2.1)-(2.2)]{BernHalpernSadunTaubes1987} which takes the viewpoint of inserting a ``gauge fixing'' term as in Remark~\ref{rem:gauge-fix-term}).
 Recently \cite{CharalambousGrossCMP} applies this gauge fixing to deterministic Yang-Mills flow to define a kind of nonlinear negative-index Sobolev space (the motivation is to find possible orbit spaces where non-Abelian gauge theoretic measures could live on) and they refer to this trick as Zwanziger-Donaldson-Sadun gauge fixing,
see also \cite{CharalambousGrossJMP,Gross2016YM} for subsequential works along this line.
\end{remark}

With the above linear example in mind, our strategy is then to apply the DeTurck trick 
together with the theory of regularity structures \cite{Regularity} which provide the solution theory 
by regularizing the equation, renormalization, and passing to the limit. 
In the regularization step, we will consider the dynamics of {\it lattice gauge theory},
which was first proposed by Wilson \cite{Wilson1974}, see for instance the physics book \cite{Rothe2012Book} on this  extensively studied topic.
Lattice gauge theory provides us with discretization of Hamiltonian \eqref{e:potential}
which preserves the exact gauge symmetry;
 this is discussed in Section~\ref{sec:Discretization}.
One main reason  we are particularly interested in the lattice gauge theory is that one would hope to obtain gauge invariant continuum limit, so we start from regularizations that preserve the exact symmetry.
Another reason  we   choose a gauge invariant regularization
is that 
our arguments such as  DeTurck trick and Ward identities rely on gauge invariance of the nonlinearities.
\footnote{In discussion with I.Bailleul we learned  
a way of continuous regularization by a function of covariant Laplacian multiplying on the noise in the physics literature, see \cite[Eq.~(2.1),(2.4)]{BernHalpernSadunTaubes1987}. It would be very interesting to study this rigorously, in particular show that this regularization yields the same limit.}

In applying this strategy, several issues arise and need to be addressed.
 The dynamic of our lattice gauge theory 
is slightly nonlocal so we will need to localize these terms and control the errors (see Section~\ref{sec:Discrete regularity structure theory}).
In the renormalization step, the renormalization terms must also preserve the gauge invariance
in order to make  our gauge tuning argument work
and more importantly in order to obtain a gauge invariant limiting dynamic;
so we should prove that any ``symmetry-breaking'' renormalizations will actually cancel out
(these are discussed in Section~\ref{sec:GaugeTrans} and Section~\ref{sec:renormalization}).


\subsection*{Main results}

In Section~\ref{sec:Discretization} we will discretize the Hamiltonian \eqref{e:potential} and the stochastic PDE, which  preserves the exact gauge symmetry. 
Here assuming all the  precise definitions of discretizations (for which we refer to  Section~\ref{sec:Discretization}) we state our main results.
Let $\eps$ be the lattice spacing.
For each $t\ge 0$, the discretized field $A^\eps$ will be a real valued function 
on the set of edges of the lattice (with component $A_1^\eps$ on the horizontal edges and component $A_2^\eps$ on the vertical edges), 
while  the discretized field $\Phi^\eps$ will be a complex valued function 
on the set of vertices of the lattice.
With such a discretization at hand, we say that two  pairs $(A^\eps,\Phi^\eps)$ and  
$(B^\eps,\Psi^\eps)$ are {\it gauge equivalent} if there exists a real valued function $f^\eps$
on the vertices of the lattice
such that
\begin{equs} [e:Dgauge-equ]
A^\eps_j(e) =
	B^\eps_j(e) -\nabla^\eps_j  f^\eps(e) \;,
\qquad
\Phi^\eps(x ) =  e^{-i\lambda f^\eps(x)} \Psi^\eps(x) \;,
\end{equs}
for every edge $e$,  and vertex $x$ and $j\in\{1,2\}$. Here
$\nabla^\eps_j  f^\eps$ is the finite difference of $f^\eps$.
Denote the set of vertices by $\Lambda_\e$ and the set of edges by $\CE_\e$.
The main result of the article is then the following.
The spaces $\CC^\eta$,  $\CC^{\delta, \alpha}_{\eta, T^*}$
and distances $\Vert \cdot ; \cdot \Vert_{\CC^\eta}^{(\eps)} $
and $\Vert \cdot ; \cdot \Vert_{\CC^{\delta, \alpha}_{\bar{\eta}, T_\eps}}^{(\eps)} $
 will be defined  in 
\eqref{e:AlphaNorm}, \eqref{e:SpaceTimeNorm},
 \eqref{e:diffCalpha} and  \eqref{e:DHolderDist}.

\begin{theorem}\label{theo:main}
For every $\eps =2^{-N}$ with $N\in\N$, let $\CH^\eps$ be an gauge invariant 
discretization of \eqref{e:potential} given by \eqref{e:d-potential} below, and
let $(A^\eps(t), \Phi^\eps(t))_{t\ge 0}$ be the solution to the It\^o system parametrized by
the set of lattice sites $\Lambda_\e$ and the set of edges $\CE_\e$
\begin{equs}   [e:DLangevin-0]
dA^\eps(e) &= -\frac{\partial\CH^\eps}{\partial A^\eps(e)}\eps^{-2} dt
	+ dW_t^\eps(e)      \qquad\qquad (e\in\CE_\e)\;,
\\
d \Phi^\eps (x) &= -\frac{\partial \CH^\eps}{\partial \bar\Phi^\eps(x)} \eps^{-2}dt
- C^{(\eps)} \Phi^\eps
+ dW_t^\eps(x)
  \quad (x\in \Lambda_\e)\;,
\end{equs}
whose explicit form is given in  \eqref{e:APhi-long},
with initial data $(\mathring A^\eps, \mathring \Phi^\eps)$.
Suppose that 
the initial data satisfy almost surely
\begin{equ} [e:init-data-conv]
\lim_{\eps \to 0} 
\Big( \Vert \mathring \Phi;  \mathring \Phi^\eps \Vert^{(\eps)}_{\CC^\eta} 
+\Vert \mathring A;  \mathring A^\eps \Vert^{(\eps)}_{\CC^\eta}\Big)
= 0\;,
\end{equ}
for $\mathring A,\mathring \Phi \in \CC^\eta$ with $\eta>-\tfrac12$.

Then, there exists $(B^\eps(t),\Psi^\eps(t))$ which is gauge equivalent with
$(A^\eps(t),\Phi^\eps(t))$ for each $t>0$ and 
$(B^\eps(0),\Psi^\eps(0))=(\mathring A^\eps, \mathring \Phi^\eps)$  
such that the following holds:
 for every $\alpha < 0$ there is a sequence of renormalization constants $C^{(\eps)} =O( \log\eps)$,  
 a sequence of stopping times $T_\eps$ satisfying $\lim_{\eps \to 0} T_\eps = T^*$ in probability, and processes $B,\Psi\in \CC^{\delta, \alpha}_{\eta, T^*}$
  such that, for every $\bar{\eta} < \eta \wedge \alpha$, and for any $\delta > 0$ small enough, one has the  limit in distribution
\begin{equ}[e:PhiConvergence]
\lim_{\eps \to 0} \Big(
	\Vert \Psi; \Psi^\eps \Vert^{(\eps)}_{\CC^{\delta, \alpha}_{\bar{\eta}, T_\eps}}
	+\Vert B; B^\eps \Vert^{(\eps)}_{\CC^{\delta, \alpha}_{\bar{\eta}, T_\eps}}
\Big) = 0\;.
\end{equ}
\end{theorem}

The above theorem does not make any claim about convergence of 
the original fields 
$(A^\eps,\Phi^\eps)$. 
However, it is really the convergence of {\it gauge invariant observables} 
that is of true interest. Once we obtain convergence of gauge and scalar 
fields in a suitable gauge, 
we then have convergence of local observables such as 
curvature, composite fields of $\Phi^\e$ and its gauge covariant derivative, as well as some global observables such as 
``loop'' observables. These are defined in Section~\ref{sec:Discretization} and their convergences
are given in Theorem~\ref{theo:observables} below.

\begin{remark}
With the method developed in this article, it should be straightforward to obtain a local solution theory for
such a U(1) theory in two spatial dimensions  with
``Ginzburg-Landau potentials", that is, with polynomial terms 
$P(|\Phi|^2)$ 
 in the Hamiltonian $\mathcal H$. These are important models in superconductivity,
with the ``vortices dynamics" of particular interest;
to our best knowledge the stochastic dynamics of vortices have been only defined intuitively so far
(see for instance  \cite{iyer2016model}).
\end{remark}

We conclude the introduction by discussing the earlier related literature.
The Higgs model (i.e. the formal measure \eqref{e:formal-measure} with Hamiltonian \eqref{e:potential}, or its variants and generalizations) has been of great interest in constructive quantum field theory.
Let us mention the previous works on this model, especially those related with our continuum limit question.
In a series of papers
\cite{BrydgesDiamag,BrydgesGauge1,BrydgesGauge2}
 Brydges, Fr\"ohlich and Seiler
showed several important properties of 
general lattice gauge theories such as 
diamagnetic inequalities, correlation inequalities, reflection positivity, etc.
Focusing on two dimensional Higgs model, with an additional polynomial interaction $P(|\Phi|^2)$ in the Hamiltonian, 
they then proved in \cite[Section~2]{BrydgesGauge3} the continuum limit of (not necessarily gauge invariant)
correlation functions, using the diamagnetic inequalities, Nelson / hypercontractive estimate and some power counting results for Feynman graphs;
in \cite[Section~4]{BrydgesGauge3}, the infinite volume limit for 
generating functional of $F_A$ and $\Wick{|\Phi|^2}$ 
 is obtained using the correlation inequalities and Osterwalder-Schrader axioms are verified  (except for clustering) for  correlations of these gauge invariant local fields.
The series of papers by Balaban \cite{MR677999,MR679203,MR701926}
studied the quantum Higgs models of type  \eqref{e:potential}
with $|\Phi|^4$ interaction where
 $\Phi$ is an $N$-component field, in both two and three space dimensions restricted to a finite volume,
and
proved upper and lower bounds uniform in lattice spacings
on the partition function
using renormalization group transformation techniques.
Using these renormalization group  techniques,
working also in  both two and three space dimensions,
King \cite[Theorem~2.1]{MR824095} proved continuum limit of generating functional 
of the fields $F_A$ and $\Wick{|\Phi|^2}$ 
and in \cite{MR826868} combining with correlation inequalities he proved infinite volume limit and verified some of the
Osterwalder-Schrader axioms such as reflection positivity.
Clustering properties and mass gap are discussed in 
\cite{MR772615,MR782971,MR928228}.

In the pure gauge case without the scalar field $\Phi$ (i.e. $\lambda=0$ in the Hamiltonian $\mathcal H$), a lot more is known since the quantum theory is Gaussian after fixing a suitable gauge. For continuum limit results of these pure Abelian gauge models, see \cite{MR728862} in three dimensional case and \cite{MR891949} in four dimensional case.
Note that some of the literature is concerned with gauge theory models in ``compact'' setting, for instance in \cite[Section~3]{MR891949} the components of the lattice fields take values in the unit circle (i.e. the Lie group), whereas we work in the ``non-compact'' setting where the lattice fields take values in the Lie algebra.

%

There is also much interest recently in the convergence problem of discretized singular SPDEs.
The present paper proves convergence of discretized systems  based upon the  discrete regularity structure theory developed in \cite{hairer2015discrete}.
More recently,
\cite{ErhardHairer} developed a more general framework for studying such discrete SPDEs.
In the paracontrolled distribution approach,
convergence results of discretized singular SPDEs are proved in 
\cite{MR3406823} for KPZ equation,
 \cite{ChoukGairingPerkowski,martin2017paracontrolled} on parabolic Anderson model and \cite{ZhuZhuLatticePhi4} on dynamical $\Phi^4$ model.
 The paper \cite{CannizzaroMatetski} studied general {\it nonlocal} discretizations in the context of KPZ equation.
These papers show that the discrete solutions converges to the {\it same} limits as the ones obtained via continuous regularizations which were proved previously; while
for our model, the (symmetry-preserving) discretization
is the object we start with.

Finally, let's point out that establishing long time solution  will be certainly interesting, 
and will possibly shed some light on constructive field theory (see \cite{MourratWeberComesDown} for the case of $\Phi^4$ model). Generalization to non-Abelian gauge groups is also an interesting future direction;
for the earlier results in the equilibrium see 
\cite{MR1006295,MR1337973,MR2006374} in two spatial dimensions and
\cite{MR1230032,MR998658} in three and four spatial dimensions and the literature therein.



\vspace{2ex}
{\bf Acknowledgement.} The author is very grateful to Ismael Bailleul, David Brydges and Martin Hairer  for the helpful discussions on  stochastic quantization and gauge theories, and would also like to thank
Konstantin Matetski for helpful discussions on discrete regularity structures.
 The author is partially supported by NSF through DMS-1712684 / DMS-1909525.

\section{Lattice gauge theory and observables} \label{sec:Discretization}

\subsection{Lattice gauge theory}
\label{sec:Lattice gauge theory}
We henceforth consider discretizations of \eqref{e:potential} and \eqref{e:APhi}. Let $\T_\eps^2$ be the dyadic discretization of the unit torus $\T^2$
with mesh size $\eps=2^{-N}$ for $N\in\N$, 
and let $\Lambda_\e$, $\CE_\e$ and $\CP_\e$ be the sets of vertices, {\it undirected} edges and squares \footnote{Each square is simply an $\eps\times\eps$ area surrounded by four edges. The notation $\CP_\e$ here should not be confused with the operator $\CP^\e$  introduced in \eqref{e:RP-PR}} of the lattice $\T_\eps^2$. 
Let $\CE^1_\e$  and  $\CE^2_\e$ be  the set of horizontal and vertical
undirected edges respectively.
Let $\be_1=(\eps,0)$ and $\be_2=(0,\eps)$.

For $j\in\{1,2\}$,
let  $\vec{\CE^j_\e}$ and $\cev{\CE^j_\e}$ be two copies of $\CE^j_\e$,
which consist of {\it directed} edges that are oriented in the positive and negative directions along  $\be_j$  respectively.
Let $\vec{\CE_\e} = \vec{\CE^1_\e} \cup \vec{\CE^2_\e}$,
and
 $\cev{\CE_\e} = \cev{\CE^1_\e} \cup \cev{\CE^2_\e}$.
For a directed edge $e$, the edge $-e$ denotes the edge with the reversed direction.
A directed edge is often denoted by $(x,y)$ while an undirected edge is often denoted by $\{x,y\}$, where $x,y\in\Lambda_\e$ are two nearest neighbor vertices. 
For any $e\in\vec{\CE}_\e^j$ we will often denote by $e_-$ and $e_+$ 
the starting and terminating vertices of $e$, namely
\[
e=(e_-,e_+) \in \vec{\CE}_\e^j\;,
\qquad 
\mbox{where }
e_+ = e_- + \be_j \;.
\]
According to these definitions we have
$|\Lambda_\e|=\eps^{-2}$ and
$|\CE_\e|=2 \eps^{-2}$.

A vertex $x\in \Lambda_\e$ can be naturally viewed as a point in $\T^2$.
For an edge $e$, we will also sometimes think of it as a point in $\T^2$,
that is, the point $\frac{e_- +e_+}{2}$.
For two edges $e,\tilde e$,
we will sometimes write $e-\tilde e \eqdef \frac{e_- +e_+}{2}-\frac{\tilde e_- +\tilde e_+}{2} \in \R^2$;
and for an edge $e$ and a grid point $x$ 
we write 
$e-x\eqdef \frac{e_- +e_+}{2}-x$.

{\it Discretized fields.}
A discretized complex field $\Phi^\eps \in \C^{\Lambda_\e}$ is a function  which assigns each vertex $x\in\Lambda_\e$  with a complex number $\Phi^\eps(x) =\Phi^\eps_1(x)+i\Phi^\eps_2(x)$.
A discretized vector field (i.e. gauge field) $A^\eps=(A^\eps_1,A^\eps_2)\in \R^{\CE_\e}$ is a function
which assigns
each horizontal (resp. vertical) undirected edge $e\in\CE_\e$
 with a real number $A^\eps_1(e)$ (resp. $A^\eps_2(e)$).
When $e$ is known in the context as horizontal or vertical, 
we sometimes simply write $A^\eps(e)$. 
The field $A^\eps \in \R^{\CE_\e}$ can be naturally extended  to the directed edges:
If $(x,y)\in \vec{\CE_\e}$,
$A^\eps(x,y):=A^\eps(\{x,y\})$; 
if $(x,y)\in\cev{\CE_\e}$, $A^\eps(x,y):=-A^\eps(\{x,y\})$.

For $\Phi^\eps \in \R^{\Lambda_\e}$ or $\C^{\Lambda_\e}$,
$j\in\{1,2\}$, $x\in\Lambda_\e$, $\be\in\{\pm  \be_j\}$
we define the discrete derivative 
\[
(\nabla^\eps_{\be} \Phi^\eps)(x)  
 \eqdef \eps^{-1} \big( \Phi^\eps(x+\be)-\Phi^\eps(x) \big)
\]
and set $(\nabla^\eps_j \Phi^\eps)(x) \eqdef (\nabla^\eps_{\be_j} \Phi^\eps)(x) $
so that one has $(\nabla^\eps_{-\be_j} \Phi^\eps)(x)= - (\nabla^\eps_j \Phi^\eps)(x-\be_j)$.
Discrete derivatives are also defined  on the fields on the edges.
For $j,k\in\{1,2\}$, $e\in \CE^k_\e$, $\be\in\{\pm  \be_j\}$ and
$A^\e_k\in \R^{\CE^k_\e}$, 
\begin{equ}[e:nabla-j-Ak]
(\nabla^\eps_\be  A^\eps_k )(e) \eqdef 
 \eps^{-1} \big(  A_k(e+ \be)-A_k(e) \big) 
\end{equ}
where $e+ \be = \{x+ \be ,y+ \be \}   \in \CE^k_\e$ if $e=\{x,y\}$,
so that $e+ \be$ and $e$ are  end-to-end if $j=k$ 
and are parallel if $j\neq k$.
We also set $(\nabla^\eps_j A^\eps)(x) \eqdef (\nabla^\eps_{\be_j} A^\eps)(x) $.


\begin{remark} \label{rem:grad-Phi}
For $\Phi^\e\in\R^{\Lambda_\e} $ or $\C^{\Lambda_\e}$,
 we will sometimes also write $\nabla^\e_j \Phi^\e(e)$ instead of $\nabla^\e_j \Phi^\e(x)$   where $e=(x,x+\be_j)$, meaning that we regard $\nabla^\e_j \Phi^\e$ as a field on $\CE^j_\e$.
\end{remark}

\vspace{1ex}

{\it Discretized curl. } 
Let $A^\e \in \R^{\CE_\e}$. For each square $p\in\mathcal P_\eps$, summing over the values of $A^\e$ along its four  adjacent edges counter-clock-wisely then yields the value of $F^\e_{A^\e}$ on $p$:
\begin{equ}[e:discFA]
F^\e_{A^\e}(p) \eqdef \eps^{-1} \Big(A^\e(e^E_p)+A^\e(e^N_p)-A^\e(e^W_p)-A^\e(e^S_p)\Big)
\end{equ}
where $e^E_p,e^N_p,e^W_p,e^S_p \in \vec{\CE_\e}$ are the  edges 
to the east, north, west and south of 
 the square $p$ as shown in the middle part of Figure~\ref{fig:discrete}. 
 $F^\e_{A^\e}$ is often called the (discrete) curl or curvature of $A^\e$.
It is easy to see that $F^\e_{A^\e}(p)$ is invariant under $A^\e\to A^\e+\nabla^\e f^\e$ for any function $f^\e \in \R^{\Lambda_\e}$, since summing a gradient over a loop gives zero.

\begin{figure}[h] 
\centering
\begin{tikzpicture}
  \draw[step=1.5cm] (-3.5,-2) grid (7,2);
\node at (-3.2,1.8) {$\Phi$}; \filldraw (-3,1.5) circle (0.1) ;
\node at (-1.2,1.2) {$\Phi$}; \filldraw (-1.5,1.5) circle (0.1) ;
\node at (-3.2,-0.3) {$\Phi$}; \filldraw (-3,0) circle (0.1) ;
\node at (-2.2,1.8) {$D^A_1 \Phi$}; \draw[->,very thick] (-2.7,1.5) -- (-1.8,1.5) ; 
\node at (-3.5,0.7) {$D^A_2 \Phi$}; \draw[->,very thick] (-3,0.3) -- (-3,1.2) ; 
\node at (0.7,0.25) {$A_1(e_p^N)$}; \draw[->,very thick] (0.2,0) -- (1.3,0) ; 
\node at (-0.65,-0.7) {$A_2(e_p^W)$};\draw[->,very thick] (0,-1.3) -- (0,-0.2) ; 
\node at (0.7,-1.8) {$A_1(e_p^S)$}; \draw[->,very thick] (0.2,-1.5) -- (1.3,-1.5) ; 
\node at (2.1,-0.7) {$A_2(e_p^E)$};\draw[->,very thick] (1.5,-1.3) -- (1.5,-0.2) ; 
\node at (0.7,-0.7) {$F_A(p)$};
\node at (5.3,0.3) {$A_1$}; \draw[->,very thick] (4.8,0) -- (5.7,0) ; 
\node at (4.2,-0.8) {$A_2 $}; \draw[->,very thick] (4.5,-1.2) -- (4.5,-0.3) ; 
\node at (3.8,0.3) {$A_1$}; \draw[->,very thick] (3.3,0) -- (4.2,0) ; 
\node at (4.2,0.8) {$A_2 $}; \draw[->,very thick] (4.5,0.3) -- (4.5,1.2) ; 
\node at (5,-0.3) {div$A$}; \filldraw (4.5,0) circle (0.1) ;
\end{tikzpicture}
\caption{Picture for $F^\e_{A^\e}$, $D_j^{A^\e}$ and $\mbox{div}^\e A^\e$;  we drop the $\e$ dependence in the picture to make it a bit cleaner.}
\label{fig:discrete}
\end{figure}
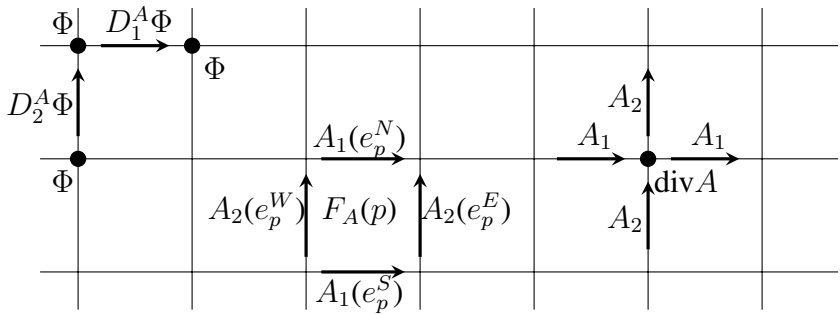




\vspace{2ex}

{\it Discretized gauge covariant derivative.}
We introduce a gauge-invariant discretization of the covariant derivative.
For $x\in \Lambda_\e$ and $\Phi^\e \in \C^{\Lambda_\e}$, let
\begin{equ}[e:dDAj]
(D_j^{A^\eps} \Phi^\eps) (x) 
= \eps^{-1} \Big(e^{-i\eps \lambda A^\eps(x,x+\be_j)} \Phi^\eps(x+\be_j) -\Phi^\eps(x)\Big) \;.
\end{equ}
With this discretization, the covariance property \eqref{e:gauge-trans} \eqref{e:covariance} remains true:
for any  $f^\e \in \R^{\Lambda_\e}$,
\begin{equs} 
( & D_j^{A^\eps +\nabla^\eps f^\e} (e^{i\lambda f^\e}\Phi^\eps)) (x) \\
&= \eps^{-1} \Big(
	e^{-i\eps \lambda A^\eps(x,x+\be_j) 
		-i\lambda f^\e(x+\be_j) + i\lambda f^\e(x)} 
			\Big( e^{i\lambda f^\e(x+\be_j) }\Phi^\eps(x+\be_j)  \Big)
	-e^{i\lambda f^\e(x) }\Phi^\eps(x)\Big)\\
&= e^{i\lambda f^\e(x) } (D_j^{A^\eps } \Phi^\eps) (x) \;.
\label{e:DCovariance}
\end{equs}
It is easy to check that its adjoint operator is given by
\begin{equ} [e:adjoint]
((D_j^{A^\eps})^* \Phi^\eps) (x) 
= \eps^{-1} \Big(e^{-i\eps \lambda A^\eps(x,x-\be_j)} \Phi^\eps(x-\be_j) -\Phi^\eps(x)\Big) \;,
\end{equ}
which is the discrete covariant derivative along the ``negative direction" $-\be_j$.
Here according to our aforementioned convention
 $(x,x-\be_j) \in \cev{\CE^j_\e}$ and 
 $A^\eps(x,x-\be_j) = - A^\eps(\{x-\be_j,x\})$.

\begin{remark}
Note that a ``naive" discretization of $D^A_j \Phi$ such as 
\[
\eps^{-1}\big( \Phi^\eps(x+\be_j) -\Phi^\eps(x) \big)- i\lambda A^\eps(x,x+\be_j) \Phi^\eps(x)
 \]
  would violate
the gauge covariance property, and hence the gauge invariance of the Hamiltonian.
As discussed in the introduction,
we choose gauge invariant discretization 
in order to implement the gauge tuning argument 
in Section~\ref{sec:GaugeTrans}, and also because we aim to obtaining the ``right" limit that inherits gauge symmetry.
There is obviously another choice of  discretization:
$
(\tilde D_j^{A^\eps} \Phi^\eps) (x)  \eqdef
\eps^{-1} \Big( \Phi^\eps(x+\be_j) - e^{i\eps \lambda A^\eps(x,x+\be_j)}\Phi^\eps(x)\Big) 
$,
which also satisfies the covariance property, 
but it only differs with \eqref{e:dDAj} by a phase factor,
which does not change the Hamiltonian \eqref{e:d-potential}.
One might also think that the  covariance derivative along the ``negative" directions
as defined in \eqref{e:adjoint}
should also be included in the Hamiltonian to look more ``symmetric",
but obviously \eqref{e:adjoint} is just equal to $- (\tilde D_j^{A^\eps} \Phi^\eps) (x-\be_j) $
so it does not matter under $\sum_x |\cdot|^2$.
\end{remark}

{\it Discretized divergence.} The discrete divergence for a vector field $A^\eps\in \R^{\CE^\e}$ gives a scalar field on $\Lambda_\e$ and is defined as
\begin{equ}[e:def-discrete-div]
\left(\mbox{div}^\e A^\e\right) (x) \eqdef 
\eps^{-1} \left( A^\eps (e_x^N)  -A^\eps (e_x^S) +A^\eps (e_x^E) -A^\eps (e_x^W) \right)
\end{equ}
where $e_x^N, e_x^S, e_x^E, e_x^W \in \vec{\CE_\e}$ are the edges to the north, south, east and west of the site $x$ respectively, see the right picture in Figure~\ref{fig:discrete}.

With the above discrete objects at hand, our discretized Hamiltonian is then written as
\begin{equs} [e:d-potential]
\mathcal H^\eps(A^\eps,\Phi^\eps) &=
\frac{\eps^2}{2} \sum_{p\in \CP_\e}  F_{A^\eps}(p)^2 
+ \frac{\eps^2}{2} \sum_{x \in \Lambda_\e} \sum_{j=1,2} | (D^{A^\eps}_j \Phi^\eps) (x)|^2 \\
&= \mathcal H_F^\eps(A^\eps)+\mathcal H_D^\eps(A^\eps,\Phi^\eps)  
\end{equs}
where $\mathcal H_F^\eps(A), \mathcal H_D^\eps(A,\Phi)$ denote the two terms respectively.
With all the aforementioned invariance / covariance properties enjoyed by the quantities on the right hand side of \eqref{e:d-potential}, one has that $\mathcal H^\eps(A^\eps,\Phi^\eps)$ is invariant under
 \begin{equ} \label{e:disc-gauge-trans}
 \tilde A^\e(e)=A^\e(e)+\left(\nabla^\eps f^\e\right)(e) \;,
 \qquad \tilde \Phi^\e (x)= e^{i\lambda f^\e (x)} \Phi^\e (x)\;,
 \end{equ}
for any function $f\in \R^{\Lambda_\e}$.
The fact that our discretization preserves gauge symmetry will be crucial for the rest of the paper, especially the gauge tuning arguments in Section~\ref{sec:GaugeTrans}.

\subsection{Dynamics of lattice gauge theory}
The (renormalized) Langevin dynamic for the discretized Hamiltonian \eqref{e:d-potential},
sped up by a factor $\eps^{-2}$,
is given by
 the following system of  It\^o
stochastic differential equations parametrized by 
$\CE_\e \sqcup \Lambda_\e$
(with $t$ understood as the macroscopic time variable)
\begin{equs} [e:DLangevin]
dA^\eps(e) &= -\frac{\partial\CH^\eps}{\partial A^\eps(e)}\eps^{-2} dt
	+ dW_t^\eps(e)\qquad\qquad (e\in\CE_\e)\;,
\\
d \Phi^\eps (x) &= -\frac{\partial \CH^\eps}{\partial \bar\Phi^\eps(x)} \eps^{-2}dt
- C^{(\eps)} \Phi^\eps(x)
+ dW_t^\eps(x)
\quad (x\in \Lambda_\e)\;,
\end{equs}
where 
$\{W^\eps(e)\}_{e\in\CE_\e}$, $\{\mbox{Re}W^\eps(x),\mbox{Im}W^\eps(x)\}_{x\in\Lambda_\e}$
are $|\CE_\e|+2|\Lambda_\e|=4\eps^{-2}$ independent real valued Brownian motions all with quadratic covariation 
\[
d[W^\eps(e)]_t=d[\mbox{Re}W^\eps(x)]_t=d[\mbox{Im}W^\eps(x)]_t=\eps^{-2}dt \;.
\]
Note that we have inserted a renormalization term
$- C^{(\eps)} \Phi^\eps$,  as alluded in the Introduction.

Let $\xi^\eps(e)=\dot W^\eps(e)$, 
$\zeta^\eps(x)=\dot W^\eps(x)$ be the time derivatives of the Brownian motions, i.e. the discrete white noises.
We will  write $\xi^\eps_j(e)$ with $j=1$ (resp. $j=2$)
if $e$ is a horizontal (resp. vertical) edge.

\begin{lemma} \label{lem:APhi-discrete}
The equations \eqref{e:DLangevin} satisfied by $(A^\eps,\Phi^\eps)$ has the following explicit form
\begin{equs} 
\partial_{t}A^\eps_{1} (e)
	&= \eps^{-1} \Big( F_{A^\eps} (p_e^S) -F_{A^\eps} (p_e^N) \Big)
	+\eps^{-1}\lambda \,\textup{Im}\Big(
		e^{-i\eps \lambda A_1^\eps(e)} \Phi^\eps(e_+)\bar\Phi^\eps(e_-)\Big)
	+\xi^\eps_1(e) \\
\partial_{t}A^\eps_{2} (e)
	&=   \eps^{-1} \Big( F_{A^\eps} (p_e^E) -F_{A^\eps} (p_e^W) \Big)
	+\eps^{-1}\lambda \,\textup{Im}\Big(
		e^{-i\eps \lambda A_2^\eps(e)} \Phi^\eps(e_+)\bar\Phi^\eps(e_-)\Big)
	+\xi^\eps_2(e) \\
\partial_t \Phi^\eps (x)&= \Delta_{A^\eps}^\eps \Phi^\eps (x) - C^{(\eps)} \Phi^\eps(x)
		+\zeta^\eps(x)
			\label{e:APhi-long}
\end{equs}
%
for every $x\in\Lambda_\e$ and $e\in\vec{\CE}_\e$,
where 
\begin{equ}[e:discCovLap]
\Delta_{A^\eps}^\eps \Phi^\eps (x) \eqdef
\eps^{-2} 
\Big( \sum_{\be}  
  		e^{-i\eps \lambda A^\eps(x,x+\be)}\Phi^\eps(x+\be)
		-4 \Phi^\eps(x)  \Big) 
\end{equ}
with $\be$ summed over $ \{\pm \be_1, \pm\be_2\}$ is called the discrete gauge covariant Laplacian,
and we have used the notation for directed edge $e=(e_-,e_+)$,
and $p_e^S$, $p_e^N$ are the squares immediately below and above 
the horizontal edge $e$ respectively (as shown in the left picture of Figure~\ref{fig:discrete-equ}),
and
$p_e^W$, $p_e^E$ are the squares immediately to the left and right of
the vertical edge $e$ respectively.
\end{lemma}

 The linear part of the equation for $A^\eps_1$ has the form
\begin{equs} [e:defd2d12]
\eps^{-1} & \Big( F_{A^\eps} (p_e^S) -F_{A^\eps} (p_e^N) \Big) \\
&= \eps^{-2} \big( A^\eps_1(\dot e)+A^\eps_1(\underaccent{\Cdot}{e})  -2A^\eps_1(e) \big) 
 - \eps^{-2} \big( A^\eps_2({e}^\Cdot) +A^\eps_2(\leftidx{_{\Cdot}}{e})
	-A^\eps_2({e}_\Cdot)- A^\eps_2(\leftidx{^{\Cdot\!}}{e})\big)
\end{equs}
for every horizontal edge $e$,
where the dotted edges are the edges around $e$ as shown in   the left picture of Figure~\ref{fig:discrete-equ} (which also illustrates the coefficients in front of $A^\eps_j$ 
on the right hand side of \eqref{e:defd2d12}),
and it is clearly a discretization of 
$\partial_2 \partial_2 A_1 - \partial_1\partial_2 A_2$.
The linear part in the equation for $A_2^\eps$
for every vertical edge $e$ is an analogous discretization of
$\partial_1 \partial_1 A_2 - \partial_1\partial_2 A_1$.

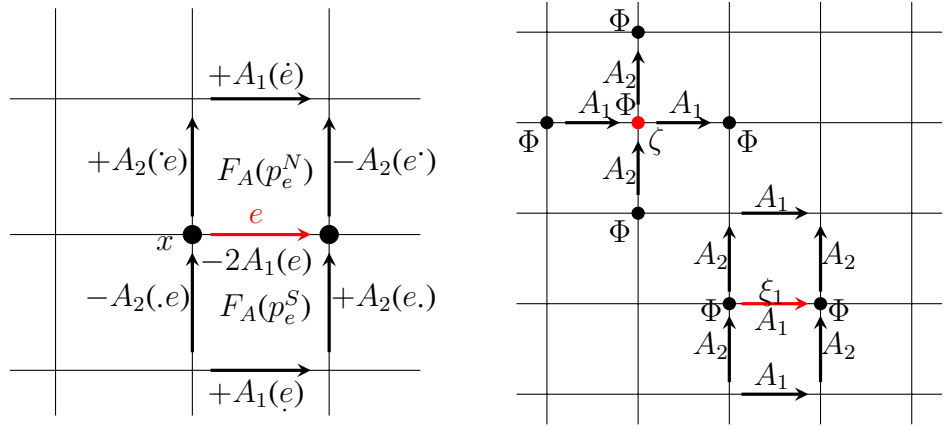
\begin{figure}[h]
\centering
\begin{tikzpicture}[scale=1.2]
  \draw[step=1.5cm] (-2,-3.5) grid (2.5,1);
\node at (0.7,0.25) {$+A_1(\dot e)$}; \draw[->,very thick] (0.2,0) -- (1.3,0) ; 
\node at (-0.6,-0.7) {$+A_2(\leftidx{^{\Cdot\!}}{e})$};
	\draw[->,very thick] (0,-1.3) -- (0,-0.2) ; 
\node at (0.7,-1.8) {$-2A_1(e)$}; \draw[->,very thick,red] (0.2,-1.5) -- (1.3,-1.5) ; 
\node[red] at (0.7,-1.3) {$e$};
\node at (2.1,-0.7) {$-A_2({e}^\Cdot)$};\draw[->,very thick] (1.5,-1.3) -- (1.5,-0.2) ; 
\node at (0.7,-3.3) {$+A_1(\underaccent{\Cdot}{e})$}; \draw[->,very thick] (0.2,-3) -- (1.3,-3) ; 
\node at (-0.6,-2.2) {$-A_2(\leftidx{_{\Cdot}}{e})$};\draw[->,very thick] (0,-2.8) -- (0,-1.7) ; 
\node at (2.1,-2.2) {$+A_2({e}_\Cdot)$};\draw[->,very thick] (1.5,-2.8) -- (1.5,-1.7) ; 
\filldraw (1.5,-1.5) circle (0.1) ;
\filldraw (0,-1.5) circle (0.1) ;
\node at (-0.3,-1.6) {$x$}; 
\node at (0.8,-0.8) {$F_A(p_e^N)$}; 
\node at (0.8,-2.3) {$F_A(p_e^S)$}; 
\end{tikzpicture}
\hspace{3ex}
\begin{tikzpicture}[scale=0.8]
  \draw[step=1.5cm] (-3.5,-3.5) grid (3.5,3.5);
\node at (-1.7,1.8) {$\Phi$}; \filldraw[red] (-1.5,1.5) circle (0.1) ;
\node at (-1.2,1.2) {$\zeta$};
\node at (0.3,1.2) {$\Phi$}; \filldraw (0,1.5) circle (0.1) ;
\node at (-1.8,-0.3) {$\Phi$}; \filldraw (-1.5,0) circle (0.1) ;
\node at (-3.3,1.2) {$\Phi$}; \filldraw (-3,1.5) circle (0.1) ;
\node at (-1.8,3.2) {$\Phi$}; \filldraw (-1.5,3) circle (0.1) ;
\node at (-0.7,1.8) {$A_1$}; \draw[->,very thick] (-1.2,1.5) -- (-0.3,1.5) ; 
\node at (-1.8,0.7) {$A_2 $}; \draw[->,very thick] (-1.5,0.3) -- (-1.5,1.2) ; 
\node at (-2.2,1.8) {$A_1$}; \draw[->,very thick] (-2.7,1.5) -- (-1.8,1.5) ; 
\node at (-1.8,2.3) {$A_2 $}; \draw[->,very thick] (-1.5,1.8) -- (-1.5,2.7) ; 
\node at (0.7,0.25) {$A_1$}; \draw[->,very thick] (0.2,0) -- (1.3,0) ; 
\node at (-0.3,-0.7) {$A_2$};\draw[->,very thick] (0,-1.3) -- (0,-0.2) ; 
\node at (0.7,-1.8) {$A_1$}; \draw[->,very thick,red] (0.2,-1.5) -- (1.3,-1.5) ; 
\node at (0.7,-1.3) {$\xi_1$};
\node at (1.8,-0.7) {$A_2$};\draw[->,very thick] (1.5,-1.3) -- (1.5,-0.2) ; 
\node at (0.7,-2.7) {$A_1$}; \draw[->,very thick] (0.2,-3) -- (1.3,-3) ; 
\node at (-0.3,-2.2) {$A_2$};\draw[->,very thick] (0,-2.8) -- (0,-1.7) ; 
\node at (1.8,-2.2) {$A_2$};\draw[->,very thick] (1.5,-2.8) -- (1.5,-1.7) ; 
\node at (1.8,-1.6) {$\Phi$}; \filldraw (1.5,-1.5) circle (0.1) ;
\node at (-0.3,-1.6) {$\Phi$}; \filldraw (0,-1.5) circle (0.1) ;
\end{tikzpicture}
\caption{The left picture illustrates the linear part of the evolution of $A_1^\e(e)$.
The right picture shows that 
 the evolution of the field $A_1^\e$ on the red edge depends on the noise on that edge and values of the fields surrounding that edge;   the evolution of  the field $\Phi^\e$ on  the red vertex depends on the noise on that vertex and values of the fields surrounding that vertex. Dependence on $\eps$ is dropped for cleaner picture.
 }
\label{fig:discrete-equ}
\end{figure}

\begin{proof}[of Lemma~\ref{lem:APhi-discrete}]
The proof is by straightforward computations. Recall from \eqref{e:d-potential} that $\CH^\e=\CH^\e_F + \CH^\e_D$.
For a horizontal edge $e$ as shown in the picture,
 differentiating $-\eps^{-2}\CH_F^\eps$ w.r.t. $A^\eps_1(e)$ yields precisely \eqref{e:defd2d12}.
Regarding $-\eps^{-2} \frac{\partial\CH^\eps_D}{\partial A^\eps_1(e)}$
the only term that contributes to this derivative is
 \[
- \frac{1}{2} | (D^{A^\eps}_1 \Phi^\eps) (x)|^2 
 =-\frac{\eps^{-2}}{2} \Big|e^{-i\eps \lambda A_1^\eps(e)} \Phi^\eps(x+\be_1) -\Phi^\eps(x)\Big|^2
 \]
 where $e=(x,x+\be_1)$, and its derivative
with respect to $A_1^\eps(e)$ is
\begin{equs}
- &\eps^{-2} \mbox{Re}\Big((-i\eps\lambda)\,e^{-i\eps \lambda A_1^\eps(e)} \Phi^\eps(x+\be_1)
	\Big((e^{i\eps \lambda A_1^\eps(e)} \bar\Phi^\eps(x+\be_1) -\bar\Phi^\eps(x)\Big) \Big) \\
&=
-\eps^{-2}\mbox{Re}\Big((i\eps\lambda)\,e^{-i\eps \lambda A_1^\eps(e)} \Phi^\eps(x+\be_1)
	\bar\Phi^\eps(x)\Big)
=
\eps^{-1}\lambda\,\mbox{Im}\Big(e^{-i\eps \lambda A_1^\eps(e)} \Phi^\eps(x+\be_1)
	\bar\Phi^\eps(x)\Big) \;,
\end{equs}
which is the claimed nonlinear term in the equation for $A_1^\e$.
The derivation for  the equation for $A_2^\e$ follows in the same way.
Turning to the equation for $\Phi^\eps(x)$, note that 
$-\eps^{-2} \frac{\partial\CH^\eps_D}{\partial \Phi^\eps(x)}=-(D_j^{A^\eps} )^*(D_j^{A^\eps} \Phi^\eps) (x) $ which equals
\begin{equs}
    - & \eps^{-2} \Big[    e^{-i\eps \lambda A^\eps(x,x-\be_j)}
	\Big(e^{-i\eps \lambda A^\eps(x-\be_j,x)} \Phi^\eps(x) -\Phi^\eps(x-\be_j)\Big)\\
 & \qquad -\Big(e^{-i\eps \lambda A^\eps(x,x+\be_j)} \Phi^\eps(x+\be_j) -\Phi^\eps(x)\Big) \Big] \\
  & =\eps^{-2} \Big[ e^{-i\eps \lambda A^\eps(x,x-\be_j)}\Phi^\eps(x-\be_j)
  		+e^{-i\eps \lambda A^\eps(x,x+\be_j)}\Phi^\eps(x+\be_j)
		-2 \Phi^\eps(x)  \Big] 
\end{equs}
and summing $j$ over $1,2$ gives the equation for $\Phi^\eps(x)$.
\end{proof}

\begin{remark}
By Taylor expanding the exponential factors into polynomials in the field $A^\e$ and only retaining the leading terms one obtains the ``approximate'' equations
\begin{equs} [e:APhi-long1]
\partial_{t}A_{1}^\e 
	&=\partial_{2}^{2}A_{1}^\e
	-\partial_{1}\partial_{2}A_{2}^\e
	- \frac{ i \lambda}{2} \Big(\bar{\Phi}\partial_{1}\Phi^\e-\Phi^\e\partial_{1}\bar{\Phi}^\e-2i\lambda A_{1}^\e|\Phi^\e|^{2}\Big)
	+\xi_{1}^\e \\
\partial_{t}A_{2}^\e 
	&=\partial_{1}^{2}A_{2}^\e-\partial_{1}\partial_{2}A_{1}^\e
	- \frac{ i \lambda}{2} \Big(\bar{\Phi}^\e\partial_{2}\Phi^\e-\Phi^\e\partial_{2}\bar{\Phi}^\e-2i\lambda A_{2}^\e|\Phi^\e|^{2}\Big)
	+\xi_{2}^\e 	\\
\partial_t \Phi^\e &= 
\Delta\Phi^\e -i \lambda \sum_j\partial_j (A_j^\e \Phi^\e)-i \lambda \sum_j A_j^\e (\partial_j \Phi^\e) 
	- \lambda^2 \sum_j (A_j^\e)^2 \Phi^\e + \zeta^\e \;.
\end{equs}
It is then clearer that this system is discretization of \eqref{e:APhi}.
We will carry out such an expansion in the proof of Lemma~\ref{lem:poly-form} below. 
These equations appear
in the physics review on stochastic quantization \cite[(4.32,4.38)]{damgaard1987},
except that in their  convention there is not the factor $\frac12$ in the equation for $A$.
By simple power counting as in \cite{Regularity}, is critical spatial dimension is four.
Note that the Da Prato-Debussche method \cite{MR2016604} would not work for this model, even it is in two dimensions, because of  the term $\Psi\nabla\Psi$.
\end{remark}

\subsection{Norms and distances}

With the discretized functions and processes as above we define
norms and distances for them.

{\it 1. H\"{o}lder spaces and test function spaces.}
For $r > 0$, we denote by $\CC^r(\R^d)$ the usual H\"{o}lder space on $\R^d$, and by $\CC^r_0(\R^d)$  the space of compactly supported $\CC^r$-functions. 
Denote by $\CB^r_0(\R^d)$  the set of 
$\CC^r_0(\R^d)$-functions  supported in the unit ball centered at the origin and with the $\CC^r$-norm bounded by $1$.
For $\varphi \in \CB^r_0(\R^d)$, $\lambda > 0$ and $x, y \in \R^{d}$ we define 
\begin{equ} [e:rescale-phi]
\varphi_x^\lambda(y) \eqdef \lambda^{-d} \varphi(\lambda^{-1}(y-x)) \;.
\end{equ}
For $\alpha < 0$, we define the space $\CC^\alpha(\R^d)$ to consist of $\zeta \in \CS'(\R^d)$, belonging to the dual space of the space of $\CC^{r}_0$-functions, with $r > -\lfloor \alpha \rfloor$, and such that
\begin{equ}[e:AlphaNorm]
\Vert \zeta \Vert_{\CC^\alpha} \eqdef \sup_{\varphi \in \CB^r_0} \sup_{x \in \R^d} \sup_{\lambda \in (0,1]} \lambda^{-\alpha} |\langle \zeta, \varphi_x^\lambda \rangle| < \infty\;.
\end{equ}
We also define Holder spaces
 over space-time $\R^{d+1}$ as follows.
For $\varphi \in \CB^r_0(\R^{d+1})$, $\lambda > 0$ and $(t,x), (s,y) \in \R\times \R^d$ we define 
\begin{equ} [e:rescale-phi-hat]
\hat\varphi_{(t,x)}^\lambda(s,y) \eqdef \lambda^{-d-2} 
\varphi(\lambda^{-2}(s-t),\lambda^{-1}(y-x)) \;.
\end{equ}
For $\alpha < 0$, we define the space $\hat\CC^\alpha(\R^{d+1})$ to consist of $\zeta \in \CS'(\R^{d+1})$, belonging to the dual space of the space of $\CC^{r}_0$-fucntions, with $r > -\lfloor \alpha \rfloor$, and such that
\begin{equ}[e:AlphaNorm-hat]
\Vert \zeta \Vert_{\hat\CC^\alpha} 
\eqdef \sup_{\varphi \in \CB^r_0(\R^{d+1})}
 \sup_{t \in \R} \sup_{x \in \R^d} \sup_{\lambda \in (0,1]} \lambda^{-\alpha} |\langle \zeta, \hat\varphi_{(t,x)}^\lambda \rangle| < \infty\;.
\end{equ}

{\it 2. The space $\CC^{\delta, \alpha}_{\eta}\bigl([0,T], \R^d\bigr)$.}
For a function $\R \ni t \mapsto \zeta_t$ we define the operator $\delta^{s, t}$ by
\begin{equ}[e:deltaTime]
 \delta^{s, t} \zeta \eqdef \zeta_t - \zeta_s\;,
\end{equ}
and for $\delta > 0$, $\eta \leq 0$ and $T > 0$, we define the space $\CC^{\delta, \alpha}_{\eta}\bigl([0,T], \R^d\bigr)$ to consist of the functions $(0, T] \ni t \mapsto \zeta_t \in \CC^{\alpha}(\R^d)$, such that the following norm is finite
\begin{equ}[e:SpaceTimeNorm]
\Vert \zeta \Vert_{\CC^{\delta, \alpha}_{\eta, T}} \eqdef \sup_{t \in (0, T]} \onorm{t}^{-\eta} \Vert \zeta_t \Vert_{\CC^\alpha} + \sup_{s \neq t \in (0, T]} \onorm{t, s}^{-\eta} \frac{\Vert \delta^{s, t} \zeta \Vert_{\CC^{\alpha - \delta}}}{|t-s|^{\frac{\delta}{2}}}\;,
\end{equ}
where $\onorm{t} \eqdef |t|^{\frac12} \wedge 1$ and $\onorm{t, s} \eqdef \onorm{t} \wedge \onorm{s}$.

\vspace{1ex}

{\it 3. Comparing functions on lattice and continuum.}
We will need to compare discrete functions on the lattice 
(either on the grid points or on the edges)
with their continuous counterparts, and for this
we first introduce the following convention.
For a continuous function $\phi: \R^d\to \R$, a site $x\in \Lambda_\eps$
can be naturally identified with a point $x\in\R^d$, thus giving the meaning of
$\phi(x)$.
For an edge $e = (e_-,e_+) \in \vec{\CE}_\e$, we define $\phi(e) \eqdef \phi(\frac{e_- +e_+}{2})$.

The values of $\phi$ on differences of grid points or edges 
are then defined in a similar way. For instance,
the quantity $\phi(e-\tilde e) $ for two edges $e,\tilde e$ then stands for 
$\phi(\frac{e_- +e_+}{2}-\frac{\tilde e_- +\tilde e_+}{2}) $,
and 
the quantity $\phi(e-x) $ for an edge $e$ and a grid point $x$ then stands for 
$\phi(\frac{e_- +e_+}{2}-x) $.

Given this convention,
in order to compare discrete functions 
 $\zeta^\eps \in \R^{\Lambda_\eps}$ on the dyadic grid
with their continuous counterparts $\zeta \in \CC^\alpha(\R^d)$ with $\alpha \leq 0$, we introduce the following ``distance''
\begin{equ} [e:diffCalpha]
\Vert \zeta; \zeta^\eps \Vert^{(\eps)}_{\CC^\alpha} 
\eqdef \sup_{\varphi \in \CB^r_0} \sup_{x \in \R^d} \sup_{\lambda \in [\eps,1]} \lambda^{-\alpha} |\langle \zeta, \varphi_x^\lambda \rangle - \langle \zeta^\eps, \varphi_x^\lambda \rangle_\eps|\;,
\end{equ}
where 
\begin{equ}[e:DPairing]
 \langle \zeta^\eps, \varphi_x^\lambda \rangle_\eps
   \eqdef \eps^d \sum_{y \in \Lambda_\eps} \zeta^\eps(y) \varphi_x^\lambda(y)\;.
\end{equ}
For discrete functions 
 $\zeta^\eps \in \R^{\CE^j_\e}$ with 
 $j\in\{1,2\}$ on the edges,
we define $\Vert \zeta; \zeta^\eps \Vert^{(\eps)}_{\CC^\alpha} $ in the same way
but with $ \langle \zeta^\eps, \varphi_x^\lambda \rangle_\eps$ defined by
\begin{equ}[e:DPairing-e]
 \langle \zeta^\eps, \varphi_x^\lambda \rangle_\eps
   \eqdef \eps^d \sum_{e \in \CE^j_\e} \zeta^\eps(e) \varphi_x^\lambda(e)\;.
\end{equ}
Here $\varphi_x^\lambda(y)$ and $\varphi_x^\lambda(e)$ are defined according to the aforementioned convention 
as well as \eqref{e:rescale-phi}.
For $\zeta^\eps: \R \times \Lambda_\eps \to \R$ and its continuous counterpart
 $\zeta \in \hat\CC^\alpha(\R^{d+1})$ with $\alpha \leq 0$,
 we also define the ``distance"  $\Vert \zeta; \zeta^\eps \Vert^{(\eps)}_{\hat\CC^\alpha} $
 in the analogous way with the parabolically scaled test function \eqref{e:rescale-phi-hat}
 and scalar product that is discrete in space and continuous in time.
 
 \vspace{1ex}
{\it 4. Comparing processes on lattice and continuum.}
Assume that we are given  a  space-time distribution $\zeta$, 
 and a function $\zeta^\eps$ 
 on $(0,T]\times \Lambda_\eps $ or on $(0,T]\times \CE^j_\eps $ with $j\in\{1,2\}$.
For $\delta > 0$ and $\eta \leq 0$, 
we define
\begin{equ}[e:DHolderDist]
\Vert \zeta; \zeta^\eps \Vert^{(\eps)}_{\CC^{\delta, \alpha}_{\eta, T}}
 \eqdef
  \sup_{t \in (0, T]} \enorm{t}^{-\eta} \Vert \zeta_t; \zeta_t^\eps \Vert_{\CC^\alpha}^{(\eps)} 
  + \sup_{s \neq t \in (0, T]} \enorm{s, t}^{-\eta} 
  	\frac{\Vert \delta^{s, t} \zeta; \delta^{s, t}\zeta^\eps \Vert_{\CC^{\alpha - \delta}}^{(\eps)} }
		{\bigl(|t-s|^{\frac12} \vee \eps\bigr)^{\delta}},
\end{equ}
where $\enorm{t} \eqdef \onorm{t} \vee \eps$ and $\enorm{s, t} \eqdef \enorm{s} \wedge \enorm{t}$, 
and $\Vert \zeta; \zeta^\eps \Vert^{(\eps)}_{\CC^\alpha} $
is defined as above with the discrete pairing defined as in 
 \eqref{e:DPairing} or \eqref{e:DPairing-e},
 depending on the domain where $\zeta^\eps$ is defined.
 
 Furthermore, we define the norms $\Vert \zeta^\e \Vert^{(\eps)}_{\CC^\alpha}$ and
 $\Vert \zeta^\eps \Vert^{(\eps)}_{\CC^{\delta, \alpha}_{\eta, T}}$ in the same way as in \eqref{e:AlphaNorm} and \eqref{e:SpaceTimeNorm}, but using the discrete pairing \eqref{e:DPairing} and \eqref{e:DPairing-e}.
 
If $\zeta^\eps=(\zeta_1,\zeta_2)$ is vector valued or  $\zeta^\eps=\zeta_1+i\zeta_2$ is complex valued, and so is $\zeta$, then $\Vert \zeta; \zeta^\eps \Vert^{(\eps)}_{\CC^{\delta, \alpha}_{\eta, T}} \eqdef \sum_{k=1}^2 \Vert \zeta_k; \zeta^\eps_k \Vert^{(\eps)}_{\CC^{\delta, \alpha}_{\eta, T}}$.
The other norms or distances 
for vector or complex fields 
are defined in the same way.


\subsection{Gauge invariant observables}
\label{sec:observables}

We now define a collection of discrete gauge invariant observables
and state the convergence result for them, as mentioned in the introduction.

As discussed below \eqref{e:discFA}, 
given a discrete vector field $A^\e$,
the curvature field $F^\e_{A^\e}$ is gauge invariant.

We can also define composite scalar fields $\Wick{|\Phi^\eps|^{2n}}$ as follows.
Given a stationary complex valued Gaussian field  $\psi^\e$ on 
$\Lambda_\e$,
recall that for processes $\Phi^\eps$ on $ \Lambda_\e$ 
and its complex conjugate $\bar\Phi^\eps$,
the Wick power (with respect to the Gaussian measure of $\psi^\e$) is given by
\begin{equ}[e:generating-Wick]
\Wick{(\Phi^\eps)^p (\bar{\Phi}^\eps)^q } (x)
=
\frac{\partial^p}{\partial t_1^p} 
\frac{\partial^q}{\partial t_2^q} 
e^{t_1 \Phi^\eps(x) + t_2 \bar\Phi^\eps (x)
	- \frac12 \E \big((t_1 \psi^\eps (x)+ t_2 \bar\psi^\eps(x))^2 \big)}
\Big|_{t_1=t_2=0}
\qquad (x\in\Lambda_\e) \;.
\end{equ}
From this it is easy to see that  for any integer $n\geq 1$,
the quantity
$\Wick{|\Phi^\eps|^{2n}} = \Wick{(\Phi^\eps\bar\Phi^\eps)^n }$
is a polynomial in $\Phi^\eps\bar\Phi^\eps$ of order $n$.
Therefore $\Wick{|\Phi^\eps|^{2n}}$ is gauge invariant
since $\Phi^\eps\bar\Phi^\eps$ is invariant under the gauge transformation \eqref{e:disc-gauge-trans}.
In the sequel $\Phi^\e$ will be taken as the solution to our discrete equation and thus time dependent, and  $\psi^\e \eqdef K^\e \ast_\e \zeta^\eps$,
where $K^\eps$ is the  truncated discrete heat kernel as given in \cite[Lemma~5.4]{hairer2015discrete}.

Another type of composite field is given by 
\begin{equ} [e:def-obs-PhiDAPhi]
\bar\Phi^\eps (e_-) (D_j^{A^\eps} \Phi^\eps)(e)
-C^\e_{\bar\Phi D^A \Phi}
\end{equ}
 for $j\in\{1,2\}$ and $e=(e_-,e_+)\in\vec\CE^j_\e$,
 where $C^\e_{\bar\Phi D^A \Phi}$ is a renormalization constant.
By \eqref{e:DCovariance},
this field is also gauge invariant under the transformation \eqref{e:disc-gauge-trans}.

Besides these local observables,
we also define some global observables as follows.
To start with, given a simple $\CC^2$ curve $C$ 
with a parametrization $\br(\si)$ for $\si\in[0,1]$ with $|\dot \br(\si)|\neq 0$,
and $\eps>0$ sufficiently small (depending on $C$),
we introduce a notion of {\it regular approximation} of $C$ as follows.
First, we find {\it a collection $P^C_{\e}$ of squares} in $\CP_\e$
which covers the curve $C$,  in an inductive way as follows.
Letting $p_1 \in \CP_\e$  be the square that contains $C(\br(0))$,
we impose that $p_1\in P^C_{\e}$.
Assuming that  $p_i\in P^C_{\e}$ for some $i\ge 1$,
if $C$ exits $p_i$ from edge $e$ (generically $C$ will not intersect with $p_i$ at the corner of $p_i$), then we add $p_{i+1} \in \CP_\e$ into the collection $P^C_{\e}$ such that $p_{i+1} \cap p_i = e$, and repeat this procedure until $C$ is completely covered by such squares. 
By construction, provided that $\eps>0$ is sufficiently small, every vertex $x\in\Lambda_\e$
is shared by at most $3$ squares in the collection $P^C_{\e}$.
As an example, in the following picture with a curve $C$ running from left to right as $\si$ increases, it exits $p_i$ from the right edge of $p_i$, so $p_{i+1}$  is the  adjacent square on the right.
\begin{center}
\includegraphics[scale=0.6]{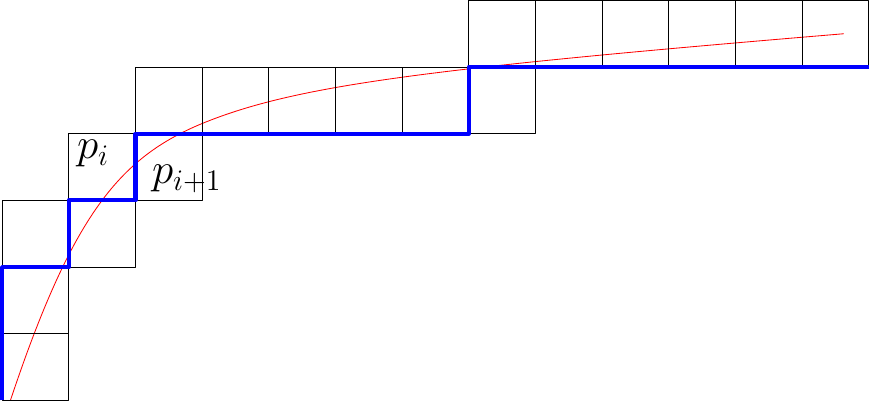}
\end{center}
By construction every point $\br(\si)\in C$ lies in a unique square in $P^C_{\e}$,
therefore this collection $P^C_{\e}$ specifies a partition of $[0,1]$ into subintervals:
\[
0=\sigma_0 <\sigma_1 <\cdots <\sigma_{|P^C_{\e}|}=1
\] 
such that $\si\in [\si_{i},\si_{i+1}]$ if and only if the point $\br(\si)$ is in the corresponding square.
With slight abuse of notation
we will denote this partition of $[0,1]$ also by $P^C_{\e}$.

We then select {\it one} edge from each $p \in  P^C_{\e}$, such that
these edges concatenate forming a discrete curve $C^\e$. We call $C^\e$
 a {\it regular approximation} of $C$.
 In the above picture a regular approximation $C^\e$ is drawn in blue.
 By construction we have a one-to-one correspondence between edges in $C^\e$
 and subintervals in the partition $P^C_{\e}$.
 
 \begin{remark}\label{rem:hori-vert-tang}
For a $\CC^2$ curve $C$, provided that $\eps>0$ is sufficiently small,
for any point $\br(\si) =(\br_1(\si) ,\br_2(\si) )\in C$ such that $\dot{\br}_2(\si)=0$, namely the tangent line is horizontal,
the edge  $e(\si)\in C^\e$ is necessarily an horizontal edge.
Indeed, in this situation a small section of $C$ containing $\br(\si)$ must be covered by 
a few squares in $\CP_\e$  concatenated horizontally,
and by inspection (or simply looking at the above picture) the edges of  $C^\e$ which are associated to these squares 
must also be horizontal,
in order
 to ensure that each square in $\CP_\e$ in  associated with a {\it unique} edge in $C^\e$.
The same holds near a point with a vertical tangent line.
\end{remark}

Let $C\subset \T^2$ be a given  simple $\CC^2$ 
closed curve.
Let  $C^\e$  be a discrete loop which is a regular approximation of $C$.
Let $A^\e\in \R^{\CE_\eps}$ be discrete vector field.
One can then consider the 
loop observables
\begin{equ}[e:def-loop-obs]
\mathcal O_{C,\e} \eqdef
\e \sum_{e\in C^\e} A^\e(e)
\qquad
\mbox{and}
\qquad
\tilde{\mathcal O}_{C,\e} \eqdef
e^{i\e\sum_{e\in C^\e} A^\e(e)}\;.
\end{equ}
These are obviously gauge invariant since $\sum_{C^\e} \nabla f^\e =0$ for any $f\in \R^{\Lambda_\e}$.


\begin{theorem}\label{theo:observables}
Let $(A^\e,\Phi^\e)$ be the local It\^o solution to \eqref{e:APhi-long} with initial data $(\mathring A^\eps, \mathring \Phi^\eps)$.
As $\eps\to 0$, 
the gauge invariant observables 
$F^\eps_{A^\e}$ and 
\begin{equ}
\bar\Phi^\eps(\cdot_-) (D_j^{A^\eps} \Phi^\eps)(\cdot)
- C^\e_{\bar\Phi D^A \Phi}
\end{equ}
for $j\in\{1,2\}$ and  $C^\e_{\bar\Phi D^A \Phi}=O(\e^{-1})$
given in \eqref{e:defC-composite} below
 converge in distribution with respect to the distance 
$\Vert \cdot ; \cdot \Vert_{\CC^{\delta, \alpha-1}_{\bar{\eta}, T_\eps}}^{(\eps)} $.
For any integer $n\geq 1$, the gauge invariant observables
$\Wick{|\Phi^\eps|^{2n}}$
converge in distribution with respect to the distance 
$\Vert \cdot ; \cdot \Vert_{\CC^{\delta, \alpha}_{\bar{\eta}, T_\eps}}^{(\eps)} $, 
for every $\alpha<0$
and $\delta,\bar\eta$ are as in Theorem~\ref{theo:main}.

Moreover, as $\eps\to 0$, the loop observables $\mathcal O_{C,\e}$ 
defined in \eqref{e:def-loop-obs} 
via the solutions $(A^\e,\Phi^\e)$
converge in $\CC([0,T],\R)$
in distribution,
and the loop observables $\tilde{\mathcal O}_{C,\e}$ in \eqref{e:def-loop-obs} 
converge in $\CC([0,T],\C)$
in distribution.
\end{theorem}

\begin{remark}
There are other observables which could be constructed using our method.
For instance, given a two-dimensional subdomain $\Omega$ of the torus $\T^2$,
and let $\Omega_\e$ be a suitable approximation to $\Omega$, which is a union of a collection of squares in $\CP_\e$. 
One can consider the gauge invariance observable
$\e^2 \sum_{p\in \Omega_\e} F^\e_{A^\e}$. 
In view of the above Corollary, the field $ F^\e_{A^\e}$ converges in 
the space of regularity $\alpha-1$ for $\alpha<0$,
so one would expect that the second moment of this observable 
behaves like $\e^4 \sum_{p,\bar p\in \Omega_\e} |p-\bar p|^{2(\alpha-1)}$
which seems to diverge. However, using the definition of $F^\e_{A^\e}$
together with a discrete Green's theorem, this observable
can be re-written in terms of (possibly a sum of, if the boundary of $\Omega$ is not connected) the observable $\mathcal O_{C,\e}$.
 
 Another type of observables is the ``string observables'' which are formally written as ``$\bar \Phi (x) e^{-i\lambda \int_x^y A} \Phi(y)$'' where $\int_x^y A$ denotes the line integral of $A$ along a given curve from $x$ to $y$. Since for {\it fixed} $x,y$ even when $\lambda=0$, the formal expression 
 $\bar \Phi (x)  \Phi(y)$ would be meaningless in the continuum, we may need to define it in an averaged sense.
Assume that $C\subset \T^2$ is a given  simple smooth 
open curve,
and that $C^\e$  is a discrete curve which is a regular approximation of $C$.
For smooth test functions $\phi_1,\phi_2$ on $\T^2$ we define 
\begin{equ}[e:def-string-obs]
\mathcal O_{C,\e}^{\phi_1,\phi_2} \eqdef
\e^2
\sum_{x,y\in C^\e}
\phi_1(x)
\Big(
\bar\Phi^\e(x)\,
e^{ - i\e \lambda \sum_{e\in C^\e(x\to y)} A^\e(e)}\,
\Phi^\e(y)
\Big)
\phi_2(y)
\end{equ}
where $C^\e(x\to y) \subset C^\e$
is the part of $C^\e$ between the two sites $x,y$.
It is also  gauge invariant
since the quantity in the parenthesis is gauge invariant.

A variant of this is the ``smeared string observable'' (for instance \cite{MR836009}) formally written as
\[
\bar \Phi (x) e^{-i\lambda \int_x^y A\cdot \nabla U} \Phi(y)
\]
where $\int_x^y A\cdot \nabla U$ is the line integral
of the scalar function $A\cdot \nabla U$
and $U$ is the Coulomb potential generated by 
a negative charge at $x$ and a positive charge at $y$,
namely $-\Delta U = -\delta_x + \delta_y$.
We do not rigorously construct these observables in the present article.
\end{remark}

\section{Gauge transformations and Ward identity}\label{sec:GaugeTrans}



\subsection{DeTurck trick}

In this section we apply the DeTurck trick, which amounts to introducing a time dependent family of gauge transformations that can turn the system \eqref{e:APhi} 
or \eqref{e:APhi-long}
into a parabolic system. The argument relies crucially on the gauge invariant discretization
discussed in the previous section.

Fix $\eps>0$. Let $(B^\eps, \Psi^\eps)$ solve the following It\^o system: 
\begin{equs} [e:BPsi]
\partial_{t}B^\eps_j (e)
	&=\Delta^\eps B^\eps_j (e)
	+\eps^{-1}\lambda \,\textup{Im}\Big(
		e^{-i\eps \lambda B_j^\eps(e)} \Psi^\eps(e_+)\bar\Psi^\eps(e_-)\Big)
	+\xi^\eps_j (e)    \\
\partial_t \Psi^\eps (x)&=  \Delta_{B^\eps}^\eps \Psi^\eps (x)
 + i\lambda \, \textup{div}^\eps B^\eps (x)\,\Psi^\eps(x)
 	- C^{(\eps)} \Psi^\eps(x)
		+e^{i\lambda \int_0^t \textup{div}^\eps B^\eps(s,x)\,ds}  \zeta^\eps(x)
\end{equs}
for $j=1,2$ 
with the same initial condition $(\mathring A^\e, \mathring \Phi^\e)  $ as that of $(A^\eps,\Phi^\eps)$,
where the directed edge $e=(e^-,e^+)$ and
 $\Delta_{B^\eps}^\eps \Psi^\eps (x)$ is the gauge covariant Laplacian defined in 
\eqref{e:discCovLap}.
In the equation for $B^\e$,
the discrete Laplacian $\Delta^\eps$ acting on a field on the edges is defined in the natural way, namely,
 for a horizontal edge $e$,
\begin{equ} [e:LapB]
\Delta^\eps B^\eps_1 (e)
= \eps^{-2} \Big(  B^\eps_1 (e^N)  +B^\eps_1 (e^S) 
	+ B^\eps_1 (e^E) + B^\eps_1 (e^W) - 4 B^\eps_1 (e)\Big)
\end{equ}
where $e^N$, $e^S$ are the horizontal edges above (``north") and below (``south") $e$ respectively (they are parallel with $e$),
and  $e^W$, $e^E$ are the horizontal edges on the left (``west") and right (``east") of $e$ respectively (they are aligned and end-to-end with $e$).
For a vertical edge $e$ the quantity $\Delta^\eps B^\eps_2 (e)$ is defined in the analogous way.
We emphasize that the last term 
in the equation for $\Psi^\eps$ is understood as the It\^o product
of $\zeta^\eps$ and 
the process $e^{i\lambda \int_0^t \textup{div}^\eps B^\eps(s)\,ds}$
which is adapted to the filtration generated by the 
Brownian motions $W^\eps$ introduced in \eqref{e:DLangevin},
since the It\^o solution $(B^\eps,\Psi^\eps)$ is adapted to this filtration.


For each $\eps>0$, $x\in\Lambda_\e$, $e\in\CE_\e$, $t>0$, we then define
\begin{equs} [e:APhi-via-BPsi]
A^\eps_j(t , e) &\eqdef  
	B^\eps_j(t , e) -\int_0^t \nabla^\eps_j  \textup{div}^\eps B^\eps(s , e)\,ds 
\\
\Phi^\eps(t ,x ) &\eqdef   e^{-i\lambda \int_0^t \textup{div}^\eps B^\eps(s , x)\,ds} \Psi^\eps(t ,x) \;.
\end{equs}
The quantity $\nabla^\eps_j  \textup{div}^\eps B^\eps( e)$ has the following explicit form (by \eqref{e:def-discrete-div} and Remark~\ref{rem:grad-Phi}):
\begin{equs} [e:explicit-DdivB]
\nabla^\eps_1  \textup{div}^\eps B_1^\eps( e)
&= \eps^{-2} \big( B^\eps_1(e\cdot)
	+B^\eps_2({e}^\Cdot)  -B^\eps_2({e}_\Cdot)  -B^\eps_1(e) \big) \\
& \qquad - \eps^{-2} \big( B^\eps_1(e)+ B^\eps_2(\leftidx{^{\Cdot\!}}{e})
 	-B^\eps_2(\leftidx{_{\Cdot}}{e})
	-B^\eps_2(\cdot e) \big)
\end{equs}
for a horizontal edge $e$, which is illustrated in the following picture,
and similarly for a vertical edge.
\begin{equ}[e:pic-grad-div]
\begin{tikzpicture}[scale=1.2,baseline=-25]
  \draw[step=1.5cm] (-2,-3.5) grid (3.5,0.8);
\node at (2.2, -1.25) {$+B_1(e \cdot)$}; \draw[->,very thick] (1.7,-1.5) -- (2.8,-1.5) ; 
\node at (-0.6,-0.7) {$-B_2(\leftidx{^{\Cdot\!}}{e})$};
	\draw[->,very thick] (0,-1.3) -- (0,-0.2) ; 
\node at (0.7,-1.8) {$-2B_1(e)$}; 
	\draw[->,very thick,red] (0.2,-1.5) -- (1.3,-1.5) ; 
\node[red] at (0.7,-1.3) {$e$};
\node at (2.1,-0.7) {$+ B_2({e}^\Cdot)$};\draw[->,very thick] (1.5,-1.3) -- (1.5,-0.2) ; 
\node at (-0.8,-1.25) {$+B_1(\cdot e)$}; 
	\draw[->,very thick] (-1.3,-1.5) -- (-0.2,-1.5) ; 
\node at (-0.6,-2.2) {$+B_2(\leftidx{_{\Cdot}}{e})$};\draw[->,very thick] (0,-2.8) -- (0,-1.7) ; 
\node at (2.1,-2.2) {$-B_2({e}_\Cdot)$};\draw[->,very thick] (1.5,-2.8) -- (1.5,-1.7) ; 
\filldraw (1.5,-1.5) circle (0.1) ;
\filldraw (0,-1.5) circle (0.1) ;
\end{tikzpicture}
\end{equ}

\begin{remark}
The U(1) gauge transformation $g^\eps (t)\eqdef   e^{i\lambda \int_0^t \textup{div}^\eps B^\eps(s)\,ds}$
 satisfies 
\[
\partial_t g^\eps  = \left( i \lambda \textup{div}^\eps B^\eps\right)\,g^\eps
 \quad \mbox{with } g^{(\eps)} (0)=1 \;.
\]
In general gauge theories with non-Abelian gauge groups one should expect that the gauge transformation solves 
a differential equation when applying the DeTruck trick.
\end{remark}

\begin{lemma} \label{lem:nonlin-deturck}
For $\eps>0$, let
$(A^\eps,\Phi^\eps)$ be defined as in \eqref{e:APhi-via-BPsi},
where $(B^\eps,\Psi^\eps)$ is the solution to \eqref{e:BPsi} with initial condition $(\mathring A^\eps, \mathring \Phi^\eps)$. 
Then $(A^\eps,\Phi^\eps)$
satisfy the system \eqref{e:APhi-long} with the same initial condition
 $(\mathring A^\eps, \mathring \Phi^\eps)$. 
\end{lemma}
\begin{proof}
From the definition of the transformation \eqref{e:APhi-via-BPsi} we can immediately see that $(A^\eps,\Phi^\eps)$ and $(B^\eps,\Psi^\eps)$ satisfy the same initial condition.
\footnote{Here the fact that the transformation \eqref{e:APhi-via-BPsi}  is identity at $t=0$
is only a matter of convenience. One could also change the initial condition to a gauge equivalent one.}

For $j=1,2$, by the transformation \eqref{e:APhi-via-BPsi},
one has
\begin{equs}[e:dtA-in-B]
\partial_{t} A^{\eps}_{j} (e)
&= \partial_{t} B^{\eps}_{j} (e)-\nabla^\eps_j  \textup{div}^\eps B^\eps( e) \\
&=\Delta^\eps B^\eps_j -\nabla^\eps_j  \textup{div}^\eps B^\eps( e)
	+\eps^{-1}\lambda \,\textup{Im}\Big(
		e^{-i\eps \lambda B_j^\eps(e)} \Psi^\eps(e_+)\bar\Psi^\eps(e_-)\Big)
	+\xi^\eps_j (e) \;.
\end{equs}
By \eqref{e:defd2d12}, \eqref{e:LapB} and \eqref{e:explicit-DdivB}, one has
(this is actually best seen by comparing the left picture of Figure~\ref{fig:discrete-equ} and the picture \eqref{e:pic-grad-div})
\begin{equ}[e:DeltaB-DdivB]
\Delta^\eps B^\eps_1 -\nabla^\eps_1  \textup{div}^\eps B^\eps( e)
= \eps^{-1} ( F_{B^\eps} (p_e^S) -F_{B^\eps} (p_e^N) )
\end{equ}
and the same is true with all the subscripts $1,2$ swapped.

With \eqref{e:DeltaB-DdivB} at hand we argue that one can replace $(B^\eps,\Psi^\eps)$  by $(A^\eps,\Phi^\eps)$  
on the right hand side of \eqref{e:dtA-in-B} by invoking gauge invariance, and thus show that $A_j^\eps$ satisfies \eqref{e:APhi-long}.
Indeed, by the fact that $F^\e(p)$ defined in \eqref{e:discFA}
annihilate any gradient of a function on the lattice sites, 
together with the fact that the transformation \eqref{e:APhi-via-BPsi}
is precisely shifting the $B^\eps$ by a gradient, one has $F^\e_{A^\eps}=F^\e_{B^\eps}$ and thus \eqref{e:DeltaB-DdivB} equals $\eps^{-1} ( F_{A^\eps} (p_e^S) -F_{A^\eps} (p_e^N) )$.
As for the nonlinear term on the right hand side of \eqref{e:dtA-in-B}, by \eqref{e:APhi-via-BPsi} we have
\begin{equs} 
		e^{-i\eps \lambda B_j^\eps(e)}& \Psi^\eps(e_+)\bar\Psi^\eps(e_-) 
=
	e^{-i\eps \lambda A_j^\eps(e)} 
	e^{-i\eps \lambda \int_0^t \nabla^\eps_j  \textup{div}^\eps B^\eps(s , e)\,ds}
	\cdot   \Psi^\eps(e_+)\bar\Psi^\eps(e_-)\\
&=
	e^{-i\eps \lambda A^\eps(e)} 
	\Big( e^{-i \lambda \int_0^t   \textup{div}^\eps B^\eps(s , e_+)\,ds}
	\Psi^\eps(e_+) \Big)
	\Big( e^{ i \lambda \int_0^t  \textup{div}^\eps B^\eps(s , e_-)\,ds}
	\bar\Psi^\eps(e_-) \Big)  \\
&=
	e^{-i\eps \lambda A^\eps(e)} 
	\Phi^\eps(e_+)\bar\Phi^\eps(e_-)\;. \label{e:change-nonlin-B}
\end{equs}
Therefore $A^\eps$ satisfies \eqref{e:APhi-long}.

Regarding the equation for $\Phi^\eps$, note that by \eqref{e:BPsi} and \eqref{e:APhi-via-BPsi}  one has
\begin{equs}
\partial_t \Phi^\eps(x )&=  e^{-i\lambda \int_0^t \textup{div}^\eps B^\eps(s , x)\,ds} 		\partial_t \Psi^\eps(x) 
- i \lambda  \textup{div}^\eps B^\eps(x )\,\Phi^\eps(x ) \\
&= e^{-i\lambda \int_0^t \textup{div}^\eps B^\eps(s , x)\,ds} 
	\Big( \Delta_{B^\eps}^\eps \Psi^\eps (x) - C^{(\eps)} \Psi^\eps(x) \Big)
		+  \zeta^\eps(x) \;.
\end{equs}
Note that when taking the time derivative of the product on the right hand side of \eqref{e:APhi-via-BPsi},
the usual product rule applies without It\^o correction term. 
Also, note that the term $i \lambda(\textup{div}^\eps B^{\eps}) \Psi^{\eps}$
appearing  in \eqref{e:BPsi} has been used to cancel (after multiplying the exponential factor $ e^{-i\lambda \int_0^t \textup{div}^\eps B^\eps(s , x)\,ds} $)
the term $ - i \lambda(\textup{div}^\eps B^{\eps}) \Phi^{\eps}$ which arises from
the time derivative of  the exponential factor.
 We also have the gauge covariance of the discrete covariant Laplacian, namely, with $f^\e(x) \eqdef \int_0^t \textup{div}^\eps B^\eps(s , x)\,ds$ 
and by 
\eqref{e:discCovLap}, \eqref{e:APhi-via-BPsi}
we have
\begin{equs}  [e:DCovLap]
e^{-i\lambda f^\e(x)} &
\Delta_{B^\eps}^\eps \Psi^\eps (x)
=
e^{-i\lambda f^\e(x)} 
 \Big( \sum_{\be} 
  		e^{-i\eps \lambda B^\eps(x,x+\be)}\Psi^\eps(x+\be)
		-4 \Psi^\eps(x)  \Big) \\
&=
 \sum_{\be} 
  		e^{-i\eps \lambda B^\eps(x,x+\be)}
		e^{i\eps \lambda \nabla^\eps_\be f^\e(x)} 
		\Phi^\eps(x+\be)
		-4 \Phi^\eps(x) 
\\
&=
\sum_{\be} 
  		e^{-i\eps \lambda A^\eps(x,x+\be)}\Phi^\eps(x+\be)
		-4 \Phi^\eps(x) 
=
\Delta_{A^\eps}^\eps \Phi^\eps 
\end{equs}
where $\be$ summed over $ \{\pm \be_1, \pm\be_2\}$.
Note also that
$ C^{(\eps)} e^{-i\lambda \int_0^t \textup{div}^\eps B^\eps(s , x)\,ds}  \Psi^\eps(x) 
= C^{(\eps)}  \Phi^\eps(x)  $.
Therefore $\Phi^\eps$   satisfies the second equation in \eqref{e:APhi-long}.
\end{proof}


At this stage, the factor $e^{i\lambda \int_0^t \textup{div}^\eps B^\eps(s)\,ds} $ in front of $ \zeta^\eps$
in \eqref{e:BPsi} looks harsh to deal with since this factor depends on the solution $B^\eps$ itself. However, we observe that the noise is also gauge invariant in a distributional sense, which is manifest in our It\^o regularization:
as a general  fact, for a vector valued It\^o equation
$d\vec{X_t} = M(t,X) \,d \vec{B_t}$ where the matrix valued process $M$ is adapted to the filtration generated by $B_t$ and satisfies
$ MM^T=\mathrm{Id}$,
then
$\vec{X_t} $ has quadratic variation $t\,\mathrm{Id}$ and thus $\vec{X_t} $ is distributed as a Brownian motion.
We write this as a lemma for our case.

\begin{lemma} \label{lem:Levy}
For each $j\in\{1,2\} $ let $\xi_j$ and $\zeta_j$ be independent space-time
white noises over $L^2(\R\times \T^2)$ which are all independent  of the noises $\{W^\eps(e)\}_{e\in\CE_\e}$ and $\{W^\eps(x)\}_{x\in\Lambda_\e}$ in \eqref{e:DLangevin},
and define discretized white noises 
\begin{equ} [e:average-noise]
\xi_j^\e (t,e)
\eqdef \eps^{-2} \langle \xi(t,\Cdot),\one_{|\Cdot - e|\le \e/2}\rangle\;,
\qquad
\zeta_j^\e  (t,x)\eqdef \eps^{-2} \langle \zeta(t,\Cdot),\one_{|\Cdot - x|\le \e/2}\rangle\;,
\end{equ}
and $\zeta^\e = \zeta^\e _1+i \zeta^\e_2$.
The solution to the following It\^o system
\minilab{e:BPsi-new}
\begin{equs} 
\partial_{t}B^\eps_j (e)
	&=\Delta^\eps B^\eps_j (e)
	+\eps^{-1}\lambda \,\textup{Im}\Big(
		e^{-i\eps \lambda B_j^\eps(e)} \Psi^\eps(e_+)\bar\Psi^\eps(e_-)\Big)
	+\xi^\eps_j (e) \;,   \label{e:BPsi-new-B} \\
\partial_t \Psi^\eps (x)&= \Delta_{B^\eps}^\eps \Psi^\eps (x)
 +i\lambda \textup{div}^\eps B^\eps (x)\,\Psi^\eps(x)
		-C^{(\eps)}  \Psi^\eps(x)   +  \zeta^\eps(x)
					\label{e:BPsi-new-Psi}
\end{equs}
for $ j\in\{1,2\}$ where $x\in\Lambda_\e$, $e=(e_-,e_+)\in\vec{\CE}^j_\e$
with initial condition $(\mathring A^\e, \mathring \Phi^\e)  $
has the same law  as the solution to the  It\^o system \eqref{e:BPsi},
that is, the (strong) solution to \eqref{e:BPsi-new} is a weak solution to \eqref{e:BPsi}.
\end{lemma}

\begin{proof}
Since the process
$e^{i\lambda \int_0^t \textup{div}^\eps B^\eps(s)\,ds}$
is adapted to  the filtration generated by the 
Brownian motions $\{W^\eps(e), \mbox{Re}W^\eps(x),\mbox{Im}W^\eps(x)\}_{e\in\CE_\e, x\in\Lambda_\e}$ introduced in \eqref{e:DLangevin},
using the fact that the transformation $e^{i\lambda \int_0^t \textup{div}^\eps B^\eps(s)\,ds}$ is orthogonal,
one has that  
\[
\Big\{  W_t^\eps(e), 
\mbox{Re} \int_0^t 
	e^{i\lambda \int_0^s \textup{div}^\eps B^\eps(s,x)\,ds} dW_s^\eps(x),
\mbox{Im} \int_0^t 
	e^{i\lambda \int_0^t \textup{div}^\eps B^\eps(s,x)\,ds} dW_s^\eps(x)
\Big\}_{e\in\CE_\e, x\in\Lambda_\e}
\]
has $(\eps^{-2}t )\mathrm{Id}_{4\eps^{-2}}$
as its quadratic covariation, where $\mathrm{Id}_{4\eps^{-2}}$ is a $4\eps^{-2}$ by $4\eps^{-2}$ identity matrix (recall that $\eps^{-1}=2^{N}$ is an integer).
Thus by L\'evy's  characterization  of Brownian motion
the above process is a collection of $4\eps^{-2}$ independent standard 
Brownian motions sped up by $\eps^{-2}$.
The statement of the lemma then immediately follows.
\end{proof}

Thanks  to the above lemma, we will simply work with \eqref{e:BPsi-new},
without changing the distribution of our solution.

\begin{remark}
Another possible approach, at least in the Abelian case we are considering, 
is to project the field $A^\eps$ into the {\it divergence-free} subspace as 
discussed in Remark~\ref{rem:Helmholtz}.
This would be more reminiscent of the procedure of fixing the ``Coulomb gauge''
in the corresponding quantum field theory.
In the end the analysis might turn out to be  similar,
because setting $\textup{div}^\eps B^\eps =0$
 Eq.~\eqref{e:BPsi-new} would 
have the same form as  \eqref{e:APhi-long}   since $\Delta^\eps = \nabla^\eps  \mbox{div}^\eps+ \mbox{curl}^* \mbox{curl}$.
\end{remark}

\subsection{Preparation for regularity structures}

To study \eqref{e:BPsi-new} via regularity structures,
we start by expanding the nonlinearities into polynomial forms with  remainders.

\begin{lemma} \label{lem:poly-form}
Equation \eqref{e:BPsi-new}
can be written in the following form:
\begin{equs} [e:BPsi-poly-complex]
\partial_{t}  B^\eps_j  (e)
	  = \Delta^\eps B^\eps_j (e)
& + \lambda \,\textup{Im}\Big( \nabla_j^\eps\Psi^\eps(e_-)\bar\Psi^\eps(e_-)\Big) \\
&- \lambda^2 \textup{Re} \Big( B_j^\eps(e)  \Psi^\eps(e_+) \bar\Psi^\eps(e_-) \Big) 
+ R_{B_j^\eps}^\eps (e)
	 +\xi^\eps_j(e) \\
\partial_t \Psi^\eps (x)
=  \Delta^\eps  \Psi^\eps(x)
	& -i \lambda  \sum\nolimits_{\be  } B^\eps(x,x+\be) \nabla^\eps_\be \Psi^\eps (x)
	 \\
& - \tfrac12 \lambda^2   \sum\nolimits_{\be} B^\eps(x,x+\be)^2 \,\Psi^\eps(x+\be)
- C^{(\eps)}  \Psi^\eps(x)  
 +R_{\Psi^\eps}^\eps(x) +\zeta^\eps(x)
\end{equs}
\begin{equs} [e:defRBePsix]
R_{B_j^\eps}^\eps  (e)  
& \eqdef \eps^{-1} \lambda\, \textup{Im} 
\Big( \tilde F_1 \left(-i \eps \lambda B_j^\eps(e) \right) 
\Psi^\eps(e_+) \bar\Psi^\eps(e_-) \Big)
\\
R_{\Psi^\eps}^\eps(x)
& \eqdef
\eps^{-2} \sum\nolimits_{\be} 
  		\tilde F_2\left(-i\eps \lambda B^\eps(x,x+\be)\right)
		\,\Psi^\eps(x+\be)
\end{equs}
where $\tilde F_1(z) \eqdef e^z - 1- z$ and $\tilde F_2(z) \eqdef e^z - 1- z-z^2/2$;  
the sums are over  $\be\in \{\pm \be_1,\pm \be_2\}$, and
 $e=(e_-,e_+)\in\vec{\CE}^j_\e$.

\end{lemma}

\begin{proof}
Taylor expanding the factor $e^{-i\eps \lambda B_j^\eps(e)}$ in the nonlinearity of \eqref{e:BPsi-new-B}
 yields
\begin{equs}
\eps^{-1}  &   \lambda\,\mbox{Im}\Big(e^{-i\eps \lambda B_j^\eps(e)} \Psi^\eps(e_+)
	\bar\Psi^\eps(e_-)\Big)\\
&=
\eps^{-1}\lambda \,\textup{Im}\Big( \Psi^\eps(e_+)\bar\Psi^\eps(e_-)\Big)
	-\lambda^2\textup{Im}\Big(iB_j^\eps(e)\Psi^\eps(e_+)\bar\Psi^\eps(e_-)\Big)
	+R_{B_j^\eps}^\eps  (e) \;.
\end{equs}
Writing 
$
\Psi^\eps(e_+) = \Psi^\eps(e_-) + \eps \nabla^\eps_j  \Psi^\eps(e_-) 
$
and noting that $\mbox{Im}\big(\Psi^\eps(e_-)\bar\Psi^\eps(e_-)\big) =0$
one obtains the first equation in \eqref{e:BPsi-poly-complex}.

For each $j\in\{1,2\}$,
expanding  $e^{-i\eps \lambda B^\eps(x,x+\be)}$ in the term $\Delta_{B^\eps}^\eps \Psi^\eps (x)$ of \eqref{e:BPsi-new-Psi}:
\begin{equs}
\eps^{-2} & \Big( \!   \sum_{\be = \pm \be_j }   
  		e^{-i\eps \lambda B_j^\eps(x,x+\be)}\Psi^\eps(x+\be)
		 -2 \Psi^\eps(x)  \Big) 
		 = \Delta_j^\eps  \Psi^\eps(x)
	\\
&\quad
 - i \eps^{-1} \lambda\! \!  \sum_{\be = \pm \be_j }\! \! 
	B_j^\eps(x,x+\be) \,\Psi^\eps(x+\be)  
- \frac{\lambda^2}{2} \! \!  \sum_{\be = \pm \be_j }\! \! 
	B_j^\eps(x,x+\be)^2 \, \Psi^\eps(x+\be)
 +R_{\Psi^\eps}^\eps(x) .
\end{equs}
It is straightforward to check that the second term on the right hand side  is equal to (recall that $\nabla^\eps_j B$ is defined in \eqref{e:nabla-j-Ak})
\begin{equ}[e:a-term-contain-div-cancel]
 -i \lambda  \sum_{\be = \pm \be_j} 
	B^\eps(x,x+\be) \nabla^\eps_\be \Psi^\eps (x)
	-i \lambda \nabla^\eps_j B^\eps(x-\be_j ,x) \Psi^\eps(x) \;.
\end{equ}
The second term in \eqref{e:a-term-contain-div-cancel}, upon summed over $j\in\{1,2\}$, precisely cancel with the term
$i\lambda\textup{div}^\e B^\eps (x)\,\Psi^\eps(x)$
in \eqref{e:BPsi-new-Psi}. The other terms when summing over $j\in\{1,2\}$ precisely
give the right hand side of the second equation in
\eqref{e:BPsi-poly-complex}.
\end{proof}

Recall
that the definition of the spaces $\CC^{\delta, \alpha}_{\eta, T_\eps}$
depends on a choice of regularity $r>0$ of test functions.
In the following lemma, we explicitly record this dependence
and write the spaces as $\CC^{\delta, \alpha,r}_{\eta, T_\eps}$, and we turn the system into a slightly different one that only lives on the lattice sites $\Lambda_\eps$.
\footnote{Alternatively one could 
 introduce modeled distributions on {\it both} vertices and edges, and  two reconstruction operators (one on $\Lambda_\e$ and one on $\CE_\e$).
This however would probably cause notational complication.}

\begin{lemma} \label{lem:move-to-vertices}
Let the real-valued processes $\{ (B_j^\eps(t,x),\Psi_j^\eps(t,x)): x\in\Lambda_\eps,j\in\{1,2\},t\in[0,T_\eps]\}$ be the  solution to
\minilab{e:BPsi-poly}
\begin{equs} 
\partial_{t}  B^\eps_j  & (x)
	  = \Delta^\eps  B^\eps_j (x)
 + \lambda \Big( \Psi_1^\eps(x) \nabla_j^\eps\Psi_2^\eps(x) 
		-\Psi_2^\eps(x) \nabla_j^\eps\Psi_1^\eps(x)  \Big) \\
& \qquad \qquad
- \lambda^2  B_j^\eps(x)  \sum_{k=1,2} \Psi_k^\eps(x+\be_j) \Psi_k^\eps(x)
+ R_{B_j^\eps}^\eps (x)
	 +\xi^\eps_j(x) 
\label{e:BPsi-poly1}
\\
\partial_t \Psi_j^\eps & (x)
 =  \Delta^\eps   \Psi_j^\eps(x)
	  -(-1)^{j} \lambda  \sum_{k=1,2 \atop \ell\neq j} 
	  \Big(
B_k^\eps(x) \nabla^\eps_{\be_k}\Psi_{\ell}^\eps (x)
+ \nabla^\eps_{\-\be_k} B_k^\eps (x) \Psi_{\ell}^\eps (x) 
- \nabla^\eps_{\-\be_k} (B_k^\eps \Psi_{\ell}^\eps )(x)
\Big)
\\
&  \qquad \qquad
- \tfrac12 \lambda^2   \sum_{k=1,2} \Big(
	B^\eps_k(x)^2 \,\Psi_j^\eps(x+\be_k)
	+ B^\eps_k(x- \be_k)^2 \,\Psi_j^\eps(x-\be_k) \Big)
\\
&  \qquad \qquad
 - C^{(\eps)} \Psi_j^\eps(x) +R_{\Psi^\eps_j}^\eps(x) +\zeta_j^\eps(x) \;,
\label{e:BPsi-poly2}
\end{equs}
where $\xi^\eps_j(x) \eqdef \xi^\eps_j(x,x+\be_j) $ and
$\zeta_j^\eps(x)$ as in \eqref{e:average-noise}, and
$R_{B_j^\eps}^\eps  (x) \eqdef R_{B_j^\eps}^\eps  (x,B^\eps,\Psi^\eps) $ and $R_{\Psi_j^\eps}^\eps  (x) \eqdef R_{\Psi_j^\eps}^\eps  (x,B^\eps,\Psi^\eps) $ are defined as 
\begin{equs} [e:defRBxPsix]
R_{B_j^\eps}^\eps  (x,B^\eps,\Psi^\eps)   
& 
\eqdef 
\eps^{-1} \lambda\, \textup{Im} 
\Big( \tilde F_1 \left(-i \eps \lambda B_j^\eps(x) \right) 
(\Psi_1^\eps +i\Psi^\e_2)(x+\be_j) (\Psi_1^\eps - i\Psi^\e_2)(x) \Big)
\\
R_{\Psi_j^\eps}^\eps(x,B^\eps,\Psi^\eps)
& 
\eqdef
\eps^{-2} \!\!\!\!
\sum_{k\in\{1,2\} \atop \be=\pm\be_k } 
  		\textup{Re} \Big(
			 (-i)^{j-1} \tilde F_2\left(-i\eps \lambda B_k^\eps(\mathbf{x})\right)
			\,(\Psi_1^\eps +  i\Psi^\e_2)(x+\be) \Big) \;,
\end{equs}
with $\mathbf{x}=x$ if $\be=+\be_k$ and  $\mathbf{x}=x+\be$ if $\be=-\be_k$, and $\tilde F_1, \tilde F_2$ are the functions defined below \eqref{e:defRBePsix}.

If these solutions  converge in probability
to a limiting process $(B,\Psi)$ on $[0,T^*]$ in
$\CC^{\delta, \alpha,r-1}_{\eta, T_\eps}$ 
for every $\alpha<\bar\alpha$
as $\eps\to 0$
where the stopping time $T^*=\lim_{\eps\to 0} T_\eps$ (in probability),
then the solutions 
to \eqref{e:BPsi-poly-complex} converge in probability
to the same limit $(B,\Psi)$ in
$\CC^{\delta, \alpha,r}_{\eta, T_\eps}$
for every $\alpha<\bar\alpha$.
\end{lemma}

\begin{proof}
With $\Psi^\eps=\Psi^\eps_1+i\Psi^\eps_2$
we readily rewrite \eqref{e:BPsi-poly-complex} into the following real form 
\begin{equs} [e:BPsi-poly-e]
\partial_{t}  B^\eps_j  (e)
	  = \Delta^\eps B^\eps_j (e)
& + \lambda \Big( \Psi_1^\eps(e_-) \nabla_j^\eps\Psi_2^\eps(e_-) 
		-\Psi_2^\eps(e_-) \nabla_j^\eps\Psi_1^\eps(e_-)  \Big) \\
&- \lambda^2  B_j^\eps(e)  \sum_{k=1,2} \Psi_k^\eps(e_+) \Psi_k^\eps(e_-)
+ R_{B_j^\eps}^\eps (e)
	 +\xi^\eps_j(e) \\
\partial_t \Psi_j^\eps (x)
=  \Delta^\eps  \Psi_j^\eps(x)
	&  -(-1)^{j} \lambda  \sum\nolimits_{\be  } B^\eps(x,x+\be) \nabla^\eps_\be \Psi_{3-j}^\eps (x)
	 \\
& - \tfrac12 \lambda^2   \sum\nolimits_{\be} B^\eps(x,x+\be)^2 \,\Psi_j^\eps(x+\be)
 -C^{(\eps)}  \Psi_j^\eps(x)  
  +R_{\Psi^\eps_j}^\eps(x) +\zeta_j^\eps(x)
\end{equs}
where  $\be$
is summed over $\{\pm \be_1,\pm \be_2\}$.
Here $R_{B_j^\eps}^\eps(e)$ is as in \eqref{e:defRBePsix},
and $R_{\Psi_j^\eps}^\eps(x)$ is 
such that $R_{\Psi^\eps}^\eps(x) $ in \eqref{e:defRBePsix} equals $R_{\Psi_1^\eps}^\eps(x) +iR_{\Psi_2^\eps}^\eps(x) $,
where in \eqref{e:defRBePsix} every incidence of   $\Psi^\eps$
should be replaced by $\Psi^\eps_1+i\Psi^\eps_2$.

 Denote by $(\tilde B^\eps, \tilde \Psi^\eps)$ the solution to \eqref{e:BPsi-poly}, and
by $(B^\eps, \Psi^\eps)$ the solution to \eqref{e:BPsi-poly-e}. 
Noting that \footnote{Be cautious that product rule $\partial(fg)=\partial f g+f\partial g$ does not exactly hold on lattice.}
\begin{equs}
- \nabla^\eps_{\-\be_k} (B_k^\eps \Psi_{3-j}^\eps )(x)
&+ B_k^\eps(x) \nabla^\eps_{\be_k}\Psi_{3-j}^\eps (x)
+ \nabla^\eps_{\-\be_k} B_k^\eps (x) \Psi_{3-j}^\eps (x)
\\
&=
B_k^\eps(x) \nabla^\eps_{\be_k} \Psi_{3-j}^\eps (x)
		-B_k^\eps(x-\be_k) \nabla^\eps_{\-\be_k} \Psi_{3-j}^\eps (x)  \;,
\end{equs}
one can see that 
in going from \eqref{e:BPsi-poly-e} to \eqref{e:BPsi-poly} we have just
re-indexed 
 the collection of processes $B^\eps$ by
\[
\tilde B_j^\eps(t,x) \eqdef B_j^\eps(t,(x,x+\be_j))  \;,
\qquad
\mbox{i.e.}\quad
B^\eps_j(t,e)=\tilde B_j^\eps(t,e^-)
\quad \mbox{where } e=(e^-,e^+) \;,
\]
and $ \tilde \Psi_j^\eps(t,x) \eqdef \Psi_j^\eps(t,x)$ for every $t,x$ and every realization of the noises. 

By assumption we have 
$\lim_{\eps\to 0}
\Vert B; \tilde B^\eps 
	\Vert^{(\eps)}_{\CC^{\delta, \alpha,r-1}_{\eta, T_\eps}}=0$
in probability. To control
$\Vert B;  B^\eps \Vert^{(\eps)}_{\CC^{\delta, \alpha,r}_{\eta, T_\eps}}$,
we first estimate the first term on the right hand side of \eqref{e:DHolderDist}.
Write
\begin{equs}
\Big| \langle B_j^\eps(t) & , \varphi_x^\lambda \rangle_\eps 
 	-  \langle B_j(t), \varphi_x^\lambda \rangle \Big| \\
& \le
\Big| \langle B_j^\eps(t), \varphi_x^\lambda \rangle_\eps 
		-\langle \tilde B_j^\eps(t), \varphi_x^\lambda \rangle_\eps\Big|
+  \Big|  \langle \tilde B_j^\eps(t), \varphi_x^\lambda \rangle_\eps
 	-  \langle B_j(t), \varphi_x^\lambda \rangle \Big|  \;.
\end{equs}
By convergence of $ \tilde B^\eps$ the second term multiplied by $\lambda^{-\alpha} \enorm{t}^{-\eta} $ goes to zero uniformly in
$\varphi \in \CB^r_0$, $x \in \R^d$ and $\lambda \in [\eps,1]$. The first term equals 
(recall that 
$\varphi_x^\lambda(e)=\varphi_x^\lambda(\frac{e^-+e^+}{2})$)
\begin{equs}
\Big|
 \eps^d \sum_{e \in \CE^j_\e}&   B_j^\eps(t,e) \varphi_x^\lambda(e)
 - \eps^d \sum_{y \in \Lambda_\eps} \tilde B_j^\eps(t,y) \varphi_x^\lambda(y)\Big| \\
&= \Big|\eps^d \sum_{y \in \Lambda_\eps} \tilde B_j^\eps(t,y) 
	\Big( \varphi_x^\lambda(y+\tfrac{\eps}{2}\be_j)  -\varphi_x^\lambda(y)  \Big) \Big|  \;.
\end{equs}
Define 
$\psi(z)\eqdef \varphi(2z+\tfrac{\eps}{2\lambda}\be_j) - \varphi(2z)$.
Since $\varphi$ is supported in the unit ball centered at origin, so is $\psi$ for $\eps\le \lambda$.
One can also prove that there exists a constant $c_r$ only depending on $r$
such that $\|\psi\|_{\CC^{r-1}}\le c_r \tfrac{\eps}{\lambda}\|\varphi\|_{\CC^{r}}\le c_r \tfrac{\eps}{\lambda}$,
namely
$c_r^{-1} \tfrac{\lambda}{\eps}\psi \in \CB_0^{r-1}$.
Thus the above quantity is bounded  by
\[
\Big|\eps^d \sum_{y \in \Lambda_\eps} \tilde B_j^\eps(t,y) 
	\lambda^{-d} \psi(\tfrac{y-x}{2\lambda}) \Big|
\lesssim
\Big|\eps^d \sum_{y \in \Lambda_\eps} \tilde B_j^\eps(t,y) 
	(c_r^{-1}\tfrac{\lambda}{\eps}\psi)^{2\lambda}_x (y) \tfrac{\eps}{\lambda}\Big|
\lesssim  \lambda^\alpha \tfrac{\eps}{\lambda} |t|_\eps^\eta
\lesssim \eps^\theta \lambda^{\alpha-\theta}|t|_\eps^\eta
\]
for any $\theta\in(0,1)$, uniformly in $\lambda\in[\eps,1]$,
where we used  the uniform bound on $\tilde B_j^\eps$.
So choosing $\theta$ small enough we obtain that 
$ \lim_{\eps\to 0} 
\sup_{t \in (0, T]} \enorm{t}^{-\eta} \Vert B(t); B^\eps(t) \Vert_{\CC^\alpha}^{(\eps)} =0$ for every $\alpha<\bar\alpha$.

The second term on the right hand side of \eqref{e:DHolderDist} can be bounded in the same way.
\end{proof}

%
%
%

\begin{proposition}\label{prop:parabolic}
Let $(B^\eps(t), \Psi^\eps(t))_{t\ge 0}$ be the  solution to the system \eqref{e:BPsi-poly}.
Assume that the initial data $(B^\eps(0), \Psi^\eps(0)) = (\mathring A^\eps,\mathring \Phi^\eps)$    satisfies
\eqref{e:init-data-conv}.
Then  for every $\alpha < 0$ there is a sequence of renormalisation constants $C^{(\eps)} =O( \log\eps)$  and a sequence of stopping times $T_\eps$ satisfying $\lim_{\eps \to 0} T_\eps = T^*$ in probability such that, for every $\bar{\eta} < \eta \wedge \alpha$, and for any $\delta > 0$ small enough, one has   the following limit in probability:
\begin{equ} [e:Convergence-Prob]
\lim_{\eps \to 0} 
	\Vert   (\Psi^\eps,B^\eps),( \Psi,B)  \Vert^{(\eps)}_{\CC^{\delta, \alpha}_{\bar{\eta}, T_\eps}}
 = 0\;.
\end{equ}
\end{proposition}

\begin{proof}[of Theorem~\ref{theo:main}]
The proof of this main theorem
is an immediate consequence of  Proposition~\ref{prop:parabolic} together  with
Lemma~\ref{lem:nonlin-deturck}, \ref{lem:Levy}, \ref{lem:poly-form} and \ref{lem:move-to-vertices}.
\end{proof}

The rest of the paper is devoted to proving Proposition~\ref{prop:parabolic} and convergence of observables.
Before introducing the machinery of regularity structures
 we derive an important identity that will used later for cancellation of renormalization in Section~\ref{sec:mom} as well as showing convergence of certain observables in Section~\ref{sec:conv-obs}.

\subsection{Ward identity}
\label{sec:Ward}

A reader who is familiar with renormalization 
techniques would imagine that the equations  \eqref{e:BPsi-new} or \eqref{e:BPsi-poly}
would need more renormalizations besides the mass renormalization 
 $-C^{(\e)} \Psi^\eps(x)$.
In particular,
in view of the term $B_j^\e \Psi^\e \bar\Psi^\e$,
  a mass renormalization  for $B_j^\eps$
would seem to be necessary
in order to counter the divergence of $\E [\Psi^\e \bar\Psi^\e]$. 
However, with a mass renormalization  for $B_j^\eps$
the proof of Lemma~\ref{lem:nonlin-deturck} would break down since 
the type of argument as in \eqref{e:change-nonlin-B}
which allows us to replace $(B^\eps,\Psi^\eps)$ by $(A^\eps,\Phi^\eps)$
would not work for a term $\tilde C^\e B^\eps$ for some constant $\tilde C^\e$.
We will see that actually there will be several contributions to a mass renormalization  for $B_j^\eps$
and these different contributions cancel out.
This cancellation is due to gauge symmetry.
One can prove such cancellation by some elementary tricks (see Remark~\ref{rem:fav-id}),
but here we derive a version of ``Ward identity'' 
which will help us better understand this cancellation; such identities often arise in quantum field theories with symmetries.
\footnote{After we posted the first version of the paper we learned that \cite{bruned2019geometric} developed a more systematic way of exploring the relation between symmetries and renormalizations which allows the authors to treat much more complicated models.}

Fix $\eps>0$.
Consider Eq.\eqref{e:BPsi-new}, 
but now on $\Lambda_\e^M$ which is the discrete torus of length size $M\gg 1$ and lattice spacing $\e$, and
 with initial condition 
 \begin{equ} [e:zero-init-data]
 B^\eps (-T)=0\;,\qquad \Psi^\eps (-T)=0
 \end{equ}
at time $-T$ for $T>0$. 
Once we have obtained the desired identity we will take the infinite volume limit $M\to \infty$ and send $T\to \infty$. 
Let $\CE^{M,j}_\e$ for $j\in\{1,2\}$ denote the sets of horizontal and vertical edges of $\Lambda_\e^M$.
With a slight abuse of notation we write $P_M^\e$ for the transition probability of continuous time random walk 
 either  on the vertices $\Lambda_\e^M$, or  the edges $\CE^{M,j}_\e$ for some $j\in\{1,2\}$;
they are the essentially same kernel since $\Lambda_\e^M$ and $\CE^{M,j}_\e$
only differ by a small translation,
 and it will be clear which case is under consideration from the context.


Fix a function $f^\e$ on $\Lambda_\e^M$ which is independent of time.
Let 
\[
B_f^\eps = B^\eps + \nabla^\eps f^\e\;,
\qquad
\hat\Psi_f^\eps = \Psi^\eps e^{i\lambda f^\e}\;,
\]
where $(B^\e,\Psi^\e)$ the the solution to the above initial value problem.
By \eqref{e:zero-init-data} one has
\[
 B_f^\eps (-T)=\nabla^\eps f^\e \;, \qquad \hat\Psi_f^\eps (-T)=0\;.
\]
%
By gauge invariance of the nonlinearity as shown in \eqref{e:change-nonlin-B}, one has  
\begin{equs} 
B^\eps_{f,j} (t,e)
&=
B^\eps_j (t,e) + \nabla_j^\eps f^\e(e) \\
&=   P_M^\eps \ast_{\eps,T} \Big( 
	\eps^{-1}\lambda \,\textup{Im}\Big(
		e^{-i\eps \lambda  B_{f,j}^\eps(\Cdot)}  \hat\Psi_f^\eps(\Cdot_+)\bar{\hat\Psi}_f^\eps(\Cdot_-)\Big)
	+\xi^\eps_j (\Cdot)  \Big) (t,e)
+ \nabla^\eps_j f^\e (e) \;.
\end{equs}
Here, 
 for $\Lambda \in \{\Lambda^M_\e,\CE^{M,1}_\e,\CE^{M,2}_\e \}$, $x\in\Lambda$,
and any space-time function $g^\e \in C(\R,\R^\Lambda)$ we have introduced the convolution notation
\[
P_M^\e \ast_{\e,T}  g^\e (t,x)
	\eqdef \int_{-T}^\infty \sum_{y\in\Lambda}   P_M^\eps(t-s,x-y)  g^\e(s,y) \,ds \;.
\]
By the covariance property \eqref{e:DCovLap} of the covariant Laplacian one also has 
\begin{equ} 
\partial_t \hat\Psi_f^\eps (x)
=
  e^{i\lambda f^\e(x)}  \partial_t \Psi^\eps (x)
  = \Delta_{B_f^\eps}^\eps \hat \Psi_f^\eps (x)
 +i\lambda \textup{div}^\eps (B_f^\eps - \nabla^\eps f^\e)(x)    \,\hat \Psi_f^\eps(x)
		+   \hat\zeta^\eps(x)\;,
\end{equ}
where $\hat\zeta^\eps(x) \eqdef e^{i\lambda f^\e(x)} \zeta^\eps(x)$.
Let $\Psi_f^\e$ solve
\begin{equ}[e:Psi-tilde]
\partial_t \Psi_f^\eps (x)
  = \Delta_{B_f^\eps}^\eps  \Psi_f^\eps (x)
 +i\lambda \textup{div}^\eps (B_f^\eps - \nabla^\eps f^\e)(x)    \,\hat \Psi_f^\eps(x)
		+   \zeta^\eps(x)\;,
\end{equ}
  with initial condition $\Psi_f^\e(-T)=0$. We have 
  $(B_f^\e, \Psi_f^\e) \stackrel{law}{=} (B_f^\e, \hat \Psi_f^\e)$. 

The perturbative expansion in $\lambda$ for $(B_f^\e,\Psi_f^\e)$ will now depend on $f^\e$.
For  the ``0-th order'' perturbation we have
\begin{equs}[e:def-bj-psi]
b_j^\e = b_j^\e(f^\e) & \eqdef  B^\eps_{f,j} |_{\lambda=0} 
= P_M^\eps \ast_{\e,T} \xi^\eps_j
+ \nabla^\eps_j f^\e\;,
\\
\psi^\e  &\eqdef  \Psi_f^\e  |_{\lambda=0} = 
	 P_M^\e \ast_{\e,T} \zeta^\e \;.
\end{equs}
(We drop the $M$ dependence in the notation $b_j^\e $ and $\psi^\e$ for simplicity.)
We will also need the ``1st order'' perturbation of $\Psi_f^\e$:
differentiating \eqref{e:Psi-tilde} with respect to $\lambda$ and by similar computation as in proof of Lemma~\ref{lem:poly-form}, one can  check  that
\begin{equ} 
\partial_t \left(\partial_\lambda \Psi_f^\e(x) |_{\lambda=0}\right) 
= \Delta^\e \partial_\lambda \Psi_f^\e(x)  |_{\lambda=0}
 -i  \sum\nolimits_{\be  }  B_f^\eps(x,x+\be) \nabla^\eps_\be \Psi_f^\eps (x) |_{\lambda=0}
 - i\Delta^\eps f^\e (x)    \, \Psi_f^\eps(x) |_{\lambda=0}
\end{equ}
where $\be\in\{\pm\be_1,\pm\be_2\}$, namely
\begin{equ} [e:1st-der-Psi]
\partial_\lambda \Psi_f^\e |_{\lambda=0} 
= -i P_M^\e \ast_{\e,T}  \Big( \sum\nolimits_{\be  } b^\eps(\cdot,\cdot+\be) \nabla^\eps_\be \psi^\eps 
+\Delta^\eps f^\e     \, \psi^\eps 
 \Big)\;.
\end{equ}
Note that this quantity depends on $f^\e$ via $b^\e$ and $\Delta f^\e$.

For any $x\in\Lambda^M_\e$ and $t\ge -T$ consider the following observable 
\begin{equ}[e:Ward-observ]
Z^\e_j (t,x) \eqdef 
\eps^{-1} \,\textup{Im}\Big(
		e^{-i\eps \lambda  B_{f,j}^\eps(t,e)}  \Psi_f^\eps(t,x+\be_j)\bar{\Psi}_f^\eps(t,x)\Big)
\end{equ}
where 
$e=\{x,x+\be_j\}$.

Applying \eqref{e:def-bj-psi} and \eqref{e:1st-der-Psi} we compute  (dropping the variable $t$ for simplicity of notation)
\begin{equs}
\partial_\lambda  &  Z^\e_j  (t,x) |_{\lambda=0}
=  
\textup{Re}\Big(
		 b_j^\e  (e)  \psi^\e (x+\be_j)  \bar{\psi }^\e(x) \Big) \\
 & \qquad -  
 \eps^{-1}  \textup{Re}\Big(
		 P_M^\e \ast_{\e,T} \Big(  \sum\nolimits_{\be  }  b^\eps(\cdot,\cdot+\be) \nabla^\eps_\be \psi^\eps 
		 +\Delta^\eps f^\e     \, \psi^\eps   \Big)(x+\be_j)  \cdot \bar{\psi}^\e(x)\Big)
\\
& \qquad  + 
\eps^{-1} \textup{Re}\Big(
		\psi^\e(x+\be_j) \cdot  P_M^\e \ast_{\e,T} \Big(  \sum\nolimits_{\be  } b^\eps(\cdot,\cdot+\be) \nabla^\eps_\be  \bar\psi^\eps  
		+\Delta^\eps f^\e  \, \bar\psi^\eps \Big)(x)\Big) 
\\
&=  
 - \sum_{k} 
		 b_j^\e  (e)  \psi^\e_k (x+\be_j)  \psi_k^\e(x)
\\
& \qquad - \sum_{k}  
		 \nabla^\e_j P_M^\e \ast_{\e,T} \Big(  \sum\nolimits_{\be  } b^\eps(\cdot,\cdot+\be) \nabla^\eps_\be \psi_k^\eps 
		 +\Delta^\eps f^\e  \, \psi_k^\eps     \Big)(x)  \cdot  \psi_k^\e(x)
\\
& \qquad  + \sum_{k} 
		\nabla^\e_j \psi_k^\e(x) \cdot  P_M^\e \ast_{\e,T} \Big(  \sum\nolimits_{\be  }  b^\eps(\cdot,\cdot+\be) \nabla^\eps_\be \psi_k^\eps  +\Delta^\eps f^\e  \, \psi_k^\eps \Big)(x)  
\end{equs}
where $k$ sums over $\{1,2\}$, $\psi^\e = \psi^\e_1+ i \psi^\e_2$.
Here,  we have expanded a factor in the third line as
$\psi^\e(x+\be_j) =\psi^\e(x) +\e \nabla^\e_j \psi^\e(x) $,
and expanded  a factor in the second line as $P_M^\e \ast_{\e,T} ( \cdots )(x+\be_j) = P_M^\e \ast_{\e,T} ( \cdots )(x) + \e \nabla^\e_j P_M^\e \ast_{\e,T} ( \cdots )(x)$;
the two terms of order $\e^{-1}$  cancel each other,
therefore the factor $\e^{-1}$ does not show up in the 
last two lines of the above equation.

It is important to note that $Z^\e_j (t,x)$ is
gauge-invariant in law, in the sense that although $( B_f^\e, \Psi_f^\e)$ depends on $f^\e$,
the law of the quantity $Z^\e_j (t,x)$ is independent of $f^\e$.
 Indeed, $Z^\e_j (t,x)$ would be deterministically gauge invariant
  by the calculation \eqref{e:change-nonlin-B} if $\Psi_f^\e$ was replaced by $\hat\Psi_f^\e$; on the other hand
    $(B_f^\e, \Psi_f^\e) \stackrel{law}{=} (B_f^\e, \hat \Psi_f^\e)$.
Therefore the expectation of $\partial_\lambda    Z^\e_j  (t,x) |_{\lambda=0}$  is  independent of $f^\e$.

Replacing $f^\e$ by $\alpha f^\e$ in the above arguments,
the expectation of $\partial_\alpha \partial_\lambda Z^\e_j  (t,x) |_{\lambda=0}$  must then be zero, namely,
\begin{equs}
 \E &\Big[
-  \Big\langle \nabla_j^\e f^\e ,\,  
	\sum\nolimits_{k} \delta(\cdot-x) \psi^\e_k (\cdot +\be_j)  \psi_k^\e (\cdot) \Big\rangle_{\Lambda_\e^M} 
\\
&  - \sum_\be \Big\langle  \nabla^\e_\be f^\e ,\, \sum_{k}  
		 \psi_k^\e(x) \int_{-T}^\infty \nabla^\e_j P_M^\e (t-s,x-\cdot)  \nabla^\eps_\be \psi_k^\eps(s,\cdot)\,ds
	\Big\rangle_{\Lambda_\e^M} 
\\
&  + \sum_\be \Big\langle  \nabla^\e_\be f^\e ,\, \sum_{k}  
		\nabla^\e_j \psi_k^\e(x) \int_{-T}^\infty P_M^\e (t-s,x-\cdot)  \nabla^\eps_\be \psi_k^\eps(s,\cdot)\,ds
	\Big\rangle_{\Lambda_\e^M}  \Big]
\\
&  +  \Big\langle \Delta^\e f^\e ,\,\CU_j^\e
\Big\rangle_{\Lambda_\e^M} 
=0
\end{equs}
where all the functions without time variables are evaluated at $t$, and the function $\CU_j^\e(y)$, whose precise form will not actually matter much in the sequel, is given by
 \begin{equs}
\sum_{k}   \int_{-T}^\infty \Big(
P_M^\e(t-s,x-y) \E\Big[ \psi^\e_k  (s,y)  \,\nabla^\e_j\psi_k^\e(x) \Big]
 -\nabla^\e_j P_M^\e(t-s,x-y) \E\Big[ \psi^\e_k  (s,y)  \,\psi_k^\e(x)\Big]
  \Big)\,ds \;.
\end{equs}
Summing over $j\in\{1,2\}$ and integrating by parts,
noting that $f^\e$ is arbitrary, one obtains
\begin{equ} [e:Ward]
(\textup{div}^\e \CV^\e)(t,y) + \Delta^\e \CU^\e(t,y) =0 
	\qquad \forall y\in \Lambda^M_\e
\end{equ}
where the vector field $\CV^\e= (\CV^\e_1,\CV^\e_2)$ is defined as
\begin{equs} [e:def-vecfield-V]
\CV^\e & (y,y+\be_\ell) \eqdef
-  \sum_{k} \delta(y-x) \E\Big[ \psi^\e_k (t,y +\be_\ell)  \psi_k^\e(t,y)\Big] \\
& \quad
- \sum_{\be\in\{0,\be_\ell\}} \sum_{k,j}  
		\int_{-T}^\infty \nabla^\e_j P_M^\e (t-s,x-y-\be) \, \E\Big[ \nabla^\eps_\ell \psi_k^\eps(s,y) \psi_k^\e(t,x) \Big]\,ds
\\
& \quad + \sum_{\be\in\{0,\be_\ell\}} \sum_{k,j}  
		 \int_{-T}^\infty P_M^\e (t-s,x-y-\be) \,\E\Big[ \nabla^\eps_\ell \psi_k^\eps(s,y) \nabla^\e_j \psi_k^\e(t,x)\Big]\,ds
\end{equs}
for $\ell\in\{1,2\}$, and the function $\CU^\e \eqdef \sum_{j=1}^2\CU^\e_j$;
we suppressed their dependence on $M,T$ for cleaner notation.
Note that the left hand side of \eqref{e:Ward} is analytic in $t>-T$ so once we have obtained  \eqref{e:Ward}
for small $t-(-T)$ we have it for any $t>-T$.
We then take an infinite volume limit $M\to \infty $   as well as $T\to \infty$,
so that \eqref{e:Ward} holds with the infinite volume heat kernel $P^\e$ in place of $P^\e_M$ and the time integrals over entire $\R$.
We then have that there exists a function $\CW^\e$ such that
\footnote{This is by a discrete version of Helmholtz decomposition \eqref{e:Helmholtz} and its proof is exactly the same as its continuous version and thus is omitted.}
\[
\CV^\e + \nabla^\e \CU^\e = \textup{curl}^*\CW^\e  \quad \big(= (-\nabla^\e_2 \CW^\e, \nabla^\e_1 \CW^\e)\big)  \;.
\]
It is easy to see that $\CU^\e$ and $\CV^\e$ both decay to zero at infinity, so $\CW^\e$ must go to a constant at infinity;
we can well shift $\CW^\e$ by constant so that $\CW^\e$ also decays to zero at infinity. Summing the above identity over horizontal or vertical edges of  $\Z^2_\e$, one has
\begin{equ} [e:Ward-use]
\sum_{e\in \CE^\ell(\Z^2_\e)} \CV^\e (e)=0\;,
\qquad
\ell\in\{1,2\}\;.
\end{equ}
This is precisely what we will need for cancellation of mass renormalization of the gauge field
as well as the convergence of the composite field observable (see Lemma~\ref{lem:proof-PhiDAPhi}).

\begin{remark}
Identities of the type \eqref{e:Ward} are discovered in quantum field theories and usually referred to as Ward or Ward-Takahashi identities, see  \cite[Eq.~(A.3)]{BrydgesGauge2} and \cite[Eq.~(2.27)]{MR701926} for examples of  such identities in the context of quantum Abelian gauge theory (i.e. Higgs model). Note that the ``useless'' term of the form $\Delta^\e \CU^\e$ did not appear in this reference; this is probably because their gauge fixing condition amounts to imposing the gauge field to be divergence free, 
and without the div term  in \eqref{e:Psi-tilde} 
the term $\Delta^\e \CU^\e$  would not appear.
For a reader who is familiar with quantum electrodynamics (QED),  \eqref{e:Ward} (without the term $\Delta^\e \CU^\e$) is reminiscent to the QED Ward identity $\sum_j k_j \CM^j (k)=0$ in Fourier space where $\CM(k)$ is an amplitude for some QED process involving an external photon (i.e. gauge field) with momentum $k$.
\end{remark}

\section{Discrete regularity structure theory}
\label{sec:Discrete regularity structure theory}

\subsection{Regularity structure}

From this section we aim to prove the limit 
claimed in Proposition~\ref{prop:parabolic}
using the theory of regularity structures \cite{Regularity} and its discrete version \cite{hairer2015discrete}. 
We will  first 
build regularity structures for the ``unrenormalized version" of Eq.~\eqref{e:BPsi-poly}, 
namely setting $C^{(\eps)}=0$,
and then show  that the renormalized equation given by the theory is indeed \eqref{e:BPsi-poly} with a divergent constant $C^{(\eps)}$.
We start by briefly recalling the framework.

To start with we consider a parabolic scaling $\s = (2, 1,1)$ over space-time $\R^{2+1}$.
We set $|\s| \eqdef 4$, denote by $| x|$ the $\ell^\infty$-norm of a point $x \in \R^2$, 
and define $\|z\|_\s \eqdef |t|^{1/2} \vee | x|$ to be the $\s$-scaled $\ell^\infty$-norm of $z=(t,x) \in \R^{3}$. For a multiindex $k \in \N^{3}$ 
we define 
$|k|_\s \eqdef 2 k_0+  k_1+k_2$, 
and for $k \in \N^2$ with the scaling $(1,  1)$
 we denote the respective norm by $|k|$. 
Here $\N=\{0,1,2,\cdots\}$. 
Moreover,
for $r > 0$, we denote by $\CC^r(\R^2)$ the usual H\"{o}lder space on $\R^2$, 
by $\CC^r_0(\R^2)$ we denote the space of compactly supported $\CC^r$-functions and by $\CB^r_0(\R^2)$ we denote the set of $\CC^r$-functions, compactly supported in $B(0,1)$ (the unit ball centered at the origin) and with the $\CC^r$-norm bounded by $1$.

Recall from  \cite{Regularity} that a {\it regularity structure} $\ST = (\CT, \CG)$ consists of two objects:
a graded vector space $\CT = \bigoplus_{\alpha \in \CA} \CT_\alpha$, 
where $\CA \subset \R$ is a finite set of
``homogeneities'';
and a {\it structure group} $\CG$ of linear transformations of $\CT$, such that for every $\Gamma~\in~\CG$, every $\alpha \in \CA$ and every $\tau \in \CT_\alpha$ one has
$\Gamma \tau - \tau \in  \bigoplus_{\beta < \alpha} \CT_\beta$. 
One example of a regularity structure denoted by $\poly\CT$
is given by the ``abstract polynomials'' in $2+1$ 
indeterminates $X_i$, with $i = 0, 1, 2$, 
where $\CA$ consists of the natural numbers
$\alpha  \leq r$ for some fixed  $r >0$,
and $\CT_\alpha$ contains all monomials $X^k = X_0^{k_0}X_1^{k_1}X_2^{k_2}$ of scaled degree $|k|_{\s}=\alpha$; the structure group $\poly\CG$ is then simply the group of translations in $\R^{3}$
acting on $X^k$ by $h \mapsto (X-h)^k$. 

We build 
the regularity structure relevant for the analysis
of  \eqref{e:BPsi-poly}. For now we ignore the slight non-locality of the nonlinear terms  in  \eqref{e:BPsi-poly} and will control the discrepancy later.

We write $\CU_{B_j}$, $\CU_{\Psi_j}$, $\CU_{B_k}^{\-k}$, $\CU^{ j}_{\Psi_k}$
for $j,k\in\{1,2\}$
%
 for collections of  formal expressions
that will describe the solutions $B_j$, $\Psi_j$, $\nabla_{\-\be_k} B_k$ and 
$\nabla_{\be_j} \Psi_k$.
(Here,  the notation  $\nabla_{\pm\be_j} f$ should be thought of as discrete derivatives of $f$ along the direction $\pm\be_j$
in the discrete setting, while in the continuum setting $\nabla_{\-\be_j}f$
should be simply thought of as $-\partial_j f$.)
We  write   $\CV_{B_j}$, $\CV_{\Psi_j}$ for $j\in\{1,2\}$
 for a collection of formal expressions
 useful to describe the right hand side of the equation for $B_j$ and $\Psi_j$ respectively. 
We decree that 
all the above sets contain at least the polynomials
of  $X_0,X_1,X_2$.
We then introduce additional symbols
$\Xi_{\xi_j} \in \CV_{B_j}$
 and $\Xi_{\zeta_j} \in \CV_{\Psi_j}$ with $j\in\{1,2\}$, 
 as well as $\CI$, $\CI'_j$ and $\CI'_{\- j}$,
where $\Xi_{\xi_j}$, $\Xi_{\zeta_j}$ will be interpreted as an abstract representation of (discretization of) the 
driving noises $\xi_j$ and $\zeta_j$, 
and $\CI$, $\CI'_j$ and $\CI'_{\- j}$ will be interpreted as
 the operation of convolving with a truncation of the heat kernel  and its spatial derivative along the  direction $\be_j$ and $-\be_j$  respectively.
 
In view of the structure of the equation  \eqref{e:BPsi-poly}, 
to  describe the fixed point problem
we should at least also decree that 

1) If $\tau_1 \in  \CU_{\Psi_k}$ and  $ \tau_2 \in \CU^j_{\Psi_\ell}$
where $k\neq \ell$,
then $ \tau_1\tau_2\in\CV_{B_j}$.

2) If $\tau_1 \in \CU_{B_j}$ and    
 $ \tau_2 ,\tau_3 \in \CU_{\Psi_k} $, then
	$\tau_1\tau_2 \tau_3 \in\CV_{B_j} $.
	
3) If 	 $\tau\in \CV_{B_j}$ then $\CI(\tau)\in\CU_{B_j} $ and  
$\CI'_{\-k}(\tau)\in\CU_{B_j}^{\-k}$.

4)  If $\tau_1 \in \CU_{B_k}$ and  $\tau_2 \in \CU_{\Psi_\ell} $,
then 
$\tau_1\tau_2 \in\CV_{\Psi_{j}} $ and
$\CI'_{\-k}(\tau_1\tau_2) \in\CU_{\Psi_{j}} $,
 for $j\neq \ell$.

5)  If $(\tau_1,\tau_2) \in (\CU_{B_k} \times \CU^{k}_{\Psi_\ell} )
\cup (\CU_{B_k}^{\-k} \times \CU_{\Psi_\ell} )$ 
then 
$\tau_1\tau_2 \in\CV_{\Psi_{j}} $, $\CI(\tau_1\tau_2) \in \CU_{\Psi_{j}} $
and $\CI_m'(\tau_1\tau_2) \in \CU^m_{\Psi_{j}} $
 for $j\neq \ell$ and any $m\in\{1,2\}$.
 
6)  If $\tau_1,\tau_2 
 	\in    \CU_{B_k}$ and
$\tau_3 \in \CU_{\Psi_j} $,
then 
$ \tau_1  \tau_2  \tau_3 \in\CV_{\Psi_j} $,
$\CI(\tau_1\tau_2\tau_3) \in \CU_{\Psi_{j}} $
and $\CI_m'(\tau_1\tau_2\tau_3) \in \CU^m_{\Psi_{j}} $ for any $m\in\{1,2\}$.

We furthermore decree that $\tau\bar\tau=\bar\tau \tau$, 
$\CI(X^k)=\CI'_j(X^k)=\CI'_{\- j}(X^k)=0$, and $\tau\one = \tau$.
We then define 
\begin{equ}
\CF\eqdef \cup_{j,k  
\in\{1,2\}}  
\Big( \CU_{B_j}    
\cup \CU_{\Psi_k}  
 \cup \CU_{B_k}^{\-k}
 \cup \CU^{ j}_{\Psi_k}
\cup  \CV_{B_j} \cup \CV_{\Psi_k} \Big)\;.
\end{equ}


For each formal expression $\tau$, the homogeneity $|\tau| \in\R$
 is assigned in the following way.
Set
\begin{equ}
|\Xi_{\xi_j}|=|\Xi_{\zeta_j}| = -2-\bar\kappa
\qquad
j\in\{1,2\}
\end{equ}
where $\bar\kappa>0$ is a fixed small number,
and we define the homogeneity for every formal expression by
\begin{equ}
|\tau\bar\tau| = |\tau|+|\bar\tau|\;,
\qquad
|\CI(\tau)| 
=|\tau|+2 \;,
\qquad
|\CI'_j(\tau)| =|\CI'_{\- j}(\tau)|=|\tau|+1 \;.
\end{equ}
The linear space $\CT$ is then defined as the linear span
of $\CF$,
and $\CT_\alpha$ 
is the subspace spanned by $\{\tau\,:\,|\tau|= \alpha\}$.
The product is defined on the  regularity structure  by construction.
The sector $\langle\CU_{\Psi_k}  \rangle$
comes with an abstract derivative 
$\DD_j : \langle\CU_{\Psi_k}  \rangle \to \langle\CU^j_{\Psi_k}  \rangle$ 
 and  likewise $\DD_{\- j} : \langle\CU_{B_k}  \rangle\to \langle\CU^{\- j}_{B_k}  \rangle$ 
which decrease homogeneity by $1$.
By a simple power-counting argument, one finds
that as long as $\bar\kappa<1/4$, 
the sets $\{\tau\in\CF:|\tau|<\gamma\}$ are finite for every $\gamma \in \R$,
reflecting the fact that the  equation \eqref{e:BPsi-poly}
in two space dimensions is subcritical.

%
We will represent the relevant elements in graphic notations.
We use $\;\tikz [baseline=-3] \node[xi] {\tiny $1$};\;$  and $\;\tikz [baseline=-3] \node[xi] {\tiny $2$};\;$
to represent the expressions $\Xi_{\xi_1}$ and $\Xi_{\xi_2}$ respectively,
and 
$\;\tikz [baseline=-3] \node[zeta] {\tiny $1$}; \;$ and $\;\tikz [baseline=-3] \node[zeta] {\tiny $2$}; \;$
to represent the expressions $\Xi_{\zeta_1}$ and  $\Xi_{\zeta_2}$ respectively,
and also use $\;\tikz [baseline=-3,draw=symbols] \draw (0,0) -- (0.4,0);  \;$ 
 to represent the operator $\CI$,
 and 
 $\;\tikz [baseline=-3,draw=symbols] \draw (0,0) -- node[color=symbols,above=-0.6ex] {\tiny $j$} (0.4,0);$
 or
  $\;\tikz [baseline=-3,draw=symbols] \draw (0,0) -- node[color=symbols,above=-0.6ex] {\tiny $\-j$} (0.4,0);$
 to represent the operator $\CI'_j$ or $\CI'_{\-j}$.
Joining of formal expressions by their roots
denotes their product. 
For example,
$\<Psi1dPsi2>=\CI(  \CI(\Xi_{\zeta_1})  \CI'_j (\Xi_{\zeta_2}))$.


%

The elements in $\CF$  are completely determined by the above generating rule. 
\footnote{These elements can be generated by specifying a ``rule'' as in \cite{bruned2016algebraic}.}
We list some of these elements
which have negative 
homogeneities other than $\Xi_{\xi_j},\Xi_{\zeta_j} $ that will be useful later.

At the homogeneity  $- 1 - 2\bar\kappa$ we have
\begin{equ}[e:homo1-2kappa]
\<PsikdPsil> (\mbox{where } k\neq \ell)
\qquad
\<BdPsi--1>
\qquad
\<dBPsi>
\end{equ}
At the homogeneity  $- 1 - \bar\kappa$ we have
\[
\<Ij-zeta-k>
\qquad
\<I-k-xi-k>
\]
At the homogeneity
$ - 3 \bar\kappa $  
 we have
\[ 
	\<BjPsikPsik> 
	\qquad 
	\<BjBjPsik> 
\]
One can find many other elements of negative homogeneities  in $\CF$ in Section~\ref{sec:renorm}.

One can also build a structure group $\CG$ acting on $\CT$  as done in \cite{Regularity}
in such a way that
the operation $\CI$ satisfies $\CI \Gamma \tau - \Gamma \CI \tau \in \poly{\CT}$ for every $\tau \in \CT$ and $\Gamma \in \CG$, 
and $\Gamma(\tau\bar\tau)=(\Gamma\tau )(\Gamma \bar\tau )$
for every $\Gamma \in \CG$ and $\tau,\bar\tau\in\CT$ with $\tau\bar\tau\in\CT$,
and $\CG$ acts on $\poly{\CT}$ by translations.

A problem with this regularity structure $\ST = (\CT, \CG)$ is that 
$\Xi_{\xi_j}$,  $\Xi_{\zeta_j}$ and the elements in \eqref{e:homo1-2kappa} will not  be realized by our model as processes which are continuous in time. 
To circumvent this problem we actually work with a smaller regularity structure  $\hat \ST\subset \ST $ as in \cite[Section~3.2]{hairer2015discrete}. 
Instead of generating the  regularity structure from $\Xi_{\xi_j} \in \CV_{B_j}$
 and $\Xi_{\zeta_j} \in \CV_{\Psi_j}$
 and the abstract polynomials in these sets, we actually define 
a set of generating elements $\CF^{gen} $ which consists of  the abstract polynomials
and the following symbols
\begin{equ} 
  \<Izetaj> ,\quad  
   \<I-BdPsi> ,\quad  \<I(d-kBP)> ,
   \quad  \<Ixij>, \quad  
  \<I-PsikdPsil> ,  
  \label{e:F-gen-list}
\end{equ}
for $j,k,\ell\in\{1,2\}$ where in the last symbol $k\neq \ell$.
We then define the set $\hat \CF$ as the subset of $\CF$ generated by $\CF^{gen} $ via
the aforementioned procedure.
The advantage of this manipulation is that
$\Xi_{\xi_j}, \Xi_{\zeta_j}$ and elements \eqref{e:homo1-2kappa} 
are not included in $\hat\CF$.
These elements could not be realized as processes continuous in time if they were included in 
our regularity structure.
The  symbols $ \<Izetaj>$ and $ \<Ixij>$
have homogeneity  $-\bar\kappa $  and the other symbols have homogeneity
$1-2\bar\kappa$.

We then define the {\it truncated} regularity structure $\hat{\ST} = (\hat\CT, \CG)$
with $\hat\CT \eqdef \mbox{span}\{\tau \in \hat{\CF}\,:\, |\tau| \le r\} \subset \CT$, which is closed under $\CG$ \cite[Remark~3.5]{hairer2015discrete}.

\subsection{Inhomogeneous models}
\label{sec:Inhomo-models}

Given a regularity structure $\ST = (\CT, \CG)$, recall from \cite{hairer2015discrete}
that an {\it inhomogeneous  model} $(\Pi, \Gamma, \Sigma)$ consists of the following three elements.
First, a collection of maps $\Gamma^t : \R^{2} \times \R^{2} \to \CG$, parametrized by $t \in \R$, such that 
\begin{equ}[e:GammaDef]
\Gamma^t_{x x}=1\;, \qquad 
\Gamma^t_{x y} \Gamma^t_{y z} = \Gamma^t_{x z}\;,
\qquad
\forall x, y, z \in \R^{2} \;, \; t \in \R
\end{equ}
and their action  on $\poly{\CF} $ is given by $\Gamma^t_{x y} X^k  =    (X-(0,y-x))^k$. 
Secondly, a collection of maps $\Sigma_x : \R \times \R \to \CG$, parametrized by $x \in \R^2$, such that
\begin{equ}[e:SigmaDef]
\Sigma^{t t}_{x}=1\;, \qquad \Sigma^{s r}_{x} \Sigma^{r t}_{x} = \Sigma^{s t}_{x}\;, \qquad \Sigma^{s t}_{x} \Gamma^{t}_{x y} = \Gamma^{s}_{x y} \Sigma^{s t}_{y}\;,
\qquad
\forall x \in \R^{2} \;, \;  s, r, t \in \R
\end{equ}
and 
 $\Sigma^{s t}_{x} X^k  =    (X-(t-s,0))^k$. 
Finally, a collection of linear maps $\Pi^t_x: \CT \to \mathcal{S}'(\R^{2})$, such that
\begin{equ}[e:PiDef]
\Pi^t_{y} = \Pi^t_x \Gamma^t_{x y}\;, \quad \bigl(\Pi_x^t X^{(0, \bar k)}\bigr)(y) = (y-x)^{\bar k}\;, \quad \bigl(\Pi_x^t X^{(k_0, \bar k)}\bigr)(y) = 0\;,
\end{equ}
for all $x, y \in \R^{2}$, $t \in \R$, $\bar{k} \in \N^{2}$, $k_0 \in \N$ such that $k_0 > 0$.
Moreover, it requires that for any $\gamma > 0$ and every $T > 0$, there is a constant $C$ such that
\minilab{Model}
\begin{equs}\label{e:PiGammaBound}
| \langle \Pi^t_{x} \tau, \varphi_{x}^\lambda \rangle| 
\leq 
C \Vert \tau \Vert \lambda^{l} &\;, 
\qquad 
\Vert \Gamma^t_{x y} \tau \Vert_{m} \leq C \Vert \tau \Vert | x-y|^{l - m}\;,\\
\Vert \Sigma^{s t}_{x} \tau \Vert_{m} &\leq C \Vert \tau \Vert |t - s|^{(l - m)/2}\;,\label{e:SigmaBound}
\end{equs}
uniformly over all $\tau \in \CT_l$, with $l \in \CA$ and $l < \gamma$, all $m \in \CA$ such that $m < l$, all $\lambda \in (0,1]$, all $\varphi \in \CB^r_0(\R^2)$ with $r > -\lfloor\min \CA\rfloor$, and all $t, s \in [-T, T]$ and $x, y \in \R^2$ such that $|t - s| \leq 1$ and $|x-y| \leq 1$.
In addition, we say that the map $\Pi$ has {\it time regularity} $\delta > 0$, if the bound
\begin{equs}
\label{e:PiTimeBound}
| \langle \bigl(\Pi^t_{x} - \Pi^{s}_{x}\bigr) \tau, \varphi_{x}^\lambda \rangle| 
\leq C \Vert \tau \Vert |t-s|^{\delta/2} \lambda^{l - \delta}\;,
\end{equs}
holds for all $\tau \in \CT_l$ and the other parameters as before.
%
We can also define a {\it discrete inhomogeneous model} $(\Pi^\eps, \Gamma^\eps, \Sigma^\eps)$ which consists of the collections of maps 
\begin{equs}
\Pi_x^{\eps, t}: \CT \to \R^{\Lambda_\eps}\;, 
\qquad \Gamma^{\eps, t} : \Lambda_\eps \times \Lambda_\eps \to \CG\;, 
\qquad \Sigma^{\eps}_x : \R \times \R \to \CG\;,
\end{equs}
which have all the algebraic properties of their continuous counterparts \eqref{e:GammaDef}, \eqref{e:SigmaDef} and \eqref{e:PiDef} with the spatial variables restricted to the grid $\Lambda_\eps$. Additionally,
\footnote{This does not follow  
automatically from the discrete analogue of \eqref{e:PiGammaBound} since these
are only assumed to hold for test functions at scale $\lambda \ge \eps$.}
 we require $\bigl(\Pi^{\eps, t}_x \tau\bigr) (x) = 0$, for all $\tau \in \CT_l$ with $l > 0$, 
 and all $x\in\Lambda_\eps$ and $t \in \R$.

For a model $Z = (\Pi, \Gamma, \Sigma)$, we denote by $\Vert \Pi \Vert_{\gamma; T}$, $\Vert \Gamma \Vert_{\gamma;T}$ and $\Vert \Sigma \Vert_{\gamma;T}$ the smallest constants $C$ such that the bounds on $\Pi$, $\Gamma$ and $\Sigma$ in \eqref{e:PiGammaBound} and \eqref{e:SigmaBound} hold. We then define 
$
\VERT Z \VERT_{\gamma; T} \eqdef \Vert \Pi \Vert_{\gamma; T} + \Vert \Gamma \Vert_{\gamma; T} + \Vert \Sigma \Vert_{\gamma; T}
$
and we can define
 the ``distance'' $\VERT Z; \bar{Z} \VERT_{\gamma; T} $ between two models by taking the corresponding differences (in $\Vert\cdot\Vert_{\gamma; T}$ norms) of $(\Pi, \Gamma, \Sigma)$.
%
To take time regularity into account
we also set $\Vert \Pi \Vert_{\delta, \gamma; T} \eqdef \Vert \Pi \Vert_{\gamma; T} + C$, where $C$ is the smallest constant such that the bound \eqref{e:PiTimeBound} holds, and we define
$
\VERT Z \VERT_{\delta, \gamma; T} \eqdef \Vert \Pi \Vert_{\delta, \gamma; T} + \Vert \Gamma \Vert_{\gamma; T} + \Vert \Sigma \Vert_{\gamma; T}
$, and
 the ``distance'' $\VERT Z; \bar{Z} \VERT_{\delta, \gamma; T}$ 
 is defined analogously, as in \cite[Remark~2.5]{hairer2015discrete}.

%


In the discrete setting
we define the quantities $\Vert \Pi^\eps \Vert^{(\eps)}_{\gamma; T}$, 
$\Vert \Gamma^\eps \Vert_{\gamma;T}^{(\eps)}$ and  $\Vert \Sigma^\eps \Vert_{\gamma;T}^{(\eps)}$
to be the smallest constants $C$ such that the bounds \eqref{e:PiGammaBound}, \eqref{e:SigmaBound}
hold uniformly in $x, y \in \Lambda_\eps$, $\lambda \in [\eps,1]$ 
and the other parameters as in the continuum setting,
and
measure the time regularity of $\Pi^\eps$ as in \eqref{e:PiTimeBound},
with the discrete pairing \eqref{e:DPairing} in place of the standard one,
and $|t - s|^{\frac12} \vee \eps$ in place of $|t - s|^{\frac12}$.
%
The other norms and distances 
 $\VERT \cdot \VERT^{(\eps)}_{\gamma,T}$, 
 $\VERT \cdot \VERT^{(\eps)}_{\delta,\gamma,T}$
 and  $\VERT \cdot ; \cdot \VERT^{(\eps)}_{\delta,\gamma,T}$
 are defined by analogy with their continuous counterparts,
 as in \cite[Section~4.1]{hairer2015discrete}.

With the specific regularity structure defined above, 
we define a discrete model $Z^\eps = (\Pi^\eps, \Gamma^\eps, \Sigma^\eps)$ on $\ST^{gen}$
as follows.
Let $K^\eps$ be the  truncated discrete heat kernel as given in \cite[Lemma~5.4]{hairer2015discrete}.
Define
\begin{equs}
\Pi^{\eps,t}_x \, \<Izetaj>  (y) 
	= ( K^\e   \ast_\e \,\zeta^\e_j ) \,(t,y) \;, 
\qquad
\Pi^{\eps,t}_x \,   \<Ixij>  (y) 
	= ( K^\e \ast_\e \,\xi^\e_j )\, (t,y) \;,  
\end{equs}
where $\ast_\eps$ denotes the convolutions on $\R \times \Lambda_\eps$;
note that the expressions on the right hand sides 
do not depend on the base point $x$.
%
%
The group elements $\Gamma^{\eps, t}_{xy},\Sigma^{\eps,st}_x$
act on the above symbols as identities.
Furthermore, we denote the functions on $\R \times \Lambda_\eps$
\begin{equs} [e:def-Y-Ybar]
Y_{k j}^\eps
&\eqdef 
K^\eps \ast_\eps 
	\Big(\big(K^\e  \ast_\e \,\zeta^\e_k \big)
	\big( \nabla^\e_j K^\e  \ast_\e \,\zeta^\e_{3-k} \big)\Big)
\\
\bar Y_{k \ell }^\eps
&\eqdef 
K^\eps \ast_\eps 
\Big( \big(K^\e  \ast_\e \,\xi^\e_k\big)
\big( \nabla^\e_k K^\e  \ast_\e \,\zeta^\e_{\ell}\big)
\Big)
\\
\tilde Y_{k \ell }^\eps
&\eqdef 
K^\eps \ast_\eps 
\Big( \big(\nabla^\e_{\-\be_k} K^\e  \ast_\e \,\xi^\e_k\big)
\big(  K^\e  \ast_\e \,\zeta^\e_{\ell}\big)
\Big)
\\
\end{equs}
 and set  for $k\neq \ell$
\begin{equs}[e:Model-on-Yguy]
\bigl(\Pi^{\eps, t}_x \<I-PsikdPsil> \bigr)(y) 
= Y_{kj}^\eps(t,y) - Y_{kj}^\eps&(t,x)\;,
\qquad 
\Gamma^{\eps, t}_{x y} \<I-PsikdPsil> 
= \<I-PsikdPsil>  
	- \left( Y_{k j}^\eps(t,y) - Y_{kj}^\eps(t,x)\right) \one\;,
\\
\Sigma^{\eps, s t}_{x} \<I-PsikdPsil> 
= \<I-PsikdPsil>  &
	- \left( Y_{k j}^\eps(t,x) - Y_{k j}^\eps(s,x)\right) \one\;.
\end{equs}
We define the action of $(\Pi^{\eps},\Gamma^\eps,\Sigma^\eps)$
on $\<I-BdPsi> ,  \<I(d-kBP)> $ in the same way with  $\bar Y^\eps$ and $\tilde Y^\e$.


%


We then extend the model to the entire $\hat\CF$ in the canonical way.
For elements of the form $\tau_1\tau_2$   we define
\minilab{e:canon}
\begin{equs} 
(\Pi^{\eps,t}_{x} \tau_1\tau_2)(y) 
	= (& \Pi^{\eps,t}_{x}  \tau_1)( y) 
		\cdot  (\Pi^{\eps,t}_{x} \tau_2)(y) \;,
	\label{e:canon1} \\
\Gamma^{\eps, t}_{xy}(\tau_1\tau_2)
	= \big( \Gamma^{\eps, t}_{xy}\tau_1 \big) \,
		 \big(\Gamma^{\eps, t}_{xy}\tau_2\big) \;,
&\qquad
\Sigma^{\eps,st}_x (\tau_1\tau_2) 
	= \big(\Sigma^{\eps,st}_x \tau_1 \big) \,\big( \Sigma^{\eps,st}_x \tau_2\big)
	\label{e:canon2}
\end{equs}
for all $x,y\in\Lambda_\eps$, $s,t\in \R$.

%
%
For the symbols of the form $\CI \tau$ with  $\tau \in \CT_\alpha$  we define the actions by the models as follows.
\begin{equs}[e:PiIntegral]
\Pi^{\eps,t}_{x} \left(\CI \tau  \right)(y)
&\eqdef
 \int_{\R} \langle \Pi^{\eps, s}_{x} \Sigma_x^{\eps, s t} \tau, K^\e_{t-s}(y - \cdot)\rangle_\eps \, ds
- \Pi^{\eps,t}_{x} \left( \CJ_{t, x} \tau \right)(y)  \;, \\
\mbox{where} \qquad
\CJ_{t, x} \tau 
&\eqdef \sum_{|k|_{\s} < \alpha + 2} \frac{X^k}{k!} 
	\int_{\R} \langle\Pi^{\eps,s}_{x} \Sigma_x^{\eps, s t} \tau, D^k K^\e_{t-s}(x - \cdot)\rangle_\eps \, ds\;.
\end{equs}
Here $k \in \N^{3}$ and the derivative 
$D^k=\partial_t^{k_0} (\nabla^\eps_1)^{k_1} (\nabla^\eps_2)^{k_2}$. Moreover, we require that
\begin{equs}[e:GammaSigmaIntegral]
\Gamma_{x y}^{\eps,t} \big(\CI \tau\big) 
	&= \big(\CI + \CJ_{t, x}\big) \Gamma_{x y}^{\eps,t} \tau 
	-\Gamma_{x y}^{\eps,t} \big( \CJ_{t, y}\tau\big) \;,\\
\Sigma_x^{\eps, st} \big(\CI \tau\big) 
	&= \big(\CI + \CJ_{s, x}\big)\Sigma_x^{\eps,st} \tau
	 -\Sigma_x^{\eps,st} \big( \CJ_{t, x}\tau\big)\;,
\end{equs}
for all $s, t \in \R$ and $x, y \in \Lambda_\eps$.
For elements of the form $\CI'_{\pm j}\tau$ the model action is defined in the same way as \eqref{e:PiIntegral}
with $K^\e$ replaced by $\nabla^\eps_{\pm\be_j} K^\e$ 
and with $k$ sums over $|k|_{\s} < \alpha + 1$.
It satisfies
\begin{equ}[e:model-compat-der]
\Pi^{\eps,t}_x (\DD_{\pm j}(\CI\tau))(y)
= \Pi^{\eps,t}_x (\CI'_{\pm j}\tau)(y) 
	= \nabla_{\pm j}^\eps \big( \Pi^{\eps,t}_{x} (\CI\tau) \big) (y)
\qquad
\forall x,y\in\Lambda_\eps, \; t\in\R \;.
\end{equ}

We call a discrete model $Z^\eps = (\Pi^{\eps}, \Gamma^{\eps}, \Sigma^{\eps})$ defined on $\hat{\ST}$ {\it admissible}, if it satisfies the identities \eqref{e:canon2}, \eqref{e:PiIntegral} and \eqref{e:GammaSigmaIntegral};
the set of admissible models is denoted by $\MM_\e$.
We also denote by $\MM_0$ the set of admissible models defined in continuum with all the  discrete maps and scalar products in \eqref{e:canon2}, \eqref{e:PiIntegral} and \eqref{e:GammaSigmaIntegral} replaced by their continuous counterparts.
For each $\eps>0$, $(\Pi^\eps,\Gamma^\e, \Sigma^\e)$ defined above
  is  an admissible discrete inhomogeneous model.

\subsection{Modeled distributions}

We recall from \cite{hairer2015discrete} the notion of modeled distributions,
both in continuum and
on a grid. 
Given a regularity structure $\ST = (\CT, \CG)$ with an inhomogeneous model $Z=(\Pi, \Gamma, \Sigma)$,  
$\gamma, \eta \in \R$ and  $T > 0$, we consider maps 
$H : (0, T] \times \R^{d} \to \CT_{<\gamma}$ and define
\begin{equs}[e:ModelledDistributionNormAbs]
 \Vert H \Vert_{\gamma, \eta; T} 
 \eqdef 
 \sup_{ \substack{ t \in (0,T]  \\ x \in \R^2}} 
 \sup_{ \substack{ l \in \CA \\ l < \gamma}} \onorm{t}^{(l - \eta) \vee 0} \Vert H_t(x) \Vert_l
+ \sup_{t \in (0,T]} \sup_{\substack{x \neq y \in \R^d \\ | x - y | \leq 1}} 
	\sup_{ \substack{ l \in \CA \\ l < \gamma}} \frac{\Vert H_t(x) - \Gamma^{t}_{x y} H_t(y) \Vert_l}{\onorm{t}^{\eta - \gamma} | x - y |^{\gamma - l}}
\end{equs}
where $\onorm{t} \eqdef |t|^{1/2} \wedge 1$. 
The space of {\it inhomogeneous modeled distributions} $\CD^{\gamma, \eta}_T(Z)$ consists of 
all such functions $H$ such that
\begin{equ}[e:ModelledDistributionNorm]
\VERT H \VERT_{\gamma, \eta; T} 
\eqdef \Vert H \Vert_{\gamma, \eta; T} 
	+ \sup_{\substack{s \neq t \in (0,T] \\ | t - s | \leq \onorm{t, s}^{2}}} \sup_{x \in \R^d} 
	\sup_{ \substack{ l \in \CA \\ l < \gamma}}
	\frac{\Vert H_t(x) - \Sigma_x^{t s} H_{s}(x) \Vert_l}{\onorm{t, s}^{\eta - \gamma} |t - s|^{(\gamma - l)/2}} < \infty\;.
\end{equ}
Here $\onorm{t, s} \eqdef \onorm{t} \wedge \onorm{s}$.
For a discrete inhomogeneous model $Z^\eps=(\Pi^\eps, \Gamma^\eps, \Sigma^\eps)$,
and for a function $H : (0, T] \times \Lambda_\eps \to \CT_{<\gamma}$, 
we define
$ \Vert H \Vert^{(\eps)}_{\gamma, \eta; T} $ and $\VERT H \VERT^{(\eps)}_{\gamma, \eta; T} $
the same way as \eqref{e:ModelledDistributionNormAbs} and \eqref{e:ModelledDistributionNorm}
 but with spatial variables restricted to the grid $\Lambda_\eps$,
$|t|_0$ replaced by $|t|_\eps \eqdef \onorm{t} \vee \eps$
and $\onorm{t, s}$  replaced by $\enorm{t,s}\eqdef \enorm{s} \wedge \enorm{t}$.
We call such functions $H$ {\it discrete modeled distributions}
and denote the space for them by $\CD^{\gamma, \eta}_{\e,T}(Z^\eps)$.
Both $\CD^{\gamma, \eta}_{T}$ and $\CD^{\gamma, \eta}_{\e,T}$ come with products that are defined point-wisely.
The space $\CD^{\gamma, \eta}_{T}$ comes with derivative operators
still denoted by $\DD_j$ and  $\DD_{\-j}$ for $j\in\{1,2\}$
which act
on modeled distributions point-wisely.
On the space $\CD^{\gamma, \eta}_{\e,T}$,
the operator $\DD_j$ also acts point-wisely as before,
while the  linear operator $\DD_{\- j}$ acts point-wisely 
on all elements except 
on polynomials
in a slightly nonlocal way: 
for $\langle \nabla\phi , \mathbf X\rangle \eqdef \sum_{j=1}^d \nabla_j \phi \mathbf X_j$
 with $\nabla\phi =(\nabla_1\phi,\cdots,\nabla_d\phi): \Z^d\to \R^d$, one has
$(\DD_{\- j} \langle \nabla\phi , \mathbf X\rangle )(x) 
\eqdef -\nabla_j \phi(x-\be_j) \one $.
We can check that
\footnote{Here $\nabla\phi$ is merely a notation for a $d$ component field and does not mean a gradient of a function. We also remark that in this paper we will only need abstract polynomials of order one. For higher order polynomials it is a bit more subtle because the finite difference of  $x^k$ is not  $k x^{k-1}$ for $k\ge 2$.}
\begin{equs}
\Pi_x^{\eps, t}  \DD_{ j} \langle \nabla\phi , \mathbf X\rangle 
&= \nabla_{\be_j} \Pi_x^{\eps, t}   \langle \nabla\phi , \mathbf X\rangle = (\nabla_j \phi) (x) \;,
\\
\Pi_x^{\eps, t}  \DD_{\- j} \langle \nabla\phi , \mathbf X\rangle 
&= \nabla_{\- \be_j} \Pi_x^{\eps, t}   \langle \nabla\phi , \mathbf X\rangle  = -(\nabla_j \phi)(x-\be_j) \;.
\end{equs}
We will also sometimes write $\CD^{\gamma, \eta}_{0,T}$
which stands for $\CD^{\gamma, \eta}_{T}$.

One can compare two modeled distributions 
$H \in \CD^{\gamma, \eta}_T(Z)$ and $\bar{H} \in \CD^{\gamma, \eta}_T(\bar{Z})$
by the distance $\VERT H; \bar{H} \VERT_{\gamma, \eta; T}$,
as well as two discrete modeled distributions 
$H \in \CD^{\gamma, \eta}_{\e,T}(Z^\eps)$ and $\bar{H} \in \CD^{\gamma, \eta}_{\e,T}(\bar Z^\eps)$
by the distance $\VERT H; \bar{H} \VERT_{\gamma, \eta; T}^{(\e)}$
defined in the natural way as in \cite[Section~2.2 and 4.1]{hairer2015discrete}.

The inhomogeneous modeled distributions can be ``realized" as
$\CC^\alpha$ distributions.
By \cite[Theorem~2.11]{hairer2015discrete} or \cite[Theorem~3.10]{Regularity}, 
letting  $\alpha \eqdef \min \CA < 0$,
for every $\eta \in \R$, $\gamma > 0$ and $T > 0$, there is a unique family of linear 
operators $\CR_t : \CD_T^{\gamma, \eta} \to \CC^{\alpha}(\R^d)$, parametrized by $t \in (0,T]$, called reconstruction operators,
such that the bound
\begin{equ}[e:Reconstruction]
|\langle \CR_t H_t - \Pi^t_x H_t(x), \varphi_x^\lambda \rangle| \lesssim \lambda^\gamma \onorm{t}^{\eta - \gamma} \Vert H \Vert_{\gamma, \eta; T} \Vert \Pi \Vert_{\gamma; T}
\end{equ}
holds uniformly in $H \in \CD^{\gamma, \eta}_T$, $t \in (0,T]$, $x \in \R^d$, $\lambda \in (0,1]$ and $\varphi \in \CB^r_0(\R^d)$ with $r > -\lfloor \alpha\rfloor$.
If furthermore the map $\Pi$ has time regularity $\delta > 0$, then by \cite[Theorem~2.11]{hairer2015discrete}, for any $\tilde{\delta} \in (0, \delta]$ such that $\tilde{\delta} \leq (m - \zeta)$ for all $\zeta, m \in \left((-\infty, \gamma) \cap \CA\right) \cup \{\gamma\}$ such that $\zeta < m$, 
one has
$ \Vert \CR H \Vert_{\CC^{\tilde{\delta}, \alpha}_{\eta - \gamma, T}} \lesssim \Vert \Pi \Vert_{\delta, \gamma; T} \bigl(1 + \Vert \Sigma \Vert_{\gamma; T} \bigr) \VERT H \VERT_{\gamma, \eta; T} $ 
and $\CR$ is locally Lipschitz continuous in $H$ and the model.

\begin{definition}\label{def:DReconstruct}
Given a discrete model $Z^\eps = (\Pi^\eps, \Gamma^\eps, \Sigma^\eps)$ and a discrete modeled distribution $H$ we define the {\it discrete reconstruction map} $\CR^{\eps}$ by $\CR^{\eps}_t = 0$ for $t \leq 0$, and
\begin{equ}[e:DReconstructDef]
\big(\CR^{\eps}_t H_t\big)(x) \eqdef \big(\Pi_x^{\eps, t} H_t(x) \big)(x)\;, \qquad (t, x) \in (0, T] \times \Lambda_\eps\;.
\end{equ}
\end{definition}

For a discrete model $Z^\eps = (\Pi^\eps, \Gamma^\eps, \Sigma^\eps)$ the analogous bound \eqref{e:Reconstruction}
holds with discrete pairing, $|t|_\eps$ and discrete norms in place of the continuous ones, uniformly in
 $x \in \Lambda_\eps$,  $\lambda \in [\eps, 1]$ and other variables as in continuum.
The map $\CR^\eps$ is also locally Lipschitz continuous in $H$ and the model.

Finally as in \cite[Remark~4.6]{hairer2015discrete} 
one can compare a discrete model $Z^\eps = (\Pi^\eps, \Gamma^\eps, \Sigma^\eps)$ with a continuous model $Z = (\Pi, \Gamma, \Sigma)$ by distance 
\begin{equ} [e:ZeZ-distance]
\VERT Z; Z^\eps \VERT^{(0,\eps)}_{\delta, \gamma; T} 
\eqdef
\Vert \Pi ; \Pi^\eps \Vert^{(0,\eps)}_{\delta, \gamma; T} 
+
\Vert \Gamma; \Gamma^\eps \Vert^{(0,\eps)}_{\gamma; T} 
+
\Vert \Sigma; \Sigma^\eps \Vert^{(0,\eps)}_{\gamma; T} 
\end{equ}
 where
 \begin{equs}
\Vert \Pi &; \Pi^\eps \Vert^{(0,\eps)}_{\delta, \gamma; T} 
\eqdef 
\sup_{\varphi, x, \lambda, l, \tau} \sup_{t \in [-T, T]}
	 \lambda^{-l} | \langle \Pi^t_{x} \tau, \varphi_{x}^\lambda \rangle - \langle \Pi^{\eps, t}_{x} \tau, \varphi_{x}^\lambda \rangle_\eps|\\
&
+ \sup_{\varphi, x, \lambda, l, \tau} \sup_{\substack{s \neq t \in [-T, T] \\ |t-s| \leq 1}}
	 \lambda^{-l + \delta} \frac{| \langle \bigl(\Pi^t_{x} - \Pi^s_{x}\bigr) \tau, \varphi_{x}^\lambda \rangle - \langle \bigl(\Pi^{\eps, t}_{x} - \Pi^{\eps, s}_{x}\bigr) \tau, \varphi_{x}^\lambda \rangle_\eps|}{\bigl(|t-s|^{\frac12} \vee \eps\bigr)^{\delta}}\;,
\end{equs}
where   $\varphi \in \CB^r_0$, $x \in \Lambda_\eps$, $\lambda \in [\eps, 1]$, $l < \gamma$ and $\tau \in \CT_l$ with $\Vert \tau \Vert = 1$, and $\Vert \Gamma; \Gamma^\eps \Vert^{(0,\eps)}_{\gamma; T} $ and
$\Vert \Sigma; \Sigma^\eps \Vert^{(0,\eps)}_{\gamma; T} $
are defined in the analogous way (of course without the need of the second line here that measures time regularity).
We can  also compare discrete and continuous modeled distributions in the analogous way by 
$\VERT H; H^\eps \VERT^{(0,\eps)}_{\gamma, \eta; T} $.
 With these notions for $H \in \CD^{\gamma, \eta}_T(Z)$ and a discrete modeled distribution $H^\eps$ one can prove the estimate
\begin{equ}[e:RH-ReHe]
\Vert 
	\CR H; \CR^\eps H^\eps
 \Vert^{(0,\eps)}_{\CC^{\tilde{\delta}, \alpha}_{\eta - \gamma, T}} 
\lesssim 
\VERT H; H^\eps \VERT^{(0,\eps)}_{\gamma, \eta; T} 
+ \VERT Z; Z^\eps \VERT^{(0,\eps)}_{\delta, \gamma; T} + \eps^\theta\;,
\end{equ}
for $\tilde \delta > 0$ and $\theta > 0$ small enough.

We can define  a nonlocal operator $\CP^\e$ on the space $\CD^{\gamma, \eta}_{\e,T}$
such that 
\begin{equ}[e:RP-PR]
\CR_t^\e \CP^\e = P^\e \ast_\e \CR_t^\e
\end{equ}
where $P^\e$ is the discrete heat kernel, see 
\cite[Eq.(4.35)(5.18)]{hairer2015discrete} for definition of $\CP^\e$ which utilizes a decomposition of $P^\eps$ given in
\cite[Eq.(5.13)]{hairer2015discrete}. We can also define an operator $\CP_{\-k}^\e \eqdef \DD_{\-k}\CP^\e$.

Fixing $\alpha\in (-\frac43,-1)$, 
$\gamma>|\alpha|$, $\eta\in (-\frac12,1)$, we now formulate the abstract fixed point problem in the space of modeled distributions  $\CD^{\gamma, \eta}_{\e,T}$
for the system of equations \eqref{e:BPsi-poly}.
Set
\footnote{
We expect (see \eqref{e:expand-sol} below) our solutions to have the form
$
\bPsi^\e_j = \<Izetaj> + (\cdots)$
where ``$(\cdots)$'' take values in the subspace  spanned by $\one$ and elements with strictly positive homogeneity; so the symbols $ \<Psi1dPsi2> $ - which is not in the truncated regularity structure $\hat{\ST}$ as one may worry - 
will be cancelled by the ``leading order'' part of $\bPsi_1^\e \DD_j \bPsi_2^\e$
in \eqref{e:FhatB-fields}. Same remark for the other abstract symbols of such kind in \eqref{e:FhatB-fields}.
}
\begin{equs} [e:FhatB-fields]
\hat F_{B_j^\eps} (\bB^\eps,  \bPsi^\eps)
& \eqdef
\lambda \left(  \bPsi_1^\e \DD_j \bPsi_2^\e
		 - \bPsi_2^\e \DD_j \bPsi_1^\e \right)
- \lambda^2  \sum_{k=1,2} \bB_j^\e
	 (\bPsi_k^\e)^2 
-
\lambda\Big( \<Psi1dPsi2> - \<Psi2dPsi1> \Big) 
\\
\hat F_{\Psi_j^\eps} (\bB^\eps,    \bPsi^\eps)
&  \eqdef  -(-1)^j \lambda  \sum_{k=1,2 \atop  \ell\neq j}  
		 \Big(\bB_k^\e \DD_k \bPsi^\e_{\ell}  
		  +  \bPsi^\e_{\ell} \DD_{\- k} \bB_k^\e \Big)
 -  \lambda^2   \sum_{k=1,2}
		 (\bB_k^\e)^2 \,\bPsi_j^\e
		 \\
	&\qquad\qquad
+(-1)^j \lambda \sum_{ k=1,2 \atop \ell\neq j } 
	 \Big(\<BdPsi--1>+ \<dBPsi> \Big)
\\
\hat F^{(k)}_{\Psi_j^\eps} (\bB^\eps,    \bPsi^\eps)
&  \eqdef  (-1)^j \lambda  \sum_{\ell\neq j}  
		 \bB_k^\e  \bPsi^\e_{\ell} 
\end{equs}
for modeled distributions $\bB_j^\e$ taking values in $\langle \CU_{B_j} \rangle$
and $\bPsi_j^\e $ taking values in $\langle \CU_{\Psi_j} \rangle$ with $j\in\{1,2\}$.

The abstract fixed point problem then reads
\minilab{e:abs-fp}
\begin{equs} 
\bB^\eps_j  
	 & = \CP^\eps 
	 	\Big(\hat F_{B_j^\e}(\bB^\eps,\bPsi^\eps)+ \boldsymbol R_{B_j^\eps}^\eps\Big)
	 +S^\e \mathring A_{j}^\eps 
+ \<Ixij> 
+
\lambda\Big(
	\CY^\e_{1j}- \CY^\e_{2j}
	\Big)     \one
+\widetilde{\boldsymbol R}_{B^\eps_j}^\eps
	\label{e:abs-fp1}
\\
\bPsi^\eps_j 
&=  \CP^\eps 
	\Big(\hat F_{\Psi_j^\e}(\bB^\eps,\bPsi^\eps)+ \boldsymbol R_{\Psi^\eps_j}^\eps\Big)
	+\sum_{k} \CP_{\-k}^\e \hat F^{(k)}_{\Psi_j^\e}(\bB^\eps,\bPsi^\eps)
	\\
	& \qquad\qquad \qquad
	+S^\e \mathring \Phi_{j}^\eps 
+ \<Izetaj>     
-(-1)^j \lambda \sum_{ k=1,2 \atop \ell\neq j } 
	 \Big(\bar{\CY}^\e_{k \ell} - \tilde{\CY}^\e_{k \ell} \Big) \one
+\widetilde{\boldsymbol R}_{\Psi^\eps_j}^\eps 
	\label{e:abs-fp2}
\end{equs}
where 
$\CY,\bar\CY,\tilde\CY$ are defined the same way 
as $Y,\bar Y,\tilde Y$  in \eqref{e:def-Y-Ybar}, but 
with a slight tweak that the left-most $K^\e$ appearing in each equation of \eqref{e:def-Y-Ybar}
is replaced by the un-truncated heat kernel $P^\e$.
Moreover
 $S^\e$ is the discrete semi-group so that 
$S^\e \mathring A_{j}^\eps , S^\e \mathring \Phi_{j}^\eps $
are naturally lifted into $\CD^{\gamma, \eta}_{\e,T}$
as \cite[Lemma~3.6]{hairer2015discrete}.
Finally the ``remainder'' terms are defined as
\minilab{e:remainders}
\begin{equs} 
{}&
\boldsymbol R_{B^\eps_j}^\eps   (x) 
	\eqdef 
	R_{B^\eps_j}^\eps (x; \CR_t^\eps \bB^\eps,\CR_t^\eps \bPsi^\eps)
	\one \;,
			\label{e:remainders1}
\\
{}&
\boldsymbol R_{\Psi^\eps_j}^\eps  (x)\eqdef 
R_{\Psi^\eps_j}^\eps (x; \CR_t^\eps \bB^\eps,\CR_t^\eps\bPsi^\eps)\,\one  \;,
			\label{e:remainders2}
\\
{}&
\widetilde{\boldsymbol R}_{B^\eps_j}^\eps  
 \eqdef
-  \lambda^2 
		P^\e \ast_\e
		\Big(\sum_{k=1,2}\eps\,
		(\CR^\eps \bB_j^\eps)
		(\nabla_{\be_j}^\e \CR^\eps \bPsi_k^\eps)
		( \CR^\eps \bPsi^\eps_k)
		- 2 c_B^\e  \CR^\eps \bB_j^\eps  \Big) \circ \langle \one, \mathbf X\rangle   \;,
			\label{e:remainders3}
 \\
			 \label{e:remainders4}
{}&
\widetilde{\boldsymbol R}_{\Psi^\eps_j}^\eps  
\eqdef
-\lambda^2
	P^\e \ast_\e
	\Big[\sum_{k=1,2} \frac{\eps}{2}\,\Big( (\CR^\eps \bB_k^\eps)^2
		(\nabla_{\be_k}^\e \CR^\eps \bPsi_j^\eps)
	+ (\CR^\eps \bB_k^\eps)(\cdot -\be_k)^2
		(\nabla_{\- \be_k}^\e \CR^\eps \bPsi_j^\eps)
\\
& \qquad \qquad\qquad\quad
	+
	(\CR^\eps \bB_k^\eps) 
	(\nabla^\e_{\-\be_k}\CR^\eps \bB_k^\eps)
	( \CR^\eps \bPsi_j^\eps)
	+ (\CR^\eps \bB_k^\eps) (\cdot -\be_k)
	(\nabla^\e_{\-\be_k}\CR^\eps \bB_k^\eps)
	( \CR^\eps \bPsi_j^\eps)\Big)\\
& \qquad \qquad\qquad\quad
 - c^\e_{\Psi}  \CR^\eps \bPsi_j^\eps
\Big]  \circ \langle \one, \mathbf X\rangle 
\end{equs}  
with $R_{B^\eps}^\eps$ and $R_{\Psi^\eps}^\eps$ defined in
\eqref{e:defRBxPsix}. We will specify the constants $c_B^\e$ and $c_{\Psi}^\e$ later but let's immediately remark that they will converge to finite limits. Here we used a 
 shorthand notation 
 \begin{equ}[e:circ1X]
 F\circ \langle\one,\mathbf X\rangle\eqdef F\one +  \langle \nabla F,\mathbf X\rangle \;.
 \end{equ}


We will renormalize the models in the next section. On the other hand,
assuming that the abstract fixed point problem is locally wellposed, we would like to see what equation we obtain 
using the reconstruction operator defined in Definition~\ref{def:DReconstruct} for the un-renormalized models:

\begin{lemma} \label{lem:unren-equ}
Suppose that  $c^\e_{\Psi}=c^\e_{B}=0$, and that $(\bB^\e,\bPsi^\e)$ is a solution to the abstract fixed point problem \eqref{e:abs-fp}. Then
$ (B^\eps,\Psi^\eps)\eqdef (\CR_t^\eps \bB^\eps,  \CR_t^\eps \bPsi^\eps)$ satisfy 
\eqref{e:BPsi-poly} with $C^{(\eps)} =0$.
\end{lemma}

We will turn on the constants $c^\e_{\Psi}, c^\e_{B}$
in Lemma~\ref{lem:renorm-equ} when we do renormalizations.

\begin{proof}
This can be readily checked, 
by action of reconstruction $\CR^\e$ for the above canonical model
 on both sides of \eqref{e:abs-fp}, and using
the property \eqref{e:RP-PR} of the operator $\CP^\e$, the definition of 
reconstruction operator $\CR_t^\eps$ in \eqref{e:DReconstructDef},
the definition of the canonical model \eqref{e:canon}, and the 
compatibility of the abstract derivative \eqref{e:model-compat-der}.

%
In particular, the ``remainder'' terms \eqref{e:remainders3} and  \eqref{e:remainders4}
are defined in order to exactly obtain
the slightly nonlocal cubic terms in \eqref{e:BPsi-poly}  upon
action by the reconstruction:
\[
\CR^\eps \CP^\eps  \Big(- \lambda^2  \sum_{k=1,2} \bB_j^\e
	 (\bPsi_k^\e)^2 \Big)
+\CR^\eps \widetilde{\boldsymbol R}_{B^\eps_j}^\eps 
=
- \lambda^2  \sum_{k=1,2} 
P^\eps  \ast_\eps 
 \Big(B_j^\e (\Psi_k^\e)^2
+ \eps B_j^\eps
		\nabla_{\be_j}^\e \Psi_k^\eps
		 \Psi^\eps_k\Big)
\]
where the terms in the parenthesis equal to the cubic term  $B_j^\eps(x)  \Psi_k^\eps(x+\be_j) \Psi_k^\eps(x)$  in \eqref{e:BPsi-poly1}, and by straightforward computation
{\small
\[
\CR^\eps \CP^\eps  \Big(- \lambda^2  \sum_{k=1,2}  (\bB_k^\e)^2 \,\bPsi_j^\e \Big)
+\CR^\eps \widetilde{\boldsymbol R}_{\Psi^\eps_j}^\eps 
=
- \frac{\lambda^2}{2}  \sum_{k=1,2} 
P^\eps  \ast_\eps 
 \Big((B^\eps_k)^2 \,\Psi_j^\eps(\cdot+\be_k)
	+ B^\eps_k(\cdot- \be_k)^2 \,\Psi_j^\eps(\cdot-\be_k)  \Big)
\]}
recovering the cubic terms in \eqref{e:BPsi-poly2}.
\end{proof}

\section{Renormalization and bounds on the models}
\label{sec:renormalization}

\subsection{Renormalization} \label{sec:renorm}

The models defined in the previous section will not converge as $\eps\to 0$,
and to obtain a limit we  need to renormalize them. 
%
For each $\e>0$ let $M_\e:\CT\to \CT$ be the map
\begin{equs}
{}& M_\e  =\exp \Big(-  C_1^{(\e)} L_1^{(k)}-C_2^{(\e)} L_2^{(k)}
- C_{3,j}^{k(\e)} L_{3,j}^{k(\ell)} 
-C_{3,-\ell}^{-\ell(\e)} L_{3,-\ell}^{-\ell(\ell)} 
 -C_{4,j}^{k(\e)} L_{4,j}^{k(\ell)} 
 	\label{e:def-Meps}
 \\
& -C_{4,-k}^{k(\e)}  L_{4,-k}^{k(\ell)}
 -C_{4,-\ell}^{\ell(\e)}    L_{4,-\ell}^{\ell(\ell)}
  - C_{5,j}^{-k(\e)}      L_{5,j}^{-k(\ell)}
   -C_{6,\ell}^{-\ell(\e)}         L_{6,\ell}^{-\ell(\ell)}
-C_{7,j}^{-k(\e)} L_{7,j}^{-k(\ell)}
-C_{7,-\ell}^{-\ell(\e)} L_{7,-\ell}^{-\ell(\ell)}
\Big) 
\end{equs}
 where Einstein's notation is used (all the indices are summed over $\{1,2\}$), and the renormalization constants are specified below, 
 and  the nilpotent linear operators on $\CT$ are given as follows: for $k,j,\ell\in\{1,2\}$. let
\begin{gather}
L_1^{(k)}: 	\<PsikPsik> \to \one \;, 
\qquad 
 L_2^{(j)}: 	\<BjBj> \to \one\;, 
\notag \\
L_{3,j}^{k(\ell)}:  \<dj-Psil-Idk-Psil>  \to \one\;, 
\qquad 
L_{3,-\ell}^{-\ell(\ell)}:  \<d-lBlI(d-eBe)>\to \one\;, 
 \notag \\
 L_{4,j}^{k(\ell)}: \<Psi-dIdPsi> \to \one \;,
 \qquad
  L_{4,-k}^{k(\ell)}: \<dk-Psi-IdkPsi> \to \one \;,
\qquad
  L_{4,\ell}^{-\ell(\ell)}: \<BldlI(d-eBe)>\to \one \;,
 	\label{e:def-Ls}  \\
   L_{5,j}^{-k(\ell)}:  \<Ped-kdeIPe>\to \one \;,
   \qquad
    L_{6,\ell}^{-\ell(\ell)}:  \<Bld-edlI(Be)>\to \one \;, 
    \notag  \\
L_{7,j}^{-k(\ell)}:  \<Ped-kI(dePe)>\to \one \;, 
\qquad
L_{7,-\ell}^{-\ell(\ell)}:  \<d-lBld-eI(Be)>\to \one \;. 
	\notag
\end{gather}

%
%
To define the renormalization constants 
we introduce some graphical notations. Each arrow $\tikz [baseline=-3] \draw[kernel] (0,0) to (1,0);$ represents 
the kernel $K^\e(y-x)$, and each arrow 
$\tikz [baseline=-3] \draw[kernel] (0,0) to node [midway,above,font=\scriptsize] {$j$}  (1,0);$
represents 
the kernel $\nabla^\e_{\be_j} K^\e(y-x)$,
and each arrow 
$\tikz [baseline=-3] \draw[kernel] (0,0) to node [midway,above,font=\scriptsize] {$-j$}  (1,0);$
represents 
the kernel $\nabla^\e_{\-\be_j} K^\e(y-x)$,
with $x$ and $y$ being the starting and end space-time points of the arrow respectively.
Each dot  $\;\tikz [baseline=-3] \node[dot]  at (0,0) {};\;$
represents a space-time point that is integrated out,
and each green dot $\;\tikz [baseline=-3] \node[root]   at (0,0) {};\;$
represents the origin of space-time which is not integrated.
Define the following renormalization constants: 
\begin{gather}  
C_2^{(\eps)} \eqdef
\begin{tikzpicture}  [baseline=10]
\node[root]	(root) 	at (0,0) {};
\node[dot]		(top)  	at (0,1) {};			
\draw[kernel,bend left =60]   (top)  to  (root) ;
\draw[kernel,bend right =60]   (top)  to  (root) ;
\end{tikzpicture}
\qquad\qquad
C_{3,j}^{k(\eps)}  \eqdef
\begin{tikzpicture}  [baseline=10]
\node[root]	(root) 	at (0,0) {};
\node[dot]		(left)  	at (-0.7,1) {};
\node[dot]		(right)  	at (0.7,1) {};			
\draw[kernel,bend right=20] (left) to  (root);
\draw[kernel, bend left=20]   (right)  to 
	node [midway,right,font=\scriptsize] {$j$} (root) ;
\draw[kernel, bend right=20]  (right)  to 
	node [midway,above,font=\scriptsize] {$k$} (left);
\end{tikzpicture}
\qquad
C_{3,-k}^{-k(\eps)} \eqdef
\begin{tikzpicture}  [baseline=10]
\node[root]	(root) 	at (0,0) {};
\node[dot]		(left)  	at (-0.7,1) {};
\node[dot]		(right)  	at (0.7,1) {};			
\draw[kernel,bend right=20] (left) to  (root);
\draw[kernel, bend left=20]   (right)  to 
	node [midway,right,font=\scriptsize] {$-k$} (root) ;
\draw[kernel, bend right=20]  (right)  to 
	node [midway,above,font=\scriptsize] {$-k$} (left);
\end{tikzpicture}
	\notag
\\
C_{4,j}^{k(\eps)}  \eqdef 
\begin{tikzpicture}  [baseline=10]
\node[root]	(root) 	at (0,0) {};
\node[dot]		(left)  	at (-0.7,1) {};
\node[dot]		(right)  	at (0.7,1) {};			
\draw[kernel,bend right=20] (left) to 
	node [midway,left,font=\scriptsize] {$j$} (root);
\draw[kernel,bend left=20]   (right)  to  (root) ;
\draw[kernel,bend right=20]  (right)  to 
	node [midway,above,font=\scriptsize] {$k$} (left);
\end{tikzpicture}
\qquad
C_{4,-k}^{k(\eps)}  \eqdef
\begin{tikzpicture}  [baseline=10]
\node[root]	(root) 	at (0,0) {};
\node[dot]		(left)  	at (-0.7,1) {};
\node[dot]		(right)  	at (0.7,1) {};			
\draw[kernel,bend right=20] (left) to
	 node [midway,left,font=\scriptsize] {$-k$} (root);
\draw[kernel, bend left=20]   (right)  to  (root) ;
\draw[kernel, bend right=20]  (right)  to 
	node [midway,above,font=\scriptsize] {$k$} (left);
\end{tikzpicture}
\qquad
C_{4,k}^{-k(\eps)}  \eqdef
\begin{tikzpicture}  [baseline=10]
\node[root]	(root) 	at (0,0) {};
\node[dot]		(left)  	at (-0.7,1) {};
\node[dot]		(right)  	at (0.7,1) {};			
\draw[kernel,bend right=20] (left) to 
	node [midway,left,font=\scriptsize] {$k$}  (root);
\draw[kernel, bend left=20]   (right)  to  (root) ;
\draw[kernel, bend right=20]  (right)  to 
	node [midway,above,font=\scriptsize] {$-k$} (left);
\end{tikzpicture}
	\label{e:def-Cs}
\\
C_{5,j}^{-k(\eps)} =C_{6,j}^{-k(\eps)}  \eqdef 
\begin{tikzpicture}  [baseline=10]
\node[root]	(root) 	at (0,0) {};
\node[dot]		(left)  	at (-0.7,1) {};
\node[dot]		(right)  	at (0.7,1) {};			
\draw[kernel,bend right=20] (left) to 
	node [midway,left,font=\scriptsize] {$-k,j$} (root);
\draw[kernel,bend left=20]   (right)  to  (root) ;
\draw[kernel,bend right=20]  (right)  to  (left);
\end{tikzpicture}
\qquad
C_{7,j}^{-k(\eps)}  \eqdef 
\begin{tikzpicture}  [baseline=10]
\node[root]	(root) 	at (0,0) {};
\node[dot]		(left)  	at (-0.7,1) {};
\node[dot]		(right)  	at (0.7,1) {};			
\draw[kernel,bend right=20] (left) to 
	node [midway,left,font=\scriptsize] {$-k$} (root);
\draw[kernel,bend left=20]   (right)  to 
	node [midway,right,font=\scriptsize] {$j$} (root) ;
\draw[kernel,bend right=20]  (right)  to  (left);
\end{tikzpicture}
\qquad
C_{7,-k}^{-k(\eps)}  \eqdef 
\begin{tikzpicture}  [baseline=10]
\node[root]	(root) 	at (0,0) {};
\node[dot]		(left)  	at (-0.7,1) {};
\node[dot]		(right)  	at (0.7,1) {};			
\draw[kernel,bend right=20] (left) to 
	node [midway,left,font=\scriptsize] {$-k$} (root);
\draw[kernel,bend left=20]   (right)  to
	 node [midway,right,font=\scriptsize] {$-k$} (root) ;
\draw[kernel,bend right=20]  (right)  to  (left);
\end{tikzpicture}
	\notag
\end{gather}
Furthermore we define
\[
c_B^\e \eqdef
\eps\;
\begin{tikzpicture}  [baseline=10]
\node[root]	(root) 	at (0,0) {};
\node[dot]		(top)  	at (0,1) {};			
\draw[kernel,bend left =60]   (top)  to node [midway,right,font=\scriptsize] {$j$}  (root) ;
\draw[kernel,bend right =60]   (top)  to  (root) ;
\end{tikzpicture}
\;=\;
\eps \int_{\R\times \Lambda_\e} K^\e(-z) \nabla^\e_{\be_j} K^\e (-z) \,dz \;.
\]	
\begin{equs} [e:def-C1]
C_1^{(\eps)} & 
\;\; \eqdef \;\;
\sum_{k=1,2}
\Big(  C_{3,j}^{k(\eps)}  -C_{4,j}^{k(\eps)}
+C_{5,j}^{-k(\eps)}
- C_{7,j}^{-k(\eps)} \Big) 
-  c_B^\e\;.
\end{equs}
\[
c_\Psi^\e \eqdef
\eps\;
\begin{tikzpicture}  [baseline=10]
\node[root]	(root) 	at (0,0) {};
\node[dot]		(top)  	at (0,1) {};			
\draw[kernel,bend left =60]   (top)  to node [midway,right,font=\scriptsize] {$-j$}  (root) ;
\draw[kernel,bend right =60]   (top)  to  (root) ;
\end{tikzpicture}
\;-\;
\eps\;
\begin{tikzpicture}  [baseline=10]
\node[root]	(root) 	at (0,0) {};
\node[dot]		(top)  	at (0,1) {};			
\draw[kernel,bend left =60]   (top)  to node [midway,right,font=\scriptsize] {$j$}  (root) ;
\draw[kernel,bend right =60]   (top)  to  (root) ;
\end{tikzpicture} \;.
\]
Note that the values of the last three constants  do not  depend on $j$ since the integrals
remain the same values when the spatial coordinates are switched.
The choice of $C_1^{(\eps)}$ is such that 
in the renormalized equation for $B^\e$ derived below,
all the renormalization constants will precisely cancel out
(at the same time we can ensure that the models converge).
These constants then determine for each $\eps>0$ an element 
$M_\eps$ of the renormalization group. 
The maps $M_\eps$ acting on the canonical models $Z^\e$
then yield a sequence of models $\hat Z^\e$ which are called renormalized models. This follows from the general result
in \cite{bruned2016algebraic}.

In fact all the constants with two distinct indices $k\neq j$,  
converge to finite limits,
which follows from the following lemma (``parity''). 
The other renormalization constants defined in \eqref{e:def-Cs}, \eqref{e:def-C1} diverge logarithmically.
The constant $c_B^\e$ and $c_\Psi^\e$ converge to finite but non-zero limits by Lemma~\ref{lem:int-P-dP}.

\begin{lemma} \label{lem:some-are-finite}
For $k\neq j$, one has
\begin{equ}[e:even-odd-ha]
\lim_{\e \to 0}
\int_{(\R\times \Lambda^\e)^2} K^\e (x-y)\,
	\nabla^\e_{\pm \be_j} K^\e (x-z) \, 
	\nabla^\e_{\pm \be_k}  K^\e (z-y)\,dz\,dy
	<\infty \;.
\end{equ}
The inequality also holds with the last factor of the integrand replaced by $\nabla^\e_{\pm \be_k}  K^\e (y-z)$.
\end{lemma}
\begin{proof}
We only prove the case where the two discrete derivatives in 
 \eqref{e:even-odd-ha} are $\nabla^\e_{\be_j}$ and $\nabla^\e_{\be_k}$;
the other cases follow in the same way.
Since the truncated heat kernel $K^\e(x^{(0)},x^{(1)},x^{(2)})$ for $x\in \R\times \Lambda_\e$
is even in both spatial variables $x^{(1)}$ and $x^{(2)}$,
simultaneously flipping the signs of the $j$-th spatial variables of $x,y,z$  in \eqref{e:even-odd-ha} we see that the integral   in \eqref{e:even-odd-ha} is equal to
the same integral with $\nabla^\e_{\be_j}$
replaced by $\nabla^\e_{-\be_j}$. Therefore \eqref{e:even-odd-ha}
is equal to $1/2$ times
\[
-\e \int_{(\R\times \Lambda^\e)^2} K^\e (x-y)\,
	\nabla^\e_{-\be_j} \nabla^\e_{\be_j} K^\e (x-z) \, 
	  \nabla^\e_{\be_k}  K^\e (z-y)\,dz\,dy\;.
\]
It is easy to show that $\nabla^\e_{-\be_j} \nabla^\e_{\be_j}  K^\e$
is a kernel of order $-4$ (see \eqref{def:DSingKer} below). By Lemma~\ref{lem:OpDSingKer} below and a bound 
$\int_{\R\times \Lambda^\e } 
	\e \,(\|z\|_\s \vee \e)^{-5} dz =O(1)$
 we obtain \eqref{e:even-odd-ha}.
\end{proof}

In the previous section, we saw in Lemma~\ref{lem:unren-equ}  an ``un-renormalized''  equation by applying
the reconstruction operator $\CR^\eps$ for the model $Z^\e$
on 
the abstract equation \eqref{e:abs-fp}.
Its solution would not converge. 
In the sequel, we will 
consider the abstract equation \eqref{e:abs-fp}
with every incidence of the reconstruction operator $\CR^\eps$ in the remainder terms
\eqref{e:remainders} replaced by the reconstruction operator denoted by ${\hat\CR}^\e$
for the renormalized model $\hat Z^\e$.

\begin{lemma}\label{lem:renorm-equ}
Suppose that $(\bB^\e,\bPsi^\e)$ is a solution to the abstract fixed point problem \eqref{e:abs-fp}. Then
$(\hat\CR_t^\eps \bB^\eps,  \hat\CR_t^\eps \bPsi^\eps)$ satisfy 
\eqref{e:BPsi-poly}, with the constant $C^{(\e)}$  in \eqref{e:BPsi-poly}  given by
\begin{equ} [e:value-of-C]
C^{(\e)} = 
- 2 \lambda^2 
   C_2^{(\e)} 
 + \lambda^2 \sum_{k=1,2}
    \Big(
  C_{3,k}^{k(\eps)} +C_{4,-k}^{k(\eps)} 
   \Big)  + \lambda^2 c^\e_{\Psi}
   \;.
\end{equ}
\end{lemma}

\begin{proof}
In view of the abstract equation \eqref{e:abs-fp} (with $\CR^\eps$ in the remainder terms
\eqref{e:remainders} replaced by ${\hat\CR}^\e$),
the solution $(\bB^\e,\bPsi^\e)$ in  $\CD^{\gamma, \eta}_{\e,T}$ must have the following expansion truncated at homogeneity $1$
\begin{equs} [e:expand-sol]
\bB^\eps_j(x) &= \<Ixij> + b_j(x) \one 
	+ \lambda \varphi_1(x) \;\<I-dPsi2>
	- \lambda \varphi_2(x) \;\<I-dPsi1>
	+ \lambda \<I-Psi1dPsi2> - \lambda \<I-Psi2dPsi1>
	+ \langle \nabla b_j(x), \mathbf X \rangle  \\
\bPsi^\eps_j(x) 
&= \<Izetaj> + \varphi_j(x) \one 	+ \langle\nabla\varphi_j(x),\mathbf X\rangle \\
&
	-(-1)^{j} \lambda
	   \sum_{k=1,2 \atop \ell\neq j} 
	 \Big[ 
		b_k(x)\<I-dPsil> 
		+\phi_\ell (x)\<I(I-k(Bk))> 
		-\phi_\ell  (x)\<d-kI(I(Bk))> 
		-b_k(x) \<d-kI(I(Pe))> 
	 +
		 \<I-BdPsi> +  \<I(d-kBP)> - \<d-kI(BP)>  \Big]
\end{equs}
	 for functions  \footnote{We omit the dependence of $b_j ,\phi_j,\nabla b_j ,\nabla\phi_j$ on $\e$  in our notation.}
$b_j ,\phi_j:\Lambda_\e \times \R \to \R$ and  
$\nabla b_j ,\nabla\phi_j:\Lambda_\e \times \R \to \R^2$ and $\nabla_i\phi_j$ denotes the $i$-th component of $\nabla\phi_j$.
The modeled distributions 
$\DD_{i} \bPsi^\eps_j$ (resp.
 $\DD_{\- i} \bPsi^\eps_j$)
 then has expansion derived directly:
 namely, every symbol gets an extra ``derivative'' decoration
$i$ (resp. $\-i$) 
on the bottom line, and the abstract polynomial
part is given by $ \nabla_i \varphi_j(x) \one $
(resp. $-\nabla_{i} \varphi_j(x-\be_i) \one $). 
$\DD_{\-j}\bB^\eps_j$ is expanded in the same way.


It is then straightforward to check that
the nonlinear term $\hat F_{B_j^\e} = \hat F_{B_j^\e}(\bB^\eps,   \bPsi^\eps)$
defined in \eqref{e:FhatB-fields}  has the following expression,
truncated at homogeneity $0$:
\begin{equs}
\,&\hat F_{B_j^\e} (x)= \\
 &\sum_{k}  (-1)^k\lambda \Big( 
		 \varphi_{3-k}(x) \; \<Ij-zeta-k>  
		-   \nabla_j\varphi_{3-k}(x) \; \<Izetak>
		+ \langle \nabla\varphi_{3-k}(x), \<X-Ij-zeta-k> \rangle + \varphi_{3-k}(x)\nabla_j \varphi_k(x)\one\Big)
		\\
	&  + \lambda^2 \sum_{k,\ell} \Big[
		\<dj-Psi-IBdkPsi> 
		 -   \<Psi-dIj-BkdPsi>  
		 +\<djPeI(d-kBkPe)>
		 -\<PedjI(d-kBkPe)>
		+   \<Ped-kdjI(BkPe)> 
		-   \<djPed-kI(BkPe)> 
		\Big]  
		\\
	&+  \lambda^2 \sum_{k,\ell } b_k(x)  
	 \bigg[ 
	  \<dj-Psil-Idk-Psil>
	-   \<Psi-dIdPsi>
	+   \<Ped-kdeIPe>
	-  \<Ped-kI(dePe)>	 \bigg]
	+ \Big(\cdots\Big)
		\\
	& 
	-\lambda^2 \sum_{k} \bigg[
	 \<BjPsikPsik>   
	+ b_j(x) \,\<PsikPsik> + \varphi_k(x)^2 \; \<Ixij> 
	 +b_j(x)\varphi_k^2(x)\one
	 + 2 \varphi_k(x) \<BjPsik>
	+ 2 b_j(x) \<Izetak>
	  \bigg].
\end{equs}
Here, $k,\ell$ sum over $\{1,2\}$,
and $(\cdots)$ denotes a linear combination of symbols 
 on which $M_\e$ acts trivially (and they arise from suitably modifying symbols in the line above ``$(\cdots)$'' by abstract polynomials);
 for instance, it includes a term $- \lambda^2\varphi_\ell(x) \<dI-BdPsi>$, which arises from replacing an incidence of $\<Izetal>$ in $- \lambda^2  \<Psi-dIj-BkdPsi>$ by $\varphi_\ell(x)\one$. We refrain from listing all these terms because they do not contribute any renormalization at all in the renormalized equation.

Applying the reconstruction operator $\hat\CR^\e$ associated to the renormalized model $\hat\Pi^\eps$
on both sides of \eqref{e:abs-fp},
using $(\hat\Pi^{\eps,t}_x \tau) (t,x)=(\Pi^{\eps,t}_x M_\eps \tau) (t,x) $, 
following the analogous arguments in proof of Lemma~\ref{lem:unren-equ} and noting that
\begin{equs}
\hat\CR^\e \bB_j^\eps = K^\eps \ast_\e \xi^\eps_j + b_j
\qquad
\hat\CR^\e \bPsi_j^\eps = K^\eps \ast_\e \zeta^\eps_j + \varphi_j
\end{equs}
 one has that 
 $(\hat\CR^\e \bB^\eps,\hat\CR^\e \bPsi^\eps)$ satisfy  \eqref{e:BPsi-poly1}
with the following extra terms on the right hand side:
\begin{equs}[e:to-be-zero]
 2\lambda^2 
 \sum_{k=1,2}
\Big( - C_{3,j}^{k(\eps)}  + C_{4,j}^{k(\eps)}
-C_{5,j}^{-k(\eps)}
+ C_{7,j}^{-k(\eps)} 
\Big) &
\hat\CR^\e \bB_j^\eps  (\Cdot)
+ 2 \lambda^2 c_B^\e  \hat\CR^\e \bB_j^\eps
\\
&+ 2\lambda^2 C_1^{(\eps)} \hat\CR^\e \bB_j^\eps  (\Cdot) 
\end{equs}
where the factor $2$ comes from summing over an index,
and the term
$2 \lambda^2 c_B^\e  \hat\CR^\e \bB_j^\eps$
arises from 
the ``remainder'' term $\widetilde{\boldsymbol R}_{B^\eps_j}^\eps $ given in
\eqref{e:remainders3}.

By definition of $C_1^{(\eps)}$ given in \eqref{e:def-C1} we have
\[
\eqref{e:to-be-zero}\; =\; 0\;.
\]

It is also straightforward to check that 
$\hat F_{\Psi_j^\e} (x)  =  \sum_{m=1}^4 \hat f^{(m)}_{\Psi_j^\e} (x)$ with
(indices sum over $\{1,2\}$ if not specified)
\begin{equs}
 \hat f^{(1)}_{\Psi_j^\e} (x)  & \eqdef 
 -(-1)^j \lambda  \sum_{k=1,2 \atop \ell\neq j}   
\Big[
	 \nabla_k \varphi_{\ell}(x) \<Ixik>
	 +b_k(x)  \<Ik-zeta-l>
	 +\varphi_{\ell}(x) \<I-k-xi-k>
	 +\nabla_{\- k} b_k(x)  \<Izetal>	
	\\
	& \;
	+ (b_k(x) \nabla_k\phi_{\ell} (x) - \nabla_{ k} b_k(x-\be_k)\phi_{\ell} (x)  )\one
	+  \Big\langle 
		\nabla b_k(x),   \<X-Ik-zeta-l>      \Big\rangle 
	+  \Big\langle 
		\nabla \phi_\ell(x),   \<X-I-k-xi-k>      \Big\rangle 	
	\Big] 
\\
 \hat f^{(2)}_{\Psi_j^\e} (x) & \eqdef 
	 \lambda^2 \sum_{k,\ell}  
	\Big( \<B-dIBdPsi> + \<d-kPI(BdeP)>
	+\<PdkI(d-eBP)>
	+\<d-kPI(d-eBP)>
	-\<Pd-edkI(BP)> -\<d-kPd-eI(BP)>
	\Big)
	\\
&\qquad + \lambda^2 \sum_{k,\ell}   
	\phi_j(x) \Big(
	 \<BkdkI(d-eBe)> + \<d-kBkI(d-eBe)>
	 -\<Bkd-edkI(Be)> - \<d-kBkd-eI(Be)>
	\Big)
+ \Big(\cdots \Big)
\\
 \hat f^{(3)}_{\Psi_j^\e} (x) & \eqdef 
	  \lambda^2  \sum_{k=1,2 \atop \ell\neq j}  
	  \Big(
		 \<dkPsil-IPsijdkPsil> 
		 - \<dkPsil-IPsildkPsij>
		  +\<Ped-kI(PjdkPe)> 
		 - \<Ped-kI(PedkPj)>
		 \Big) 
		 	\\
& \qquad
	+ \lambda^2 \sum_{k=1,2 \atop \ell\neq j} 
	 \Big[
		 \phi_j(x) \<dkPsil-IdkPsil>
		 -\phi_\ell(x) \<dkPsil-IdkPsij>
		 +\phi_j(x) \<dk-Psi-IdkPsi>
		 -\phi_\ell(x) \<dk-Psi-IdkPsi2>
		 + \Big(\cdots \Big)
		 \Big]
\\
 \hat f^{(4)}_{\Psi_j^\e} (x) &\eqdef 
	- \lambda^2 
		 \sum_{ k}
	 \Big[
	\<BkBkPsij>
	+ \varphi_j(x)    \<BkBk>
	+2 b_k(x)    \<BkPsij>  
	\\
	&\qquad\qquad\qquad
	+ 2\phi_j(x) b_k(x)   \<Ixik>
	+ b_k^2(x)  \<Izetaj>
	+ \phi_j(x) b_k^2(x)\one\Big]
\end{equs}
where 
again $(\cdots)$ denotes linear combination of symbols 
 on which $M_\e$ acts trivially, which arise from suitably modifying symbols in the first line  by abstract polynomials,
 and do not contribute any renormalization at all in the renormalized equation.
One also has (recall  \eqref{e:FhatB-fields})
\begin{equs}
\hat F^{(k)}_{\Psi_j^\e} (x) =
 (-1)^j \lambda \sum_{\ell\neq j} 
 \Big( \<BPsi>+
  \varphi_{\ell}(x) \<Ixik>
	 + b_k(x)  \<Izetal>	
	+ b_k(x) \phi_{\ell} (x) \one
	\Big) \;.
\end{equs}
Applying the reconstruction operator $\hat\CR^\e$ on \eqref{e:abs-fp} as before,
we see that 
$(\hat\CR^\e \bB^\eps,\hat\CR^\e \bPsi^\eps)$ satisfy  \eqref{e:BPsi-poly2}
with the following extra terms on the right hand side
\begin{equs}
 \lambda^2 
 \sum_{k=1,2}
\Big(- C_{4,k}^{-k(\eps)}  -C_{3,-k}^{-k(\eps)}
+C_{6,k}^{-k(\eps)}
&+C_{7,-k}^{-k(\eps)} 
-C_{3,k}^{k(\eps)}
- C_{4,-k}^{k(\eps)} 
\Big) 
\hat\CR^\e \bPsi_j^\eps
- \lambda^2 c^\e_{\Psi}\hat\CR^\e \bPsi_j^\eps
\\
&+ 2\lambda^2 C_2^{(\eps)} \hat\CR^\e \bPsi_j^\eps . 
\end{equs}
By simple summation by parts $\sum_y \nabla^\e_\be f (x-y) g(y-z)=\sum_y f(x-y)\nabla^\e_{\be}g(y-z)$, one has 
\[
-C_{4,k}^{-k(\eps)}  
+C_{6,k}^{-k(\eps)}
=
-C_{3,-k}^{-k(\eps)}
+C_{7,-k}^{-k(\eps)} 
 =0 \;.
\]
Therefore the renormalization constant $C^{(\e)}$ on the right hand side
is given by \eqref{e:value-of-C}.
\end{proof}

\begin{remark}
By the argument in Remark~\ref{rem:fav-id} one can prove that $C^{(\e)}=-\lambda^2 C_2^{(\e)}+O(1) =O(| \log\e|)$
which indeed diverges logarithmically. We omit the detail of this calculation since it is straightforward, but only remind the reader that the mass renormalization $C^{(\e)} \Phi^\eps$ {\it does not} break gauge invariance.
\end{remark}

\subsection{Moment bounds}
\label{sec:mom}


Our random models $\hat Z^\e$
 are stationary and Gaussian models 
 (i.e. built from Gaussian noises).
They belong to a finite inhomogeneous Wiener chaos:
 there exist kernels $\CW^{(\eps; k)}\tau$ such that $\bigl(\CW^{(\eps; k)} \tau\bigr)(z) \in H^{\otimes k}_\eps$ for $z \in \R \times \Lambda_\eps$
 where $H_\eps\eqdef L^2([-T,T]\times \Lambda_\eps)$, and
\begin{equs}[e:PiWiener]
\langle
\hat\Pi^{\eps, t}_0 \tau, \varphi\rangle_\eps &= \sum_{k \leq \VERT \tau \VERT} I^\eps_k\Big(\int_{\Lambda_\eps} \varphi(y)\, \bigl(\CW^{(\eps; k)}\tau\bigr)(t,y)\, dy\Big)\;,
\end{equs}
where $ \VERT \tau \VERT$ is the number of total occurrences of $\Xi_{\xi_j}$ or $\Xi_{\zeta_j}$
for $j=1$ or $2$ in $\tau$, and
$I^\eps_k$ is the $k$-th order Wiener integrals with respect to $\xi^\eps_j$ or $\zeta^\e_j$
depending on the particular types of noise $\Xi$ occurred in $\tau$.
%
%
Then we define
\begin{equs}[e:CovDef]
\bigl(\CK^{(\eps; k)}\tau\bigr)(z_1, z_2) \eqdef \langle \bigl(\CW^{(\eps; k)}\tau\bigr)(z_1), \bigl(\CW^{(\eps; k)}\tau\bigr)(z_2) \rangle_{H_\eps^{\otimes k}}\;,
\end{equs}
for $z_1 \neq z_2 \in \R \times \Lambda_\eps$.
The functions $\CK^{(\eps; k)}\tau$ will depend on the time variables $t_1$ and $t_2$ only via $t_1 - t_2$, i.e.
$\bigl(\CK^{(\e,k)}\tau\bigr)_{t_1 - t_2}(x_1, x_2) \eqdef \bigl(\CK^{(\e,k)}\tau\bigr)(z_1, z_2)$
where $z_i = (t_i, x_i)$.

We will then apply the following result with $d=2$.
Recall that 
the norm $\VERT Z^\e \VERT_{\delta, \gamma; T}^{(\eps)}$
on discrete inhomogeneous models $Z$ is defined in Section~\ref{sec:Inhomo-models},
and the distance  $\VERT Z; Z^\e \VERT_{\delta, \gamma; T}^{(0,\eps)}$ defined in 
\eqref{e:ZeZ-distance}.
We will also use the following set in the statement of the following result
\begin{equ}[e:TMinus]
\hat\CF^{-} \eqdef \left(\{\tau \in \hat\CF: |\tau| < 0\} \cup \CF^{gen}\right) \setminus \poly{\CF}\;.
\end{equ}

\begin{proposition}
\label{p:CovarianceConvergence}
Let $Z^\eps$ be an admissible stationary Gaussian discrete model on the truncated regularity structure $\hat{\ST} = (\hat\CT, \CG)$. 
Let $\kappa > 0$ be such that the bounds
\begin{equs}[e:GaussModelBoundSigmaGamma]
\E \Vert \Gamma^{\eps, t}_{x y} \tau \Vert_{m}^2 \lesssim | x-y|^{2(|\tau| - m) + \kappa}\;,
\qquad
\E \Vert \Sigma^{\eps, s t}_{x} \tau \Vert_{m}^2 \lesssim \bigl(|t - s|^{\frac12} \vee \eps\bigr)^{2 (|\tau| - m) + \kappa}\;,
\end{equs}
hold for $\tau \in \CF^{gen} \setminus \poly{\CF}$, for all $s, t \in [-T, T]$, all $x, y \in \Lambda_\eps$, all $m \in \hat{\CA}$ such that $m < |\tau|$.
Assume that for some $\delta > 0$ and for each $\tau \in \hat\CF^{-}$ the bounds
\begin{equs}[e:GaussModelBound]
\E |\langle \Pi^{\eps, t}_x \tau, \varphi_x^\lambda\rangle_\eps|^2 &\lesssim \lambda^{2 |\tau| + \kappa}\;, \\
\E |\langle \bigl(\Pi^{\eps, t}_x - \Pi^{\eps, s}_x\bigr) \tau, \varphi_x^\lambda\rangle_\eps|^2 &\lesssim \lambda^{2 (|\tau| - \delta) + \kappa} \bigl(|t-s|^{\frac12} \vee \eps\bigr)^{2 \delta}\;,
\end{equs}
hold uniformly in all $\lambda \in [\eps,1]$, all $s, t \in [-T, T]$, all $x \in \Lambda_\eps$ and all $\varphi \in \CB^r_0(\R^d)$ with $r > - \lfloor\min \hat{\CA}\rfloor$.
%
Then, one has
$\E \big(\VERT Z^\eps \VERT_{\bar \delta, \gamma; T}^{(\eps)}\big)^p \lesssim 1$  for every $\gamma > 0$, $p \geq 1$ and $\bar \delta \in [0, \delta)$. 

Furthermore, let $Z=(\Pi,\Gamma,\Sigma)$ be an admissible stationary Gaussian discrete model on $\hat{\ST}$ such that 
 for some $\theta>0$,
the maps $\Gamma^\e - \Gamma$ and $\Sigma^\e - \Sigma$
satisfy \eqref{e:GaussModelBoundSigmaGamma} with proportionality constants of order $\e^{2\theta}$ and
\begin{equs}[e:GaussModelBound-diff]
\E | \langle \Pi^{\eps, t}_x \tau, \varphi_x^\lambda\rangle_\eps
	-\langle \Pi^{t}_x \tau, \varphi_x^\lambda\rangle |^2 
&\lesssim \e^{2\theta} \lambda^{2 |\tau| + \kappa}\;,
 \\
\E |\langle \bigl(\Pi^{\eps, t}_x - \Pi^{\eps, s}_x\bigr) \tau, \varphi_x^\lambda\rangle_\eps
	-\langle \bigl(\Pi^{t}_x - \Pi^{s}_x\bigr) \tau, \varphi_x^\lambda\rangle
|^2 
&\lesssim \e^{2\theta} \lambda^{2 (|\tau| - \delta) + \kappa} \bigl(|t-s|^{\frac12} \vee \eps\bigr)^{2 \delta}\;,
\end{equs}
uniformly in all parameters as above.
Then, one has
$\E \big(\VERT Z ; Z^\eps  \VERT_{\bar \delta, \gamma; T}^{(0,\eps)}\big)^p \lesssim \e^{\theta p}$  for every $\gamma > 0$, $p \geq 1$ and $\bar \delta \in [0, \delta)$. 
\end{proposition}

\begin{proof}
The bound on $\VERT Z^\eps \VERT_{\bar \delta, \gamma; T}^{(\eps)}$ is precisely the content of \cite[Thoerem~6.1]{hairer2015discrete}.  The bound on $\VERT Z ; Z^\eps  \VERT_{\bar \delta, \gamma; T}^{(0,\eps)}$ follows in the same way as \cite[Thoerem~10.7]{Regularity}.
\end{proof}

One useful way to verify conditions 
in the above proposition is, by \cite[Prop~6.2]{hairer2015discrete}, that it holds once we have
 that for every $\tau \in \hat\CF^{-}$ there are values $\alpha > |\tau| \vee (-d/2)$ and $\delta \in (0, \alpha + d / 2)$ such that 
\begin{equs}[e:CovarianceBounds]
|\big(\CK^{(\eps; k)}\tau\big)_0\big(x_1, x_2\big)| 
&\lesssim \sum_{\zeta \geq 0} \big( |x_1| \vee |x_2|\big)^\zeta \bigl(|x_1- x_2| \vee \eps \bigr)^{2 \alpha - \zeta}\;,
\\
\frac{|\delta^{0, t} \big(\CK^{(\eps; k)}\tau\big)\big(x_1, x_2\big)|}{\bigl(|t|^{\frac12} \vee \eps \bigr)^{2\delta}} 
&\lesssim \sum_{\zeta \geq 0} \big( |x_1| \vee |x_2| \vee |t|^{\frac12}\big)^\zeta \bigl(|x_1- x_2| \vee \eps\bigr)^{2 \alpha - 2 \delta - \zeta}\;,
\end{equs}
 for $x_1, x_2 \in \Lambda_\eps$ and $k \leq \VERT \tau \VERT$, where the operator $\delta^{0, t}$ is defined in \eqref{e:deltaTime}, and the sums run over finitely many values of $\zeta \in [0, 2 \alpha - 2 \delta + d)$. 
To verify \eqref{e:CovarianceBounds},
recall that for a  function $\CK$ on $\R \times \Lambda_\eps$ supported in a ball around the origin, we say that it is of order $\alpha\in\R$ if for some $m\in\N$ the quantity
\begin{equation}\label{def:DSingKer}
\VERT \CK\VERT_{\alpha;m}^{(\eps)}\eqdef \max_{|k|_\s\leq m}\sup_{z\in \R \times \Lambda_\eps} \frac{\left| D^k_\eps \CK(z)\right|}{(\|z\|_\s\vee \eps)^{\alpha-|k|_\s}}
\end{equation}
is bounded uniformly in $\eps$, where $k=(k_0,k_1,k_2)\in\N^3$, 
 and  $D^k_\eps\eqdef \partial_t^{k_0} (\nabla^\eps_1)^{k_1}(\nabla^\eps_2)^{k_2}$. In particular, the truncated heat kernel $K^\e$ in our case is of order $-2$.
 Often the kernels $\CK^{(\e;k)}\tau$ will depend on the spatial variables $x_1,x_2$ via the difference $x_1-x_2$ and comes with an order of singularity i.e. $\CK^{(\e;k)}\tau$  is a function on $\R \times \Lambda_\eps$, and supported in a ball around the origin; in this case one can easily prove that if $\CK^{(\e;k)}\tau$  is order $2\alpha<0$ then \eqref{e:CovarianceBounds} holds with $\zeta=0$ and $\delta\in  (0, (\alpha + \frac{d}{2})\wedge 1)$.

To obtain the orders of the kernels we will frequently use the following results from \cite[Lemma~7.2 and 7.4]{hairer2015discrete}
with $|\s|=4$ and $d=2$.

\begin{lemma}\label{lem:OpDSingKer}
Let $K_1^\eps$ and $K_2^\eps$ be functions on $[-T,T]\times \Lambda_\eps$ of order $\zeta_1$ and $\zeta_2\in\R$ respectively. Let $\bar{\zeta}\eqdef \zeta_1+\zeta_2+|\s|$.
Then for any $m\in\N$ the following bounds hold uniformly in $\eps$ for $\e$ sufficiently small.
\begin{enumerate}
\item
$\VERT K_1^\eps K_2^\eps \VERT^{(\eps)}_{\zeta_1+\zeta_2;m}\lesssim \VERT K_1^\eps\VERT^{(\eps)}_{\zeta_1;m}\VERT K_2^\eps\VERT^{(\eps)}_{\zeta_2;m}$.
%
\item 			\label{item:convolution}
$\VERT K_1^\eps\ast_\eps K_2^\eps\VERT^{(\eps)}_{\bar{\zeta};m}\lesssim \VERT K_1^\eps\VERT^{(\eps)}_{\zeta_1;m}\VERT K_2^\eps\VERT^{(\eps)}_{\zeta_2;m}$
provided that $\zeta_1\wedge \zeta_2>-|\s|$ and $\bar{\zeta}<0$.
\item 			\label{item:positive}
If $\bar{\zeta}\in (0,2) \setminus\N$, then
$\VERT \bar{K}^\eps\VERT^{(\eps)}_{\bar{\zeta};m}\lesssim \VERT K_1^\eps\VERT^{(\eps)}_{\zeta_1;\bar{m}}\VERT K_2^\eps\VERT^{(\eps)}_{\zeta_2;\bar{m}}$ where
$\bar{m}=m\vee(\lfloor\bar{\zeta}\rfloor +2)$ and
\[
\bar{K}^\eps(z)
\eqdef K_1^\eps\ast_\eps K_2^\eps(z)
	-\sum_{|k|_\s<\bar{\zeta}} 
	\frac{z^k}{k!} D_\eps^k (K_1^\eps\ast_\eps K_2^\eps) (0)\;.
\]
\item				\label{item:renormalized}
If
$\zeta_1 \in \bigl(-|\s|-1, -|\s|\bigr]$ and $\zeta_2 \in \bigl(-2 |\s|-\zeta_1, 0\bigr]$, then 
\begin{equ} [e:RenorConv]
\VERT{\bigl(\mathscr{R} K^\e_1\bigr) \ast_\eps K^\e_2}\VERT_{\bar{\zeta}; m} \lesssim \VERT{K^\e_1}\VERT_{\zeta_1; m} \VERT{K^\e_2}\VERT_{\zeta_2; m + 2}
\end{equ}
where 
$\left(\mathscr{R} K^\e\right)(\varphi) \eqdef \int_{\R \times \Lambda_\eps} K^\e(z) \left( \varphi(z) - \varphi(0) \right) dz$
 for any compactly supported test function $\varphi$ on $\R^{d+1}$.
\end{enumerate}
\end{lemma}

\begin{proposition} \label{prop:moments}
Let $\hat Z^\e = (\hat \Pi^\e, \hat \Gamma^\e)$ be the renormalized inhomogeneous  model  defined above.
There exists a sufficiently small $\delta>0$ 
such that
$\E \big[\VERT Z^\eps \VERT_{\bar \delta, \gamma; T}^{(\eps)}\big]^p \lesssim 1$
 for every $\gamma > 0$, $p \geq 1$ and $\bar \delta \in [0, \delta)$,
 uniformly in $\eps\in(0,1]$.
\end{proposition}

\begin{proof} 
We will find the kernels $\CK^\e$ for each of the symbols and verify the condition \eqref{e:CovarianceBounds}.

We will denote by $\;\tikz [baseline=-3] \node[zeta] {}; \;$ a generic noise, and a thick line for a spatial derivative of heat kernel.
For the symbols $ \<Izeta> $ 
the chaos expansion \eqref{e:PiWiener} only involves the first chaos (i.e. $k=1$),
and one has for both cases 
\[
(\CK^{(\eps,1)} \tau) (z_1,z_2)=\langle K^\e(z_1- \cdot ), K^\e(z_2- \cdot) \rangle_{H_\e}
\]
which is of order $0^-$ (i.e. order $\alpha$ for any 
 $\alpha<0$) as a function of $z_1-z_2$ by item \ref{item:convolution} of Lemma~\ref{lem:OpDSingKer}. 
By the discussion above Lemma~\ref{lem:OpDSingKer}, \eqref{e:CovarianceBounds} holds.
For symbol $\<I'zeta>$ 
since 
 $\hat\Pi_x^{\e,t} \<I'zeta> = \nabla_j^\e \hat\Pi_x^{\e,t}\<Izeta> $ and $\Gamma^\e,\Sigma^\e$ act on them trivially as they do for $\<Izeta>$, the bounds 
\eqref{e:PiGammaBound}, \eqref{e:SigmaBound}  and \eqref{e:PiTimeBound} still almost surely hold uniformly in $\e$. 
Regarding  the symbols $X\<I'zeta>$ 
again using item \ref{item:convolution} of Lemma~\ref{lem:OpDSingKer} 
one obtains a bound
\[
|\big(\CK^{(\eps; 1)}\tau\big)_0\big(x_1, x_2\big)| 
\lesssim \big( |x_1| \vee |x_2|\big)^2 \bigl(|x_1- x_2| \vee \eps \bigr)^{-2 }\;,
\]
as well as a similar bound on $\delta^{0,t} \big(\CK^{(\eps; 1)}\tau\big)$.
Successively applying item \ref{item:convolution} of Lemma~\ref{lem:OpDSingKer} and changes of variables
also yields kernels 
$\CK^{(\eps,1)} \tau$ for $\tau=\<dIdIzeta> $ 
which is of order $0^-$.

For the symbols of the type $\<I-PsidPsi>$
in the list \eqref{e:F-gen-list},
by conspection for our models the two noises  are always independent (recall that 
$\xi^\e_1,\xi^\e_2,\zeta^\e_1,\zeta^\e_2$ are independent).
It is straightforward to check using Lemma~\ref{lem:OpDSingKer}
(especially item~\ref{item:positive}) that for both symbols we have
\begin{equs} 
|\big(\CK^{(\eps; 2)}\tau\big)_0\big(x_1, x_2\big)| 
\lesssim  \big( |x_1| \vee |x_2|\big)^{2-\kappa}\;,
\quad
\frac{|\delta^{0, t} \big(\CK^{(\eps; 2)}\tau\big)\big(x_1, x_2\big)|}{\bigl(|t|^{\frac12} \vee \eps \bigr)^{2\delta}} 
\lesssim  \big( |x_1| \vee |x_2| \vee |t|^{\frac12}\big)^{2-\kappa-2\delta} 
\end{equs}
for $\kappa>0$,
so \eqref{e:CovarianceBounds} holds with $\alpha=1-\frac{\kappa}{2}$. 
The desired bounds on $\<dIV>$  
 then immediately follow.

We now focus on objects with two noises,
three heat kernels with two spatial derivatives. 

Let
 $\tau=\<Psi-dIdPsi>$.
Recall the graphical notation in \eqref{e:def-Cs}.
Moreover,
we will denote by
 $\tikz [baseline=-3] \node[vab]   at (0,0) {\tiny $j$};$  for the noise $\xi^\e_j$, and  $\tikz [baseline=-3] \node[var]   at (0,0) {\tiny $j$};$
 for the noise $\zeta^\e_j$,
and a green arrow $\tikz [baseline=-3] \draw[testfcn] (0.8,0) to (0,0);$ represents a rescaled test function $\phi^\lambda$.
The entire graph with, say, $m$ occurrences of   $\tikz [baseline=-3] \node[vab]   at (0,0) {};$ or  $\tikz [baseline=-3] \node[var]   at (0,0) {};$
then represents the $m$-th Wiener integral
against the noises represented by these 
$\tikz [baseline=-3] \node[vab]   at (0,0) {};$
or $\tikz [baseline=-3] \node[var]   at (0,0) {};$.
One then has
\begin{equ} [e:Psi-dIdPsi]
\langle
\hat\Pi^{\eps, t}_0 \tau, \varphi^\lambda \rangle_\eps
\quad =\quad
\begin{tikzpicture}[scale=0.8,baseline=-13]
\node[var] at (0.6,0.3) (top) {\tiny $\ell$};
\node[dot] at (0,-0.4)  (mid) {};
\node[dot] at (0.6,-1.1) (bot) {};
\node[var] at (1.2, -0.4) (botup) {\tiny $\ell$};
\node[root] at (0.6,-1.7) (root) {};
\draw[kernel] (top) to node[above] {\tiny $k$}  (mid);
\draw[kernel]  (mid) to node[left] {\tiny $j$} (bot); 
\draw[kernel]  (botup) to (bot); 
\draw[testfcn] (bot) to (root); 
\end{tikzpicture}
\quad 
+ 
\quad
\begin{tikzpicture}[scale=0.8,baseline=-13]
\node[dot] at (0,0)  (mid) {};
\node[dot] at (0,-0.8) (bot) {};
\node[dot] at (1.2, -0.3) (botup) {};
\node[root] at (0,-1.7) (root) {};
\draw[kernel,bend right=40] (botup) to   node[above] {\tiny $k$} (mid);
\draw[kernel]  (mid) to node[left] {\tiny $j$} (bot); 
\draw[kernel,bend left=40]  (botup) to  (bot); 
\draw[testfcn] (bot) to (root); 
\end{tikzpicture}
\quad
- 
\quad
  C_{4,j}^{k(\e)} 
 \;
\begin{tikzpicture}[scale=0.8,baseline=-12]
\node[dot] at (0,0)  (mid) {};
\node[root] at (0,-0.8) (root) {};
\draw[testfcn] (mid) to (root); 
\end{tikzpicture}
\end{equ}
%
For the first graph (i.e. second chaos), 
Lemma~\ref{lem:OpDSingKer} then yields a kernel 
$\CK^{(\eps,2)} $ of order $0^-$.
For the other two terms (i.e. zeroth chaos) 
the renormalization constant $  C_{4,j}^{k(\e)}$
is precisely  defined  such that the two zeroth chaos terms  cancel.  The symbols $ \<dk-Psi-IdkPsi> , \<BldlI(d-eBe)>$ are treated analogously with renormalization constants $C_{4,-k}^{k(\e)}$
and $C_{4,\ell}^{-\ell(\e)}$ respectively.

The symbols $\<dj-Psii-Idk-Psil> $ and $ \<d-lBlI(d-eBe)>$
are treated in the same way so we only
consider the first one. We  introduce one more graphical notation: a barred arrow   $\tikz \draw[kernel1]  (0,0) to (1.2,0);$ represents $K^\e( z-w)-K^\e(-w)$ where
$w$ and $z$ are the coordinates of the starting and end point of the arrow respectively. One then has
\begin{equ} [e:boundingUb]
\langle
\hat\Pi^{\eps, t}_0 \tau, \varphi^\lambda
\rangle_\eps
\quad =\quad
\begin{tikzpicture}[scale=0.8,baseline=-13]
\node[var] at (0.6,0.3) (top) {\tiny $\ell$};
\node[dot] at (0,-0.4)  (mid) {};
\node[dot] at (0.6,-1.1) (bot) {};
\node[var] at (1.2, -0.4) (botup) {\tiny $i$};
\node[root] at (0.6,-1.7) (root) {};
\draw[kernel] (top) to node[above] {\tiny $k$}  (mid);
\draw[kernel1]  (mid) to  (bot); 
\draw[kernel]  (botup)  to node[right] {\tiny $j$} (bot); 
\draw[testfcn] (bot) to (root); 
\end{tikzpicture}
\quad 
+ 
\quad
 \delta_{i,\ell}\;\;
\begin{tikzpicture}[scale=0.8,baseline=-13]
\node[dot] at (0,0)  (mid) {};
\node[dot] at (0,-0.8) (bot) {};
\node[dot] at (1.2, -0.3) (botup) {};
\node[root] at (0,-1.7) (root) {};
\draw[kernel,bend right=40] (botup) to   node[above] {\tiny $k$} (mid);
\draw[kernel1]  (mid) to  (bot); 
\draw[kernel,bend left=40]  (botup) to node[below] {\tiny $j$}  (bot); 
\draw[testfcn] (bot) to (root); 
\end{tikzpicture}
\quad
- 
\quad
  \delta_{i,\ell}  C_{3,j}^{k(\e)} 
 \;
\begin{tikzpicture}[scale=0.8,baseline=-12]
\node[dot] at (0,0)  (mid) {};
\node[root] at (0,-0.8) (root) {};
\draw[testfcn] (mid) to (root); 
\end{tikzpicture}
\end{equ}
For the first graph (i.e. second chaos)
applying item \ref{item:convolution} and \ref{item:positive} of Lemma~\ref{lem:OpDSingKer}
one  obtains a bound
\[
|\big(\CK^{(\eps; 2)}\tau\big)_0\big(x_1, x_2\big)| 
\lesssim  \big( |x_1| \vee |x_2|\big)^{2-2\delta} \bigl(|x_1- x_2| \vee \eps \bigr)^{-2 -2\delta}
\]
and
a similar bound on $\delta^{0,t} \big(\CK^{(\eps; 2)}\tau\big)$.
Assuming now $i=\ell$ and considering the second graph,
we bound the function represented by the barred arrow as a sum of
two separate parts:
 a kernel $K^\e$ represented by an arrow pointing towards the origin and a kernel $K^\e$ represented by an arrow pointing towards the tip of the green arrow
 \begin{equ}
 \quad
-
\quad
\begin{tikzpicture}[scale=0.8,baseline=-13]
\node[dot] at (0,0)  (mid) {};
\node[dot] at (0,-0.8) (bot) {};
\node[dot] at (1.2, -0.3) (botup) {};
\node[root] at (0,-1.7) (root) {};
\draw[kernel,bend right=40] (botup) to   node[above] {\tiny $k$} (mid);
\draw[kernel, bend right=50]  (mid) to  (root); 
\draw[kernel,bend left=40]  (botup) to node[below] {\tiny $j$}  (bot); 
\draw[testfcn] (bot) to (root); 
\end{tikzpicture}
\quad
+
\quad
\begin{tikzpicture}[scale=0.8,baseline=-13]
\node[dot] at (0,0)  (mid) {};
\node[dot] at (0,-0.8) (bot) {};
\node[dot] at (1.2, -0.3) (botup) {};
\node[root] at (0,-1.7) (root) {};
\draw[kernel,bend right=40] (botup) to   node[above] {\tiny $k$} (mid);
\draw[kernel]  (mid) to  (bot); 
\draw[kernel,bend left=40]  (botup) to node[below] {\tiny $j$}  (bot); 
\draw[testfcn] (bot) to (root); 
\end{tikzpicture}
\quad
- 
\quad
  C_{3,j}^{k(\e)}
 \;
\begin{tikzpicture}[scale=0.8,baseline=-12]
\node[dot] at (0,0)  (mid) {};
\node[root] at (0,-0.8) (root) {};
\draw[testfcn] (mid) to (root); 
\end{tikzpicture}
\end{equ}
By Lemma~\ref{lem:OpDSingKer}
the first term here is then deterministically bounded by $\lambda^{\alpha}$ for any $\alpha<0$.
Our choice of $C_{3,j}^{k(\e)}$
precisely cancel  the second term.

All the other objects with two {\it independent} noises,
three heat kernels with two spatial derivatives 
such as 
$ \<B-dIdPsi>$ give $\CK^{(\eps,2)} $ of order $0^-$
by Lemma~\ref{lem:OpDSingKer}, 
 thus the desired bounds  \eqref{e:CovarianceBounds} hold.

The symbol $\tau= \<Psi-dIBdPsi>$ is treated in a similar way as $\<Psi-dIdPsi>$ above. In fact we have
\begin{equ} [e:Psi-dIBdPsi]
\langle
\hat\Pi^{\eps, t}_0 \tau, \varphi^\lambda 
\rangle_\eps
\quad =\quad
\begin{tikzpicture}[scale=0.8,baseline=-13]
\node[var] at (0.6,0.3) (topright) {\tiny $\ell$};
\node[vab] at (-0.4,0.3) (topleft) {\tiny $j$};
\node[dot] at (0.1,-0.4)  (mid) {};
\node[dot] at (0.6,-1.1) (bot) {};
\node[var] at (1.2, -0.4) (botup) {\tiny $\ell$};
\node[root] at (0.6,-1.7) (root) {};
\draw[kernel] (topright) to node[above] {\tiny $j$}  (mid);
\draw[kernel] (topleft) to   (mid);
\draw[kernel]  (mid) to node[left] {\tiny $k$} (bot); 
\draw[kernel]  (botup) to (bot); 
\draw[testfcn] (bot) to (root); 
\end{tikzpicture}
\quad 
+ 
\quad
\begin{tikzpicture}[scale=0.8,baseline=-13]
\node[vab] at (0,0.4) (top) {\tiny $j$};
\node[dot] at (0,-0.3)  (mid) {};
\node[dot] at (0,-1.1) (bot) {};
\node[dot] at (1.2, -0.6) (botup) {};
\node[root] at (0,-1.7) (root) {};
\draw[kernel] (top) to (mid);
\draw[kernel,bend right=40] (botup) to   node[above] {\tiny $j$} (mid);
\draw[kernel]  (mid) to node[left] {\tiny $k$} (bot); 
\draw[kernel,bend left=40]  (botup) to  (bot); 
\draw[testfcn] (bot) to (root); 
\end{tikzpicture}
\quad
- 
\quad
C_{4,k}^{j(\e)}
 \;
\begin{tikzpicture}[scale=0.8,baseline=-12]
\node[vab] at (0,0.4) (top) {\tiny $j$};
\node[dot] at (0,-0.5)  (mid) {};
\node[root] at (0,-1.3) (root) {};
\draw[kernel] (top) to (mid);
\draw[testfcn] (mid) to (root); 
\end{tikzpicture}
\end{equ}
Applying  Lemma~\ref{lem:OpDSingKer} to the first graph yields kernel $\CK^{(\eps,3)} $ of order $0^-$.
Denoting by $\tikz \draw[kernelBig]  (0,0) to (1,0);$  the renormalized kernel $\mathscr R K_3^\e$ where
\begin{equ} [e:def-K-3]
K_3^\e (x-z) \eqdef \int_{\R\times \Lambda^\e} K^\e (x-y)\,
	\nabla^\e_{k} K^\e (x-z) \, \nabla^\e_{j}  
	K^\e (z-y)\,dy\;,
\end{equ}
and $\mathscr R$ is defined below \eqref{e:RenorConv},
the   second  graph is then equal to
\begin{equ}
\begin{tikzpicture}[scale=0.8,baseline=-13]
\node[vab] at (0,0.4) (top) {\tiny $j$};
\node[dot] at (0,-0.4)  (mid) {};
\node[dot] at (0,-1.1) (bot) {};
\node[root] at (0,-1.8) (root) {};
\draw[kernel]  (top) to (mid); 
\draw[kernelBig]  (mid) to (bot); 
\draw[testfcn] (bot) to (root); 
\end{tikzpicture}
\quad
+
\quad
\begin{tikzpicture}[scale=0.8,baseline=-13]
\node[vab] at (0,0.4) (top) {\tiny $j$};
\node[dot] at (0,-0.3)  (mid) {};
\node[dot] at (0,-1.1) (bot) {};
\node[dot] at (1.2, -0.6) (botup) {};
\node[root] at (0,-1.8) (root) {};
\draw[kernel,bend right=40] (top) to (bot);
\draw[kernel,bend right=40] (botup) to   node[above] {\tiny $j$} (mid);
\draw[kernel]  (mid) to node[right] {\tiny $k$} (bot); 
\draw[kernel,bend left=40]  (botup) to  (bot); 
\draw[testfcn] (bot) to (root); 
\end{tikzpicture}
\end{equ}
By Lemma~\ref{lem:OpDSingKer} (item \ref{item:convolution} and item \ref{item:renormalized}),
the first graph yields a kernel 
$\CK^{(\eps,1)} $ of order $0^-$.
The other graph is then cancelled by the last term in \eqref{e:Psi-dIBdPsi}.

Consider now  $\tau=\<dPsi-IPsidPsi> $ where $i\neq \ell$. One has
\begin{equ}[e:dPsi-IPsidPsi]
\langle
\hat\Pi^{\eps, t}_0 \tau, \varphi^\lambda
\rangle_\eps
\quad 
=
\quad
\begin{tikzpicture}[scale=0.8,baseline=-10]
\node[var] at (0.5,0.6) (topright) {\tiny $\ell$};
\node[var] at (-0.5,0.6) (topleft) {\tiny $i$};
\node[dot] at (0,-0.4)  (mid) {};
\node[dot] at (0.6,-1.3) (bot) {};
\node[var] at (1.2, -0.4) (botup) {\tiny $k$};
\node[root] at (0.6,-1.9) (root) {};
\draw[kernel] (topright) to node[right] {\tiny $j$}  (mid);
\draw[kernel]  (topleft) to (mid); 
\draw[kernel1]  (mid) to (bot); 
\draw[kernel]  (botup) to node[right] {\tiny $j$} (bot); 
\draw[testfcn] (bot) to (root); 
\end{tikzpicture}
\quad 
+ 
\quad
\delta_{i,k}
\begin{tikzpicture}[scale=0.8,baseline=-10]
\node[var] at (0,0.6) (topleft) {\tiny $\ell$};
\node[dot] at (0,-0.4)  (mid) {};
\node[dot] at (0,-1.3) (bot) {};
\node[dot] at (1.2, -0.3) (botup) {};
\node[root] at (0,-1.9) (root) {};
\draw[kernel,bend right=40] (botup) to   (mid);
\draw[kernel]  (topleft) to node[left] {\tiny $j$} (mid); 
\draw[kernel1]  (mid) to (bot); 
\draw[kernel,bend left=40]  (botup) to node[right] {\tiny $j$} (bot); 
\draw[testfcn] (bot) to (root); 
\end{tikzpicture}
\quad 
+ 
\quad
\delta_{\ell,k}
\begin{tikzpicture}[scale=0.8,baseline=-10]
\node[var] at (0,0.6) (topleft) {\tiny $i$};
\node[dot] at (0,-0.4)  (mid) {};
\node[dot] at (0,-1.3) (bot) {};
\node[dot] at (1.2, -0.3) (botup) {};
\node[root] at (0,-1.9) (root) {};
\draw[kernel,bend right=40] (botup) to   node[above] {\tiny $j$} (mid);
\draw[kernel]  (topleft) to (mid); 
\draw[kernel1]  (mid) to (bot); 
\draw[kernel,bend left=40]  (botup) to node[right] {\tiny $j$} (bot); 
\draw[testfcn] (bot) to (root); 
\end{tikzpicture}
\quad
- 
\quad
\delta_{\ell,k}
C_{3,j}^{j(\e)} 
\begin{tikzpicture}[scale=0.8,baseline=-10]
\node[var] at (0,0.3) (topleft) {\tiny $i$};
\node[dot] at (0,-0.8)  (mid) {};
\node[root] at (0,-1.6) (root) {};
%
\draw[kernel]  (topleft) to (mid); 
\draw[testfcn] (mid) to (root); 
\end{tikzpicture}
\end{equ}
The first term (i.e. the third chaos) results in a kernel with bound
\[
|\big(\CK^{(\eps; 3)}\tau\big)_0\big(x_1, x_2\big)| 
\lesssim  \big( |x_1| \vee |x_2|\big)^2 \bigl(|x_1- x_2| \vee \eps \bigr)^{-2 - \kappa}
\]
for $\kappa>0$ and a similar bound on $\delta^{0,t}(\CK^{(\eps; 3)}\tau)$.
For the second term (i.e. the first chaos), we bound its second moment as follows.
Denote by $y$ and $x$ the starting and end points of  
$\tikz \draw[kernel1]  (0,0) to (1.2,0);$ respectively, so that it 
represents $K^\e( x-y)-K^\e(-y)$. We bound the integral with $K^\e( x-y)$
and the integral with $K^\e(-y)$ separately. 
For the integral with $K^\e(x-y)$, one obtains a kernel $\CK^{(\eps; 1)}$ of order $\alpha$
for any $\alpha<0$.
For the integral with the function $K^\e( -y)$,
if $\|x-y\|_\s \le 2 \|y\|_\s$, we bound $|K^\e( -y)|\lesssim (\|x-y\|_\s \vee\e)^{-2}$
and the bound follows in the same way as that for the integral with $K^\e(x-y)$.
On the other hand, assume that $\|x-y\|_\s > 2 \|y\|_\s$, then one necessarily has $\|x\|_\s <  \frac32 \|x-y\|_\s$,
so we can multiply the integrand by $\|x\|_\s^{-\delta}\|x-y\|_\s^{\delta}$
for a sufficiently small $\delta>0$
and bound $|K^\e( x-y)|\|x-y\|_\s^{\delta}\lesssim (\|y\|_\s \vee\e)^{-2+\delta}$.
Note that now the integration over $y$ and the integration over $x$ factorize
and thus can be computed independently.
Applying Lemma~\ref{lem:OpDSingKer}, together with $\int\phi^\lambda(x)\|x\|_\s^{-\delta}dx\le \lambda^{-\delta}$ we obtain the desired bound for the second graph on the right hand side of \eqref{e:dPsi-IPsidPsi}.
The last two terms on the right hand side of  \eqref{e:dPsi-IPsidPsi} are equal to $\delta_{\ell,k}$ times
\begin{equ}
\begin{tikzpicture}[scale=0.8,baseline=-10]
\node[var] at (0,0.5) (top) {\tiny $i$};
\node[dot] at (0,-0.4)  (mid) {};
\node[dot] at (0,-1.3) (bot) {};
\node[root] at (0,-2) (root) {};
\draw[kernel]  (top) to (mid); 
\draw[kernelBig]  (mid) to (bot); 
\draw[testfcn] (bot) to (root); 
\end{tikzpicture}
\quad
-
\quad
\begin{tikzpicture}[scale=0.8,baseline=-10]
\node[var] at (0,0.5) (topleft) {\tiny $i$};
\node[dot] at (0,-0.4)  (mid) {};
\node[dot] at (0,-1.3) (bot) {};
\node[dot] at (1.2, -0.3) (botup) {};
\node[root] at (0,-2) (root) {};
%
\draw[kernel,bend right=40] (botup) to   node[above] {\tiny $j$} (mid);
\draw[kernel]  (topleft) to (mid); 
\draw[kernel,bend right=40]  (mid) to (0,-2); 
\draw[kernel,bend left=40]  (botup) to node[right] {\tiny $j$} (bot); 
\draw[testfcn] (bot) to (root); 
\end{tikzpicture}
\end{equ}
where  this time  $\tikz \draw[kernelBig]  (0,0) to (1,0);$
represents the renormalized kernel $\mathscr R \bar K_3^\e$ where
\begin{equ}
\bar K_3^\e (x-z) \eqdef \int_{\R\times \Lambda^\e}
 \nabla^\e_{j}  K^\e (x-y)\,
	K^\e (x-z) \, \nabla^\e_{j}  K^\e (z-y)\,dy\;.
\end{equ}
Lemma~\ref{lem:OpDSingKer} (item \ref{item:convolution} and item \ref{item:renormalized})
immediately gives the desired bound for the first term by $\lambda^{\alpha}$ for any $\alpha<0$.
The second term can be treated in the same way as the second graph on the right hand side of \eqref{e:dPsi-IPsidPsi}.
The other similar symbols such as $\<dPsi-IBdPsi>$ are  bounded in the analogous  way. 

Note that 
symbols of the form
$\<B-dIBdPsi> $  (the top edge with derivative carries a noise that is independent from all other noises)
have negative homogeneity but do not need renormalization. Indeed, 
the model  yield a process with the third chaos and first chaos:
\begin{equ}
\begin{tikzpicture}[scale=0.8,baseline=-10]
\node[var] at (0.5,0.6) (topright) {\tiny $j$};
\node[vab] at (-0.5,0.6) (topleft) {\tiny $\ell$};
\node[dot] at (0,-0.4)  (mid) {};
\node[dot] at (0.6,-1.3) (bot) {};
\node[vab] at (1.2, -0.4) (botup) {\tiny $k$};
\draw[kernel] (topright) to node[right] {\tiny $\ell$}  (mid);
\draw[kernel]  (topleft) to (mid); 
\draw[kernel]  (mid) to node[left] {\tiny $k$} (bot); 
\draw[kernel]  (botup) to  (bot); 
\end{tikzpicture}
\quad 
+ 
\quad
\delta_{\ell,k}
\begin{tikzpicture}[scale=0.8,baseline=-10]
\node[var] at (0,0.6) (topleft) {\tiny $j$};
\node[dot] at (0,-0.4)  (mid) {};
\node[dot] at (0,-1.3) (bot) {};
\node[dot] at (1.2, -0.3) (botup) {};
\draw[kernel,bend right=40] (botup) to   (mid);
\draw[kernel]  (topleft) to node[left] {\tiny $\ell$} (mid); 
\draw[kernel]  (mid) to node[left] {\tiny $k$} (bot); 
\draw[kernel,bend left=40]  (botup) to  (bot); 
\end{tikzpicture}
\end{equ}
Using Lemma~\ref{lem:OpDSingKer} we obtain
kernels $\CK^{(\e,3)}$ and $\CK^{(\e,1)}$ both of order $\alpha$ for $\alpha<0$.

The model acting on 
 $ \<BjPsik>$    
yields processes that live in the second homogeneous  chaos
by independence of $\xi^\e$ and $\zeta^\e$.
By Lemma~\ref{lem:OpDSingKer}, $\CK^{(\eps,2)} $
is of order  $0^-$, thus the desired bounds  \eqref{e:CovarianceBounds} hold. 
Turning to  $\tau= \<BjBj>$,  its second chaos component gives a kernel
$\CK^{(\e,2)}\tau$  of order $0^-$,
and the zeroth chaos component of $\langle
\hat\Pi^{\eps, t}_0 \tau, \varphi\rangle_\eps $
(recall definition \eqref{e:def-Cs})
\[
\int\!\!\sum_{y,z\in\Lambda_\e} 
K^\e (t-s,y-z)^2 
\phi(y) \,ds
 - C_2^{(\e)}\sum_{y\in\Lambda_\e}  \phi(y) =0\;.
\]

For the symbol $ \<PsikPsik>$,
the second chaos component again results in a kernel $\CK^{(\e,2)}$ of order $0^-$.
The zeroth component is now equal to
\begin{equ} [e:cherry-need-Ward]
\int\!\!\sum_{y,z\in\Lambda_\e} 
K^\e(t-s ,y-z)^2\, 
\phi(y) \,ds
 - C_1^{(\e)}  \sum_{y\in\Lambda_\e}  \phi(y) 
\end{equ}
and recall that we have {\it defined} 
$C_1^{(\e)}$ by \eqref{e:def-C1}.
In fact, the above expression would vanish
if $C_1^{(\eps)}$ was replaced by 
$C_2^{(\eps)}$;
thus it remains to show that $C_1^{(\eps)}-C_2^{(\eps)}$ converges to a finite limit.
Recall from \cite[Lemma~5.4]{hairer2015discrete} that   we have $P^\e=K^\e + R^\e$, where
$P^\e$ is the discrete heat kernel, and $R^\e$ is compactly supported with norm
$\|R^\e\|_{\CC^r}$ bounded uniformly in $\eps$.
Therefore, 
 denoting by a thick blue  arrow $\tikz \draw[kernelblue]  (0,0) to (1,0);$ for $P^\e$, we have that
$C_1^{(\eps)}-C_2^{(\eps)}$ equals 
\begin{equs}  [e:C1e-C2e]
 \sum_{k} \Big(
\begin{tikzpicture}  [baseline=10,scale=0.8]
\node[root]	(root) 	at (0,0) {};
\node[dot]		(left)  	at (-0.7,1) {};
\node[dot]		(right)  	at (0.7,1) {};			
\draw[kernelblue,bend right=20] (left) to  (root);
\draw[kernelblue, bend left=20]   (right)  to 
	node [midway,right,font=\scriptsize] {$j$} (root) ;
\draw[kernelblue, bend right=20]  (right)  to 
	node [midway,above,font=\scriptsize] {$k$} (left);
\end{tikzpicture}
\;
-
\;
\begin{tikzpicture}  [baseline=10,scale=0.8]
\node[root]	(root) 	at (0,0) {};
\node[dot]		(left)  	at (-0.7,1) {};
\node[dot]		(right)  	at (0.7,1) {};			
\draw[kernelblue,bend right=20] (left) to 
	node [midway,left,font=\scriptsize] {$j$} (root);
\draw[kernelblue,bend left=20]   (right)  to  (root) ;
\draw[kernelblue,bend right=20]  (right)  to 
	node [midway,above,font=\scriptsize] {$k$} (left);
\end{tikzpicture}
\; +\;
\begin{tikzpicture}  [baseline=10,scale=0.8]
\node[root]	(root) 	at (0,0) {};
\node[dot]		(left)  	at (-0.7,1) {};
\node[dot]		(right)  	at (0.7,1) {};			
\draw[kernelblue,bend right=20] (left) to 
	node [midway,left,font=\scriptsize] {$-k,j$} (root);
\draw[kernelblue,bend left=20]   (right)  to  (root) ;
\draw[kernelblue,bend right=20]  (right)  to  (left);
\end{tikzpicture}
\;-\;
\begin{tikzpicture}  [baseline=10,scale=0.8]
\node[root]	(root) 	at (0,0) {};
\node[dot]		(left)  	at (-0.7,1) {};
\node[dot]		(right)  	at (0.7,1) {};			
\draw[kernelblue,bend right=20] (left) to 
	node [midway,left,font=\scriptsize] {$-k$} (root);
\draw[kernelblue,bend left=20]   (right)  to 
	node [midway,right,font=\scriptsize] {$j$} (root) ;
\draw[kernelblue,bend right=20]  (right)  to  (left);
\end{tikzpicture}
\Big)
-
\;\;
\begin{tikzpicture}  [baseline=10,scale=0.8]
\node[root]	(root) 	at (0,0) {};
\node[dot]		(top)  	at (0,1) {};			
\draw[kernelblue,bend left =60]   (top)  to  (root) ;
\draw[kernelblue,bend right =60]   (top)  to  (root) ;
\end{tikzpicture}
\;-\;
\eps\;
\begin{tikzpicture}  [baseline=10,scale=0.8]
\node[root]	(root) 	at (0,0) {};
\node[dot]		(top)  	at (0,1) {};			
\draw[kernelblue,bend left =60]   (top)  to node [midway,right,font=\scriptsize] {$k$}  (root) ;
\draw[kernelblue,bend right =60]   (top)  to  (root) ;
\end{tikzpicture}
\qquad
\end{equs}
plus some terms  which converge to finite limits.
Note that \eqref{e:C1e-C2e} remains the same
with $\sum_k$ replaced by $\sum_j$: indeed
\eqref{e:C1e-C2e} are equal for $j\in\{1,2\}$,
so we can sum \eqref{e:C1e-C2e} over $j$ and 
then by the same reason drop the sum over $k$.
We can then rename $k$ by $\ell$ so that \eqref{e:C1e-C2e} equals
\begin{equs}  [e:C1e-C2e-rewrite]
 \sum_{j} \Big(
\begin{tikzpicture}  [baseline=10,scale=0.8]
\node[root]	(root) 	at (0,0) {};
\node[dot]		(left)  	at (-0.7,1) {};
\node[dot]		(right)  	at (0.7,1) {};			
\draw[kernelblue,bend right=20] (left) to  (root);
\draw[kernelblue, bend left=20]   (right)  to 
	node [midway,right,font=\scriptsize] {$j$} (root) ;
\draw[kernelblue, bend right=20]  (right)  to 
	node [midway,above,font=\scriptsize] {$\ell$} (left);
\end{tikzpicture}
\;
-
\;
\begin{tikzpicture}  [baseline=10,scale=0.8]
\node[root]	(root) 	at (0,0) {};
\node[dot]		(left)  	at (-0.7,1) {};
\node[dot]		(right)  	at (0.7,1) {};			
\draw[kernelblue,bend right=20] (left) to 
	node [midway,left,font=\scriptsize] {$j$} (root);
\draw[kernelblue,bend left=20]   (right)  to  (root) ;
\draw[kernelblue,bend right=20]  (right)  to 
	node [midway,above,font=\scriptsize] {$\ell$} (left);
\end{tikzpicture}
\; +\;
\begin{tikzpicture}  [baseline=10,scale=0.8]
\node[root]	(root) 	at (0,0) {};
\node[dot]		(left)  	at (-0.7,1) {};
\node[dot]		(right)  	at (0.7,1) {};			
\draw[kernelblue,bend right=20] (left) to 
	node [midway,left,font=\scriptsize] {$-\ell,j$} (root);
\draw[kernelblue,bend left=20]   (right)  to  (root) ;
\draw[kernelblue,bend right=20]  (right)  to  (left);
\end{tikzpicture}
\;-\;
\begin{tikzpicture}  [baseline=10,scale=0.8]
\node[root]	(root) 	at (0,0) {};
\node[dot]		(left)  	at (-0.7,1) {};
\node[dot]		(right)  	at (0.7,1) {};			
\draw[kernelblue,bend right=20] (left) to 
	node [midway,left,font=\scriptsize] {$-\ell$} (root);
\draw[kernelblue,bend left=20]   (right)  to 
	node [midway,right,font=\scriptsize] {$j$} (root) ;
\draw[kernelblue,bend right=20]  (right)  to  (left);
\end{tikzpicture}
\Big)
-
\;\;
\begin{tikzpicture}  [baseline=10,scale=0.8]
\node[root]	(root) 	at (0,0) {};
\node[dot]		(top)  	at (0,1) {};			
\draw[kernelblue,bend left =60]   (top)  to  (root) ;
\draw[kernelblue,bend right =60]   (top)  to  (root) ;
\end{tikzpicture}
\;-\;
\eps\;
\begin{tikzpicture}  [baseline=10,scale=0.8]
\node[root]	(root) 	at (0,0) {};
\node[dot]		(top)  	at (0,1) {};			
\draw[kernelblue,bend left =60]   (top)  to node [midway,right,font=\scriptsize] {$\ell$}  (root) ;
\draw[kernelblue,bend right =60]   (top)  to  (root) ;
\end{tikzpicture}
\qquad
\end{equs}
for each $\ell\in\{1,2\}$.
We prove that \eqref{e:C1e-C2e-rewrite} has a finite limit using the ``Ward identity'' derived in
Section~\ref{sec:Ward}. 
With our graphical notation, the  identity \eqref{e:Ward-use} with $t=x=0$ can be  represented as  
\begin{equ} [e:WardGraph]
-
\quad
\begin{tikzpicture}  [baseline=10]
\node[root]	(root) 	at (0,0) {};
	\node[empt,yshift=2,xshift=12] at (root)  {\scriptsize $\be_\ell$};
\node[dot]		(top)  	at (0,1) {};			
\draw[kernelblue,bend left =60]   (top)  to  (root) ;
\draw[kernelblue,bend right =60]   (top)  to  (root) ;
\end{tikzpicture}
\quad
-
\quad
\sum_{j,\be\in\{0,\be_\ell\}}
\;
\begin{tikzpicture}  [baseline=10]
\node[root]	(root) 	at (0,0) {};
	\node[empt,yshift=2,xshift=-12] at (root)  {\scriptsize $-\be$};
\node[dot]		(left)  	at (-0.7,1) {};
\node[dot]		(right)  	at (0.7,1) {};			
\draw[kernelblue,bend right=20] (left) to 
	node [midway,left,font=\scriptsize] {$j$} (root);
\draw[kernelblue,bend left=20]   (right)  to  (root) ;
\draw[kernelblue,bend right=20]  (right)  to 
	node [midway,above,font=\scriptsize] {$\ell$} (left);
\end{tikzpicture}
\quad
+
\quad
 \sum_{j,\be\in\{0,\be_\ell\}} \;
\begin{tikzpicture}  [baseline=10]
\node[root]	(root) 	at (0,0) {};
	\node[empt,yshift=2,xshift=-12] at (root)  {\scriptsize $-\be$};
\node[dot]		(left)  	at (-0.7,1) {};
\node[dot]		(right)  	at (0.7,1) {};			
\draw[kernelblue,bend right=20] (left) to  (root);
\draw[kernelblue, bend left=20]   (right)  to 
	node [midway,right,font=\scriptsize] {$j$} (root) ;
\draw[kernelblue, bend right=20]  (right)  to 
	node [midway,above,font=\scriptsize] {$\ell$} (left);
\end{tikzpicture}
\quad
=\quad 0
\end{equ}
for each $\ell\in\{1,2\}$,
where we dropped the summation over $k\in\{1,2\}$
in \eqref{e:def-vecfield-V} because  \eqref{e:def-vecfield-V}  only depends on $k$ through the noise $\zeta^\e_k$ and this dependence is gone upon taking expectation.
Here, a lattice vector $\be$ (of length $\e$)
near a tip of an arrow means that the kernel represented by the arrow is evaluated at the point $z+\be$ if $z$ is the end point of the arrow; 
for instance  
$\begin{tikzpicture} [baseline=-3] \draw[kernelblue] (0,0) to node [midway,above,font=\scriptsize] {$k$}  (1,0);
\node[empt] at (1,0.1)  {\scriptsize $\be$};\end{tikzpicture}$
 represents 
the kernel $\nabla^\e_k P^\e(y+\be-x)$,
with $x$ and $y$ being the starting and end points of the arrow respectively.

Now since $\sum_{y} \nabla^\e_{\-\be_\ell}P^\e(x-y) P^\e(y-z) = -\sum_{y}  P^\e(x-\be_\ell-y) \nabla^\e_{\be_\ell} P^\e(y-z)$ we see that  the four triangular graphs in \eqref{e:C1e-C2e-rewrite} are {\it exactly} the same as the four triangular graphs in \eqref{e:WardGraph}. 
The sum of the two last terms of \eqref{e:C1e-C2e-rewrite}
is then obviously equal to
 the first term of \eqref{e:WardGraph}.


The symbols 
	$\<BjPsikPsik> $ and
	$\<BjBjPsik> $ 
are then analyzed in the same way as $\<PsikPsik>$ and $ \<BjBj>$ so we omit the details.
\end{proof}

\begin{remark}\label{rem:fav-id}
In the previous proof we explained boundedness of $C_1^{(\eps)}-C_2^{(\eps)}$
by gauge symmetry,
but of course one can also prove it via the more elementary arguments.
Indeed,
one can start by ignoring  the terms in \eqref{e:C1e-C2e}
for $k\neq j$, because they converge to zero as $\e\to 0$
due to $x\to -x$ symmetry.
One then has (for instance see \cite[Lemma~4.2]{CorwinShenTsai}, or \cite[Section~6]{KPZJeremy} for one spatial dimension case in continuum)
\footnote{The constant in front of the left hand side of \cite[(4.9)]{CorwinShenTsai} is $\frac{1}{d}$ instead of $2$ as here, because
the heat kernel is defined as $(\partial_t - \frac{1}{2d}\Delta)^{-1} $ therein.}
\begin{equ} \label{e:d-dim-id}
2 \sum_{x\in \Z^d} \sum_{j=1}^d \int_{-\infty}^\infty
	 \nabla_j P^\e (t+s,x+y) \nabla_j P^\e (t+s',x+y') \,dt
= P^\e(|s-s'|,y-y') \;, 
\end{equ}
for all $s,s'\in\R$ and $y,y'\in \Z^d$.
Therefore
\begin{equ}
\begin{tikzpicture}  [baseline=10]
\node[root]	(root) 	at (0,0) {};
\node[dot]		(left)  	at (-0.7,1) {};
\node[dot]		(right)  	at (0.7,1) {};			
\draw[kernelblue,bend right=20] (left) to  (root);
\draw[kernelblue, bend left=20]   (right)  to 
	node [midway,right,font=\scriptsize] {$j$} (root) ;
\draw[kernelblue, bend right=20]  (right)  to 
	node [midway,above,font=\scriptsize] {$j$} (left);
\end{tikzpicture}
 \; = \; \frac12 \sum_{k=1}^2 \;
\begin{tikzpicture}  [baseline=10]
\node[root]	(root) 	at (0,0) {};
\node[dot]		(left)  	at (-0.7,1) {};
\node[dot]		(right)  	at (0.7,1) {};			
\draw[kernelblue,bend right=20] (left) to  (root);
\draw[kernelblue, bend left=20]   (right)  to 
	node [midway,right,font=\scriptsize] {$k$} (root) ;
\draw[kernelblue, bend right=20]  (right)  to 
	node [midway,above,font=\scriptsize] {$k$} (left);
\end{tikzpicture} 
 \; = \; \frac14
\begin{tikzpicture}  [baseline=10]
\node[root]	(root) 	at (0,0) {};
\node[dot]		(top)  	at (0,1) {};			
\draw[kernelblue,bend left =60]   (top)  to  (root) ;
\draw[kernelblue,bend right =60]   (top)  to  (root) ;
\end{tikzpicture}
\end{equ}
The first equality here is by the fact that the value of the first graph is the same for $j\in\{1,2\}$, and in the last equality we used the identity \eqref{e:d-dim-id} to compute the integration 
of the upper-right vertex.
The other graphs can be manipulated in a similar way, for instance,
\begin{equ}
\begin{tikzpicture}  [baseline=10]
\node[root]	(root) 	at (0,0) {};
\node[dot]		(left)  	at (-0.7,1) {};
\node[dot]		(right)  	at (0.7,1) {};			
\draw[kernelblue,bend right=20] (left) to 
	node [midway,left,font=\scriptsize] {$j$} (root);
\draw[kernelblue,bend left=20]   (right)  to  (root) ;
\draw[kernelblue,bend right=20]  (right)  to 
	node [midway,above,font=\scriptsize] {$j$} (left);
\end{tikzpicture}
 \;=\; -
\begin{tikzpicture}  [baseline=10]
\node[root]	(root) 	at (0,0) {};
\node[dot]		(left)  	at (-0.7,1) {};
\node[dot]		(right)  	at (0.7,1) {};			
\draw[kernelblue,bend right=20] (left) to  (root);
\draw[kernelblue,bend left=20]   (right)  
	to node [midway,right,font=\scriptsize] {$j$} (root) ;
\draw[kernelblue,bend right=20]  (right)    
	to node [midway,above,font=\scriptsize] {$j$} (left);
\end{tikzpicture}
= -\frac14
\begin{tikzpicture}  [baseline=10]
\node[root]	(root) 	at (0,0) {};
\node[dot]		(top)  	at (0,1) {};			
\draw[kernelblue,bend left =60]   (top)  to  (root) ;
\draw[kernelblue,bend right =60]   (top)  to  (root) ;
\end{tikzpicture}
\end{equ}
where we applied integration by parts twice to shift the derivative in the $j$-th coordinate from the arrow on the left to the arrow on the top and then to the arrow on the right.
This also proves finiteness of \eqref{e:C1e-C2e}.
\end{remark}

\section{Convergence of solutions and observables}


In the following proposition we will assume that 
our models $\hat\Pi^{\eps} $
satisfy the following uniform bound 
(writing $\tau=\hat\Pi^{\e}\tau=\hat\Pi^{\eps,t}_x\tau$ for short here)
\minilab{e:assump-models}
\begin{equs}
| & \e^\kappa  \hat\Pi^{\e}   \, \<Ixij> (z) |   +
|\e^\kappa \hat\Pi^{\e} \, \<Izetak> (z) | \;  \lesssim 1 \;,
\qquad \mbox{and}
	\label{e:assump-models-1}
\\
 \e \, 	  \<Ixij>  \;
	 \<Izetak> \;
	\nabla^\e_j  \<Izetak> \;,
& \qquad
 \e \,  \<Ixik> (\cdot - \be)^2 \;
	\nabla^\e_k  \<Izetaj> \;,
 \qquad
 \e \, \<Ixik> (\cdot - \be) \;
	(\nabla^\e_{\-\be_k}  \<Ixik> ) \;
	\<Izetaj>
 	\label{e:assump-models-2}
\\
\e \, \<Izetak>  \;
	 \nabla^\e_j    \<Izetak> \;,
& \quad
\e \, \<Ixik> (\cdot - \be)  
	\; \nabla^\e_{\-\be_k}    \<Ixik> \;,
\quad
\e \, \<Ixik> (\cdot - \be)  \;  \nabla^\e_k  \<Izetaj>\;,
\quad
\e \, (\nabla^\e_{\-\be_k}  \<Ixik>)  \;   \<Izetaj>\;,
 	\label{e:assump-models-3}
\end{equs}
are all uniformly bounded in the $\| \Cdot \|^{(\e)}_{\hat\CC^{-\kappa}}$ norm
for some sufficiently small $\kappa>0$ and any $j,k\in\{1,2\}$ and $\be\in\{0,\be_k\}$.

\begin{proposition}  \label{prop:sol-abs}
Let $\alpha\in (-\frac43,-1)$, 
$\gamma\in (|\alpha|,2)$, $\eta\in (-\frac12,0)$,
and $\eps \in (0,1]$.
Then, for  final time $T>0$ sufficiently small and 
for any model $\hat\Pi^{\eps} $ in $\MM_\eps$ such that 
\eqref{e:assump-models} hold,
 there exists a unique solution to 
the fixed point problem \eqref{e:abs-fp} in 
$\CD^{\gamma,\eta}_{\eps,T}$. 
The existence time $T$ can be chosen uniformly over $\eps \in (0,1]$,
over bounded sets of initial conditions in  $\CC_\eps^{\gamma,\eta}$, 
and over bounded
sets in $\MM_\eps$ satisfying the above uniform bounds \eqref{e:assump-models}.
\end{proposition}

\begin{proof}
For the local polynomial nonlinear terms in fixed point problem \eqref{e:abs-fp}, 
we can verify the assumption of Theorem~3.9 and Assumption~5.5 in \cite{hairer2015discrete}.
Indeed,
the maps 
$ \CD^{\gamma, \eta}_{\e,T}(V)
\times
 \CD^{\gamma, \eta}_{\e,T}(\bar V)
 \to  \CD^{\gamma+\zeta-1, 2\eta-1}_{\e,T}(\bar{\bar V})$
 given by 
 \begin{equs}
 (\bPsi_k^\e,\bPsi_{\ell}^\e) \mapsto \bPsi_k^\e \DD_j \bPsi_{\ell}^\e
& \qquad
 (\bB_k^\e,\bPsi_\ell^\e) \mapsto \bB_k^\e \DD_k \bPsi_\ell^\e
 \\
  (\bB_k^\e,\bPsi_\ell^\e) \mapsto \DD_{\-k}\bB_k^\e  \bPsi_\ell^\e
  &\qquad
   (\bB_k^\e,\bPsi_\ell^\e) \mapsto \DD_{\-k} ( \bB_k^\e  \bPsi_\ell^\e)
 \end{equs}
 are locally Lipschitz where
 $V$ and $\bar V$ are sectors with homogeneity $\zeta\in(-\frac13,0)$ and $\bar{\bar V}$ is sector with homogeneity $2\zeta-1$.
Moreover, the maps 
$ \CD^{\gamma, \eta}_{\e,T}(V)
\times
 \CD^{\gamma, \eta}_{\e,T}(\bar V)
 \to  \CD^{\gamma+2\zeta, 3\eta}_{\e,T}(\bar{\bar V})$
 given by 
 $(\bB_j^\e,\bPsi_{k}^\e) \mapsto 
 \bB_j^\e (\bPsi_{k}^\e)^2$
 and
 $(\bB_j^\e,\bPsi_{k}^\e) \mapsto 
 (\bB_j^\e)^2 \bPsi_{k}^\e$
are also locally Lipschitz, uniformly in $\eps$,
where  $V$ and $\bar V$ are as above and 
$\bar{\bar V}$ is sector with homogeneity $3\zeta$.
The conditions in    \cite[Theorem~3.9]{hairer2015discrete}
\[
\eta<\bar\eta \wedge \alpha+2,
\quad
\gamma<\bar\gamma+2,
\quad
\bar\eta>-2
\]
then hold with $\alpha,\eta,\gamma$ as above,
$\bar\gamma\eqdef (\gamma+\zeta-1)\wedge (\gamma+2\zeta)$,
$\bar\eta \eqdef (2\eta-1)\wedge (3\eta)$ and
$\beta=2$ therein. 
%
Therefore if there was not the other ``remainder'' terms \eqref{e:remainders} in \eqref{e:abs-fp},
 \cite[Theorem~5.7]{hairer2015discrete} would immediately apply
yielding unique solution to the equation \eqref{e:abs-fp} 
  over some interval $[0,T]$ that is jointly Lipschitz continuous
 in bounded sets of initial conditions and models  uniformly in $\eps$.

To control the other terms, note that if $\bB^\e,\bPsi^\e$ solve the fixed point problem \eqref{e:abs-fp} one necessarily has
\begin{equ}[e:BPsiDD]
\bB^\e_j = \<Ixij> + \bar\bB^\e_j
\qquad
\bPsi^\e_j = \<Izetaj> + \bar\bPsi^\e_j
\end{equ}
where $ \bar\bB^\e_j,\bar\bPsi^\e_j$ take values in the subspace of the regularity structure spanned by $\one$ and elements with strictly positive homogeneity.
Denote by $L^\infty_\e$ the space of  functions on  $[0,T]\times\Lambda_\e$ endowed with the supremum norm.

Consider the term
$\boldsymbol R_{B^\eps_j}^\eps$ as given by \eqref{e:remainders1}.
For any  $\kappa>0$, the maps
$
\bB^\e_j \mapsto \e^{\kappa} \CR^\e \bB^\e_j 
$ and
$
\bPsi^\e_k \mapsto \e^{\kappa} \CR^\e \bPsi^\e_k 
$
are locally Lipschitz continuous from $\CD^{\gamma, \eta}_{\e,T}$
to $L^\infty_\e$, uniformly over $\e\in (0,1]$, 
over bounded balls 
in $ \CD^{\gamma, \eta}_{\e,T}$, and
over models in $\MM_\e$ with bounded   
norm and with \eqref{e:assump-models-1}.
Moreover, regarding the function $\tilde F_1(z) = e^z - 1- z$ appearing in $R_{B^\eps_j}^\eps$
(defined below \eqref{e:defRBePsix})
 one has $ | \tilde F_1(u) - \tilde F_1(v) |
\lesssim |u-v| (|u|+|v|) $ for $u$ and $v$ bounded.
So provided that $\kappa>0$ is small enough ($\kappa<\tfrac14 $ suffices),
the map 
$
(\bB^\e_j,\bPsi^\e_1, \bPsi^\e_2) \mapsto R^\e_{B_j^\e} 
$ is locally Lipschitz continuous from $(\CD^{\gamma, \eta}_{\e,T})^3$
to $L^\infty_\e$, uniformly over the same data as above;
in fact, the norm of $R^\e_{B_j^\e}$
and the Lipschitz constant are bounded by $\e^{1-4\kappa}$.

Similarly,  regarding the term $\boldsymbol R_{\Psi^\eps_j}^\eps$ as given by \eqref{e:remainders2},
the map 
$
(\bB^\e_1,\bB^\e_2,\bPsi^\e_1, \bPsi^\e_2)
 \mapsto R^\e_{\Psi_j^\e} 
$ is  locally Lipschitz continuous from $(\CD^{\gamma, \eta}_{\e,T})^4$
to $L^\infty_\e$, uniformly over the aforementioned data, 
with the norm of $R^\e_{\Psi_j^\e}$
and the Lipschitz constant bounded by $\e^{1-4\kappa}$.
The argument for this follows in the same way as above except that
for the function $\tilde F_2$ appearing in $R_{\Psi^\eps_j}^\eps$
(defined below \eqref{e:defRBePsix})
 one has $ | \tilde F_2(u) - \tilde F_2(v) |
\lesssim |u-v| (|u|^2+|v|^2) $ for $u$ and $v$ bounded.
The maps $R^\e_{B_j^\e} \mapsto \CP^\e (R^\e_{B_j^\e} \one)$
and  $R^\e_{\Psi_j^\e} \mapsto \CP^\e (R^\e_{\Psi_j^\e} \one)$
on the right hand side of \eqref{e:abs-fp}
then map $L^\infty_\e$ into $\CD^{\gamma, \eta}_{\e,T}$
with norms bounded uniformly in $\e$ and behaves like $T^\theta$ for some $\theta>0$.

Next, consider the term $\widetilde{\boldsymbol R}_{B^\eps_j}^\eps$ 
as given by \eqref{e:remainders3}.
Write 
\begin{equ}[e:decomp-RBRPsi]
\CR^\e \bB^\e_j = b^\e_j + w^\e_j\,
\quad
\CR^\e \bPsi^\e_j = \psi^\e_j + v^\e_j \;,
\quad
\mbox{where}
\quad
b^\e_j = \hat\Pi^\e  \<Ixij>\;,
\quad
 \psi^\e_j  = \hat\Pi^\e  \<Izetak>\;.
\end{equ}
 Using \eqref{e:BPsiDD} we have the decomposition
\begin{equs}
		\e\,  (\CR^\eps \bB_j^\eps)
		&(\nabla_{\be_j}^\e \CR^\eps \bPsi_k^\eps)
		( \CR^\eps \bPsi^\eps_k)
=\e\, b^\e_j \, \psi^\e_k \, \nabla^\e_j\psi^\e_k
+\e\, w^\e_j \, \psi^\e_k \, \nabla^\e_j\psi^\e_k
+\e\, b^\e_j \, v^\e_k \, \nabla^\e_j\psi^\e_k
+\e\, b^\e_j \, \psi^\e_k \, \nabla^\e_j v^\e_k
\\
&
+\e\, \Big( b^\e_j \, v^\e_k \, \nabla^\e_j v^\e_k
+ w^\e_j \, v^\e_k \, \nabla^\e_j \psi^\e_k
+ w^\e_j \, \psi^\e_k \, \nabla^\e_j v^\e_k
+ w^\e_j \, v^\e_k \, \nabla^\e_j v^\e_k\Big)
\label{e:decom-tildeRB}
\end{equs}
By the  bound \eqref{e:assump-models-2}
the convolution of the first term on the right hand side with heat kernel can be lifted to the subspace of 
$\CD^{\gamma,\eta}_{\eps,T}$  spanned by $\one$ and $\mathbf X$, with $\CD^{\gamma,\eta}_{\eps,T}$ norm bounded uniformly in $\e>0$.
The  same holds for the second and third terms in the first line using the bound \eqref{e:assump-models-3}.
All the terms on the second line have uniformly bounded  $\|\cdot\|_{\hat\CC^{-\kappa}}^{(\e)}$ norms, by the classical Young theorem,
because for any $u^\e$ with  $\|u^\e\|_{\hat\CC^{\alpha}}^{(\e)} \lesssim 1$ with $\alpha=1-\kappa$ or  $\alpha=-\kappa$ here,  one has the uniform bound   $\|\e\nabla^\e u^\e\|_{\hat\CC^{\alpha}}^{(\e)} \lesssim 1$ with the same $\alpha$.
The remaining term 
$\e\, b^\e_j \, \psi^\e_k \, \nabla^\e_j v^\e_k = (\e^{\kappa}\, b^\e_j )\,( \e^{\kappa}\psi^\e_k )\, (\e^{1-2\kappa}\nabla^\e_j v^\e_k)$ can be controlled again by   \eqref{e:assump-models-1} and Young theorem.

The map 
$(\bB^\e_j, \bPsi^\e_1,\bPsi^\e_2) \mapsto \widetilde{\boldsymbol R}_{B^\eps_j}^\eps$
is thus locally Lipschitz continuous 
from $(\CD^{\gamma, \eta}_{\e,T})^3$ to $\CD^{\gamma, \eta}_{\e,T}$,
uniformly over $\e\in (0,1]$, over bounded balls 
in $ (\CD^{\gamma, \eta}_{\e,T})^3$, and
over models in $\MM_\e$ with bounded  
norm and satisfying the  uniform  bound
\eqref{e:assump-models}.

Finally, for the term $\widetilde{\boldsymbol R}_{\Psi^\eps_j}^\eps$ 
 given by \eqref{e:remainders4},
each of  the terms on the right hand side of  \eqref{e:remainders4}
is of the same form as the left hand side of \eqref{e:decom-tildeRB}, which is cubic with one derivative and one power of $\eps$. 
Decomposing it
 in the same way as \eqref{e:decom-tildeRB}, then
using the  bounds \eqref{e:assump-models} we have that 
its convolution  with heat kernel  can be lifted to the subspace of 
$\CD^{\gamma,\eta}_{\eps,T}$  spanned by $\one$ and $\mathbf X$, with $\CD^{\gamma,\eta}_{\eps,T}$ norm bounded uniformly in $\e>0$.

Summarizing the above bounds we then have $T>0$ and $\eps_0>0$ such that for all $\eps\in (0,\eps_0]$
the fixed point map is contractive in a ball of large enough radius provided the models are uniformly bounded as above and the initial conditions are  uniformly bounded. The solutions can be continued uniquely until the explosion time as in \cite[Sec.~7]{Regularity}.
\end{proof}

We will consider the following fixed point problem in continuum. 

\begin{proposition} 
There exists a unique solution
 $\bB_j,\bPsi_j \in \CD^{\gamma,\eta}_{0,T}$ (for $j\in\{1,2\}$) to the following fixed point problem
with an admissible model $\hat Z$
up to some terminal time $T > 0$

\begin{equs} [e:abs-fp-limit]
\bB_j  
	 & = \CP \hat F_{B_j}(\bB,\bPsi)
	 +S \mathring A_{j} + \<Ixij> 
+
\lambda\Big(
	\CY_{1j}- \CY_{2j}
	\Big)     \one 
\\
\bPsi_j 
&=  \CP \hat F_{\Psi_j}(\bB,\bPsi)
-\sum_k  \CP' \hat F^{(k)}_{\Psi_j}(\bB,\bPsi)
	 +S \mathring \Phi_{j} 
+ \<Izetaj> 
-(-1)^j \lambda \sum_{ k=1,2 \atop \ell\neq j } 
\Big(\bar{\CY}_{k \ell} - \tilde{\CY}_{k \ell} \Big) \one
\end{equs}
where
the modeled distributions $\hat F_{B_j}$, $\hat F_{\Psi_j}$ and $\hat F^{(k)}_{\Psi_j}$ are defined in the same 
way as  $\hat F_{B_j^\e}$ and $\hat F_{\Psi_j^\e}$  in Section~\ref{sec:renorm} except that the spatial variable $x$ takes values in $\T^2$,
and $S$ is the  semi-group so that 
$S\mathring A_{j}, S \mathring \Phi_{j} $
are naturally lifted into $\CD^{\gamma, \eta}_{T}$.
\end{proposition}

\begin{proof}
Existence and uniqueness of the fixed point problem \eqref{e:abs-fp-limit}
follows in the same way as in the first part of proof to 
Proposition~\ref{prop:sol-abs}, 
with all the discrete spaces
$ \CD^{\gamma, \eta}_{\e,T}$ there replaced by their continuous counter-parts
 $\CD^{\gamma, \eta}_{T}$.
\end{proof}

\begin{proof}[of Proposition~\ref{prop:parabolic}]
We prove the limit by a diagonal argument as in 
\cite{MR3628883} (and as followed later by \cite{CLTKPZ,MR3698737,hairer2015discrete,CannizzaroMatetski} in regularity structures or \cite{ChoukGairingPerkowski} in the context of paracontrolled distributions).

We take a function $\psi : \R^3 \to \R$ which is smooth, compactly supported, symmetric under $x\to -x$ and integrates to $1$, and for some $\bar{\eps} \in [\eps,1]$ we define $\psi^{\bar{\eps}}(t, x) \eqdef \bar{\eps}^{-4} \psi\bigl(\bar{\eps}^{-2} t, \bar{\eps}^{-1} x\bigr)$ and the mollified noises $\xi^{\bar \eps, 0} \eqdef \xi \ast \psi^{\bar{\eps}}$ and $\zeta^{\bar \eps, 0} \eqdef \zeta \ast \psi^{\bar{\eps}}$.  

We define an inhomogeneous continuous model $Z^{\bar\e,0}$
for our truncated regularity structure $\hat{\ST}$,
in the same way as \eqref{e:def-Y-Ybar}--\eqref{e:model-compat-der}
except that we now use
 the kernel $K$ and the noises $(\xi^{\bar \eps, 0},\zeta^{\bar \eps, 0})$, and sum over $\Lambda_\e$ replaced by $\T^2$,
  $\ast_\e$ replaced by continuous convolution $\ast$,
 backward finite difference $\nabla^\e_{\-\be_k}$ replaced by $-\partial_k$.
 We then define the renormalized model 
 $\hat Z^{\bar\e,0}$ 
 in the same way as  \eqref{e:def-Meps} -- \eqref{e:def-C1},
 except that the constants  \eqref{e:def-Cs}, \eqref{e:def-C1} (now depending on $\bar\eps$) are defined by the singular kernels $K$ and its derivatives and then contracting 
 the $\psi^{\bar{\eps}}$-regularized noises in the symbols 
 \eqref{e:def-Ls}.
For instance, $C_{4,k}^{-k(\eps)}$ in \eqref{e:def-Cs} is replaced by
 \[
 C_{4,k}^{-k(\bar\eps)}  \eqdef
\int_{\R^9} \partial_k K(-z) (-\partial_k K)(z-x) K(-y) \,\psi^{\bar{\eps},(2)}(x-y)\, dxdydz
 \]
 where $\psi^{\bar{\eps},(2)}(x-y)\eqdef \int \psi^{\bar{\eps}}(x-w)\psi^{\bar{\eps}}(y-w)dw$.
 Also, we set $c_B^{\bar\e} \eqdef 0$ and $c_\Psi^{\bar\e}\eqdef 0$.

We then let $(B^{\bar\e,0},\Psi^{\bar\e,0}) = \hat\CR^{\bar\e,0} (\bB^{\bar\e,0},\bPsi^{\bar\e,0})$ 
where 
$ \hat\CR^{\bar\e,0}$ is the reconstruction
with the model  $\hat Z^{\bar\e,0}$
and
$(\bB^{\bar\e,0},\bPsi^{\bar\e,0})$ solves the fixed point problem
\eqref{e:abs-fp-limit}. As the derivation in Lemma~\ref{lem:renorm-equ}, they solve:
\begin{equs} [e:bar-eps-equ]
\partial_{t}  B^{\bar\e,0}_j  & 
	  = \Delta  B^{\bar\e,0}_j 
 + \lambda \Big( \Psi_1^{\bar\e,0} \partial_j\Psi_2^{\bar\e,0}
		-\Psi_2^{\bar\e,0} \partial_j \Psi_1^{\bar\e,0}  \Big) 
- \lambda^2   \sum_{k=1,2}B_j^{\bar\e,0} ( \Psi_k^{\bar\e,0})^2
	 +\xi^{\bar\eps}_j 
\\
\partial_t \Psi_j^{\bar\e,0} & 
 =  \Delta^{\bar\e,0}   \Psi_j^{\bar\e,0}
	  - 2 (-1)^{j} \lambda  \sum_{k=1,2 \atop \ell\neq j}
B_k^{\bar\e,0} \partial_k \Psi_{\ell}^{\bar\e,0} 
-  \lambda^2   \sum_{k=1,2}
	(B^{\bar\e,0}_k)^2 \,\Psi_j^{\bar\e,0} 
 - C^{(\bar\eps)} \Psi_j^{\bar\e,0}  +\zeta_j^{\bar\eps} \;.
\end{equs}
Here $C^{(\bar\eps)}$ is defined as in \eqref{e:value-of-C}
with the constants replaced by the respective $\bar\e$-dependent ones.

By retracing  the proof of 
Proposition~\ref{prop:moments}, (or alternatively using the blackbox theorem from  \cite{chandra2016analytic}),
we can immediately prove that
there exists a model $\hat Z \in \MM_0$ such that
one has
\begin{equ}[e:models-smooth-limit]
\lim_{\bar\e\to 0} \E \VERT \hat Z^{\bar\e,0}, \hat Z \VERT_{\delta, \gamma; T}
=0\;.
\end{equ}
Indeed, the kernels $K$ and $K\ast \psi^{\bar{\eps}}$
satisfy all the conditions required in the proof of 
Proposition~\ref{prop:moments} to yield the uniform moment bounds for the models $\hat Z^{\bar\e,0}$.
Moreover, it is easy to identify the limiting random variables 
$\langle
\hat\Pi^{t}_0 \tau, \varphi^\lambda \rangle$ for each relavant $\tau$, which we briefly explain now.
For symbols where the renormalization
constant exactly kills the zero-th chaos
as in \eqref{e:Psi-dIdPsi}, $\langle\hat\Pi^{t}_0 \tau, \varphi^\lambda \rangle$  is simply defined as the limit of the second chaos. 
For instance,
for $ \tau=\<Psi-dIdPsi>$
 one has, in view of \eqref{e:Psi-dIdPsi},
\begin{equ} 
\langle
\hat\Pi^{t}_0 \tau, \varphi \rangle
\quad =\quad
\begin{tikzpicture}[scale=0.8,baseline=-13]
\node[var] at (0.6,0.3) (top) {\tiny $\ell$};
\node[dot] at (0,-0.4)  (mid) {};
\node[dot] at (0.6,-1.1) (bot) {};
\node[var] at (1.2, -0.4) (botup) {\tiny $\ell$};
\node[root] at (0.6,-1.7) (root) {};
\draw[kernel] (top) to node[above] {\tiny $k$}  (mid);
\draw[kernel]  (mid) to node[left] {\tiny $j$} (bot); 
\draw[kernel]  (botup) to (bot); 
\draw[testfcn] (bot) to (root); 
\end{tikzpicture}
\end{equ}
where now  all arrows represent the singular kernel $K$, decorations $k,j$ represent its respective derivatives, and noises are the white noises $\xi,\zeta$.
For
symbols $\tau=\<dj-Psii-Idk-Psil> $ (and similarly for $ \<d-lBlI(d-eBe)>$),
by the analysis below \eqref{e:boundingUb},
one has 
\begin{equ}
\langle
\hat\Pi^{ t}_0 \tau, \varphi^\lambda
\rangle
\quad =\quad
\begin{tikzpicture}[scale=0.8,baseline=-13]
\node[var] at (0.6,0.3) (top) {\tiny $\ell$};
\node[dot] at (0,-0.4)  (mid) {};
\node[dot] at (0.6,-1.1) (bot) {};
\node[var] at (1.2, -0.4) (botup) {\tiny $i$};
\node[root] at (0.6,-1.7) (root) {};
\draw[kernel] (top) to node[above] {\tiny $k$}  (mid);
\draw[kernel1]  (mid) to  (bot); 
\draw[kernel]  (botup)  to node[right] {\tiny $j$} (bot); 
\draw[testfcn] (bot) to (root); 
\end{tikzpicture}
\quad
-
\quad
\delta_{i,\ell}\,
\begin{tikzpicture}[scale=0.8,baseline=-13]
\node[dot] at (0,0)  (mid) {};
\node[dot] at (0,-0.8) (bot) {};
\node[dot] at (1.2, -0.3) (botup) {};
\node[root] at (0,-1.7) (root) {};
\draw[kernel,bend right=40] (botup) to   node[above] {\tiny $k$} (mid);
\draw[kernel, bend right=50]  (mid) to  (root); 
\draw[kernel,bend left=40]  (botup) to node[below] {\tiny $j$}  (bot); 
\draw[testfcn] (bot) to (root); 
\end{tikzpicture}
\end{equ}
For symbols with three noises such as $\tau= \<Psi-dIBdPsi>$ by the analysis below \eqref{e:Psi-dIBdPsi} one has
\begin{equ} 
\langle
\hat\Pi^{t}_0 \tau, \varphi^\lambda 
\rangle
\quad =\quad
\begin{tikzpicture}[scale=0.8,baseline=-13]
\node[var] at (0.6,0.3) (topright) {\tiny $\ell$};
\node[vab] at (-0.4,0.3) (topleft) {\tiny $j$};
\node[dot] at (0.1,-0.4)  (mid) {};
\node[dot] at (0.6,-1.1) (bot) {};
\node[var] at (1.2, -0.4) (botup) {\tiny $\ell$};
\node[root] at (0.6,-1.7) (root) {};
\draw[kernel] (topright) to node[above] {\tiny $j$}  (mid);
\draw[kernel] (topleft) to   (mid);
\draw[kernel]  (mid) to node[left] {\tiny $k$} (bot); 
\draw[kernel]  (botup) to (bot); 
\draw[testfcn] (bot) to (root); 
\end{tikzpicture}
\quad 
+ 
\quad
\begin{tikzpicture}[scale=0.8,baseline=-13]
\node[vab] at (0,0.4) (top) {\tiny $j$};
\node[dot] at (0,-0.4)  (mid) {};
\node[dot] at (0,-1.1) (bot) {};
\node[root] at (0,-1.8) (root) {};
\draw[kernel]  (top) to (mid); 
\draw[kernelBig]  (mid) to (bot); 
\draw[testfcn] (bot) to (root); 
\end{tikzpicture}
\end{equ}
where
$\tikz \draw[kernelBig]  (0,0) to (1,0);$  is re-interpreted as the renormalized kernel $\mathscr R K_3$ where
$K_3$ is defined as in \eqref{e:def-K-3} with $K^\e$ replaced by $K$.
One  can then bound the moments
of $\langle \hat\Pi^{\bar\e,t}_0 \tau
- \hat\Pi^{t}_0 \tau, \varphi^\lambda  \rangle$
by $\bar\e^\kappa$ by proceeding as in \cite{Regularity}
and using the bound
\begin{equ}
\VERT K - K \ast \psi^{\bar\e} \VERT_{-2 - \kappa; m} \lesssim \bar{\eps}^\kappa \VERT K \VERT_{-2; m + 2}
\end{equ}
for some sufficiently small $\kappa>0$.

To conclude the proof of \eqref{e:models-smooth-limit},
the only subtlety lies in the zeroth chaos of $\hat\Pi^{ t}_0 \tau$ for $\tau= \<PsikPsik>$, which is given by
\begin{equ}[e:C2be-C1be]
\lim_{\bar\e\to 0} \big( C_2^{(\bar\e)}-C_1^{(\bar\e)} \big)\;.
\end{equ}
To show that it is finite, note that $C_2^{(\bar\e)}-C_1^{(\bar\e)}$ is, up to a finite part due to the truncation of heat kernel, given by \eqref{e:C1e-C2e} without the last term there and with the discrete kernel $P^\e$ replaced by the kernels in continuum.
Since the last term in \eqref{e:C1e-C2e} i.e. $-c^\e_B$ converges to a finite limit, arguing as in Remark~\ref{rem:fav-id} yields finiteness of the zeroth chaos of this object.

As the next step, we 
discretize the noise $\xi^{\bar \eps, 0}$. Define the function
\begin{equ}
\psi^{\bar{\eps}, \eps}(t, x)
 \eqdef \eps^{-2} \int_{\R^2} \psi^{\bar{\eps}}(t, y)\, \one_{|y - x| \leq \eps/2}\, dy\;, \qquad (t,x) \in \R \times \Lambda_\eps^2\;,
\end{equ}
and the discrete noises 
\[
\xi^{\bar \eps, \eps} \eqdef \psi^{\bar{\eps}, \eps} \ast_\eps \xi^\eps \;,
\qquad
\zeta^{\bar \eps, \eps} \eqdef \psi^{\bar{\eps}, \eps} \ast_\eps \zeta^\eps \;.
\]
We define the discrete model $\hat Z^{\bar \eps, \eps}$ 
with the noises $\xi^{\bar \eps, \eps},\zeta^{\bar \eps, \eps}$,
and renormalization constants 
as in  \eqref{e:def-Cs}, \eqref{e:def-C1} using the discrete kernels $K^\e$, but then contracting 
 the $\psi^{\bar{\eps}}$-regularized discrete noises in the symbols 
 \eqref{e:def-Ls}.
For instance, $C_{4,k}^{-k(\bar\e,\eps)}$ in \eqref{e:def-Cs} is replaced by
 \[
 C_{4,k}^{-k(\bar\eps,\e)}  \eqdef
\int_{(\R\times \Lambda_\e)^3} \nabla^\e_{\be_k} K(-z) \nabla^\e_{\-\be_k} K(z-x) K(-y) \,\psi^{\bar{\eps},\e,(2)}(x-y)\, dxdydz
 \]
 where $\psi^{\bar{\eps},\e,(2)}(x-y)\eqdef \int_{\R\times \Lambda_\e} \psi^{\bar{\eps},\e}(x-w)\psi^{\bar{\eps},\e}(y-w)dw$.
Let $(B^{\bar\e,\e},\Psi^{\bar\e,\e}) = \hat\CR^{\bar\e,\e} (\bB^{\bar\e,\e},\bPsi^{\bar\e,\e})$ 
where 
$ \hat\CR^{\bar\e,\e}$ is the reconstruction
with the model  $\hat Z^{\bar\e,\e}$
and
$(\bB^{\bar\e,\e},\bPsi^{\bar\e,\e})$ solves the fixed point problem
\eqref{e:abs-fp}.
 
For any fixed $\bar\e>0$, one has that the ``remainder'' terms
$R_{B^\eps_j}^\eps$, $R_{\Psi^\eps_j}^\eps$
as well as 
$\hat\CR^{\bar\e,\e} \widetilde{\boldsymbol R}_{B^\eps_j}^\eps$, $\hat\CR^{\bar\e,\e} \widetilde{\boldsymbol R}_{\Psi^\eps_j}^\eps$ converge to zero as $\e \to 0$ in $\CC^\infty$ norms.
The convergence of $(B^{\bar\e,\e},\Psi^{\bar\e,\e}) $
to $(B^{\bar\e,0},\Psi^{\bar\e,0}) $ as $\e\to 0$
for fixed $\bar\e$ is then standard in numerical analysis \cite{MR2895081}.

It remains to bound the difference between
$(B^{\bar\e,\e},\Psi^{\bar\e,\e}) $
and $(B^{\e},\Psi^{\e})$.

It is standard to prove that for sufficiently small final time,
\begin{equ}[e:Bee-Be]
\VERT \bB_j^{\bar\e,\e};\bB_j^\e\VERT_{\gamma,\eta}^{(\e)} 
	\lesssim
	\VERT \hat Z^{\bar\e,\e} ; \hat Z^\e  \VERT_{\delta, \gamma; T}^{(\e)}
	+ \VERT \CP^\e  {\boldsymbol R}_{B^{\bar\e,\e}_j}^{\bar\e,\e} ;  \CP^\e {\boldsymbol R}_{B^{\e}_j}^{\e} \VERT_{ \gamma,\eta}^{(\e)}
	+ \VERT \widetilde{\boldsymbol R}_{B^{\bar\e,\e}_j}^{\bar\e,\e} ;  \widetilde{\boldsymbol R}_{B^{\e}_j}^{\e} \VERT_{ \gamma,\eta}^{(\e)}
\end{equ}
and similarly for $\|\bPsi_j^{\bar\e,\e};\bPsi_j^\e\|_{\gamma,\eta}^{(\e)}$.
By \cite[Lemma~7.5]{hairer2015discrete} one has the bound
uniformly in $\e\le \bar\e$
\begin{equ}
\VERT K^\e - K^\e \ast_\eps \psi^{\bar\e,\e} \VERT_{-2 - \kappa; m} \lesssim \bar{\eps}^\kappa \VERT K^\e \VERT_{-2; m + 2}
\end{equ}
From this one has the bound 
$ \E \VERT \hat Z^{\bar\e,\e} ; \hat Z^\e  \VERT_{\delta, \gamma; T}^{(\e)}
\lesssim \bar\e^\kappa$ for some $\kappa>0$ uniformly in $\e<\bar\e$.

Now consider
\[
\widetilde{\boldsymbol R}_{B^{\bar\e,\e}_j}^{\bar\e,\e}  
 \eqdef
-  \lambda^2 
		P^\e \ast_\e
		\Big(\sum_{k=1,2}\eps\,
		(B_j^{\bar\e,\e})
		(\nabla_{\be_j}^\e \Psi_k^{\bar\e,\e})
		( \Psi^{\bar\e,\e}_k)
		- c_B^{\bar\e,\e} B_j^{\bar\e,\e}  \Big) \circ \langle \one, \mathbf X\rangle   \;.
\]
Note that if $c_B^{\bar\e,\e} $ and $c_B^{\e} $ were zero,
the difference between $\widetilde{\boldsymbol R}_{B^{\bar\e,\e}_j}^{\bar\e,\e} $ and
$\widetilde{\boldsymbol R}_{B^{\e}_j}^{\e} $
would not be bounded by $O(\bar\e^\kappa)$ {\it uniformly
in} $\e$, because for fixed $\bar\e$ as $\e\to 0$,
$\widetilde{\boldsymbol R}_{B^{\bar\e,\e}_j}^{\bar\e,\e}  $
would vanish while $\widetilde{\boldsymbol R}_{B^{\e}_j}^{\e} $ would converge to a non-zero linear term.
 Now with $c_B^{\bar\e,\e} $ and $c_B^{\e} $ defined above,
we can actually show a stronger statement that for any fixed $\bar\e>0$, one has
$
\E\VERT \widetilde{\boldsymbol R}_{B^{\bar\e,\e}_j}^{\bar\e,\e}  \VERT_{\gamma,\eta}^{(\e)}
+
\E\VERT  \widetilde{\boldsymbol R}_{B^{\e}_j}^{\e}
\VERT_{\gamma,\eta}^{(\e)} 
\lesssim 
\bar\e^\kappa
$
uniformly in $\e \in (0,\e_0(\bar\e))$.
To this end, 
decompose $B^\e,\Psi^\e,B^{\bar\e,\e},\Psi^{\bar\e,\e}$ as  in \eqref{e:decomp-RBRPsi}.
It is easy to show that 
\begin{equ} [e:tilde-RB1]
 \e b_j^{\e}  \psi_k^{\e} \nabla^\e_j \psi_k^{\e}
-
 c_{B}^{\e} 
b_j^{\e}
\qquad
\mbox{and}
\qquad
 \e b_j^{\bar\e,\e}  \psi_k^{\bar\e,\e} \nabla^\e_j \psi_k^{\bar\e,\e}
-
 c_{B}^{\bar\e,\e} 
b_j^{\bar\e,\e}
\end{equ}
converge in $L^p$ to zero in $\hat\CC^\alpha$ for $\alpha<0$. Indeed, using similar arguments as in the proof of Proposition~\ref{prop:moments} and the `extra' $\eps$, the third chaos of the left hand side of \eqref{e:tilde-RB1}
vanishes in the limit;
and  the first chaos vanishes by definitions of $ c_{B}^{\e} $,
$ c_{B}^{\bar\e,\e} $.
One can also show that 
\begin{equ} 
 \e w_j^{\e} 
	\psi^{\e}_k 
	\nabla^\e_j \psi^{\e}_k 
		 -  c_{B}^{\e} w_j^{\e} 
\qquad
\mbox{and}
\qquad
 \e w_j^{\bar\e,\e} 
	\psi^{\bar\e,\e}_k 
	\nabla^\e_j \psi^{\bar\e,\e}_k 
		 -  c_{B}^{\bar\e,\e} w_j^{\bar\e,\e} 
\end{equ}
converge in $L^p$ to zero in $\hat\CC^\alpha$ for $\alpha<0$.
This follows from that 
$ \e \psi^{\e}_k 
	\nabla^\e_j \psi^{\e}_k 
		 -  c_{B}^{\e}  $ converge to zero
in $\hat\CC^\alpha$ with any $\alpha<0$
and  that $w_j^{\e} $ converges in  $\CC^\alpha$ with any $\alpha<1$,
together with  the classical Young's theorem;
same for the expression depending on $(\bar\e,\e)$.
One also has convergences of the other terms such as
$ \e \, b_j^\e  
	v_k^\e 
	 \nabla^\e_j  \psi_k^\e \to 0 $ 
and 
$ \e \, b_j^\e 
 	\psi_k^\e
	\nabla^\e_j v^\e_k \to 0 $ in $\hat\CC^\alpha$ for $\alpha<0$.
Given all these,
we conclude the claimed uniform bounds on
$
\VERT \widetilde{\boldsymbol R}_{B^{\bar\e,\e}_j}^{\bar\e,\e}  \VERT_{\gamma,\eta}^{(\e)}$ 
and 
$\VERT  \widetilde{\boldsymbol R}_{B^{\e}_j}^{\e}\VERT$.
The second term on the right hand side of \eqref{e:Bee-Be} can be shown
to be bounded by $\bar\e^\kappa$ uniformly in $\e \in (0,\e_0(\bar\e))$
by similar (even simpler) arguments.

The difference $\|\bPsi_j^{\bar\e,\e};\bPsi_j^\e\|_{\gamma,\eta}^{(\e)}$ can be bounded in the same way.
	
Summarizing the above estimates,
the convergence of the discrete solutions in probability \eqref{e:Convergence-Prob}
as well as the convergence of the stopping times 
then follow in the same way as the diagonal arguments in \cite[below~(7.16)]{hairer2015discrete}.
\end{proof}

\begin{remark}
Equations \eqref{e:value-of-C} where the noises are continuously mollified  do not enjoy the gauge invariance property, in particular Lemma~\ref{lem:Levy}.
Here we know {\it posteriorly} that \eqref{e:value-of-C} without any mass renormalization in the $B^{\bar\e,0}$ equation actually achieve the same limit; this seems to be due to the simplicity of this two-dimensional Abelian model.
In fact, if $c_B^\e$ were actually divergent (it is imaginable that this may happen  in the more sophisticated gauge theory models discussed in Section~\ref{sec:discussions}),
then the argument below \eqref{e:C2be-C1be} would not be valid anymore and thus one would have to define $C_1^{(\bar\e)}$ differently, which would lead to a nontrivial
mass renormalization for $B^{\bar\e,0}$ in the equations
 \eqref{e:value-of-C}.
\end{remark}

\subsection{Convergence of observables}
\label{sec:conv-obs}

We start to prove  Theorem~\ref{theo:observables}.
Since all the observables considered in the theorem
are gauge invariant, it suffices to show their convergence
with every incidence of $(A^\e,\Phi^\e)$ replaced
by $(B^\e,\Psi^\e)$.

\begin{lemma}\label{lem:FA-Wick}
Let $(A^\e,\Phi^\e)$ be the solutions to \eqref{e:DLangevin}.
The curvature $F^\e_{A^\e}$ and the Wick powers $\Wick{|\Phi^\eps|^{2n}}$ 
converge in distribution  with respect to the distance 
$\Vert \cdot ; \cdot \Vert_{\CC^{\delta, \alpha-1}_{\bar{\eta}, T_\eps}}^{(\eps)} $
and $\Vert \cdot ; \cdot \Vert_{\CC^{\delta, \alpha}_{\bar{\eta}, T_\eps}}^{(\eps)} $ respectively.
\end{lemma}

\begin{proof}
By definition \eqref{e:discFA}, $F^\eps_{A^\e}=F^\eps_{B^\e}$
is simply a finite difference discretization of the curl of $B^\e$.
Therefore, invoking the convergence of $B^\e$ to its distributional limit $B$ with respect to the distance 
$\Vert \cdot ; \cdot \Vert_{\CC^{\delta, \alpha}_{\bar{\eta}, T_\eps}}^{(\eps)} $,
 one immediately obtains 
 the convergence of $F^\eps_{B^\e}$ to the distributional limit
which is curl$B$,
 with respect to the distance 
$\Vert \cdot ; \cdot \Vert_{\CC^{\delta, \alpha-1}_{\bar{\eta}, T_\eps}}^{(\eps)} $.

Consider the observable 
$\Wick{|\Phi^\eps|^{2n}}=\Wick{|\Psi^\eps|^{2n}} = \Wick{(\Psi^\eps\bar\Psi^\eps)^n }$
for $n\ge 1$ where $\bar\Psi^\eps$ is the complex conjugate of $\Psi^\eps$.
One has  $\Psi^\eps =  \psi_\e + v_\e$
where $\psi_\e \eqdef K^\e \ast_\e \zeta^\e$ and
 $v_\e \to v\in \CC([0,T],\CC^\beta)$ for any $\beta<1$.
 From the definition of Wick powers \eqref{e:generating-Wick}
with respect to the Gaussian measure of $\psi_\e$,
one can also deduce 
\begin{equ}[e:Wick-product-expanded]
\Wick{(\Psi^\eps\bar\Psi^\eps)^n }
=
\sum_{p,q=0}^n {n \choose p}{n \choose q}
\Wick{  
\,\psi_\e^p \bar\psi_\e^q}
v_\e^{n-p}\bar v_\e^{n-q} \;.
\end{equ}
Note that  $\Wick{ \,\psi_\e^p \bar\psi_\e^q}$
is simply a Wick product of Gaussian process $(\psi_\e,\bar\psi_\e)$ which
converge in probability in 
$\CC([0,T],\CC^\beta)$ for any $\beta<0$.
Applying  the classical Young's theorem to the products in \eqref{e:Wick-product-expanded}
 we obtain the convergence 
in probability of $\Wick{(\Psi^\eps\bar\Psi^\eps)^n }$
and thus the convergence 
in distribution  of $\Wick{(\Phi^\eps\bar\Phi^\eps)^n }$.
Note that the logarithmic renormalization constants in the Wick products, and thus the limit 
of $\Wick{(\Phi^\eps\bar\Phi^\eps)^n }$, do not depend on the choice of the truncated heat kernel $K^\e$;
see for instance \cite[Lemma~3.1]{MR3452276}.
\end{proof}

Let
\begin{equ}[e:defC-composite]
C^\e_{\bar\Phi D^A \Phi}
\eqdef 2 \int_{\R\times \Lambda_\e} P^\e(-z) \nabla_j^\e P^\e(-z) \,dz  \;.
\end{equ}
Note that the ``parity'' symmetry $x\to -x$ 
would render a zero limit if the above expression were in continuum, but it 
does not  apply here on the lattice. Instead:
\begin{lemma}\label{lem:int-P-dP}
One has $C^\e_{\bar\Phi D^A \Phi} = -1/(4\e)$.
\end{lemma}
\begin{proof}
As the argument  in Lemma~\ref{lem:some-are-finite} one has 
\begin{equ}
C^\e_{\bar\Phi D^A \Phi}
=2 \int_{\R\times \Lambda_\e} P^\e(-z)\nabla_{-\be_j}^\e  P^\e(-z) \,dz
\end{equ}
and adding this to \eqref{e:defC-composite} one has
\begin{equs}
C^\e_{\bar\Phi D^A \Phi}&=
\eps \int_{\R\times \Lambda_\e} \!\!\!\! P^\e(-z) 
\nabla_{-\be_j}^\e \nabla_{\be_j}^\e P(-z) \,dz 
=
\frac{\eps}{d} \int_{\R\times \Lambda_\e}  \!\!\!\! P^\e(t,x) 
\Delta^\e P(t,x) \,dtdx
\\
&=
\frac{\eps}{2d} \int_{\R\times \Lambda_\e}  \!\!\!\! \partial_t(P^\e(t,x)^2) \,dtdx
=
-\frac{\eps}{2d} \int_{ \Lambda_\e}  P^\e(0,x)^2 \,dx
= 
-\frac{1}{2d\e^{d-1}}
=
-\frac{1}{4\e}\;.
\end{equs}
\end{proof}

\begin{lemma}\label{lem:proof-PhiDAPhi}
The composite field observables 
$\bar\Phi^\eps (e_-) D_j^{A^\eps} \Phi^\eps (e)-C^\e_{\bar\Phi D^A \Phi}$ as
introduced in \eqref{e:def-obs-PhiDAPhi} with $e\in \CE^j_\e$ converge in distribution  with respect to the distance 
$\Vert \cdot ; \cdot \Vert_{\CC^{\delta, \alpha-1}_{\bar{\eta}, T_\eps}}^{(\eps)} $.
\end{lemma}

\begin{proof}
By its gauge invariance 
it again suffices to consider
$\bar\Psi^\eps (e_-) D_j^{B^\eps} \Psi^\eps (e)$.
Recall that $\Psi^\e=\Psi^\e_1+ i\Psi^\e_2$.
By definition  for $e\in \CE^j_\e$ one has
\begin{equs} 
\bar\Psi^\eps & (e_-)   (D_j^{B^\eps} \Psi^\eps)(e)-
C^\e_{\bar\Phi D^A \Phi}
\\
&= \eps^{-1} \bar\Psi^\eps (e_-) 
	\Big(e^{-i\eps \lambda B_j^\eps(e)} \Psi^\eps(e_+) -\Psi^\eps(e_-)\Big) -C^\e_{\bar\Phi D^A \Phi}
\\
& = \sum\nolimits_{k} \Big( \Psi^\e_k (e_-)\nabla^\e_j \Psi^\e_k(e_-)\Big)
+ \lambda \Psi^\e_1 (e_-) B_j^\e (e)\Psi^\e_2 (e_+)
- \lambda \Psi^\e_2 (e_-) B_j^\e (e)\Psi^\e_1 (e_+)
\\
 & \quad
+ i \Big( \Psi^\e_1 (e_-)\nabla^\e_j \Psi^\e_2(e_-) -\Psi^\e_2 (e_-)\nabla^\e_j \Psi^\e_1(e_-)
- \lambda   B_j^\e (e) \sum\nolimits_{k}  \Psi^\e_k (e_-) \Psi^\e_k (e_+)
\Big) 
\\
&\quad -C^\e_{\bar\Phi D^A \Phi} + R_j^\e(e)
		\label{e:expansion-PsiDBPsi}
\end{equs}
%
where we have separated the real and imaginary parts of the process, and
$R_j^\e$ is a remainder of ``order $O(\e B^2\Psi^2)$''
\[
R_j^\e(e) \eqdef \eps^{-1} \bar\Psi^\eps (e_-) 
	\Big(e^{-i\eps \lambda B_j^\eps(e)} - 1+i\e\lambda B_j^\e(e)\Big)
	 \Psi^\eps(e_+) \;.
\]
Since 
\begin{equs}
\Big|e^{-i \eps \lambda B_j^\eps(e)} - 1+i\e\lambda B_j^\e(e)\Big|
&\le
 \Big|\cos(\eps \lambda B_j^\eps(e))-1\Big|
+ \Big|\eps \lambda B_j^\eps(e)-\sin(\eps \lambda B_j^\eps(e))\Big|
\\
& \lesssim \e^2 B_j^\eps(e)^2
\end{equs}
and $B^\e,\Psi^\e$  converge in $\CC([0,T],\CC^\beta)$ for any $\beta<0$,
we have $\sup_e |R_j^\e(e)| \to 0$.

We will now prove convergence of the real and imaginary parts of \eqref{e:expansion-PsiDBPsi} separately,
using the knowledge that 
the solution $(B^\e,\Psi^\e)$ is given by the reconstruction of the abstract solution which has the expansion \eqref{e:expand-sol}, namely, with $b_j^\e=\hat\Pi^{\e} \, \<Ixij>$ and $\psi_j^\e=\hat\Pi^{\e} \, \<Izetaj>  $ as before,
and $\hat\Pi^{\e}=\hat\Pi^{\eps,t}_x$ for short
\minilab{e:2nd-order-exp}
\begin{equs}
B_j^\eps (e) &= B_j^\eps (e_-)= b_j^\e(e_-) + w_j^\e(e_-)	\label{e:2nd-order-exp1}
\\
\Psi_j^\eps (x)&=  \psi_j^\e(x) -(-1)^{j} \lambda
	   \sum_{k=1,2 \atop \ell\neq j} 
	 \Big[\hat\Pi \<I-BdPsi> + \hat\Pi \<I(d-kBP)> - \hat\Pi\<d-kI(BP)> \Big](x)+  \tilde{v}_j^\e(x)\;.
	 	\label{e:2nd-order-exp2}
\end{equs}

Consider the real part. Let $x=e_-$ and consider $\Psi^\e_k (x)\nabla^\e_j \Psi^\e_k(x)$. 
By standard moment analysis as in Section~\ref{sec:mom},
we get convergence of the following processes with respect to the distance 
$\Vert \cdot ; \cdot \Vert_{\CC^{\delta, \alpha-1}_{\bar{\eta}, T_\eps}}^{(\eps)} $
\[
\psi^\e_k (x) \nabla^\e_j \psi^\e_k(x)
-\frac12 C^\e_{\bar\Phi D^A \Phi} \;,
\]
noting that $C^\e_{\bar\Phi D^A \Phi}$ in \eqref{e:defC-composite} is exactly such that the above renormalized process is mean zero.
One also needs to control
the ``cross terms'' between $\psi^\e_k $ and $O(\lambda)$-terms of \eqref{e:2nd-order-exp2},
namely:
\begin{equs}
\begin{tikzpicture}[scale=0.8,baseline=-13]
\node[var] at (0.6,0.3) (topright) {\tiny $\ell$};
\node[vab] at (-0.4,0.3) (topleft) {\tiny $m$};
\node[dot] at (0.1,-0.4)  (mid) {};
\node[dot] at (0.6,-1.1) (bot) {};
\node[var] at (1.2, -0.4) (botup) {\tiny $k$};
%
\draw[kernel] (topright) to node[right] {\tiny $m$}  (mid);
\draw[kernel] (topleft) to   (mid);
\draw[kernel]  (mid) to  (bot); 
\draw[kernel]  (botup) to node[right] {\tiny $j$} (bot); 
\end{tikzpicture}
\qquad
\begin{tikzpicture}[scale=0.8,baseline=-13]
\node[var] at (0.6,0.3) (topright) {\tiny $\ell$};
\node[vab] at (-0.4,0.3) (topleft) {\tiny $m$};
\node[dot] at (0.1,-0.4)  (mid) {};
\node[dot] at (0.6,-1.1) (bot) {};
\node[var] at (1.2, -0.4) (botup) {\tiny $k$};
%
\draw[kernel] (topright) to  (mid);
\draw[kernel] (topleft) to node[left] {\tiny $-m$}   (mid);
\draw[kernel]  (mid) to  (bot); 
\draw[kernel]  (botup) to node[right] {\tiny $j$} (bot); 
\end{tikzpicture}
\qquad
\begin{tikzpicture}[scale=0.8,baseline=-13]
\node[var] at (0.6,0.3) (topright) {\tiny $\ell$};
\node[vab] at (-0.4,0.3) (topleft) {\tiny $m$};
\node[dot] at (0.1,-0.4)  (mid) {};
\node[dot] at (0.6,-1.1) (bot) {};
\node[var] at (1.2, -0.4) (botup) {\tiny $k$};
%
\draw[kernel] (topright) to  (mid);
\draw[kernel] (topleft) to   (mid);
\draw[kernel]  (mid) to node[left] {\tiny $-m$}  (bot); 
\draw[kernel]  (botup) to node[right] {\tiny $j$} (bot); 
\end{tikzpicture}
\qquad
\begin{tikzpicture}[scale=0.8,baseline=-13]
\node[var] at (0.6,0.3) (topright) {\tiny $\ell$};
\node[vab] at (-0.4,0.3) (topleft) {\tiny $m$};
\node[dot] at (0.1,-0.4)  (mid) {};
\node[dot] at (0.6,-1.1) (bot) {};
\node[var] at (1.2, -0.4) (botup) {\tiny $k$};
%
\draw[kernel] (topright) to node[right] {\tiny $m$}  (mid);
\draw[kernel] (topleft) to   (mid);
\draw[kernel]  (mid) to node[left] {\tiny $j$} (bot); 
\draw[kernel]  (botup) to  (bot); 
\end{tikzpicture}
\qquad
\begin{tikzpicture}[scale=0.8,baseline=-13]
\node[var] at (0.6,0.3) (topright) {\tiny $\ell$};
\node[vab] at (-0.4,0.3) (topleft) {\tiny $m$};
\node[dot] at (0.1,-0.4)  (mid) {};
\node[dot] at (0.6,-1.1) (bot) {};
\node[var] at (1.2, -0.4) (botup) {\tiny $k$};
%
\draw[kernel] (topright) to  (mid);
\draw[kernel] (topleft) to node[left] {\tiny $-m$}   (mid);
\draw[kernel]  (mid) to node[left] {\tiny $j$} (bot); 
\draw[kernel]  (botup) to  (bot); 
\end{tikzpicture}
\qquad
\begin{tikzpicture}[scale=0.8,baseline=-13]
\node[var] at (0.6,0.3) (topright) {\tiny $\ell$};
\node[vab] at (-0.4,0.3) (topleft) {\tiny $m$};
\node[dot] at (0.1,-0.4)  (mid) {};
\node[dot] at (0.6,-1.1) (bot) {};
\node[var] at (1.2, -0.4) (botup) {\tiny $k$};
%
\draw[kernel] (topright) to  (mid);
\draw[kernel] (topleft) to   (mid);
\draw[kernel]  (mid) to node[left] {\tiny $-m,j$}  (bot); 
\draw[kernel]  (botup) to  (bot); 
\end{tikzpicture}
\end{equs}
where, it is important to notice that in  all these terms, the indices $\ell\neq k$. Thus by standard moment analysis and independence
of $\psi_k^\e$ and $\psi_\ell^\e$ one obtains convergences of each of these processes.
All the other terms in $\Psi^\e_k (x)\nabla^\e_j \Psi^\e_k(x)$ then fall into the scope of Young's theorem
and thus converge in the desired topologies.
The convergence of  the term $\Psi^\e_k (e_-) B_j^\e (e)\Psi^\e_{3-k} (e_+)$ 
is even simpler since one only needs the first order expansion in \eqref{e:2nd-order-exp},
the independence of $\psi^\e_k$,  $b_j^\e $ and $\psi^\e_{3-k}$, and Young's theorem.

Regarding the imaginary part on the right side of \eqref{e:expansion-PsiDBPsi}, 
it turns out that we must treat all the terms in the parenthesis in the last line of \eqref{e:expansion-PsiDBPsi} {\it together}, since each individual term {\it does not} converge!

Indeed, it is easy to see by independence
of $\psi_1^\e$ and $\psi_2^\e$ that 
to the ``leading order'',
one has convergences of   $\psi_1^\e \nabla^\e_j\psi_2^\e$ and $\psi_2^\e \nabla^\e_j\psi_1^\e$ (individually).
However, the ``next order''  requires more careful analysis.
Considering the ``next order'' objects from  
$ \Psi^\e_1 (e_-)\nabla^\e_j \Psi^\e_2(e_-) $ and $\Psi^\e_2 (e_-)\nabla^\e_j \Psi^\e_1(e_-)$,
we note that all terms of the form 
\[
\nabla^\e_j \psi^\e \hat\Pi \<I(d-kBP)>
\qquad
\mbox{or}
\qquad
 \psi^\e  \nabla^\e_j \hat\Pi \<I(d-kBP)>
\]
will (individually) converge,
because in the symbol $\<I(d-kBP)>$
the derivative lives on the edge $\<Ixik>$ so the first chaos has no divergence.
 It remains to consider the following objects 
\begin{equs}[e:imag-terms]
\sum_{k=1}^2 \quad &
\begin{tikzpicture}[scale=0.8,baseline=-13]
\node[var] at (0.6,0.3) (topright) {\tiny $2$};
\node[vab] at (-0.4,0.3) (topleft) {\tiny $k$};
\node[dot] at (0.1,-0.4)  (mid) {};
\node[dot] at (0.6,-1.1) (bot) {};
\node[var] at (1.2, -0.4) (botup) {\tiny $2$};
%
\draw[kernel] (topright) to node[above] {\tiny $k$}  (mid);
\draw[kernel] (topleft) to   (mid);
\draw[kernel]  (mid) to  (bot); 
\draw[kernel]  (botup) to node[right] {\tiny $j$} (bot); 
\end{tikzpicture}
-
\begin{tikzpicture}[scale=0.8,baseline=-13]
\node[var] at (0.6,0.3) (topright) {\tiny $2$};
\node[vab] at (-0.4,0.3) (topleft) {\tiny $k$};
\node[dot] at (0.1,-0.4)  (mid) {};
\node[dot] at (0.6,-1.1) (bot) {};
\node[var] at (1.2, -0.4) (botup) {\tiny $2$};
%
\draw[kernel] (topright) to  (mid);
\draw[kernel] (topleft) to   (mid);
\draw[kernel]  (mid) to node[left] {\tiny $-k$} (bot); 
\draw[kernel]  (botup) to node[right] {\tiny $j$}  (bot); 
\end{tikzpicture}
-
\begin{tikzpicture}[scale=0.8,baseline=-13]
\node[var] at (0.6,0.3) (topright) {\tiny $1$};
\node[vab] at (-0.4,0.3) (topleft) {\tiny $k$};
\node[dot] at (0.1,-0.4)  (mid) {};
\node[dot] at (0.6,-1.1) (bot) {};
\node[var] at (1.2, -0.4) (botup) {\tiny $1$};
%
\draw[kernel] (topright) to node[right] {\tiny $k$}  (mid);
\draw[kernel] (topleft) to   (mid);
\draw[kernel]  (mid) to node[left] {\tiny $j$} (bot); 
\draw[kernel]  (botup) to  (bot); 
\end{tikzpicture}
+
\begin{tikzpicture}[scale=0.8,baseline=-13]
\node[var] at (0.6,0.3) (topright) {\tiny $1$};
\node[vab] at (-0.4,0.3) (topleft) {\tiny $k$};
\node[dot] at (0.1,-0.4)  (mid) {};
\node[dot] at (0.6,-1.1) (bot) {};
\node[var] at (1.2, -0.4) (botup) {\tiny $1$};
%
\draw[kernel] (topright) to  (mid);
\draw[kernel] (topleft) to   (mid);
\draw[kernel]  (mid) to node[left] {\tiny $-k,j$} (bot); 
\draw[kernel]  (botup) to (bot); 
\end{tikzpicture}
\\
&
+
\begin{tikzpicture}[scale=0.8,baseline=-13]
\node[var] at (0.6,0.3) (topright) {\tiny $1$};
\node[vab] at (-0.4,0.3) (topleft) {\tiny $k$};
\node[dot] at (0.1,-0.4)  (mid) {};
\node[dot] at (0.6,-1.1) (bot) {};
\node[var] at (1.2, -0.4) (botup) {\tiny $1$};
%
\draw[kernel] (topright) to node[above] {\tiny $k$}  (mid);
\draw[kernel] (topleft) to   (mid);
\draw[kernel]  (mid) to  (bot); 
\draw[kernel]  (botup) to node[right] {\tiny $j$} (bot); 
\end{tikzpicture}
-
\begin{tikzpicture}[scale=0.8,baseline=-13]
\node[var] at (0.6,0.3) (topright) {\tiny $1$};
\node[vab] at (-0.4,0.3) (topleft) {\tiny $k$};
\node[dot] at (0.1,-0.4)  (mid) {};
\node[dot] at (0.6,-1.1) (bot) {};
\node[var] at (1.2, -0.4) (botup) {\tiny $1$};
%
\draw[kernel] (topright) to  (mid);
\draw[kernel] (topleft) to   (mid);
\draw[kernel]  (mid) to node[left] {\tiny $-k$} (bot); 
\draw[kernel]  (botup) to node[right] {\tiny $j$}  (bot); 
\end{tikzpicture}
-
\begin{tikzpicture}[scale=0.8,baseline=-13]
\node[var] at (0.6,0.3) (topright) {\tiny $2$};
\node[vab] at (-0.4,0.3) (topleft) {\tiny $k$};
\node[dot] at (0.1,-0.4)  (mid) {};
\node[dot] at (0.6,-1.1) (bot) {};
\node[var] at (1.2, -0.4) (botup) {\tiny $2$};
%
\draw[kernel] (topright) to node[right] {\tiny $k$}  (mid);
\draw[kernel] (topleft) to   (mid);
\draw[kernel]  (mid) to node[left] {\tiny $j$} (bot); 
\draw[kernel]  (botup) to  (bot); 
\end{tikzpicture}
+
\begin{tikzpicture}[scale=0.8,baseline=-13]
\node[var] at (0.6,0.3) (topright) {\tiny $2$};
\node[vab] at (-0.4,0.3) (topleft) {\tiny $k$};
\node[dot] at (0.1,-0.4)  (mid) {};
\node[dot] at (0.6,-1.1) (bot) {};
\node[var] at (1.2, -0.4) (botup) {\tiny $2$};
%
\draw[kernel] (topright) to  (mid);
\draw[kernel] (topleft) to   (mid);
\draw[kernel]  (mid) to node[left] {\tiny $-k,j$} (bot); 
\draw[kernel]  (botup) to (bot); 
\end{tikzpicture}
\end{equs}
where the  graphs on the first line arise from $ \Psi^\e_1\nabla^\e_j \Psi^\e_2$,
and the   graphs on the second line 
arise from $ -\Psi^\e_2\nabla^\e_j \Psi^\e_1$.
Standard moment analysis as in Section~\ref{sec:renormalization}
shows that the third chaos of each graph converges. 
Apparently, the first chaos of each graph diverges
when $k=j$ and thus
one needs a logarithmic renormalization constant for each graph in order to obtain a finite limit. 
However, the renormalizations needed for the terms in the first line of \eqref{e:imag-terms} are 
\begin{equs}  
 \sum_{k} 
\begin{tikzpicture}  [baseline=10]
\node[root]	(root) 	at (0,0) {};
\node[dot]		(left)  	at (-0.7,1) {};
\node[dot]		(right)  	at (0.7,1) {};			
\draw[kernelblue,bend right=20] (left) to  (root);
\draw[kernelblue, bend left=20]   (right)  to 
	node [midway,right,font=\scriptsize] {$j$} (root) ;
\draw[kernelblue, bend right=20]  (right)  to 
	node [midway,above,font=\scriptsize] {$k$} (left);
\end{tikzpicture}
\;-\;
\begin{tikzpicture}  [baseline=10]
\node[root]	(root) 	at (0,0) {};
\node[dot]		(left)  	at (-0.7,1) {};
\node[dot]		(right)  	at (0.7,1) {};			
\draw[kernelblue,bend right=20] (left) to 
	node [midway,left,font=\scriptsize] {$-k$} (root);
\draw[kernelblue,bend left=20]   (right)  to 
	node [midway,right,font=\scriptsize] {$j$} (root) ;
\draw[kernelblue,bend right=20]  (right)  to  (left);
\end{tikzpicture}
\;-\;
\begin{tikzpicture}  [baseline=10]
\node[root]	(root) 	at (0,0) {};
\node[dot]		(left)  	at (-0.7,1) {};
\node[dot]		(right)  	at (0.7,1) {};			
\draw[kernelblue,bend right=20] (left) to 
	node [midway,left,font=\scriptsize] {$j$} (root);
\draw[kernelblue,bend left=20]   (right)  to  (root) ;
\draw[kernelblue,bend right=20]  (right)  to 
	node [midway,above,font=\scriptsize] {$k$} (left);
\end{tikzpicture}
\; +\;
\begin{tikzpicture}  [baseline=10]
\node[root]	(root) 	at (0,0) {};
\node[dot]		(left)  	at (-0.7,1) {};
\node[dot]		(right)  	at (0.7,1) {};			
\draw[kernelblue,bend right=20] (left) to 
	node [midway,left,font=\scriptsize] {$-k,j$} (root);
\draw[kernelblue,bend left=20]   (right)  to  (root) ;
\draw[kernelblue,bend right=20]  (right)  to  (left);
\end{tikzpicture}
\end{equs}
and
the renormalizations needed for the terms in the second line of \eqref{e:imag-terms}  are the same; and, these are {\it exactly} the first four terms
 in \eqref{e:C1e-C2e}.
These together with  a  constant
\[
\sum_{k} \E\big( \psi^\e_k (e_-) \psi^\e_k (e_+)\big)
\]
which renormalizes 
 the term
$ - B_j^\e (e) \sum_{k}  \Psi^\e_k (e_-) \Psi^\e_k (e_+) $ in \eqref{e:expansion-PsiDBPsi}
yields a finite limit,
which is true due to the arguments 
 \eqref{e:C1e-C2e}, \eqref{e:C1e-C2e-rewrite}
and the graphic ``Ward identity'' \eqref{e:WardGraph}.
%
\end{proof}

\begin{remark}
Note that any  renormalization to the observable
$\bar\Phi^\eps (e_-) D_j^{A^\eps} \Phi^\eps (e)$
other than subtracting a constant $C^\e_{\bar\Phi D^A \Phi}$
would violate its gauge invariance; thanks to the above cancellation this is the only renormalization.
There is however nothing mysterious, because 
the imaginary part of this observable
is ``more or less'' the nonlinearity in the $A^\e$ equation of 
\eqref{e:APhi-long} and the ``Ward identity''
was also proved by expanding more or less the same observable,
see \eqref{e:Ward-observ}.
\end{remark}

Turning to the loop observables \eqref{e:def-loop-obs},
let $C$ be a $\CC^2$ loop with $\CC^2$ parametrization $\br :[0,1]\to \T^2$
such that 
\[
|\br(\si)-\br(\bar\si) | \asymp |\si-\bar\si|
\]
for any $\si,\bar\si\in [0,1]$ such that $|\si-\bar\si|$ is sufficiently small,
where $\asymp$
denotes being bounded from above and below up to multiplicative constants.
Let $C^\e$ be a regular approximation of $C$,
and $P^C_\e$ be the partition of $[0,1]$ as in Section~\ref{sec:observables}.
Since the solution $B^\e= K^\e\ast_\e \xi^\e + w^\e$ 
where $w^\e \to w$ in $\CC^\delta([0,T],\CC^{\alpha-\delta})$ for any $\alpha<1$ and $\delta>0$ sufficiently small,
we show convergence of $\e \sum_{e\in C^\e} (K^\e\ast_\e \xi^\e)(e) $ and  
$\e \sum_{e\in C^\e} w^\e(e) $ separately.
 
 \begin{lemma} \label{lem:loop-conv-w}
 For each $t\in (0,T]$,
\[
\lim_{\e\to 0} 
\e \sum_{e\in C^\e} w^\e(t,e) = \int_C w(t) \;.
\]
Moreover this limit holds as processes in the space $\CC([0,T],\R)$.
 \end{lemma}
 This seems to be a classical result of approximating line integrals, but we did not find a reference for exactly such a statement, 
 so we give a self-contained proof which also gives some hint on the next proof.
 \begin{proof}
 Assuming $\e$ sufficiently small, we divide the interval $[0,1]$ into subintervals 
 $0=\si_0<\si_1<\cdots<\si_M=1$ 
 such that each of $[\si_i,\si_{i+1}]$ is a union of  intervals in $P^C_\e$,
and  $\frac67 \sqrt\e \le |\si_{i+1}-\si_i|  \le \frac87 \sqrt\e$ (so $M=O(\e^{-\frac12})$):

\begin{figure}[h]
\begin{center}
\includegraphics[scale=0.4]{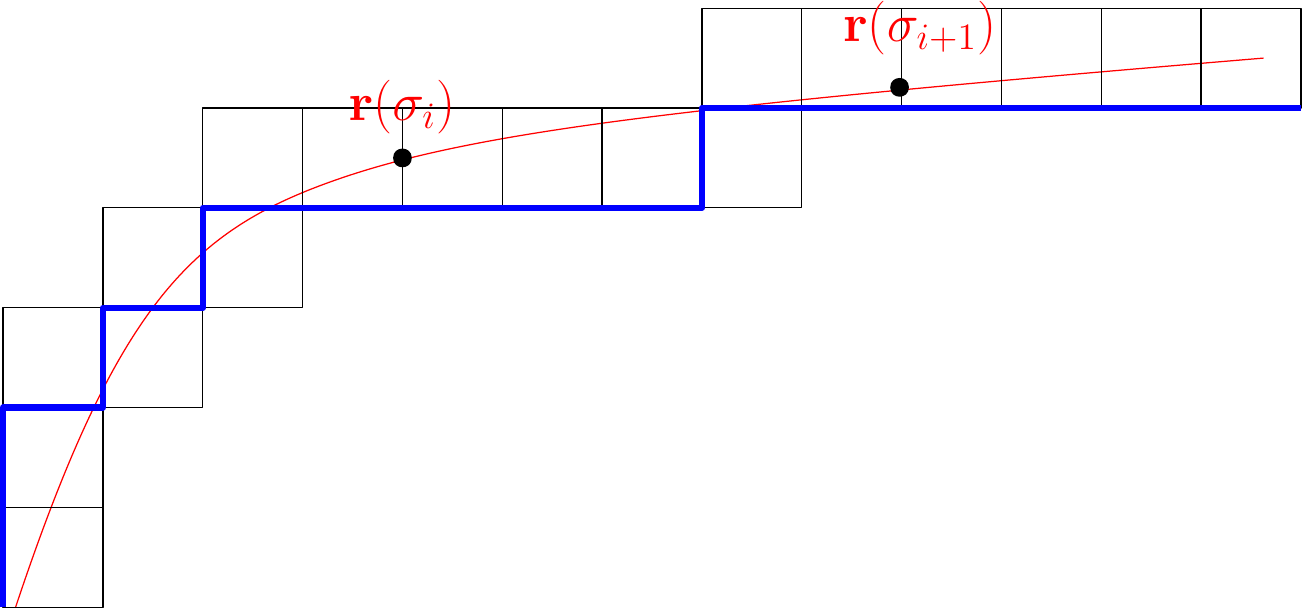}
\end{center}
\caption{A segment $[\si_i,\si_{i+1}]$.
There are $6$ distinct edges $e\in \CE_\e$
such that $e\in [\si_i,\si_{i+1}]$,
with $n_1^{[\si_i,\si_{i+1}]}=5$ and $n_2^{[\si_i,\si_{i+1}]}=1$.}
\label{fig:segment}
\end{figure}

Fixing a time $t$ and omitting it in our notation, one has
\begin{equs} [e:line-integral-cont]
\Big| & \int_C   w - \e \sum_{e\in C^\e} w^\e(e)  \Big|
\\
&\le \sum_{i=1}^M
\Big|
\int_{\si_i}^{\si_{i+1}} \!\! w(\br(\sigma))\cdot \dot\br(\si) \,d\si 
-\int_{\si_i}^{\si_{i+1}} \!\! w(\br(\si_i))\cdot \dot\br(\si_i) \, d \si 
\Big|
\\
&\qquad
+\Big| \int_{\si_i}^{\si_{i+1}} \!\! w(\br(\si_i))\cdot \dot\br(\si_i) \,d\si 
- \e \!\! \sum_{e\subset [\si_i,\si_{i+1}]} \!\!
	\Big( w_1(\br (\si_i)) \one_{e\in\CE^1_\e}  +w_2(\br (\si_i)) \one_{e\in\CE^2_\e} \Big)\Big|
\\
&\qquad
+ \Big|\e\!\! \sum_{e\subset [\si_i,\si_{i+1}]} \!\!
	\Big( w_1 (\br (\si_i)) \one_{e\in\CE^1_\e}  +w_2 (\br (\si_i)) \one_{e\in\CE^2_\e} \Big)
- \e \!\!\sum_{e\subset [\si_i,\si_{i+1}]} \!\! w^\e (e)\Big| \;,
\end{equs}
where $\dot\br$ is the derivative w.r.t. the parameter $\sigma$,
and $e\subset [\si_i,\si_{i+1}]$ means 
that  the edge $e\in\CE_\e$ corresponds to a subinterval 
of $[\si_i,\si_{i+1}]$.
For the first term on RHS of \eqref{e:line-integral-cont},
since $w\in \CC^{1-}$ and $\dot\br \in \CC^2$, 
one has 
\[
 |\dot\br(\si) -  \dot\br(\si_i) |
 \lesssim  |\si -  \si_i |
 \lesssim \e^{\frac12}
 \qquad
 \mbox{for } \si \in [\si_i, \si_{i+1}]
 \]
\[
 |w(\br(\sigma))  - w(\br(\si_i))  |
 \lesssim  |\br(\sigma)  - \br(\si_i) |^\alpha
 \lesssim \e^{\alpha/2}
 \]
  for any $\alpha<1$,
thus $\sum_{i=1}^M \int_{\si_i}^{\si_{i+1}} O(\e^{\alpha/2})\,d\si = O(\e^{-\frac12})O(\e^{\frac12})O(\e^{\alpha})\to 0$.

Consider the second term on RHS of \eqref{e:line-integral-cont}. One has that for $j\in\{1,2\}$
\[
\sum_{e\subset [\si_i,\si_{i+1}]} \!\!
	\Big( w_1(\br (\si_i)) \one_{e\in\CE^1_\e}+w_2(\br (\si_i)) \one_{e\in\CE^2_\e}   \Big)
= (w_1(\br (\si_i)) , w_2(\br (\si_i))) \cdot  (n_1^{[\si_i,\si_{i+1}]},n_2^{[\si_i,\si_{i+1}]})
\]
where $n_1^{[\si_i,\si_{i+1}]}$ (resp. $n_2^{[\si_i,\si_{i+1}]}$)
is the number of horizontal (resp. vertical) edges 
in the discrete curve $C^\e$ that correspond to subintervals of $[\si_i,\si_{i+1}]$, see Fig.~\ref{fig:segment}.

Note that $(n_1^{[\si_i,\si_{i+1}]},n_2^{[\si_i,\si_{i+1}]})$ 
 approximately gives  the tangent direction of $C$ at $\br(\si_i)$.
To be more precise, since $\br\in\CC^2$ one has
\[
\br(\si_{i+1})-\br(\si_i)
 =\dot\br (\si_{i}) (\si_{i+1}-\si_i)+O(\e) \;.
\]
On the other hand when $\e$ is sufficiently small
one has
\[
\br(\si_{i+1})-\br(\si_i) 
 =\e \,(n_1^{[\si_i,\si_{i+1}]},n_2^{[\si_i,\si_{i+1}]})+ O(\e) \;.
\]
These show that 
$\sum_{i=1}^M$ of the second term on RHS of \eqref{e:line-integral-cont} equals $\sum_{i=1}^M O(\eps) = O(\e^{-\frac12}) O(\e) $ which vanishes in the limit.

Finally since $w^\e$ converges to $w$  in $\CC^{1-}$,
and $|\br(\sigma_i) - e| \lesssim \e^{\frac12}$,
we immediately have that 
$\sum_{i=1}^M$ of the last term on RHS of \eqref{e:line-integral-cont} vanishes in the limit.

Finally, to prove continuity in time of this observable,
we can bound 
$\Big| \int_C  \delta^{s,t} w - \e \sum_{e\in C^\e} \delta^{s,t} w^\e(e)  \Big|$ in the same way as above,
replacing $w$ and $w^\e$ by $\delta^{s,t} w$ and $\delta^{s,t} w^\e$ respectively, and using  $w^\e \to w$ in $\CC^\delta([0,T],\CC^{\alpha-\delta})$.
\end{proof}

We now turn to convergence of
$\e \sum_{e\in C^\e}  K^\e\ast_\e \xi^\e (e)$ and
$e^{i\e\sum_{e\in C^\e}  K^\e\ast_\e \xi^\e (e)}$.
Fixing $t\in [0,T]$,  
let $\mathbf r: [0,1]\to \T^2$ be a parametrization of the loop $C$. Write $\mathbf r = (\mathbf r_1,\mathbf r_2)$.

We define $\int_C K\ast \xi$ to be the centered Gaussian random variable with variance
\[
\int_0^1 \!\! \int_0^1 \int_{\R\times \T^2}
K (t-s,\br(\si)-y) \, K (t-s,\br(\bar\si)-y)
\,\left( \dot{\br}(\si) \cdot  \dot{\br}(\bar\si) \right)  \, ds dy\,d\si d \bar\si
\]
which is finite since the integral over $(s,y)$ yields a function which behaves as 
$\log |\br(\si)-\br(\bar\si)  |$ for $|\si-\bar\si|$ small,
and $|\br(\si)-\br(\bar\si) | $ is bounded from above and below by  $|\si-\bar\si|$.

\begin{lemma} \label{lem:loop-conv-sing}
As $\eps\to 0$, $\e \sum_{e\in C^\e}  K^\e\ast_\e \xi^\e (e)$
converge in law to  $\int_C K\ast \xi$ in $\CC([0,T],\R)$.
\end{lemma}

\begin{proof}
Since with the fixed loop $C$ and fixed $t$, the sequence of random variables in consideration 
are centered Gaussians, to prove convergence for  each  fixed $t$ we only need to prove convergence of the variances.
Omitting $t$ in our notation, for $x\neq y\in \T^2$ let
\[
\hat K(x,y)=\hat K(x-y)
\eqdef
 \int_{\R\times \T^2}
K (t-s,x-z) \, K (t-s,y-z)
 \, ds dz
\]
and for $e,\bar e\in \CE_j^\e$ for $j\in \{1,2 \}$ (meaning $e,\bar e$ are {\it both} horizontal or {\it both} vertical)
\begin{equ}[e:def-hateK]
\hat K^\e(e,\bar e)=\hat K^\e(e-\bar e)
\eqdef
 \int_{\R} \e^2 \sum_{e'\in \CE_j^\e}
K^\e (t-s,e-e') \, K^\e (t-s,\bar e-e')
 \, ds \;.
\end{equ}
The variance of $\int_C K\ast \xi$ is then equal to
\begin{equ}[e:var-limit]
\int_0^1 \!\! \int_0^1
\hat K (\br(\si)-\br(\bar \si)) 
\,\left( \dot{\br}(\si) \cdot  \dot{\br}(\bar\si) \right)\,d\si d \bar\si
\end{equ}
and the variance of $\e \sum_{e\in C^\e}  K^\e\ast_\e \xi^\e (e)$ is equal to
\begin{equ}[e:var-eps]
\e^2 \sum_{e,\bar e \in C^\e}
\hat K^\e (e-\bar e) \,
(\one_{e,\bar e \in \CE^\e_1}+\one_{e,\bar e \in \CE^\e_2})  \;.
\end{equ}
As in the proof above we divide the interval $[0,1]$ into subintervals 
 $0=\si_0<\si_1<\cdots<\si_M=1$ 
 such that each of $[\si_i,\si_{i+1}]$ is a union of  intervals in $P^C_\e$,
so that   $ |\si_{i+1}-\si_i|  = O(\sqrt\e)$ and $M=O(\e^{-\frac12})$.
The difference between \eqref{e:var-limit} and \eqref{e:var-eps} is equal to $\mathbb S_1+\mathbb S_2+\mathbb S_3+\mathbb S_4$ where 
\begin{equs}
\mathbb S_1 
&\eqdef \sum_{|i-j|>1} \int_{\si_i}^{\si_{i+1}}\!\!\!\!\int_{\si_j}^{\si_{j+1}} 
\!\!\!\!
\Big(
\hat K(\br (\si)-\br(\bar\si)) \, \dot{\br}(\si) \cdot  \dot{\br}(\bar\si)
-
 \hat K(\br (\si_i)-\br(\si_j)) \, \dot{\br}(\si_i) \cdot  \dot{\br}(\si_j)
 \Big)
 d\si d \bar\si
\\
\mathbb S_2
&\eqdef
\sum_{|i-j|>1} \Big[ \int_{\si_i}^{\si_{i+1}}\!\!\!\int_{\si_j}^{\si_{j+1}} 
 \hat K(\br (\si_i)-\br(\si_j)) \, \dot{\br}(\si_i) \cdot  \dot{\br}(\si_j)
 \,d\si d \bar\si
 \\
 & \qquad \qquad  \qquad  \qquad \qquad\qquad
- \e^2 \!\!\!\!\! \sum_{e\subset [\si_i,\si_{i+1}] \atop \bar e\subset [\si_j,\si_{j+1}] }\!\!\!\!
\hat K (\br(\si_i)-\br (\si_j)) (\one_{e,\bar e \in \CE^\e_1}+\one_{e,\bar e \in \CE^\e_2}) \Big]
\\
\mathbb S_3
&\eqdef
\sum_{|i-j|>1} 
\e^2 \!\!\!\!\! \sum_{e\subset [\si_i,\si_{i+1}] \atop \bar e\subset [\si_j,\si_{j+1}] } \!\!\!\!
\Big( \hat K (\br(\si_i)-\br (\si_j))
-  \hat K^\e(e-\bar e)
\Big)
 (\one_{e,\bar e \in \CE^\e_1}+\one_{e,\bar e \in \CE^\e_2})
 \\
 \mathbb S_4 
&\eqdef \sum_{|i-j|\le 1}
\Big[
 \int_{\si_i}^{\si_{i+1}}\!\!\!\!\int_{\si_j}^{\si_{j+1}} 
\!\!\!
\hat K(\br (\si)-\br(\bar\si)) \, \dot{\br}(\si) \cdot  \dot{\br}(\bar\si)
 \,d\si d \bar\si
 \\
 &\qquad\qquad\qquad\qquad\qquad\qquad
 -
 \e^2 \!\!\!\!\! \sum_{e\subset [\si_i,\si_{i+1}] \atop \bar e\subset [\si_j,\si_{j+1}] } \!\!\!\!
 \hat K^\e(e-\bar e)
 (\one_{e,\bar e \in \CE^\e_1}+\one_{e,\bar e \in \CE^\e_2})
 \Big]
\end{equs}

Consider $ \mathbb S_1$.
By definition of $\hat K$ one has $|\hat K'(x)|\lesssim |x|^{-1}$. One also has $|\br(\si) -\br(\si_i) | \lesssim \e^{\frac12}$ and $|\br(\bar\si) -\br(\si_j) | \lesssim \e^{\frac12}$ as well as
\[
|\br(\si) -\br(\bar\si) |^{-1} \lesssim \frac{1}{|i-j| \sqrt{\e}}
\qquad \mbox{and} \qquad
|\br(\si_i) -\br(\si_j) |^{-1} \lesssim \frac{1}{|i-j| \sqrt{\e}}
\qquad \mbox{for } |i-j|>1 \;.
\] 
So by a Taylor remainder theorem
\begin{equ}[e:i-j-far-decay]
|\hat K(\br (\si)-\br(\bar\si))
-
 \hat K(\br (\si_i)-\br(\si_j))|
 \lesssim \frac{1}{|i-j|} \;.
\end{equ}
Noting  that $\br \in \CC^2$ one only needs to bound 
 $\mathbb S_1' $
which is the same as $\mathbb S_1$
except that
$\dot{\br}(\si) \cdot  \dot{\br}(\bar\si)$
is replaced by $\dot{\br}(\si_i) \cdot  \dot{\br}(\si_j)$.
Since $|\dot{\br}(\si) \cdot  \dot{\br}(\bar\si)-\dot{\br}(\si_i) \cdot  \dot{\br}(\si_j)|\lesssim \e^{\frac12}$
it is easy to see that $|\mathbb S_1-\mathbb S_1'|\to 0$.
 Invoking the decay \eqref{e:i-j-far-decay} one can  bound
  $\mathbb S_1' $ by
\[
|\mathbb S_1'| \lesssim \sum_{1\le i,j\le M\atop |i-j|>1}\frac{(\si_{i+1}-\si_{i}) (\si_{j+1}-\si_{j}) }{|i-j|}
\lesssim
\e^{\frac12}\e^{\frac12}  M\log M
\lesssim 
\e\cdot \e^{-\frac12}|\log\e| \to 0 \;.
\]

Consider $ \mathbb S_2$.
As in the  proof of Lemma~\ref{lem:loop-conv-w},
one has
\[
\dot \br(\si_i) (\si_{i+1}-\si_{i})=  
\e \,(n_1^{[\si_i,\si_{i+1}]},n_2^{[\si_i,\si_{i+1}]})+ O(\e)
\]
and the same holds with $i$ replaced by $j$;
 thus
\[
\int_{\si_i}^{\si_{i+1}}\!\!\!\int_{\si_j}^{\si_{j+1}} 
\, \dot{\br}(\si_i) \cdot  \dot{\br}(\si_j)
 \,d\si d \bar\si
=
\e^2 \,(n_1^{[\si_i,\si_{i+1}]},n_2^{[\si_i,\si_{i+1}]})\cdot \,(n_1^{[\si_j,\si_{j+1}]},n_2^{[\si_j,\si_{j+1}]})+O(\e^{\frac32})\;.
\]
Since $|\hat K(x)|\lesssim |x|^{-\frac18}$
and
$|\br(\si_i) -\br(\si_j) | \ge \frac{1}{\sqrt{\e}}$
for $|i-j|>1$, one has
\[
|\mathbb S_2| \lesssim \sum_{i,j} \e^{-\frac{1}{16}} \e^{\frac32}
=\e^{-1}  \e^{-\frac{1}{16}}  \e^{\frac32} \to 0 \;.
\]

Turning to $ \mathbb S_3$, as in the analysis of  $ \mathbb S_1$ one has
(recall from Section~\ref{sec:Lattice gauge theory} that in the notation $ \hat K (e-\bar e)$, $e,\bar e$ denote the mid-points of the edges)
\begin{equs}
\sum_{|i-j|>1} 
\e^2 \!\!\!\!\! \sum_{e\subset [\si_i,\si_{i+1}] \atop \bar e\subset [\si_j,\si_{j+1}] } \!\!\!\!
&
\Big|\hat K (\br(\si_i)-\br (\si_j))
- \hat K (e-\bar e)\Big|(\one_{e,\bar e \in \CE^\e_1}+\one_{e,\bar e \in \CE^\e_2})
\\
&\lesssim
 \e^2 \e^{-1} 
\sum_{|i-j|>1} \frac{|\br(\si_i) - e| +|\br(\si_j) - \bar e|}{|i-j| \sqrt{\e}}
\lesssim
\e\cdot \e^{-\frac12}|\log\e| \to 0 \;.
\end{equs}
To finish the estimate of $ \mathbb S_3$ we claim that
for $|e-\bar e|>\sqrt{\e}$,
\begin{equ} [e:Khat-Kepshat]
|\hat K(e-\bar e) -\hat K^\e(e-\bar e) | \lesssim \e^{\kappa} |e-\bar e|^{-\kappa}
\end{equ}
 for some small $\kappa>0$ so that
\begin{equs}
\sum_{|i-j|>1}  &
\e^2 \!\!\!\!\! \sum_{e\subset [\si_i,\si_{i+1}] \atop \bar e\subset [\si_j,\si_{j+1}] } \!\!\!\!
\Big|\hat K^\e (e-\bar e)
- \hat K (e-\bar e)\Big|
(\one_{e,\bar e \in \CE^\e_1}+\one_{e,\bar e \in \CE^\e_2})
\lesssim
\sum_{|i-j|>1} 
\e^2 \!\!\!\!\! \sum_{e\subset [\si_i,\si_{i+1}] \atop \bar e\subset [\si_j,\si_{j+1}] } \!\!\!\! \e^{\kappa} \e^{-2\kappa}
\\
& \lesssim 
\e^{-1}\e^2\e^{-1}\e^{\kappa} \e^{-2\kappa} \to 0 \;,
\end{equs}
where we used the fact that as long as $|i-j|>1$, $|e-\bar e|^{-\kappa} \lesssim \eps^{-2\kappa}$. To prove \eqref{e:Khat-Kepshat}, we assume $|e-\bar e|>\sqrt{\e}$ and note that for any constant $c>0$
\begin{equ}
\Big| \int_{t-c\e}^t  \int_{\T^2}
K (t-s,x-z) \, K (t-s,y-z) 
 \, ds dz \Big|
\lesssim \e^\kappa |x-y|^{-\kappa}\;,
\end{equ}
and the same bound holds in the discrete case with 
$K$ replaced by $K^\e$ and spatial integral replaced by spatial summation times $\e^2$. Indeed, $K$ coincides with $P$ in an $O(1)$ neighborhood of the origin, and one has 
$\int_{\T^2}
P (t-s,x-z) \, P (t-s,y-z) 
 \, dz = P(2t-2s, x-y)$; then it suffices to realize that 
 $f(x)\eqdef \int_0^{c\e} P(s,x) ds$ satisfies the equation $\Delta f (x)= P(c\e,x)$ because $P$ satisfies the heat equation, namely $f=P(c\e,\cdot)\ast (-\Delta)^{-1}$ from which the above bound follows.

Now to prove \eqref{e:Khat-Kepshat}, using Lemma~\ref{lem:OpDSingKer}, item 2,  
it suffices to prove that 
for  $t>c\e$ for some constant $c>0$ one has 
\[
|K^\e (z) -K (z) | \lesssim  \e^\kappa \|z\|_\s^{-\kappa}
\]
for some small $\kappa>0$.
Again it suffices to consider $|P^\e (z) -P (z) |$.
We have a local central limit theorem for continuous time random walk:
by \cite[Theorem~2.5.6]{MR2677157} 
\footnote{\cite[Theorem~2.5.6]{MR2677157}  is a statement in one spatial dimension, but it holds in higher dimensions too as remarked in the beginning of \cite[Section~2.5]{MR2677157}.}
translating into our notation, one has that as long as $\e |x| \le t/2$  (which follows from $t>c\e$ for some constant $c>0$ since $|x|=O(1)$)
\[
|P^\e (t,x) -P (t,x) | \lesssim  t^{-1} e^{-\frac{|x|^2}{4t}} \Big(\frac{\e}{\sqrt{t}} + \frac{\e |x|^3}{t^2}\Big) \;.
\]
When $|x| \le \sqrt{t}$, the dominant term in the parenthesis is $\frac{\e}{\sqrt{t}} $, and using $e^{-\frac{|x|^2}{4t}} \le 1$
we obtain the desired bound.
When $|x| > \sqrt{t}$,  the dominant term in the parenthesis is $ \frac{\e |x|^3}{t^2}$, using $e^{-\frac{|x|^2}{4t}} \lesssim t^3/|x|^6$ we again obtain the desired bound.

Finally, consider $ \mathbb S_4$.
Now $\sum_{|i-j|\le 1}$ is only a sum of $O(\e^{-\frac12})$ terms. We can estimate the two terms in $ \mathbb S_4$ separately without using the minus sign. For the first term
one can bound
\[
\Big| \hat K(\br (\si)-\br(\bar\si)) \, \dot{\br}(\si) \cdot  \dot{\br}(\bar\si)\Big| \lesssim |\si-\bar\si|^{-\frac18} \;.
\]
So
$\sum_{|i-j|\le 1} \int_{\si_i}^{\si_{i+1}}\!\!\int_{\si_j}^{\si_{j+1}} $ of the above quantity is bounded by 
$\e^{-\frac12} (\sqrt{\e})^{2-\frac18}$ which vanishes in the limit. Since $| \hat K^\e(e-\bar e)| \lesssim \e^{-\frac18}$
the same argument shows that the other term also vanishes in the limit.

The continuity in time essentially follows in the same way. To bound 
$ \E (\e \sum_{ C^\e}  K^\e\ast_\e \xi^\e (t)-\e \sum_{ C^\e}  K^\e\ast_\e \xi^\e (\bar t))^2$
it boils down to 
\eqref{e:var-eps} except that on the right hand side of \eqref{e:def-hateK}
one of the kernel $K^\e$ is replaced by $\delta^{t \bar t} K^\e (z)$ which is bounded by $|t-\bar t|^\kappa \|z\|_\s^{-\kappa}$  for some sufficiently small $\kappa>0$. Proceeding as the arguments above one gets a uniform bound by $|t-\bar t|^\kappa$; then by Gaussianity this yields higher moments bounds and the continuity in time then follows from Kolmogorov continuity theorem.
\end{proof}

\begin{proof}[of Theorem~\ref{theo:observables}]
The statements follow from Lemma~\ref{lem:FA-Wick} -- Lemma~\ref{lem:loop-conv-sing}.
The convergence of loop observables $\tilde{\mathcal O}_{C,\e}$ defined in \eqref{e:def-loop-obs}  
simply follows by continuous map theorem.
\end{proof}

\section{Discussions}\label{sec:discussions}

We conclude our paper with some informal discussions on possible extensions of our results.

\subsection{The case of three spatial dimensions}
The model under consideration is also defined in three spatial dimensions, with 
fields
$A=(A_1,A_2,A_3)= \sum_{j=1}^3 A_j \d x_j$
and $\Phi: \T^3 \to \C$:
\begin{equ} 
\mathcal H(A,\Phi) \eqdef
\frac12 \int_{\T^3} \Big( \sum_{j=1}^3 F_{A,j}(x)^2 + \sum_{j=1}^3 |D^A_j \Phi(x)|^2 \Big) \, \d^3 x  \;.
\end{equ}
Here, the curvature field $F_{A,j}$ is defined by
\begin{equs}
F_A  & = dA \\
&
= ( \partial_1 A_2  -\partial_2 A_1) \, \d x_1 \wedge \d x_2
+( \partial_2 A_3  -\partial_3 A_2) \, \d x_2 \wedge \d x_3
+( \partial_1 A_3  -\partial_3 A_1) \, \d x_1 \wedge \d x_3
\\
& =: F_{A,3}\, \d x_1 \wedge \d x_2
+  F_{A,1} \, \d x_2 \wedge \d x_3
+ F_{A,2} \, \d x_1 \wedge \d x_3
\end{equs}
and
$D^A_j \Phi$ is defined the same way except that $j\in\{1,2,3\}$.
The stochastic PDE is a system for $(A_1,A_2,A_3,\Phi)$ driven by  independent white noises $\xi_1,\xi_2,\xi_3$ as well as a complex
valued white noise $\zeta$. The system is again not parabolic, for instance, the equation for $A_1$ is of the form
\[
\partial_t A_1 = \partial_2^2 A_1 -\partial_1\partial_2 A_2
+ \partial_3^2 A_1 -\partial_1\partial_3 A_3+ \big( \cdots\big) + \xi_1
\]
where $\big( \cdots\big)$ denotes the nonlinearities.

Using the same gauge tuning trick, one obtains a parabolic system,
and the regularity of its solutions is expected to be $\alpha$ for $
\alpha < -\tfrac12$. Since its more singular than the situation in two spatial dimensions, 
to obtain a local solution theory  via lattice approximations, 
certain general tools should then be in place.
For instance, since there would be much more symbols in the regularity structure,
it  would require  a discrete version of the BPHZ theorem as in \cite{chandra2016analytic} in order to prove convergence of the models,
together with the ``blackbox'' theorems in \cite{bruned2016algebraic,bruned2017renormalising},
and presumably also a machinery to deal with renormalizations caused by symbols with ``epsilons'' as in \cite{KPZJeremy}.
The loop or string observables would be more difficult (if possible) to construct since one has to integrate a more singular gauge field along a curve.
%

\begin{remark}
It would also be interesting to construct long time solutions as done by  \cite{MourratWeberComesDown,moinat2018space,moinat2018local} for the dynamical $\Phi^4$ equation.
To obtain some bound uniform in time it would be helpful to have some damping terms,
for instance adding more gauge invariant terms 
$\mu |\Phi^2| - |\Phi|^4$ to \eqref{e:potential}.
Let's mention a very simple trick that the equation for $B$ in for instance \eqref{e:bar-eps-equ} can obtain a term  $-\lambda^2 c^2 B$
by considering the equations for $B$ and the shifted field $\Psi\mapsto \Psi+c$.
This would be reminiscent to the ``Higgs mechanism'' for the gauge field to acquire a mass
by shifting the scalar field to a point in the bottom of a mexican-hat potential.
\end{remark}

\subsection{Non-Abelian case} \label{sec:nonAbelian}
As hinted in Remark~\ref{rem:Helmholtz},
the gauge tuning method being exploited in the present paper is very likely to be  generalized into the non-Abelian case (whereas a gauge fixing with a linear decomposition by \eqref{e:Helmholtz} seems not generalizable). 
We would like to have some further discussion on this point (informally), and
also make a comparison between gauge fixing in  stochastic PDE formulation and that in functional integral formulation of quantum gauge theories. 

In general gauge theories,  with only gauge field $A$, one has $F_A \eqdef dA + A\wedge A$, where $A$ is a 1-form taking values in an Lie algebra - for instance $\mathfrak u(N)$ for some $N\ge 1$.  In the Abelian case, $N=1$, and $A\wedge A=0$ which is studied in this paper. 
In the functional integral approach, one is interested in the formal measure $e^{-\mathcal H(A)} \CD A$ where $\mathcal H(A)=
\frac12 \int_{\T^d} \tr(F\wedge *F) $
is invariant under the gauge transformation $A\to g^*A \eqdef g^{-1}Ag+g^{-1}dg$ for $g: \T^d\to U(N)$.
Note that in Abelian case $g=e^{i f}$ and this is precisely the first transformation in \eqref{e:gauge-trans}.
To make the ``measure'' normalizable, a Fadeev-Popov trick  based on the  identity
$
1= \int \delta (g(x)) \,\mbox{det}(\frac{\partial g}{\partial x})\, dx
$
 is often used.
In Abelian lattice gauge theory, writing $A_f = A+df$ so that $d^* A_f = d^* A + \Delta f$, one then writes the formal partition function as 
\[
\int e^{-\mathcal H(A)} \CD A =
\int e^{-\mathcal H(A)} \delta(d^* A_f - \omega)\, 
\mbox{det}(\frac{\delta(d^* A_f-\omega)}{\delta f}) \, 
\CD f\CD A \;.
\]
The Jacobian factor $\mbox{det}(\frac{\delta(d^* A_f-\omega)}{\delta f})=\mbox{det}(\Delta)$ can be factorized out, and using gauge invariance one can also replace $A_f$ by $A$ and factor out the ``infinite integral'' $\int \CD f$.
If $\omega=0$ this simply amounts to imposing the divergence free condition as mentioned in Remark~\ref{rem:Helmholtz}. 
(The only reason one usually introduces the field $\omega$ is that upon integrating it against a Gaussian weight one obtains a factor $e^{-\frac{(d^* A)^2}{2}}$ that is slightly more convenient to analyze.)

The  Fadeev-Popov trick gets much more complicated in non-Abelian case,
because the determinant will generally depend on $A$ and thus does not factor out.
This determinant however can be expressed by an integral over anti-commuting variables, which are called ghost fields. One then studies the model for gauge field $A$ coupled with ghost fields. 
Furthermore, fixing the gauge {\it globally} is not always possible in non-Abelian theories, due to topological obstructions, a phenomena usually referred to as Gribov ambiguity.

The   ghost fields would not show up in the stochastic PDE approach with gauge fixed by DeTurck trick; this seems to be already observed by physicists \cite{Zwanziger1981,sadun1987continuum,BernHalpernSadunTaubes1987}. In fact, for the corresponding stochastic  quantization equation
\begin{equ} [e:YM]
\frac{\partial A}{\partial t} + d_{A}^*F_{A} =  \xi 
\end{equ}
with a Lie algebra valued $d$-component space-time white noise $\xi$,
where $d_A$ is the gauge covariant derivative,
one can check by straightforward computation that $B\eqdef g^* A$ satisifes
$\frac{\partial B}{\partial t}
= g^{-1}\frac{\partial A}{\partial t}g
+ d_{B}\left(g^{-1} \frac{\partial g}{\partial t}\right)$,
so by solving $g$ from the ODE
$
g^{-1}\frac{\partial g}{\partial t} = -d_B^*B
$ and invoking gauge invariance of $d_{A}^*F_{A}$
one obtains a parabolic equation $\frac{\partial B}{\partial t}
= -d_{B}^*F_{B}
- d_{B}d_B^*B + g^{-1}\xi g$ 
where $g^{-1}\xi g \stackrel{law}{=} \xi$.
This is well-know when $\xi=0$, see  \cite[Section~6.3]{Donaldson};
and in the presence of the noise $\xi$, 
the aforementioned physics literature simply put a term $- d_{A}d_A^*A$ into the equation \eqref{e:YM} and call this a gauge-fixing term.

To obtain a local solution theory to \eqref{e:YM} via lattice approximations, one again needs some general tools as discussed in the three dimensional Abelian case.

\bibliographystyle{./Martin}
\bibliography{./refs}

\message{^^Jtimer: \the\numexpr\the\pdfelapsedtime*1000/65536\relax}

\end{document}